\documentclass[reqno, a4paper, 12pt]{amsart}
\usepackage{amssymb, url, color, mathrsfs}
\usepackage[colorlinks=true, bookmarks=true, pdfstartview=FitH, pagebackref=true, linktocpage=true]{hyperref}

\usepackage{pgf,tikz,pgfplots}
\usetikzlibrary{hobby}
\usepgfplotslibrary{fillbetween}

\usepackage[a4paper, vmargin={30mm,30mm},hmargin={20mm,20mm}, top=3.25cm, bottom=3.5cm, left=2.5cm, right=2.5cm]{geometry}

\usepackage[short,nodayofweek]{datetime}
\usepackage[pagewise,running,mathlines,displaymath, switch]{lineno}
\usepackage{mathabx}
\usepackage{times}
\usepackage{indentfirst, tabularx}

\usepackage[T5]{fontenc}
\usepackage[utf8]{inputenc}

\theoremstyle{plain}
\numberwithin{equation}{section}
\newtheorem{theorem}{Theorem}[section] 
\newtheorem{lemma}{Lemma}[section] 

\newtheorem{proposition}{Proposition}[section]

\definecolor{brown}{rgb}{0.5,0,0}
\definecolor{backgroundcolor}{rgb}{0.98, 0.92, 0.73}

\def\N{\mathbb N}
\def\R{\mathbf R} 
\def\C{\mathbb C} 
\def\cN{\mathcal N} 
\def\bZ{\mathbb Z} 

\def\sN{\mathsf N} 
\def\sP{\mathsf P}

\def\fT{\mathbf T} 
\def\bT{\mathbb T} 

\def\bS{\mathbb S} 

\def\bB{\mathscr B} 
\def\fB{\mathbf B} 
\def\fL{\mathbf L} 
\def\cL{\mathscr L} 
\def\cP{\mathsf p} 

\def\sW{\mathscr W}

\def\cE{\mathcal E} 
\def\cM{\mathcal M} 
\def\cD{\mathcal D}

\def\cA{\mathcal A} 
\def\cQ{\mathcal Q} 

\def\Csf{\mathsf c} 
\def\vk{\mathbf k} 
\def\cJ{\mathcal J} 
\def\ve{\varepsilon}

\newif\ifprint
\ifprint
\else
\fi

\author[Tiến-Tài Nguyễn]{Tiến-Tài Nguyễn }
\address[Tiến-Tài Nguyễn]{Laboratoire Analyse G\'eom\'etrie et Applications, Universit\'e Sorbonne Paris Nord,  93430 - Villetaneuse, France}
\email{\href{mailto: Tiến-Tài Nguyễn <tientai.nguyen@math.univ-paris13.fr>}{tientai.nguyen@math.univ-paris13.fr}}
\address[Tiến-Tài Nguyễn]{University of Science, Vietnam National University, Hanoi, Vietnam}
\email{\href{mailto: Tiến-Tài Nguyễn <nttai.hus@vnu.edu.vn>}{nttai.hus@vnu.edu.vn}}

\begin{document}
\allowdisplaybreaks

\setpagewiselinenumbers
\setlength\linenumbersep{100pt}

\title[Rayleigh-Taylor instability]{Nonlinear Rayleigh-Taylor instability of the viscous surface wave in an infinitely deep ocean}

\begin{abstract}
In this paper, we consider an incompressible viscous fluid in an infinitely deep ocean, being bounded above by a free moving boundary.  The governing equations are the  gravity-driven incompressible Navier-Stokes equations with variable density and no surface tension is taken into account on the free surface.  After using the Lagrangian transformation, we write the main equations in a perturbed form in a fixed domain. In the first part, we  describe a spectral analysis of the linearized equations around a hydrostatic equilibrium $(\rho_0(x_3), 0, P_0(x_3))$ for a smooth increasing density profile $\rho_0$. Precisely, we prove that there exist \textit{infinitely many} normal modes to the linearized equations by following  the operator method initiated by Lafitte and Nguyễn  \cite{LN20}.  In the second part, we study the nonlinear Rayleigh-Taylor instability around the above profile by constructing a \textit{wide class} of initial data for  the nonlinear perturbation problem departing from the equilibrium, based on the finding of infinitely many normal modes. Our nonlinear result follows the previous framework of  Guo and Strauss \cite{GS95} and also  of  Grenier \cite{Gre00} with a refinement.
\end{abstract}

\date{\bf \today \; at \, \currenttime}



\maketitle
 \tableofcontents

\section*{Acknowledgments}

The author wishes to thank Prof. David Lannes for his suggestion on this problem and Prof. Jeffrey Rauch for his advice on the nonlinear study. The author is deeply grateful to the supervision of Prof.  Olivier Lafitte and Prof. Jean-Marc Delort during PhD at Universit\'e Sorbonne Paris Nord.   This paper was done while the author was visiting Centre de Recherches Math\'ematiques (CRM)-Universit\'e de Montr\'eal (UdeM). The author gratefully acknowledges CRM-UdeM for its hospitality. This work is supported by a grant from R\'egion \^Ile-de-France.

 \section{Introduction}

This paper is devoted to the study of nonlinear Rayleigh--Taylor (RT) instability for an incompressible nonhomogeneous viscous fluid separated from air by an interface. More precisely, the fluid lies on a vertical gravity field, below some smooth interface separating it from air (see picture below).
\begin{center}
\begin{tikzpicture}[scale=1.0]
		\draw[line width = 1pt] (0,0) -- (0, 5);
		\draw[line width = 1pt] (8,0) -- (8, 5);
		\draw[thick, fill=gray, opacity=0.5] (0,0) to (0,3) to[curve through =
		{ (0.5, 3.2) (3,2.5) (7.5,4.2) }] (8,4) to (8,0);
		\draw (4,4.5) node {\textit{air}};
		\draw (3.2,2.25) node {$x_3=\eta(t,x_1,x_2)$};
		\draw (4,1.0) node {\textit{fluid}};
\end{tikzpicture}
\end{center}

We set the problem in a periodic framework. If $\bT=\R/\bZ$ and for some $L>0$, we set $\fT^2=(2\pi L\bT)^2$. We consider a three dimensional fluid of infinite depth, $2\pi L$-periodic relatively to the horizontal variables $(x_1,x_2)$, the vertical variable being denoted by $x_3$. At time $t$, the fluid occupies a domain 
\[
\Omega(t)= \{ x=(x_1,x_2,x_3) \in  \fT^2 \times \R, x_3 <\eta(t,x_1,x_2)\},
\]
where $\eta$ is a smooth enough function of its arguments, periodic in $(x_1,x_2)$. We denote by $\Gamma(t)$ the boundary of $\Omega(t)$, that is  $\Gamma(t) = \{x_3=\eta(t,x_1,x_2)\}$.
The fluid is submitted to vertical gravity field $-g e_3$ with $g>0$ and has variable density denoted by $\tilde \rho(t,x)$. 

The evolution of the velocity $\tilde u$ in the fluid domain, of the density $\tilde \rho$ and of the interface $\Gamma(t)$ is driven by the equation of fluid mechanics: we assume that $(\tilde \rho,\tilde u)$ satisfies the incompressible Navier-Stokes equation in $\Omega(t)$, with pressure $\tilde p$, that the interface moves according to the velocity of the fluid and that a convenient boundary holds at the interface $\Gamma(t)$ (without surface tension).  We get thus a system 
\begin{equation}\label{EqNS}
\begin{cases}
\partial_t \tilde \rho +\text{div}(\tilde\rho \tilde u) =0 \quad &\text{in } \Omega(t),\\
\partial_t(\tilde\rho \tilde u) + \text{div}(\tilde\rho \tilde u\otimes \tilde u) +\nabla\tilde p =\mu \Delta \tilde u - g\tilde\rho  e_3 \quad &\text{in } \Omega(t),\\
\text{div}\tilde u=0 \quad &\text{in } \Omega(t),\\
(\tilde p \text{Id} - \mu \bS \tilde u ) n= p_{atm} n  \quad&\text{on } \Gamma(t), \\
\partial_t\eta = \tilde u_3 - \tilde u_1 \partial_1\eta - \tilde u_2 \partial_2\eta \quad&\text{on } \Gamma(t).
\end{cases}
\end{equation}
 Here $\mu>0$ is the viscosity coefficient and  $\bS\tilde u= \nabla \tilde u+ \nabla \tilde u^T$ is the stress tensor. The outward normal vector to the boundary $\Gamma(t)$, $ n$ is given by 
\begin{equation}
n = \frac{(-\partial_1\eta , -\partial_2\eta, 1)^T}{\sqrt{1+|\partial_1\eta|^2+|\partial_2\eta|^2}}.
\end{equation} 
 The given constant $p_{atm}$ is the atmospheric pressure.  For a more physical description of the equations \eqref{EqNS} and the boundary conditions in \eqref{EqNS}, we refer to \cite[Sect. 1.8]{Lan13}.  

The movement of $\Omega(t)$ and $\Gamma(t)$ raises numerous mathematical difficulties. Hence, let $\R_- = (-\infty,0)$, we use the unknown  function $\eta$ and the Lagrangian coordinate transformation  to  transform the free boundary problem \eqref{EqNS} into the equivalent problem, \eqref{EqNS_Lagrangian} in a fixed domain $\Omega = \fT^2 \times \R_-$, of which the fixed upper boundary is $\Gamma = \fT^2\times \{x_3=0\}$. Let  $'=d/$, we consider two smooth functions $\rho_0$ and $P_0$ depending only on $x_3$ such that $P_0'=-g\rho_0$.  Hence, we have that $(\rho_0(x_3), 0, P_0(x_3))$ is a steady state solution of Eq. \eqref{EqNS_Lagrangian}.  We shall assume that the density $\rho_0$ is an increasing function of $x_3\in \R_-$. In that way, we have heavier fluid above lighter one and we are thus in the situation where Rayleigh--Taylor instability occurs.

Our goal is to show  nonlinear RT instability for the nonlinear equations \eqref{EqNS_Lagrangian} around the equilibrium state $(\rho_0(x_3), 0, P_0(x_3))$, assuming that \begin{equation}\label{AssumeRho1}
C_0^\infty(\R_-)  \ni \rho_0' \geq 0 \quad\text{such that } \text{supp}\rho_0'=[-a,0], \quad\text{with }a>0,
\end{equation}
and denoting that
\begin{equation}\label{AssumeRho2}
0< \rho_- = \rho_0(x_3) \text{ for all }x_3 \leq -a,\quad \rho_0(0)=\rho_+.
\end{equation} 
That means, we construct initial data of small size $\delta$, giving rise to a solution defined up to some time $T^\delta$, and which at that time has $L^2$-norm bounded from below by a fixed constant (independent of $\delta$). We refer to Section \ref{SectMainResults} for the precise statement of our main theorems and describe here our strategy of the proof.

The first step in our proof is to construct a solution of the linearization at this stationary solution of the nonlinear equations satisfied by $(\rho, u,p, \theta)$. We want this solution of the linearized equations \eqref{EqLinearized} to have growing normal modes and we look for it schematically as $U(t,x)=e^{\lambda t}V(x)$, where $\lambda$ is positive. The profile $V$ is taken as an oscillatory function of the horizontal variables $(x_1,x_2)$, with mode $\vk=(k_1,k_2)$, the $x_3$-dependence being given in terms of unknown functions of $(\vk, x_3)$. Then $U$ is a solution of the linearized equations if some function $x_3 \to \phi(\vk,x_3)$ (from which $U$ maybe reconstructed) solves a fourth order ODE on the half line $\R_-= (-\infty, 0)$, depending on $\vk$ and $\lambda$. Our first theorem asserts that one may find \textit{infinitely many} solutions to that ODE and thus get \textit{infinitely many} normal mode solutions of the linearized equations. The line of investigation is the same as in \cite{LN20}, where the case of a viscous nonhomogeneous  incompressible fluid in the whole space has been treated. 

The second part of the paper is to devoted to the proof of nonlinear instability. The spectral analysis allows us to study  the fully nonlinear perturbation equations \eqref{EqPertur}.  To this purpose, we follow the same procedure as in \cite{Tai22} for RT problem in an infinite strip with Navier-slip boundary conditions, 
\begin{enumerate}
\item[Step 1.] establish some \textit{a priori} energy estimates to the nonlinear equations,

\item[Step 2.] formulate a linear combination of normal modes to the linearized equations  \eqref{EqLinearized} to set its value at initial time $t=0$ of size $0<\delta \ll 1$ as an initial datum to the nonlinear perturbation equations, 

\item[Step 3.] obtain the difference between the local exact solution and the approximate solution in Step 2 and exploit some energy estimates for the difference,

\item[Step 4.] deduce the bound in time of the difference functions and prove the nonlinear instability.
\end{enumerate}
Our nonlinear study is  inspired by the abstract frameworks of Guo and Strauss \cite{GS95} and  of Grenier \cite{Gre00}. In the above frameworks,   only  the maximal normal mode  was used in Step 2 to approximate the nonlinear equations. Let us emphasize that, our nonlinear results show that  a \textit{wide class} of initial data (related to a linear combination of normal modes) to the nonlinear problem departing from the equilibrium is formulated  in Step 2 and it gives rise to the nonlinear instability.

The Rayleigh--Taylor instability, studied first by Lord Rayleigh in \cite{Str83} and then Taylor \cite{Tay50} is well known as a gravity-driven instability in two semi-infinite inviscid and incompressible  fluids when the heavy one is on top of the light one. It  has attracted much attention due to both its physical and mathematical importance. Two applications  worth mentioning are implosion of inertial confinement fusion capsules \cite{Lin98} and core-collapse of supernovae  \cite{RDTA00}. For a detailed physical comprehension of the linear RT instability, we refer to three survey papers  \cite{Kul91, Zho17_1, Zho17_2}. Mathematically speaking, the nonlinear study of classical RT instability is proven by Desjardins and Grenier \cite{DG06}. For the inviscid and incompressible fluid with a smooth density profile, the classical  RT instability was investigated by Lafitte \cite{Laf01}, by Guo and Hwang \cite{GH03} and by  Helffer and Lafitte \cite{HL03}.

Concerning the viscous linear RT instability, one of the first studies can be seen in the book of Chandrasekhar \cite[Chapter X]{Cha61}. He considers two uniform viscous fluid separated by a horizontal boundary and  generalize the classical result of Rayleigh and Taylor.  We refer the readers to  mathematical viscous linear/nonlinear RT studies for  two (in-)compressible channel flows in \cite{GT11}, \cite{WT12} and \cite{JTW16}.  For the incompressible fluid in the whole space $\R^3$, with a smooth density profile, we mention  the results of Jiang et. al   \cite{JJN13} and of Lafitte and the author \cite{LN20}, respectively.  Recently, the author studies the  nonlinear viscous RT instability in a slab domain $2\pi L\bT \times (-1,1)$ $(L>0)$, with Navier-slip boundary conditions. This paper continues the nonlinear RT instability \cite{Tai22} in a rigid domain to that one in a moving domain.

We organize this paper as follows. In Section \ref{SectMainResults}, from the formulation in Eulerian coordinates of the governing equations \eqref{EqNS}, we derive the formulation in Lagrangian coordinates, see \eqref{EqNS_Lagrangian}. We introduce our two theorems, Theorem \ref{ThmLinear} describing the spectral analysis of the linearized equations and Theorem \ref{ThmNonlinear} proving the nonlinear instability. Section \ref{SectLinear} is devoted to the proof of Theorem \ref{ThmLinear}. In Section \ref{SectAprioriEnergy}, we construct the \textit{a priori} energy estimates to the nonlinear  equations. In the last part, Section \ref{SectNonlinear}, we prove Theorem \ref{ThmNonlinear}.

We  employ the Einstein convention of summing over repeated indices. Throughout the paper $C > 0$ will denote  universal constants that  depend on the physical parameters of the problem, $\mu,  g, k, a$ and  $\rho_\pm$.  Such constants are allowed to change from line to line.  We will employ the notation $a\lesssim b$ $(a\gtrsim b)$ to mean that $a\leq Cb$ $(a\geq Cb)$ for a universal constant $C > 0$.

\section{The governing equations and main results}\label{SectMainResults}

\subsection{Formulation in Lagrangian coordinates}

Instead of using Eulerian coordinates,  we will  transform the free boundary problem \eqref{EqNS} into a new  problem in a fixed domain by following Beale  \cite{Bea81}. We define (see Appendix \ref{AppPoisson})
\begin{equation}
\theta := \text{Poisson extension of } \eta \text{ into } \fT^2 \times \{x_3\leq 0\}
\end{equation}
and  the following coordinate transformation:
\begin{equation}\label{ThetaTransform}
\Omega \ni x =(x_1,x_2,x_3) \mapsto (x_1,x_2, x_3+ \theta(t,x_1,x_2,x_3)) =: \Theta(t,x)=(y_1,y_2,y_3) \in \Omega(t).
\end{equation} 
If $\eta$ is sufficiently small (in an appropriate Sobolev norm), then the mapping $\Theta$ is a diffeomorphism.  Following \cite{Bea81} again, we denote by
\begin{equation}\label{NotationABKJ}
A= \partial_1\theta, \quad B=\partial_2\theta, \quad J= 1+\partial_3\theta,  \quad K= J^{-1}.
\end{equation}
From the definition of $\Theta$ \eqref{ThetaTransform}, we have
\begin{equation}\label{MatrixA}
\nabla \Theta = \begin{pmatrix} 1 &0 & 0\\ 0 & 1 & 0 \\ \partial_1\theta & \partial_2\theta & 1+\partial_3\theta \end{pmatrix}, \quad  \cA := ((\nabla \Theta)^{-1})^T = \begin{pmatrix} 1 &0 & -AK \\ 0 & 1 & -BK  \\ 0 & 0 & K \end{pmatrix}.
\end{equation}
We write the differential operators $\nabla_{\cA}, \text{div}_{\cA}, \Delta_{\cA}$ with their actions given by
\[
(\nabla_{\cA} f)_i := \sum_{j=1}^3 \cA_{ij}\partial_jf, \quad \text{div}_{\cA} X := \sum_{1\leq i,j\leq 3}\cA_{ij}\partial_jX_i, \quad \Delta_{\cA} f=\text{div}_{\cA}\nabla_{\cA} f.
\]
We write  $\cN := (-\partial_1\eta, -\partial_2\eta, 1)^T$ for the non-unit normal vector to $\Gamma(t)$, and we also write the stress tensor $\bS_{\cA} u $ as $(\bS_{\cA} u)_{ij} =  \cA_{ik}\partial_ku_j + \cA_{jk}\partial_ku_i$.

We now define the density $ \rho$, the velocity $ u$ and the pressure $ p$ on the domain $\Omega$ by the composition 
\[
 (\rho,  u, p)(t,x) = (\tilde\rho, \tilde u, \tilde p)(t,\Theta(t,x))
 \]
and  transform \eqref{EqNS} into the following system in the new coordinates
\begin{equation}\label{EqNS_Lagrangian}
\begin{cases}
\partial_t\rho - K \partial_t  \theta\partial_3\rho + \text{div}_{\cA}(\rho  u)=0 \quad&\text{in } \Omega, \\
\rho( \partial_t u - K\partial_t \theta \partial_3u+u\cdot \nabla_{\cA}u)+\nabla_{\cA} p - \mu \text{div}_{\cA}\bS_{\cA}u= -g\rho  e_3 \quad&\text{in } \Omega, \\
\text{div}_{\cA}u=0 \quad&\text{in } \Omega, \\
\partial_t\eta = u\cdot \cN \quad&\text{on } \Gamma, \\
(p\text{Id}- \mu \bS_{\cA}u)\cN =p_{atm}\cN \quad&\text{on } \Gamma.
\end{cases}
\end{equation}
\subsection{Equilibrium state and  and main results}

We now rewrite \eqref{EqNS_Lagrangian} in a perturbed formulation around the steady-state solution 
\[
(\rho(t,x),u(t,x), p(t,x), \eta(t,x_h))=(\rho_0(x_3),  0, P_0(x_3),0 ),
\]
satisfying that  $ P_0'=-g\rho_0$ and adding the condition $P_0(0)=p_{atm}$.  We define a special density and pressure perturbation by
\begin{equation}\label{PerturTerms}
\zeta = \rho-\rho_0 - \rho_0'\theta, \quad q= p-P_0 +g\rho_0\theta.
\end{equation}
We still call  the perturbations of the velocity and of the characterization of surface as $(u, \eta)$ respectively.
The equations for the perturbation $U=(\zeta,u,q,\eta)$ write
\begin{equation}\label{EqPertur}
\begin{cases}
\partial_t \zeta +\rho_0' u_3= \cQ^1(U) \quad&\text{in } \Omega,\\
\rho_0\partial_t  u+ \nabla q -\mu \Delta  u +g\zeta  e_3= \cQ^2(U) \quad&\text{in }\Omega,\\
\text{div} u= \cQ^3(U) \quad&\text{in } \Omega,\\
\partial_t \eta - u_3 =\cQ^4(U) \quad&\text{on } \Gamma,\\
((q-g\rho_+ \eta)\text{Id}-\mu \bS u)  e_3= \cQ^5(U)\quad&\text{on } \Gamma.
\end{cases}
\end{equation} 
We refer to Appendix \ref{AppNonlinearTerms} for the precise expression of $\cQ^i(1\leq i\leq 5)$ as a polynomial of $U, A,B,J,K$.  
This is the system considered from now on.

The linearized equations are  \begin{equation}\label{EqLinearized}
\begin{cases}
\partial_t\zeta+ \rho_0' u_3=0 \quad&\text{in }\Omega,\\
\rho_0\partial_t u+ \nabla q - \mu\Delta u+g\zeta e_3=0,\quad&\text{in }\Omega,\\
\text{div}u=0 \quad&\text{in }\Omega,\\
\partial_t \eta= u_3 \quad&\text{on }\Gamma,\\
((q-g\rho_+\eta)\text{Id}-\mu \bS u) e_3 =0 \quad&\text{on }\Gamma.
\end{cases}
\end{equation}
As in  \cite[Chapter XI]{Cha61}, we seek  normal modes  $U(t,x)=e^{\lambda t}V(x)$ of \eqref{EqLinearized}, which are
\begin{equation}\label{GrowingMode}
(\zeta, u, q)(t,x)= e^{\lambda t}(\omega, v, r)(x), \quad \eta(t,x_h)=e^{\lambda t} \nu(x_h).
\end{equation}
We deduce the following system on $(\omega,v,r,\nu)$, 
\begin{equation}
\begin{cases}
\lambda \omega+\rho_0'v_3=0 \quad&\text{in }\Omega,\\
\lambda\rho_0 v+ \nabla r -\mu \Delta v+g\omega  e_3 = 0\quad&\text{in }\Omega,\\
\text{div}v =0\quad&\text{in }\Omega,\\
\lambda \nu= v_3 \quad&\text{on } \Gamma,\\
((r-g\rho_+\nu)\text{Id} - \mu(\nabla v+\nabla v^T)) e_3 =0 \quad&\text{on } \Gamma. 
\end{cases}
\end{equation}
That implies 
\begin{equation}\label{RelationOmegaNu}
\omega = -\frac1{\lambda}\rho_0'v_3, \quad \nu= \frac1{\lambda}v_3|_{\Gamma}
\end{equation}
and
\begin{equation}\label{SystemV_R}
\begin{cases}
\lambda^2 \rho_0 v+\lambda \nabla r -\lambda\mu \Delta v- g\rho_0' v_3  e_3=0 \quad&\text{in }\Omega,\\
\text{div} v=0 \quad&\text{in }\Omega,\\
((\lambda r-g\rho_+v_3)\text{Id} - \lambda \mu(\nabla v+\nabla v^T)) e_3= 0 \quad&\text{on } \Gamma.
\end{cases}
\end{equation}
Let $\vk =(k_1,k_2) \in (L^{-1}\bZ)^2 \setminus \{0\}$, we  further  assume that
\begin{equation}\label{SeparationOfVR}
\begin{cases}
v_1(x) = \sin(k_1x_1+k_2x_2) \psi(\vk,x_3),\\
v_2(x) = \sin(k_1x_1+k_2x_2) \varphi(\vk, x_3),\\
v_3(x) =\cos(k_1x_1+k_2x_2) \phi(\vk, x_3),\\
r(x) = \cos(k_1x_1+k_2x_2) \pi(\vk, x_3).
\end{cases}
\end{equation}
Denote  by $k=|\vk|=\sqrt{k_1^2+k_2^2}$, we deduce from \eqref{SystemV_R} that 
\begin{equation}\label{SystFunctionsX_3}
\begin{cases}
\lambda^2\rho_0 \psi- \lambda k_1 \pi +\lambda\mu(k^2 \psi-\psi'')=0 \quad&\text{in } \R_-,\\
\lambda^2\rho_0 \varphi- \lambda k_2\pi + \lambda \mu(k^2\varphi-\varphi'') =0 \quad&\text{in } \R_-,\\
\lambda^2\rho_0\phi + \lambda\pi' +\lambda\mu(k^2\phi-\phi'')=g\rho_0'\phi \quad&\text{in } \R_-,\\
k_1\psi+k_2\varphi+\phi'=0 \quad&\text{in } \R_-.
\end{cases}
\end{equation}
At $x_3=0$, we have the boundary conditions 
\begin{equation}\label{BoundFunctionsX_3}
\begin{cases}
\mu(k_1\phi(0)-\psi'(0))=0, \\
\mu(k_2\phi(0)- \varphi'(0))=0,\\
\lambda\pi(0) -g\rho_+\phi(0) -2\lambda\mu \phi'(0)= 0.
\end{cases}
\end{equation}
We also need the decaying condition at $-\infty$,
\begin{equation}\label{FunctionsAtInfty}
\lim_{x_3\to -\infty}(\psi,\varphi,\phi,\pi)(x_3)=0.
\end{equation}

Note that, due to $\eqref{SystFunctionsX_3}_{1,2,4}$
\begin{equation}\label{RelationPi_Phi}
\pi =- \frac1{k^2}(\lambda \rho_0\phi' +\mu(k^2\phi' -\phi''')) \quad\text{in } \R_-.
\end{equation}
Hence, from \eqref{RelationPi_Phi} and $\eqref{SystFunctionsX_3}_3$, we get a fourth-order ODE for $\phi$, 
\begin{equation}\label{4thOrderEqPhi}
\lambda^2(k^2 \rho_0 \phi -(\rho_0 \phi')') +\lambda \mu(\phi^{(4)}-2k^2\phi''+k^4\phi) =gk^2\rho_0' \phi,
\end{equation}
The boundary conditions at $x_3=0$ deduced from $\eqref{SystFunctionsX_3}_4$, \eqref{BoundFunctionsX_3} and \eqref{RelationPi_Phi} are 
\begin{equation}\label{RightBoundary}
\begin{cases}
\mu(k^2\phi(0)+\phi''(0))=0, \\
 -\lambda\mu \phi'''(0)+ (3\lambda\mu k^2+ \lambda^2\rho_+)\phi'(0) +gk^2\rho_+\phi(0)=0
\end{cases}
\end{equation}
and from \eqref{FunctionsAtInfty}, we have that $\phi$ decays at $-\infty$, i.e.
\begin{equation}\label{PhiAt-Infty}
\lim_{x_3\to -\infty}\phi(x_3)=0.
\end{equation}
The finding of normal modes of the form \eqref{GrowingMode} to Eq. \eqref{EqLinearized} relies on the investigation of the characteristic values $\lambda(\vk) \in \C$ (Re$\lambda>0)$ such that Eq. \eqref{4thOrderEqPhi}-\eqref{RightBoundary}-\eqref{PhiAt-Infty} has a nontrivial solution $\phi$ living at least in $H^4(\R_-)$.

We first show that all characteristic values $\lambda(\vk)$ are real. Since our goal is to study the instability, we only consider positive $\lambda(\vk)$ in what follows.  Let $L_0 := (\|\frac{\rho_0'}{\rho_0}\|_{L^\infty(\R_-)})^{-1}$ be the characteristic length of density profile, we further obtain the  uniform upper bound $\sqrt{\frac{g}{L_0}}$ of $\lambda(\vk)$ in the following lemma,  whose proof is given in the beginning of Section \ref{SectLinear}. 

\begin{lemma}\label{LemEigenvalueReal}
For any $\vk \in (L^{-1}\bZ)^2\setminus \{0\}$, 
\begin{itemize}
\item all characteristic values $\lambda(\vk)$ are always real,
\item all characteristic values $\lambda(\vk)$ satisfy that $\lambda(\vk) \leqslant \sqrt{\frac{g}{L_0}}$.
\end{itemize}
\end{lemma}
We state our first theorem solving the ODE \eqref{4thOrderEqPhi}--\eqref{RightBoundary}--\eqref{PhiAt-Infty}.

\begin{theorem}\label{ThmLinear}
Let $\vk\in  (L^{-1}\bZ)^2\setminus\{0\}$ be fixed and let  $\rho_0$ satisfy \eqref{AssumeRho1}--\eqref{AssumeRho2}.  There exists an infinite sequence of characteristic values $\{\lambda_n(\vk)\}_{n\geq 1}$ deceasing towards 0 as $n\to \infty$,  such that for $\lambda=\lambda_n$,  Eq. \eqref{4thOrderEqPhi}-\eqref{RightBoundary}-\eqref{PhiAt-Infty} has a solution $\phi_n \in H^\infty(\R_-)$.  
\end{theorem}

Let us discuss about the strategy for proving Theorem \ref{ThmLinear}, which is in the same spirit as \cite{LN20}. To make our paper self-contained, we present all the steps below.

In the first step, on $(-\infty,-a)$, the ODE \eqref{4thOrderEqPhi}  is an ODE with constant coefficients, for which we can find explicit solutions in Proposition \ref{PropSolDecay} decaying to 0 at $-\infty$.  Hence,  we transform the problem for the normal modes on $\R_-$ into an ODE problem stated on a compact interval $(-a,0)$ with appropriate boundary conditions deduced from the outer solutions.   They are described by  \eqref{RightBoundary} stated above and \eqref{LeftBoundary}  (to be seen in Lemma \ref{LemBoundSmooth}). 

In the second step, in order to solve a fourth-order ODE \eqref{4thOrderEqPhi}  with the boundary conditions  \eqref{LeftBoundary} and  \eqref{RightBoundary}, the crucial step is to construct a  continuous and coercive bilinear form $\bB_{a,k,\lambda}$ on $H^2((-a,0))$ (see \eqref{1stBilinearForm} in Proposition \ref{PropPropertyR}),   such that the finding of a solution $\phi\in H^4((-a,0))$ of Eq. \eqref{4thOrderEqPhi}-\eqref{RightBoundary}-\eqref{LeftBoundary}   is equivalent to finding a weak solution $\phi \in H^2((-a,0))$ to  the variational problem 
\begin{equation}\label{EqVariational_Ido}
\lambda \bB_{a,k,\lambda}(\phi,\omega) = gk^2 \int_{-a}^0 \rho_0'\phi\omega  \quad\text{for all }\omega\in H^2((-a,0))
\end{equation}
and thus improving the regularity of that weak solution $\phi$. 

In the third step, as $\bB_{a,k,\lambda}$ is a coercive form on $H^2((-a,0))$, thus $\sqrt{\bB_{a,k, \lambda}(\cdot,\cdot)}$ is a norm on $H^2((-a,0))$. Let  $(H^2((-a,0)))'$ be the dual space of the functional space $H^2((-a,0))$, associated with the norm $\sqrt{\bB_{a,k, \lambda}(\cdot,\cdot)}$. In view of Riesz's representation theorem, we obtain an abstract  operator $Y_{a,k,\lambda}$ from $H^2((-a,0))$ to  $(H^2((-a,0)))'$, such that 
\begin{equation}\label{EqBilinearY}
 \bB_{a,k,\lambda}(\vartheta, \varrho) = \langle Y_{a,k,\lambda}\vartheta, \varrho\rangle \quad \text{for all } \vartheta,\varrho \in H^2((-a,0)).
\end{equation}
Hence, from \eqref{EqVariational_Ido} and \eqref{EqBilinearY}, we have that the  existence of a solution $\phi \in H^4((-a,0))$ of Eq. \eqref{4thOrderEqPhi}-\eqref{RightBoundary}-\eqref{LeftBoundary} is reduced to the finding  of a weak solution $\phi \in H^2((-a,0))$ of
\begin{equation}\label{EqWeakFormY}
\lambda Y_{a,k,\lambda}\phi = gk^2\rho_0' \phi \quad\text{ in } (H^2((-a,0)))'.
\end{equation}

Restricting to $\varrho \in C_0^\infty((-a,0))$ in \eqref{EqBilinearY}, we find the precise expression of $Y_{a,k,\lambda}$ (see Proposition \ref{PropInverseOfR}(1)), i.e. for all $\vartheta \in H^2((-a,0))$, 
\[
Y_{a,k,\lambda}\vartheta=\lambda(k^2 \rho_0 \vartheta -(\rho_0 \vartheta')') + \mu(\vartheta^{(4)}-2k^2\vartheta''+k^4\vartheta)  \quad\text{in } \mathcal{D}'((-a,0)).
\]  
Furthermore, a classical bootstrap argument (see Proposition \ref{PropInverseOfR}(2)) shows that we are able to define the inverse operator  $Y_{a,k,\lambda}^{-1}$ of $Y_{a,k,\lambda}$,  from $L^2((-a,0))$ to a subspace of $H^4((-a,0))$ requiring all elements satisfy \eqref{LeftBoundary}-\eqref{RightBoundary}.  Note that, 
because  $\phi$ belongs to $H^4((-a,0))$, these boundary conditions (involving the derivatives $\phi'', \phi'''$ of $\phi$ at $x_3=-a$ and at $x_3=0$) are well defined. Composing the above operator $Y_{a,k,\lambda}^{-1}$ with the continuous injection from $H^4((-a,0))$ to $L^2((-a,0))$ (see Proposition \ref{RemNormR}), we obtain that $Y_{a,k,\lambda}^{-1}$ is a compact and self-adjoint operator from $L^2((-a,0))$ to itself.

In the third step, we introduce $\cM$ the operator of multiplication by $\sqrt{\rho_0'}$ in $L^2((-a,0))$. Note from \eqref{EqWeakFormY} that, we thus find $v$ satisfying
\begin{equation}\label{EqMY-1M}
\frac{\lambda}{gk^2} v = \cM Y_{a,k,\lambda}^{-1}\cM v.
\end{equation}
Indeed, if $v$ satisfies \eqref{EqMY-1M},  $Y_{a,k,\lambda}^{-1}\cM v$ will satisfy \eqref{EqWeakFormY}. 

We show that the operator $\cM Y_{a,k,\lambda}^{-1}\cM$ is compact and self-adjoint from $L^2((-a,0))$ to itself (see Proposition \ref{PropOpeS}), which enables to use the spectral  theory of self-adjoint and compact operators to obtain that
\begin{center}
the discrete spectrum of the operator $\cM Y_{a,k,\lambda}^{-1}\cM$ is  an infinite sequence of eigenvalues (denoted by  $\{\gamma_n(\lambda,k)\}_{n\geq 1}$).
\end{center}
Let $v_{n,k,\lambda}$ be an eigenfunction of $ \cM Y_{a,k,\lambda}^{-1}\cM$ associated with the eigenvalue $\gamma_n(\lambda,k)$ and let 
\[
\phi_{n,k,\lambda} = Y_{a,k,\lambda}^{-1}\cM v_{n,k,\lambda} \in H^4((-a,0)),
\]
 we obtain 
\[
\gamma_n(\lambda,k) Y_{a,k,\lambda}\phi_{n,k,\lambda}= \cM^2\phi_{n,k,\lambda_n}=\rho_0'\phi_{n,k,\lambda}.
\]
Note that $\text{supp}\rho_0'=[-a,0]$, hence $\phi_{n,k,\lambda}$ is glued with the decaying solutions of \eqref{4thOrderEqPhi} in the outer region $(-\infty,-a)$ by the boundary conditions \eqref{LeftBoundary} at $x_3=-a$ to become a solution in $H^4(\R_-)$ of 
\begin{equation}\label{EqYonR}
\gamma_n(\lambda,k) Y_{a,k,\lambda}\phi_{n,k,\lambda}= \cM^2\phi_{n,k,\lambda_n}=\rho_0'\phi_{n,k,\lambda}
\end{equation}
satisfying \eqref{RightBoundary}-\eqref{PhiAt-Infty}. 

In the fourth step, from  \eqref{EqWeakFormY} and \eqref{EqYonR}, we see that the problem of finding characteristic values of \eqref{4thOrderEqPhi}-\eqref{RightBoundary}-\eqref{PhiAt-Infty} amounts to solving all the equations
\begin{equation}\label{EqFindLambda}
\gamma_n(\lambda,k)= \frac{\lambda}{gk^2}.
\end{equation}
We will show that, for each $n\geq 1$, there exists only one positive $\lambda_n$ satisfying \eqref{EqFindLambda}. They are owing to the differentiability in  $\lambda$ of $\gamma_n(\lambda,k)$ (see Lemma \ref{LemGammaCont}), which is an  extension of  Kato's perturbation theory of the spectrum of operators (see \cite{Kat95}), and  to the monotonicity of  the function $\lambda \mapsto \frac{\lambda^\alpha}{\gamma_n(\lambda, k)}$ in $\lambda>0$ for some $\alpha\in (0,1]$ through some direct computations  (see Lemma \ref{LemGammaDecrease}).  
With that $\lambda_n$ ($n\geq 1$), we have $\phi_{n,k,\lambda_n}$ is a solution in $H^\infty(\R_-)$ of \eqref{4thOrderEqPhi}-\eqref{RightBoundary}-\eqref{PhiAt-Infty} after using a bootstrap argument $H^4(\R_-)$. Theorem \ref{ThmLinear} is proven.

Once Eq. \eqref{4thOrderEqPhi}-\eqref{RightBoundary}-\eqref{PhiAt-Infty} is solved for fixed $\vk$, we go back to the linearized equations \eqref{EqLinearized}.  We obtain a finite sequence of real solutions to the linearized equations \eqref{EqLinearized} (see Proposition \ref{PropSolEqLinear}), which are $(j\geq 1)$
\[
e^{\lambda_j(\vk)t} V_j(\vk,x)= (e^{\lambda_j(\vk) t}(\zeta_j(\vk,x), u_j(\vk,x), q_j(\vk,x),\eta_j(\vk,x_h))^T.
\]
To close the linear section, we show that $\Lambda$ defined by
\begin{equation}\label{DefineLambda}
0<\Lambda :=\sup_{\vk \in (L^{-1}\bZ)^2 \setminus\{0\}}\lambda_1(\vk) \leq \sqrt\frac{g}{L_0},
\end{equation}
 is the maximal growth rate of the linearized equations, see Proposition \ref{PropSharpGrowthRate}.

We move to solve the nonlinear instability. 

The local well-posedness of \eqref{EqPertur} in our functional framework (see Proposition \ref{PropLocalSolution}) can be established similarly as in \cite[Theorem 6.3]{GT13}  for the incompressible viscous surface wave problem, which is used in \cite{WT12} for the incompressible viscous surface-internal wave problem and \cite{Wan19} for the incompressible viscous fluid with magnetic field. We refer to \cite{GT13, WT12, Wan19} for the construction of local solutions to \eqref{EqPertur} with some specific compatibility conditions. 

We  derive the \textit{a priori} energy estimate \eqref{EstAprioriEnergy} to the nonlinear  equations in Proposition \ref{PropAprioriEnergy}.

Thanks to \eqref{DefineLambda}, we define the non-empty set
\[
S_\Lambda :=\Big\{\vk\in (L^{-1}\bZ)^2\setminus\{0\}: \lambda_1(\vk) >\frac{2\Lambda}3 \Big\}.
\]
We further fix a $\vk \in S_\Lambda$. Hence, there is a unique $\sP \in \N^\star$ such that  
\begin{equation}\label{AssumeLambdaN}
\Lambda \geq \lambda_1(\vk) > \lambda_2(\vk)> \dots > \lambda_\sP(\vk) >\frac{2\Lambda}3 >\lambda_{\sP+1}(\vk) > \dots.
\end{equation}
In view of getting infinitely many characteristic values of the linearized problem, we consider  a linear combination of  normal  modes 
\begin{equation}\label{GeneralInitialCond_Ido}
 U^\sN(t,  x)=  \sum_{j=1}^\sN \Csf_j e^{\lambda_j t}V_j( x)  \quad(\text{for any natural number } \sN)
\end{equation}
to be  an approximate solution to the nonlinear equations \eqref{EqPertur}, with  constants $\Csf_j$ being chosen such that 
\begin{equation}\label{NormalizedCond_1}
\text {at least one of }\Csf_j \text{ }(1\leq j\leq \sP) \text{ is non-zero}
\end{equation}
and 
\begin{equation}\label{NormalizedCond_2} 
\begin{split}
\frac12 |\Csf_{j_m}| \| u_{j_m}\|_{L^2(\Omega)}> \sum_{j\geq j_m+1}|\Csf_j|\| u_j\|_{L^2(\Omega)} \quad(j_m:= \min \{j: 1\leq j\leq \sP, \Csf_j \neq 0\}).
\end{split}
\end{equation}

In order to prove the nonlinear instability, we would like to use  $U^\sN(0,x)$ as the initial data for the nonlinear equations \eqref{EqPertur}.  Unfortunately, $U^\sN(0,x)$ does not satisfy the compatibility conditions  (see  Proposition \ref{PropLocalSolution}). Thanks to an abstract argument from \cite[Section 5C]{JT13}, which was also used in \cite{WT12, Wan19},  we modify  $U^\sN(0,x)$ in Proposition \ref{PropModifyData} as follows:  there exists $\delta_0>0$ such that  the family of initial data 
\begin{equation}\label{ModifyData}
U^{\delta, \sN}(x)=\delta U^\sN(0,x)+ \delta^2  U_\star^{\delta,\sN}(x) 
\end{equation}
for $\delta \in (0,\delta_0)$ satisfies the compatibility conditions \eqref{CompaCond}. Eq. \eqref{EqPertur} with the initial data $U^{\delta, \sN}(x)$ \eqref{ModifyData} has a unique local strong solution $U^{\delta, \sN}(t,x)$ on $[0, T_{\max})$. 

Hence, we observe that 
\[
U^d(t)= U^{\delta,\sN}(t)- \delta U^\sN(t)
\]
 solves \eqref{EqPertur} with the initial data $U^d(0)=\delta^2  U_\star^{\delta, \sN}$ and the same nonlinear terms  $\cQ^i (1\leq i\leq 5)$ (see Eq. \eqref{EqDiff}).  For $t$ small enough, we deduce the following  bound in time  (see Proposition \ref{PropL2_NormU^d}), 
\[
\begin{split}
&\|(\zeta^d, u^d)(t)\|_{L^2(\Omega)}^2+\|\eta^d(t)\|_{L^2(\Gamma)}^2 \lesssim \delta^3 \Big(\sum_{j=j_m}^\sP|\Csf_j| e^{\lambda_j t}+ \max(0, \sN-\sP)\max_{\sP+1\leq j\leq \sN}|\Csf_j| e^{\frac23 \Lambda t} \Big)^3,
\end{split}
\]
That relies on some energy estimates of  Eq. \eqref{EqDiff} and the bound in time of  a suitable Sobolev norm of $U^{\delta,M}(t)$ (see  Proposition \ref{PropNormU^delta_M}), which we obtain thanks to  the \textit{a priori} energy estimate established in Proposition \ref{PropAprioriEnergy}. Combining those estimates, we obtain the following nonlinear result.

\begin{theorem}\label{ThmNonlinear}
Let $\rho_0$ satisfy \eqref{AssumeRho1}--\eqref{AssumeRho2}. Let $\sN$ be an arbitrary integer,  there exist  two positive constants $\nu_0, \delta_0$ sufficiently small and another constant $m_0 >0$, so that for any $\delta \in (0,\delta_0)$  the nonlinear  equations \eqref{EqPertur}  with the initial data \eqref{ModifyData}, i.e. 
\[
\delta\sum_{j=1}^{\sN} \Csf_j V_j(x) +\delta^2 U_\star^{\delta,\sN}(x),
\]
satisfying \eqref{NormalizedCond_1}-\eqref{NormalizedCond_2} admits a unique local strong solution  $U^{\delta, \sN}$ such that 
\begin{equation}\label{BoundU_2Tdelta}
\|  u^{\delta,\sN}(T^{\delta})\|_{L^2(\Omega)} \geq m_0\nu_0,
\end{equation}
where $T^\delta\in (0, T_{\max})$  is given by $\delta \sum_{j=1}^\sN|\Csf_j| e^{\lambda_j T^\delta }=\nu_0$.
\end{theorem}

\section{The linear analysis}\label{SectLinear}

In this section, we prove the linear result, Theorem \ref{ThmLinear}.  
We begin with the proof of Lemma \ref{LemEigenvalueReal}.
\begin{proof}[Proof of Lemma \ref{LemEigenvalueReal}]
Multiplying by $\overline \phi$ on both sides of \eqref{4thOrderEqPhi} and then integrating by parts, we  obtain that 
\[
\begin{split}
&\lambda^2\Big( \int_{\R_-} ( k^2 \rho_0 |\phi|^2 + \rho_0 |\phi'|^2 ) - \rho_0\phi' \overline \phi\Big|_{-\infty}^0 \Big)  + \lambda \mu  \int_{\R_-} (|\phi''+k^2\phi|^2 + 4k^2 |\phi'|^2 ) \\
&\qquad + \lambda \mu   ((\phi'''-3k^2\phi')\overline\phi -(\phi''+k^2\phi)\overline\phi') \Big|_{-\infty}^0 =gk^2\int_{\R_-}\rho_0'|\phi|^2.
\end{split}
\]
Using \eqref{RightBoundary} and \eqref{PhiAt-Infty}, we get 
\begin{equation}\label{EqVariational}
\begin{split}
\lambda^2 \int_{\R_-} ( k^2 \rho_0 |\phi|^2 + \rho_0 |\phi'|^2 )  &+ \lambda \mu \int_{\R_-} ( |\phi''+k^2\phi|^2 + 4k^2 |\phi'|^2  ) \\
& =- gk^2\rho_+ |\phi(0)|^2 + gk^2 \int_{\R_-} \rho_0'|\phi|^2 . 
\end{split}
\end{equation}
Suppose that $\lambda = \lambda_1 + i\lambda_2$, then one deduces from \eqref{EqVariational} that 
\begin{equation}\label{EqImaginaryPart}
-2\lambda_1 \lambda_2 \int_{\R_-} ( k^2 \rho_0 |\phi|^2 + \rho_0 |\phi'|^2 )  = \lambda_2\mu\int_{\R_-} ( |\phi''+k^2\phi|^2 + 4k^2 |\phi'|^2  ) .
\end{equation}
If $\lambda_2 \neq 0$, \eqref{EqImaginaryPart} leads us to
\[
-2\lambda_1 \int_{\R_-} ( k^2 \rho_0 |\phi|^2 + \rho_0 |\phi'|^2 )  =  \mu \int_{\R_-} ( |\phi''+k^2\phi |^2 + 4k^2 |\phi'|^2 )  <0,
\]
that contradiction yields $\lambda_2=0$, i.e. $\lambda$ is real.
Using \eqref{EqVariational} again, we further get that 
\[
\lambda^2 \int_{\R_-}\rho_0(k^2|\phi|^2+|\phi'|^2)  \leqslant gk^2 \int_{\R_-}\rho_0'|\phi|^2 .
\]
It tells us that $\lambda$ is  bounded by  $\sqrt{\frac{g}{L_0}}$. This finishes the proof of Lemma \ref{LemEigenvalueReal}.
\end{proof}
In view of Lemma \ref{LemEigenvalueReal}, we look for functions $\phi$ being real and we only consider the vector spaces of real-valued functions in what follows in the linear analysis. 


\subsection{Solutions on the outer region $(-\infty,-a)$ and reduction to an ODE on the finite interval $(-a,0)$}\label{SubSolOuterRegions}

\begin{proposition}\label{PropSolDecay}
Let $\tau_- =\sqrt{k^2+\lambda\rho_-/\mu}$. All solutions  decaying to 0 at $-\infty$ as $x_3 \in (-\infty,-a]$ of \eqref{4thOrderEqPhi} are of the form 
\begin{equation}\label{LeftSol}
\phi(x_3) = A_1 e^{k(x_3+a)} + A_2 e^{\tau_-(x_3+a)}.
\end{equation}
\end{proposition}
\begin{proof}
On the interval $(-\infty,-a)$, Eq. \eqref{4thOrderEqPhi} is an ODE with constant coefficients,
\begin{equation}\label{EqLeftSide}
-\lambda \rho_-( k^2 \phi - \phi'') = \mu(\phi^{(4)} - 2k^2 \phi'' +k^4 \phi).
\end{equation}
We seek $\phi$ as $\phi(x_3) = e^{rx_3}$. Hence, 
\[
-\lambda\rho_-(k^2-r^2) = \mu(r^4-2k^2r^2+k^4),
\]
which yields $r= \pm k$ or $r=\pm (k^2+\lambda \rho_-/\mu)^{1/2}$. Since $\phi$ tends to 0 at $-\infty$,  we get two  independent solutions of \eqref{EqLeftSide},
\[
\phi_1^-(x_3)= e^{kx_3}\quad\text{and}\quad\phi_2^-(x_3)=e^{(k^2+\lambda\rho_-/\mu)^{1/2}x_3}.
\]
Hence, all bounded solutions of \eqref{EqLeftSide} are of the form \eqref{LeftSol}.
Proposition \ref{PropSolDecay} is proven.
\end{proof}

Once it is proven that $\phi(x_3)$ outside $(-a,0)$ is known precisely, we look for $\phi$ on $(-a,0)$. We will show that in the following lemma. 

\begin{lemma}\label{LemBoundSmooth}
The boundary conditions of \eqref{4thOrderEqPhi} at $x_3=-a$, for $\phi\in H^4((-a,0))$, are
\begin{equation}\label{LeftBoundary}
\begin{cases}
k\tau_{-} \phi(-a) -(k+\tau_-) \phi'(-a) +\phi''(-a)=0,\\
k\tau_-(k+\tau_-) \phi(-a) - (k^2+k\tau_-+\tau_-^2)\phi'(-a)+\phi'''(-a)=0.
\end{cases}
\end{equation}
and at $x_3=0$, are \eqref{RightBoundary}.
\end{lemma}
\begin{proof}
For a solution  $\phi$ of Eq. \eqref{4thOrderEqPhi} on $(-a,0)$, the boundary conditions at $x_3= -a$ are equivalent to the fact that $\phi$ belongs to the space of decaying solutions at $-\infty$. On the one hand, it can be seen from \eqref{LeftSol}  that 
\[
\begin{pmatrix}
\phi(x_3) \\ \phi'(x_3) \\ \phi''(x_3) \\ \phi'''(x_3) 
\end{pmatrix} 
= A_1^- e^{k(x_3+a)}
\begin{pmatrix}
1 \\ k \\ k^2 \\k^3
\end{pmatrix}
+A_2^- e^{\tau_-(x_3+a)}
\begin{pmatrix}
1 \\ \tau_- \\ \tau_-^2 \\ \tau_-^3
\end{pmatrix}
\quad\text{for }x_3\leqslant -a.
\]
On the other hand, direct computations show that the orthogonal complement of the subspace  of $\R^4$ spanned by two vectors $(1,k, k^2,k^3)^T$ and $(1,\tau_-,\tau_-^2,\tau_-^3)^T$ is spanned by 
\[
(k\tau_-,-(k+\tau_-),1,0)^T \quad\text{and} \quad(k\tau_-(k+\tau_-),-(k^2+k\tau_-+\tau_-^2),0,1)^T.
\]  
The above arguments allow us to set \eqref{LeftBoundary}   as boundary conditions of Eq. \eqref{4thOrderEqPhi} at $x_3=-a$.
\end{proof}


\subsection{A bilinear form and a self-adjoint invertible operator}

\begin{proposition}\label{PropPropertyR}
Let us denote by 
\begin{equation}
\begin{split}
BV_{-a,k,\lambda}(\vartheta,\varrho) &:= \mu \left( \begin{split} 
& k\tau_-(k+\tau_-) \vartheta(-a) \varrho(-a) - k\tau_-\vartheta'(-a) \varrho(-a) \\
&\qquad -  k\tau_- \vartheta(-a) \varrho'(-a) + (k+\tau_-)\vartheta'(-a) \varrho'(-a)
\end{split}\right), \\
BV_{0,k,\lambda}(\vartheta,\varrho) &:= \mu k^2(\vartheta'(0) \varrho(0)+ \vartheta(0) \varrho'(0)) + \frac{gk^2\rho_+}{\lambda} \vartheta(0)\varrho(0),
\end{split}
\end{equation}
and
\begin{equation}\label{1stBilinearForm}
\begin{split}
\bB_{a,k,\lambda}(\vartheta, \varrho) & := BV_{0,k,\lambda}(\vartheta,\varrho) + BV_{-a,k,\lambda}(\vartheta,\varrho) + \lambda\int_{-a}^0 \rho_0(k^2\vartheta  \varrho + \vartheta'  \varrho')  \\
&\qquad\qquad+ \mu \int_{-a}^0 (\vartheta''  \varrho'' + 2k^2 \vartheta'  \varrho' +k^4 \vartheta  \varrho).
\end{split} 
\end{equation}
The bilinear form $\bB_{a,k,\lambda}$ is a continuous and coercive on $H^2((-a,0))$.

Furthermore, let  $(H^2((-a,0)))'$ be the dual space of $H^2((-a,0))$, which is associated with the norm $\sqrt{\bB_{a,k,\lambda}(\cdot,\cdot)}$, there exists a unique operator   $Y_{a,k,\lambda} \in  \mathcal{L}(H^2((-a,0)), (H^2((-a,0)))')$, that is also bijective,  such that for all $\vartheta, \varrho \in H^2((-a,0))$,
 \begin{equation}\label{EqMathcalB_a}
\bB_{a,k,\lambda}(\vartheta, \varrho) = \langle Y_{a,k,\lambda}\vartheta,  \varrho\rangle.
\end{equation}
\end{proposition}

Before proving Proposition \ref{PropPropertyR}, we state our key lemma, whose proof is postponed to Appendix \ref{AppFormulaD_k}. This yields the coercivity of $\bB_{a,k,\lambda}$ as it will appear in the proof of Proposition \ref{PropPropertyR}.
\begin{lemma}\label{LemFormulaD_k}
We have
\[\begin{split}
& \min_{\vartheta\in H^2((-a,0))} 
\left(\begin{split}
&2k^2(  \vartheta'(0)\vartheta(0)- \vartheta'(-a)\vartheta(-a)),  \vartheta \text{ satisfies the constraint}\\
&\qquad \int_{-a}^0 ((\vartheta'')^2 +2k^2(\vartheta')^2+k^4 \vartheta^2 ) =1.
\end{split}\right) = -\frac{\sinh(ka)+ka}{3\sinh(ka)-ka},
\end{split}\]
and 
\[\begin{split}
&\max_{\vartheta \in H^2((-a,0))} 
\left(\begin{split}
&2k^2(  \vartheta'(0)\vartheta(0)- \vartheta'(-a)\vartheta(-a)),  \vartheta \text{ satisfies the constraint}\\
&\qquad \int_{-a}^0 ((\vartheta'')^2 +2k^2(\vartheta')^2+k^4 \vartheta^2 ) =1.
\end{split}\right) =1.
\end{split}\]
\end{lemma}
\begin{proof}[Proof of Proposition \ref{PropPropertyR}]
Clearly, $\bB_{a,k,\lambda}$ is a bounded bilinear form on  $H^2((-a,0))$ since the terms $BV_{a,k,\lambda}(\vartheta,\varrho)$ and $BV_{0,k,\lambda}(\vartheta,\varrho)$ are well defined. We move to show the coercivity of $\bB_{a,k,\lambda}$. We have that 
\[
\begin{split}
\bB_{a,k,\lambda}(\vartheta,\vartheta) &= BV_{0,k,\lambda}(\vartheta,\vartheta) + BV_{-a,k,\lambda}(\vartheta,\vartheta) + \lambda \int_{-a}^0  \rho_0 (k^2\vartheta^2  + (\vartheta')^2 )  \\
&\qquad+  \mu \int_{-a}^0 ((\vartheta'')^2 + 2k^2 (\vartheta')^2 +k^4 (\vartheta)^2 ).
\end{split}
\]
We have
\[
\begin{split}
\frac1{\mu}BV_{-a,k,\lambda}(\vartheta,\vartheta) &=   k\tau_-(k+\tau_-) (\vartheta(-a))^2 -2 k\tau_-\vartheta(-a)\vartheta'(-a) + (k+\tau_-)(\vartheta'(-a))^2\\
&= (k+\tau_-) \Big( \vartheta'(-a) + \frac{k(k-\tau_-)}{k+\tau_-} \vartheta(-a)\Big)^2 \\
&\qquad + \frac{ k(\tau_-(k+\tau_-)^2 -k(k-\tau_-)^2) }{k+\tau_-} (\vartheta(-a))^2 -2 k^2 \vartheta(-a)\vartheta'(-a) \\
&\geq -2 k^2 \vartheta(-a)\vartheta'(-a).
\end{split}
\]
Therefore, we deduce that 
\[
\begin{split}
\frac1{\mu}\bB_{a,k,\lambda}(\vartheta,\vartheta) &\geq 2 k^2 (\vartheta(0)\vartheta'(0)- \vartheta(-a)\vartheta'(-a)) + \int_{-a}^0 ((\vartheta'')^2+2k^2(\vartheta')^2+k^4\vartheta^2).
\end{split}
\]
Notice from Lemma \ref{LemFormulaD_k} that 
\begin{equation}\label{LowerBoundB}
\frac1\mu \bB_{a,k,\lambda}(\vartheta,\vartheta)  \geq \frac{2(\sinh(ka)-ka)}{3\sinh(ka)-ka} \int_{-a}^0 ((\vartheta'')^2+2k^2(\vartheta')^2+k^4 \vartheta^2).
\end{equation}
The inequality \eqref{LowerBoundB} tell us that $\bB_{a,k,\lambda}$ is a  continuous and coercive bilinear form on $H^2((-a,0))$. It follows from Riesz's representation theorem  that there is a unique operator 
\[
Y_{a,k,\lambda} \in  \mathcal{L}(H^2((-a,0)), (H^2((-a,0)))'),
\]
 that is also bijective, satisfying \eqref{EqMathcalB_a} for all $\vartheta, \varrho \in H^2((-a,0))$. Proof of Proposition \ref{PropPropertyR} is complete.
\end{proof}

The next proposition is to devoted to studying the properties of $Y_{a,k,\lambda}$.

\begin{proposition}\label{PropInverseOfR}
We have the following results.
\begin{enumerate}
\item For all $\vartheta \in H^2((-a,0))$, 
\[
Y_{a,k,\lambda}\vartheta=\lambda(k^2 \rho_0 \vartheta -(\rho_0 \vartheta')') +\mu(\vartheta^{(4)}-2k^2\vartheta''+k^4\vartheta)  \quad\text{ in } \mathcal{D}'((-a,0)).  
\]

\item Let $f\in L^2((-a,0))$ be given, there exists a unique solution  $\vartheta \in H^2((-a,0))$ of \begin{equation}\label{EqY=f}
Y_{a,k,\lambda}\vartheta = f \text{ in } ( H^2((-a,0)))'.
\end{equation}
Moreover, $\vartheta\in H^4((-a,0))$ and satisfies the boundary conditions \eqref{LeftBoundary}--\eqref{RightBoundary}.
\end{enumerate}
\end{proposition}
The proof of Proposition \ref{PropInverseOfR} is due to a bootstrap argument, which is followed by \cite[Proposition 3.3]{LN20}. Hence we omit the details here. 

We have the following proposition on $Y_{a,k,\lambda}^{-1}$.
\begin{proposition}\label{RemNormR}
The operator $Y_{a,k,\lambda}^{-1} : L^2((-a,0)) \to L^2((-a,0))$ is compact and self-adjoint. 
\end{proposition}
We prove  Proposition \ref{PropInverseOfR} thanks to the continuous injection  from $H^4((-a,0))$ to $L^2((-a,0))$, in the same line of \cite[Proposition 3.4]{LN20}. As a result of Proposition \ref{RemNormR}, we obtain the following. 
\begin{proposition}\label{PropOpeS}
Let $S_{a,k,\lambda} := \cM Y_{a,k,\lambda}^{-1}\cM$, where $\cM$ is the operator of multiplication by $\sqrt{\rho'_0}$. The operator $S_{a,k,\lambda} : L^2((-a,0)) \to L^2((-a,0))$ is compact and self-adjoint.
\end{proposition}

As a result of the spectral theory of compact and self-adjoint operators, the point spectrum of $S_{a,k,\lambda}$ is discrete, i.e. is a  sequence $\{\gamma_n(\lambda,k)\}_{n\geqslant 1}$ of  eigenvalues of $S_{a,k,\lambda}$, associated with normalized orthogonal eigenfunctions $\{\varpi_n\}_{n\geqslant 1}$ in $L^2((-a,0))$. That means 
\[
\gamma_n(\lambda,k)\varpi_n =S_{a,k,\lambda}\varpi_n= \cM Y_{a,k,\lambda}^{-1}\cM \varpi_n.
\]
So that with $\phi_n = Y_{a,k,\lambda}^{-1}\cM \varpi_n \in H^4((-a,0))$, one has
\begin{equation}\label{EqRf_n}
\gamma_n(\lambda, k) Y_{a,k,\lambda}\phi_n =  \rho_0' \phi_n
\end{equation}
and $\phi_n$ satisfies \eqref{LeftBoundary}-\eqref{RightBoundary}. Eq.
\eqref{EqRf_n} also tells us that $\gamma_n(\lambda,k) >0$ for all $n$. Indeed, we obtain 
\begin{equation}\label{EqPhi_nB}
\gamma_n(\lambda,k) \int_{-a}^0 (Y_{a,k,\lambda}\phi_n)  \phi_n  = \gamma_n(\lambda,k) \bB_{a,k,\lambda}(\phi_n,\phi_n) = \int_{-a}^0 \rho_0' \phi_n^2 .
\end{equation}
Since $\bB_{a,k,\lambda}(\phi_n,\phi_n) >0$ and $\rho_0' >0$ on $(-a,0)$, we know that $\gamma_n(\lambda,k)$ is positive for all $n$. Hence, by reordering and using the spectral theory of compact and self-adjoint operators again, we obtain that $\{\gamma_n(\lambda,k)\}_{n\geq 1}$ is a positive sequence decreasing towards 0 as $n\to \infty$.

\subsection{Proof of Theorem \ref{ThmLinear}}

For each $n$,  we extend a solution $\phi_n$  of \eqref{EqRf_n} on $(-a,0)$ with the boundary conditions  \eqref{LeftBoundary}-\eqref{RightBoundary} to a solution of \eqref{EqRf_n} on  $\R_-$ satisfying \eqref{RightBoundary}-\eqref{PhiAt-Infty} as follows. Let 
\[
\phi_n(x_3)=  A_{n,1}e^{kx_3} + A_{n,2} e^{\sqrt{k^2+\frac{\lambda \rho_-}\mu}(x_3+a)}
\]
for some constants $A_{n,1}, A_{n,2}$.  Those constants $A_{n,1}, A_{n,2}$ are defined by 
\begin{equation}\label{EqA+}
\phi_n(-a) = A_{n,1} + A_{n,2}, \quad \phi_n'(-a) = kA_{n,1} + A_{n,2}  \sqrt{k^2+\frac{\lambda \rho_-}{\mu}}.
\end{equation}
Solving \eqref{EqA+}, we get that 
\begin{equation}\label{ConstA+}
A_{n,1} = \frac{\sqrt{k^2+\frac{\lambda \rho_-}{\mu}} \phi_n(-a) - \phi_n'(-a)}{\sqrt{k^2+\frac{\lambda \rho_-}{\mu}} -k}, \quad A_{n,2} = \frac{\phi_n'(-a) -k \phi_n(-a) }{\sqrt{k^2+\frac{\lambda \rho_-}{\mu}} -k}.
\end{equation}
In order to verify that $\phi_n$ is a solution of \eqref{4thOrderEqPhi}-\eqref{RightBoundary}-\eqref{PhiAt-Infty}, we are left to look for real values $\lambda_n$  satisfying \eqref{EqFindLambda}.

To solve \eqref{EqFindLambda}, we need the three following lemmas.
\begin{lemma}\label{LemVariationR}
There holds 
\begin{equation}\label{VariationTheta}
\max_{\vartheta \in H^2(\R_-)} 
\Big( k^2 (\vartheta(0))^2,  \vartheta \text{ satisfies the constraint } \int_{\R_- } ( (\vartheta''+k^2\vartheta)^2 +4k^2(\vartheta')^2) =1\Big)  = \frac1{2k}.
\end{equation}
\end{lemma}
The proof of Lemma \ref{LemVariationR} is postponed to Appendix \ref{AppVariationR}.
\begin{lemma}\label{LemGammaCont}
For each $n$, $\gamma_n(\lambda,k)$ and $\phi_n$ are differentiable in $\lambda$.
\end{lemma}
\begin{proof}
The proof of Lemma \ref{LemGammaCont} is the same as \cite[Lemma 3.2]{LN20}, we omit the details here.
\end{proof}
\begin{lemma}\label{LemGammaDecrease}
For each $n$,  there holds
\begin{enumerate}
\item  the function $\frac{\lambda}{\gamma_n(\lambda,k)}$ is strictly increasing as $\lambda >0$,
\item  for fixed $\ve>0$ and fixed $\alpha \in (\frac{g\rho_+L}{g\rho_+L+2\mu\ve}, 1)$, the function  $\frac{\lambda^\alpha}{\gamma_n(\lambda,k)}$ is strictly increasing as $\lambda \geq \ve>0 $.
\end{enumerate}
\end{lemma}
\begin{proof}
Let $z_n= \frac{d\phi_n}{d\lambda}$ and $\alpha>0$, it follows from  \eqref{EqRf_n} that 
\[
\lambda^\alpha Y_{a,k,\lambda}\phi_n= \frac{\lambda^\alpha}{\gamma_n(\lambda,k)} \rho_0'\phi_n.
\]
It yields, on $\R_-$,
\[
\begin{split}
&(1+\alpha) \lambda^\alpha (k^2\rho_0\phi_n- (\rho_0\phi_n')')+  \alpha\lambda^{\alpha-1} \mu (\phi_n^{(4)}-2k^2\phi_n''+k^4\phi_n)+ \lambda^\alpha Y_{a,k,\lambda}z_n \\
&= \frac{\lambda^\alpha}{\gamma_n(\lambda,k)} \rho_0'z_n+ \frac{d}{d\lambda}\Big( \frac{\lambda^\alpha}{\gamma_n(\lambda,k)}\Big)\rho_0'\phi_n.
\end{split}
\]
Equivalently, on $\R_-$, we have
\begin{equation}\label{1stEqDeriTz_n}
\begin{split}
&\lambda^\alpha (k^2\rho_0\phi_n- (\rho_0\phi_n')')+ \alpha \lambda^{\alpha-1}  Y_{a,k,\lambda}\phi_n + \lambda^{1/2} Y_{a,k,\lambda}z_n \\
&\qquad\qquad= \frac{\lambda^\alpha}{\gamma_n(\lambda,k)} \rho_0'z_n+ \frac{d}{d\lambda}\Big( \frac{\lambda^\alpha}{\gamma_n(\lambda,k)}\Big)\rho_0'\phi_n.
\end{split}
\end{equation}
Note that, at $x_3=0$,  due to \eqref{RightBoundary}, we have
\begin{equation}\label{ZnAt0}
\begin{cases}
z_n''(0)+k^2 z_n(0)=0, \\
z_n'''(0)- (3k^2+\frac{\lambda\rho_+}{\mu})z_n'(0)-  \frac{gk^2\rho_+}{\lambda\mu}z_n(0)= \frac{\rho_+}{\mu} \phi_n'(0) - \frac{gk^2\rho_+}{\lambda^2\mu}\phi_n(0).
\end{cases}
\end{equation}
Multiplying by $\phi_n$ on both sides of \eqref{1stEqDeriTz_n}, we obtain that 
\begin{equation}\label{2ndEqDeriTz_n}
\begin{split}
&\lambda^\alpha \int_{\R_-} (k^2\rho_0\phi_n- (\rho_0\phi_n')') \phi_n +\alpha\lambda^{\alpha-1} \int_{\R_-}(Y_{a,k,\lambda}\phi_n )\phi_n  +\lambda^\alpha\int_{\R_-} (Y_{a,k,\lambda}z_n) \phi_n  \\
&\qquad\qquad= \frac{\lambda^\alpha}{\gamma_n(\lambda,k)} \int_{\R_-} \rho_0'z_n \phi_n + \frac{d}{d\lambda}\Big( \frac{\lambda^\alpha}{\gamma_n(\lambda,k)}\Big) \int_{\R_-} \rho_0' \phi_n^2 .
\end{split}
\end{equation}
Thanks to the integration by parts, we have that
\begin{equation}\label{2ndEqDeriTz_nPsi_n}
 \int_{\R_-} (k^2\rho_0\phi_n- (\rho_0\phi_n')') \phi_n  =  \int_{\R_-} \rho_0(k^2 \phi_n^2+(\phi_n')^2) - \rho_+\phi_n'(0) \phi_n(0),
\end{equation}
that
\begin{equation}\label{2ndEqDeriTz_nTz_n}
\begin{split}
\int_{\R_-} (Y_{a,k,\lambda}z_n) \phi_n &=  \int_{\R_-} (Y_{a,k,\lambda} \phi_n)z_n  + \Big( \mu( z_n'''\phi_n -3k^2z_n')  \phi_n -\mu(z_n''+k^2z_n)\phi_n'-\lambda\rho_0 z_n' \phi_n \Big)(0) \\
&\qquad- \Big(\mu(\phi_n'''-3k^2\phi_n')z_n -\mu(\phi_n''+k^2\phi_n) z_n'-\lambda\rho_0 \phi_n'z_n\Big)(0)
\end{split}
\end{equation}
and that 
\begin{equation}\label{3rdEqDeriTz_nTz_n}
\begin{split}
\int_{\R_-} (Y_{a,k,\lambda}\phi_n) \phi_n &=  \lambda \int_{\R_-}  \rho_0(k^2\phi_n^2+(\phi_n')^2) +\mu \int_{\R_-}((\phi_n''+k^2\phi_n)^2+4k^2 (\phi_n')^2)\\
&\qquad+ \Big( \mu( \phi_n''' -3k^2\phi_n')  \phi_n -\mu(\phi_n''+k^2\phi_n)\phi_n'-\lambda\rho_0 \phi_n' \phi_n \Big)(0).
\end{split}
\end{equation}
Using  \eqref{ZnAt0} and \eqref{RightBoundary}, we obtain 
\begin{equation}\label{4thEqDeriTz_n}
\begin{split}
&\mu\Big( (z_n''' -3k^2z_n')  \phi_n -\mu(z_n''+k^2z_n)\phi_n'-\lambda\rho_0 z_n' \phi_n \Big)(0) - \rho_+\phi_n'(0)\phi_n(0) \\
&\qquad- \Big(\mu(\phi_n'''-3k^2\phi_n')z_n -\mu(\phi_n''+k^2\phi_n) z_n'-\lambda\rho_0 \phi_n'z_n\Big)(0) = -\frac{gk^2\rho_+}{\lambda^2}(\phi_n(0))^2,
\end{split}
\end{equation}
and 
\begin{equation}\label{5thEqDeriTz_n}
\Big( \mu( \phi_n''' -3k^2\phi_n')  \phi_n -\mu(\phi_n''+k^2\phi_n)\phi_n'-\lambda\rho_0 \phi_n' \phi_n \Big)(0)=\frac{gk^2\rho_+}{\lambda}(\phi_n(0))^2. 
\end{equation}
Owing to the above identities \eqref{2ndEqDeriTz_nPsi_n}, \eqref{2ndEqDeriTz_nTz_n}, \eqref{3rdEqDeriTz_nTz_n}, \eqref{4thEqDeriTz_n} and \eqref{5thEqDeriTz_n}, we obtain from \eqref{2ndEqDeriTz_n} that
\[
\begin{split}
\frac{d}{d\lambda}\Big(\frac{\lambda^\alpha}{\gamma_n(\lambda,k)}\Big) \int_{\R_-} \rho_0' \phi_n^2  &= (1+\alpha) \lambda^\alpha \int_{\R_-} \rho_0(k^2 \phi_n^2+(\phi_n')^2)+ ( \alpha-1)\lambda^{\alpha-2}  gk^2\rho_+ (\phi_n(0))^2 \\
&\qquad+\alpha \lambda^{\alpha-1} \mu \int_{\R_-}((\phi_n''+k^2\phi_n)^2+4k^2 (\phi_n')^2).
\end{split}
\]
We make use of Lemma \ref{LemVariationR} to get further
\begin{equation}\label{3rdEqDeriTz_n}
\begin{split}
\lambda^{2-\alpha}\frac{d}{d\lambda}\Big(\frac{\lambda^\alpha}{\gamma_n(\lambda,k)}\Big) \int_{\R_-} \rho_0' \phi_n^2  &>\alpha \lambda \mu \int_{\R_-}((\phi_n''+k^2\phi_n)^2+4k^2 (\phi_n')^2) -(1-\alpha) gk^2\rho_+ (\phi_n(0))^2\\
&\geq   \Big(\alpha \lambda \mu - \frac{(1-\alpha)g\rho_+}{2k}\Big) \int_{\R_-}((\phi_n''+k^2\phi_n)^2+4k^2 (\phi_n')^2).
\end{split}
\end{equation}
The first assertion of Lemma \ref{LemGammaDecrease} was derived from \eqref{3rdEqDeriTz_n} by letting $\alpha=1$. To prove the second assertion, we  note that $k\geq \frac1L$ since $\vk \in (L^{-1}\bZ)^2\setminus \{0\}$. Hence, for $\frac{g\rho_+L}{g\rho_+L+2\mu\ve}\leq \alpha <1$, we have
\[
\lambda k - \frac{1-\alpha}{\alpha} \frac{g\rho_+}{2\mu} \geq \frac{\ve}L -\frac{1-\alpha}{\alpha} \frac{g\rho_+}{2\mu} = \frac{\ve}L+ \frac{g\rho_+}{2\mu}- \frac1\alpha \frac{g\rho_+}{2\mu}  >0.
\]
Consequently, $\frac{\lambda^\alpha}{\gamma_n(\lambda,k)}$ is strictly increasing in $\lambda \geq \ve >0$.  Lemma \ref{LemGammaDecrease} is proven.
\end{proof}
Now we are in position to solve \eqref{EqFindLambda} and finish the proof of Theorem \ref{ThmLinear}.
\begin{proof}[Proof of Theorem \ref{ThmLinear}]
Using \eqref{EqRf_n} and \eqref{3rdEqDeriTz_nTz_n}, we know that
\[\begin{split}
\frac1{\gamma_n(\lambda,k)} \int_{\R_-} \rho_0' \phi_n^2  &=  \lambda\int_{\R_-} \rho_0(k^2\phi_n^2+ (\phi_n')^2)+ \mu \int_{\R_-}((\phi_n''+k^2\phi_n)^2+4k^2 (\phi_n')^2). 
\end{split}\]
We deduce that
\[
\frac1{L_0\gamma_n(\lambda,k)}  \geqslant \lambda  k^2.
\]
Hence
\begin{equation}\label{LimitGammaRight}
\lim_{\lambda \to +\infty} \frac{\lambda}{\gamma_n(\lambda,k)} =+\infty.
\end{equation}
Next, we prove that 
\begin{equation}\label{LimitGammaLeft}
\lim_{\lambda \to 0} \frac{\lambda}{\gamma_n(\lambda,k)} =0.
\end{equation}
For $0<\lambda \leq \ve$, due to Lemma \ref{LemGammaDecrease}(2), we have that  
\[
\frac{\lambda^\alpha}{\gamma_n(\lambda,k)} \leq \frac{\ve^\alpha}{\gamma_n(\ve,k)} \quad\text{for }\alpha = \frac{g\rho_+L}{g\rho_+L+2\mu\lambda}\in (0,1).
\]
That implies
\[
\frac{\lambda}{\gamma_n(\lambda,k)}= \lambda^{1-\alpha} \frac{\lambda^\alpha}{\gamma_n(\lambda,k)} \leq \lambda^{1-\alpha} \frac{\ve^\alpha}{\gamma_n(\ve,k)} \leq \lambda^{1-\alpha} \frac\ve{\gamma_n(\ve,k)}\to 0 \quad\text{as }\lambda \to 0,
\]
yielding \eqref{LimitGammaLeft}. Combining the two limits \eqref{LimitGammaRight} and \eqref{LimitGammaLeft} with Lemma \ref{LemGammaDecrease}(1), we deduce that, for each $n\geq 1$, there is a unique $\lambda_n$ solving \eqref{EqFindLambda}. Since $\lambda_n$ is a characteristic value, we have that $\lambda_n\in (0,\sqrt{\frac{g}{L_0}})$.

We  prove that the sequence $(\lambda_n)_{n\geq 1}$ is  decreasing. Indeed, if $\lambda_m<\lambda_{m+1}$ for some $m\geq 1$,  we make use of Lemma \ref{LemGammaDecrease}(1) to have
\[
\frac{\lambda_m}{\gamma_m(\lambda_m,k)}< \frac{\lambda_{m+1}}{\gamma_m(\lambda_{m+1},k)}.
\]
Note that $\gamma_m(\lambda_{m+1},k) > \gamma_{m+1}(\lambda_{m+1},k)$.  
That implies 
\[
gk^2= \frac{\lambda_m}{\gamma_m(\lambda_m,k)} <  \frac{\lambda_{m+1}}{\gamma_{m+1}(\lambda_{m+1},k)} =gk^2.
\]
That contradiction tells us that $(\lambda_n)_{n\geq 1}$ is a decreasing sequence.  

To complete the proof of Theorem \ref{ThmLinear},  we now prove that $\lim_{n\to \infty}\lambda_n =0$. Indeed,  suppose that $\lim_{n\to \infty}\lambda_n=c_0>0$, one has that $\lambda_n\geq c_0$ for all $n\geq 1$. By Lemma \ref{LemGammaDecrease}(1) again, it yields
\[
\frac{c_0}{\gamma_n(c_0,k)} \leq \frac{\lambda_n}{\gamma_n(\lambda_n,k)}=\frac1{gk^2},
\]
thus 
\[
c_0 \leq \frac{\gamma_n(c_0,k)}{gk^2}.
\]
Letting $n\to \infty$, we obtain a contradiction that $c_0 \leq 0$. Hence, $\lim_{n\to \infty}\lambda_n =0$.
 The proof of Theorem \ref{ThmLinear} is complete.
\end{proof}

Once we have solutions of \eqref{4thOrderEqPhi}-\eqref{RightBoundary}-\eqref{PhiAt-Infty}, we go back to the linearized equations \eqref{EqLinearized}.
\begin{proposition}\label{PropSolEqLinear}
For each $\vk=(k_1,k_2)\in  (L^{-1}\bZ)^2 \setminus \{0\}$, there exists an infinite sequence of  normal modes 
\begin{equation}
e^{\lambda_n(\vk) t}  V_n(\vk, x) = e^{\lambda_n(\vk) t}(\zeta_n(\vk,x),  u_n(\vk,x), q_n(\vk,x),\eta_n(\vk, x_h))
\end{equation}
to the linearized equations \eqref{EqLinearized}, such that 
\begin{equation}\label{RegularSolutionsLinear}
\zeta_n \in H^\infty(\Omega),  u_n \in (H^\infty(\Omega))^3 \text{ and } q_n \in H^\infty(\Omega).
 \end{equation}
\end{proposition}
\begin{proof}
For each  solution $\lambda_n\in  (0,\sqrt{\frac{g}{L_0}}) $ of \eqref{EqFindLambda}, we  have a solution $\phi_n$ in $H^4(\R_-)$ of \eqref{4thOrderEqPhi} as $\lambda=\lambda_n$, being found in Theorem \ref{ThmLinear}. Furthermore, $\phi_n \in H^\infty(\R_-)$. We find uniquely $\pi_n \in H^\infty(\R_-)$ from \eqref{RelationPi_Phi}  such that
\[\pi_n(\vk, x_3)  =
- \frac1{k^2} (\lambda_n\rho_0\phi_n'+\mu(k^2\phi_n' -\phi_n'''))(,\vk,x_3). 
\]
To look for $\psi_n$, we rewrite $\eqref{SystFunctionsX_3}$ as a second order ODE, 
\[
-\mu \psi_n'' + (\lambda_n\rho_0\psi_n + \mu k^2\psi_n - k_1\pi_n)=0.
\]
Note from \eqref{BoundFunctionsX_3} and \eqref{FunctionsAtInfty} that $\psi_n$ satisfies that $\psi_n'(0)=k_1\phi_n(0)$ and that $\lim_{x_3\to -\infty}\psi_n(x_3)=0$. By the ODE theory on a bounded interval and the domain expansion technique, we  obtain a unique solution $\psi_n \in H^\infty(\R_-)$, where the solution $\psi_n$ depends on the known functions $\phi_n$ and $\pi_n$. We get $\varphi_n$ in a similar way. Hence, $(\psi_n,\varphi_n,\phi_n,\pi_n) \in (H^\infty(\R_-))^4$ is a  solution of \eqref{SystFunctionsX_3}-\eqref{BoundFunctionsX_3}.

Following \eqref{SeparationOfVR}, we then construct the functions
\[\begin{split}
v_{1,n}(\vk, x) &= \sin(k_1x_1+k_2x_2)\psi_n(\vk, x_3),\\
v_{2,n}(\vk, x) &=\sin(k_1x_1+k_2x_2)\varphi_n(\vk, x_3),\\
v_{3,n}(\vk, x) &= \cos(k_1x_1+k_2x_2)\phi_n(\vk, x_3),\\
r_n(\vk, x_3) &=  \cos(k_1x_1+k_2x_2)\pi_n(\vk, x_3).
\end{split}\]
Keep in mind \eqref{RelationOmegaNu}, let us define also
\[
\begin{split}
\omega_n(\vk, x) &= -\frac1{\lambda_n(\vk)}\rho_0'(x_3)v_{3,n}(\vk, x_3), \\
\nu_n(\vk, x_h) &= \frac1{\lambda_n(\vk)} v_{3,n}(\vk, x_h,0).
\end{split}
\]
Hence 
\[
(\zeta_n(t,\vk, x),u_n(t,\vk, x), q_n(t,\vk, x), \eta_n(t,\vk, x_h))=  e^{\lambda_n(\vk)t}(\omega_n, v_n, r_n, \nu_n)(\vk, x)
\]
is a real-valued solution of \eqref{EqLinearized}. We claim \eqref{RegularSolutionsLinear} since $(\psi_n,\varphi_n,\phi_n,\pi_n) \in (H^\infty(\R_-))^4$.
\end{proof}

\subsection{Maximal growth rate}

We derive the following proposition on the largest characteristic value $\lambda_1$ found in Theorem \ref{ThmLinear}.
\begin{proposition}\label{PropLambda_1}
Let us recall the bilinear form $\bB_{a,k,\lambda}$ on $H^2((-a,0))$ \eqref{1stBilinearForm} and $(\lambda_1,\phi_1)$ from Theorem \ref{ThmLinear}. We have that 
\begin{equation}\label{EquivalenceVariaProb}
\begin{split}
\frac1{gk^2} &= \max_{\phi \in H^2((-a,0))} \frac{\int_{-a}^0 \rho_0' \phi^2}{\lambda_1\bB_{a,k,\lambda_1}(\phi,\phi)},
\end{split}
\end{equation}
and the extremal problem   \eqref{EquivalenceVariaProb} is attained by  $\phi_1$ restricted on $(-a,0)$ up to a constant. 

Furthermore, let us define the following bilinear form on $H^2(\R_-)$,
\[
\begin{split}
\fB_{k,\lambda}(\phi,\theta) := \lambda \int_{\R_-}\rho_0(k^2 \phi\theta +\phi' \theta') &+ \mu \int_{\R_-}((\phi''+k^2\phi) ( \theta''+k^2\theta)+4k^2\phi' \theta') \\
&\qquad\qquad+ \frac{gk^2 \rho_+}{\lambda}\phi(0) \theta(0).
\end{split}
\]
Hence, we have
\begin{equation}\label{VariaProbR_-}
\frac1{gk^2}= \max_{\phi \in H^2(\R_-)} \frac{\int_{\R_-} \rho_0' \phi^2}{\lambda_1 \fB_{k,\lambda_1}(\phi,\phi)}.
\end{equation}
The extremal problem \eqref{VariaProbR_-} is attained by  $\phi_1$ up to a constant. 
\end{proposition}

\begin{proof}[Proof of Proposition \ref{PropLambda_1}]
We divide the proof into two parts, proving \eqref{EquivalenceVariaProb} and \eqref{VariaProbR_-}, respectively.

\noindent\textbf{Part 1}. We show that \eqref{EquivalenceVariaProb} holds.
For all $\lambda>0$, we solve the variational problem 
\begin{equation}
\alpha_1(\lambda,k)= \max\Big( \int_{-a}^0 \rho_0' \phi^2 \Big| \phi \in H^2((-a,0)), \quad \lambda\bB_{a,k,\lambda}(\phi,\phi)=1\Big).
\end{equation}
Let us define the Lagrangian functional
\begin{equation}
\cL_{\bB}(\nu,\phi) =  \int_{-a}^0 \rho_0' \phi^2 - \nu (\lambda\bB_{a,k,\lambda}(\phi,\phi)-1).
\end{equation}
It follows from the Lagrange multiplier theorem that the extrema of the quotient 
\[
\frac{\int_{-a}^0 \rho_0' \phi^2}{\lambda\bB_{a,k,\lambda}(\phi,\phi)}
\]
are necessarily obtained at the stationary points $(\nu_\star,\phi_\star)$ of $\cL_{\bB}$, which satisfy
\begin{equation}\label{PsiConstraintAlpha1}
\lambda\bB_{a,k,\lambda}(\phi_\star,\phi_\star)=1
\end{equation}
and
\begin{equation}\label{EqConstraintLb}
\int_{-a}^0 \rho_0' \phi_\star \theta -  \lambda\nu_\star\bB_{a,k,\lambda}(\phi_{\star}, \theta) =0,
\end{equation}
for all $\theta \in  H^2((-a,0))$.  Restricting $\theta \in C_0^{\infty}((-a,0))$ and following the line of the proof of Proposition \ref{PropInverseOfR}, one deduces from \eqref{EqConstraintLb} that $\phi_{\star}$ has to satisfy 
\begin{equation}
\lambda\nu_\star  Y_{a,k,\lambda}\phi_{\star} = \rho_0'\phi_{\star}
\end{equation}
in a weak sense. We further get that $\phi_{\star}\in H^4((-a,0))$ and satisfies \eqref{PsiConstraintAlpha1} and the boundary conditions \eqref{LeftBoundary}-\eqref{RightBoundary}.  Hence, all stationary points  $(\nu_\star,\phi_\star)$ of $\mathcal{L}_\bB$ satisfy that,  $\lambda\nu_\star$ is an eigenvalue of the compact and self-adjoint operator $ S_{a,k,\lambda}=\cM Y_{a,k,\lambda}^{-1}\cM$ from $L^2((-a,0))$ to itself, with 
\[
\cM^{-1}Y_{a,k,\lambda} \phi_\star = \frac1{\lambda \nu_\star}\cM \phi_\star \in L^2((-a,0))
\] 
being an associated eigenfunction. That implies 
\begin{equation}\label{GammaLeqAlpha_1}
\alpha_1(\lambda,k) \leq \lambda^{-1}\gamma_1(\lambda,k).
\end{equation}

Meanwhile, since the operator $S_{a,k,\lambda}$ is  self-adjoint and positive, we thus obtain that
\[
\gamma_1(\lambda,k)=\sup_{\phi \in L^2((-a,0))} \frac{\langle S_{a,k,\lambda}\phi, \phi\rangle}{\|\phi\|_{L^2((-a,0))}^2}.
\]
Hence, for all $\phi\in L^2((-a,0))$ and  for $\psi= Y_{a,k,\lambda}^{-1}\cM\phi \in H^4((-a,0))$, we have
\[
\langle Y_{a,k,\lambda}\psi, \psi \rangle = \langle S_{a,k,\lambda}\phi, \phi\rangle,
\]
which yields
\[
\gamma_1(\lambda,k) \langle Y_{a,k,\lambda}\psi,\psi \rangle \leq  \frac{\langle S_{a,k,\lambda}\phi, \phi\rangle^2}{\|\phi\|_{L^2((-a,0)}^2} \leq  \| S_{a,k,\lambda}\phi\|_{L^2((-a,0))}^2.
\]
This yields
\[
\gamma_1(\lambda,k) \leq \sup \Big\{ \frac{\|\cM \psi\|_{L^2((-a,0))}^2}{\langle Y_{a,k,\lambda}\psi,\psi \rangle}|\psi \in H^4((-a,0)) \text{ and } \cM^{-1}Y_{a,k,\lambda}\psi \in L^2((-a,0))\Big\}.
\]
Owing to \eqref{EqMathcalB_a}, we have that
\[
\gamma_1(\lambda,k) \leq  \sup \Big\{ \frac{\int_{-a}^0 \rho_0' \psi^2}{\bB_{a,k,\lambda}(\psi,\psi)}|\psi \in H^4((-a,0)) \text{ and } \cM^{-1}Y_{a,k,\lambda}\psi \in L^2((-a,0))\Big\}.
\]
We thus obtain
\begin{equation}\label{Alpha_1LeqGamma}
\lambda^{-1}\gamma_1(\lambda,k) \leq \alpha_1(\lambda,k)
\end{equation}
The two inequalities \eqref{GammaLeqAlpha_1} and \eqref{Alpha_1LeqGamma} tell us that $\alpha_1(k,\lambda)=\lambda^{-1}\gamma_1(k,\lambda)$ for all $\lambda>0$,  from which we deduce $\alpha_1(\lambda_1,k)=\frac1{gk^2}$ and the extremal problem  \eqref{EquivalenceVariaProb} is attained by the function $\phi_1|_{(-a,0)}$ up to a constant.

\noindent\textbf{Part 2}. We  prove that \eqref{VariaProbR_-} holds. We set 
\[
\alpha_2(\lambda,k)= \max_{\phi \in H^2(\R_-)}\Big( \int_{\R_-}\rho_0' \phi^2  \Big| \lambda \fB_{k,\lambda}(\phi,\phi)=1 \Big).
\]
and consider the Lagrangian functional 
\[
\fL_{\fB}(\omega, \phi)= \int_{\R_-}\rho_0' \phi^2  - \omega ( \fB_{k,\lambda}(\phi,\phi)-1).
\]
Thanks to  Lagrange multiplier theorem again, the extrema  of the quotient 
$\frac{\int_{\R_-}\rho_0' \phi^2 }{ \lambda \fB_{k,\lambda}(\phi,\phi)}$
are necessarily obtained at the  stationary points $(\omega_\star, \Phi_\star) \in \R_+ \times H^2(\R_-)$ of $\fL_{\fB}$, which satisfy
\[
\lambda\fB_{k,\lambda}(\Phi_\star,\Phi_\star) =1, \qquad \int_{\R_-}\rho_0'\Phi_\star \theta  - \lambda\omega_\star \fB_{k,\lambda}(\Phi_\star,\theta)=0
\]
for all $\theta\in H^2(\R_-)$.  Restricting further $\theta \in C_0^\infty(\R_-)$,  we make use of a bootstrap argument to get that  $\Phi_\star$ belongs to $H^4(\R_-)$ and $\Phi_\star$ satisfies the ODE
\begin{equation}
\mu (\Phi_\star^{(4)}- 2k^2 \Phi_\star''  +k^4\Phi_\star) +\lambda (k^2\rho_0\Phi_\star - (\rho_0\Phi_\star')') = \frac1{\lambda\omega_\star}\rho_0'\Phi_\star \quad\text{on }\R_-,
\end{equation}
with the boundary condition  \eqref{RightBoundary}.  Since $\text{supp} \rho_0'=[-a,0]$, we see that $\Phi_\star$ is a solution of 
\[
\mu (\Phi_\star^{(4)}- 2k^2 \Phi_\star''  +k^4\Phi_\star) +\lambda (k^2\rho_0\Phi_\star - (\rho_0\Phi_\star')') =0 \quad\text{on }(-\infty,-a).
\]
 Then, $\Phi_\star$ on $(-\infty,-a)$ is of the form \eqref{LeftSol}. Mimicking the computations in the proof of Lemma \ref{LemBoundSmooth}, we deduce $\Phi_\star$ on $(-a,0)$ is a solution of 
\[
\lambda\omega_\star Y_{a,k,\lambda}(\Phi_\star|_{(-a,0)})= \rho_0'\Phi_\star|_{(-a,0)}= \cM^2 \Phi_\star|_{(-a,0)}
\]
with the boundary conditions \eqref{LeftBoundary}-\eqref{RightBoundary}. Set 
\begin{equation}\label{TildeBigPhi}
\tilde\Phi= \cM^{-1}Y_{a,k,\lambda}(\Phi_\star|_{(-a,0)})= \frac1{\lambda\omega_\star} \cM \Phi_\star|_{(-a,0)} \in L^2((-a,0)),
\end{equation}
it yields $\lambda\omega_\star \tilde\Phi= \cM Y_{a,k,\lambda}^{-1}\cM \tilde\Phi = S_{a,k,\lambda}\tilde\Phi$.
That means $\lambda\omega_\star$ is an eigenvalue of the compact and self-adjoint operator $S_{a,k,\lambda}$ from $L^2((-a,0))$ to itself, with $\tilde\Phi \in L^2((-a,0))$ (defined as in \eqref{TildeBigPhi}) being an associated eigenfunction. Hence, we get 
\begin{equation}\label{Alpha_2LeqGamma}
\lambda\alpha_2(\lambda,k) \leq \gamma_1(\lambda,k). 
\end{equation}

Let us recall the function $\phi_1$ from Theorem \ref{ThmLinear}.  One thus has
\begin{equation}\label{Alpha2Geq}
\alpha_2(\lambda,k) \geq \frac{\int_{\R_-}\rho_0'\phi_1^2}{\lambda \fB_{k,\lambda}(\phi_1,\phi_1)}.
\end{equation}
Note that from Proposition \ref{PropSolDecay}
\[
\phi_1(x_3) = A_1 e^{k(x_3+a)}+ A_2e^{\sqrt{k^2+\frac{\lambda_1\rho_-}\mu}(x_3+a)} \quad\text{as } -\infty<x_3<-a.
\]
Let us write $\phi_1|_{(-a,0)}$ as the function $\phi_1$ being restricted on $(-a,0)$. Hence, the direct computations show that 
\begin{equation}\label{EquiBilinearForm}
 \fB_{k,\lambda}(\phi_1,\phi_1) = \bB_{a,k,\lambda}(\phi_1|_{(-a,0)},\phi_1|_{(-a,0)}),
\end{equation}
and we keep in mind the assumption $\text{supp} \rho_0' =[-a,0]$. Then, from \eqref{Alpha2Geq} and \eqref{EquiBilinearForm}, we have
\[
\alpha_2(\lambda,k) \geq \frac{\int_{-a}^0 \rho_0' \phi_1^2}{\lambda \bB_{a,k,\lambda}(\phi_1|_{(-a,0)},\phi_1|_{(-a,0)})}.
\]
It then follows 
\begin{equation}\label{GammaLeqAlpha_2}
\alpha_2(\lambda_1,k) \geq \frac{\int_{-a}^0 \rho_0' \phi_1^2}{\lambda_1 \bB_{a,k,\lambda_1}(\phi_1|_{(-a,0)},\phi_1|_{(-a,0)})} = \frac1{gk^2}. 
\end{equation}
Combining \eqref{Alpha_2LeqGamma} and \eqref{GammaLeqAlpha_2} gives us that $\alpha_2(\lambda_1,k)=\frac1{gk^2}$ and the extremal problem \eqref{VariaProbR_-}  is attained by  $\phi_1$ up to a constant.   We finish the proof of Proposition \ref{PropLambda_1}. 
\end{proof}

Recall the definition of $\Lambda$ from \eqref{DefineLambda}, we  prove that $\Lambda$ is the maximal growth rate of the linearized equations \eqref{EqLinearized} in the following sense:
\begin{proposition}\label{PropSharpGrowthRate}
For arbitrary solution $(\zeta,u,\eta)$ of the linearized equations \eqref{EqLinearized}, the following inequalities hold 
\begin{equation}\label{IneSGR_ZetaU}
\begin{split}
&\|\zeta(t)\|_{H^1(\Omega)}^2 + \|u(t)\|_{H^1(\Omega)}^2 +\|\partial_t u(t)\|_{L^2(\Omega)}^2 +\int_0^t \| u(s)\|_{H^1(\Omega)}^2 ds \\
&\qquad\lesssim (\|\zeta(0)\|_{H^1(\Omega)}^2+\|\partial_t u(0)\|_{L^2(\Omega)}^2+ \|u(0)\|_{H^1(\Omega)}^2) e^{2\Lambda t},
\end{split}
\end{equation}
and 
\begin{equation}\label{IneSGR_Eta}
\begin{split}
&\|(\eta,\partial_t \eta)(t)\|_{H^{1/2}(\Gamma)}^2+ \int_0^t  \|\partial_t\eta(s)\|_{H^{1/2}(\Gamma)}^2 \lesssim (\|\eta(0)\|_{H^{1/2}(\Gamma)}^2 +\|\partial_t u(0)\|_{L^2(\Omega)}^2+\|u(0)\|_{H^1(\Omega)}^2) e^{2\Lambda t}.
\end{split}
\end{equation}
\end{proposition}
The proof of Proposition \ref{PropSharpGrowthRate} relies on the two lemmas below.
\begin{lemma}\label{LemEqD_tU_linearized}
For arbitrary solution $(\zeta,u,\eta)$ of the linearized equations \eqref{EqLinearized}, there holds 
\begin{equation}\label{EqD_tU_linearized}
\begin{split}
\frac12 \frac{d}{dt}\Big( \int_{\Omega}\rho_0|\partial_t u|^2  -\int_{\Omega}g\rho_0'|u_3|^2  +\int_{\Gamma}g\rho_+|u_3|^2  \Big) +\frac{\mu}2 \int_{\Omega} |\bS\partial_t u|^2  = 0.
\end{split}
\end{equation}
\end{lemma}
\begin{proof}
We differentiate $\eqref{EqLinearized}_2$ in time, multiply the resulting equation by $\partial_t u$ and then use $\eqref{EqLinearized}_1$ to obtain 
\[
\begin{split}
\int_{\Omega}\rho_0 \partial_t^2 u \cdot \partial_t u + \int_{\Omega}  \nabla \partial_tq \cdot\partial_t u - \mu \int_{\Omega} \Delta \partial_t u \cdot\partial_t u -\int_{\Omega}g\rho_0' u_3 \partial_t u_3=0.
\end{split}
\]
That is equivalent to 
\begin{equation}\label{1_EqD_tU_linearized}
\begin{split}
\frac12 \frac{d}{dt}\Big( \int_{\Omega}\rho_0|\partial_t u|^2  -\int_{\Omega}g\rho_0'|u_3|^2   \Big)  + \int_{\Omega}  \nabla \partial_tq \cdot\partial_t u - \mu \int_{\Omega} \Delta \partial_t u \cdot\partial_t u=0.
\end{split}
\end{equation}
We use the integration by parts over $\Omega$ to have 
\[
\begin{split}
\int_{\Omega}  \nabla \partial_tq \cdot\partial_t u -  \mu \int_{\Omega} \Delta \partial_t u \cdot\partial_t u  &= \int_\Gamma (\partial_t q\text{Id}- \mu \bS\partial_tu)e_3 \cdot \partial_t u -\int_\Omega \partial_t q \text{div}\partial_t u +\frac{\mu}2 \int_\Omega|\bS\partial_t u|^2 
\end{split}
\]
Thanks to  $\eqref{EqLinearized}_{3,4,5}$, we obtain 
\begin{equation}\label{2_EqD_tU_linearized}
\begin{split}
\int_{\Omega}  \nabla \partial_tq \cdot\partial_t u -  \mu \int_{\Omega} \Delta \partial_t u \cdot\partial_t u &= \int_\Gamma g\rho_+ \partial_t \eta\partial_t u_3 +\frac{\mu}2 \int_\Omega|\bS\partial_t u|^2\\
&= \int_\Gamma g\rho_+u_3 \partial_t u_3  +\frac{\mu}2 \int_\Omega|\bS\partial_t u|^2. 
\end{split}
\end{equation}
Substituting \eqref{2_EqD_tU_linearized} into \eqref{1_EqD_tU_linearized}, we conclude \eqref{EqD_tU_linearized}. 
\end{proof}

\begin{lemma}\label{LemIneLambdaSquare}
For any  $u$ such that $\text{div}u=0$, there holds 
\begin{equation}\label{IneLambdaSquare}
 \int_{\Omega}g\rho_0'|u_3|^2 \leq \int_{\Gamma}g\rho_+|u_3|^2+  \Lambda^2 \int_{\Omega}\rho_0|u|^2 + \frac12 \Lambda \int_{\Omega}\mu |\bS u|^2.
\end{equation}
\end{lemma}
\begin{proof}
Let $\vk=(k_1,k_2)\in (L^{-1}\bZ)^2$ be fixed and   $\hat f$ be the horizontal Fourier transform of $f$, i.e. 
\[
\hat f(\vk, x_3)= \int_{\fT^2} f(x_h,x_3)e^{-i \vk \cdot x_h} dx_h.
\]
We write 
\[
\hat u_1(\vk, x) =-i\psi(\vk,x_3), \quad \hat u_2(\vk, x)=-i\varphi(\vk,x_3), \quad \hat u_3(\vk, x)= \phi(\vk,x_3).
\]
Notice that for $\vk=0$, 
\[
\phi(0,0)= \int_\Gamma u_3 = \int_\Omega \text{div}u=0.
\]
Hence, together with Parseval's theorem, we have
\begin{equation}\label{Eq_grhou_3}
\begin{split}
\int_{\Gamma}g\rho_+|u_3|^2  &=   \frac1{4\pi^2L^2} \sum_{\vk  \in (L^{-1}\bZ)^2\setminus\{0\}}g \rho_+|\phi(\vk,0)|^2. 
\end{split}
\end{equation}
We may reduce to estimate \eqref{Eq_grhou_3} when $\psi, \varphi$ and $\phi$ are  real-valued  and then continue the estimate to the real and imaginary parts of $\psi, \varphi$ and $\phi$.

For each $\vk \in   (L^{-1}\bZ)^2 \setminus\{0\}$, we deduce from Proposition \ref{PropLambda_1} that
\[
\begin{split}
\int_{\R_-} g\rho_0' \phi^2(\vk,x_3)  &\leq g \rho_+(\phi(\vk,0))^2  + \lambda_1^2 \int_{\R_-}\rho_0\Big(\phi^2+\frac{(\phi')^2}{k^2} \Big)(\vk,x_3) \\
&\qquad+ \lambda_1\mu \int_{\R_-}\Big( \Big( \frac{\phi''}k+k\phi\Big)^2+4(\phi')^2\Big)(\vk,x_3) .
\end{split}
\]
It thus follows from the definition of $\Lambda$ \eqref{DefineLambda} that  
\begin{equation}\label{Ine1stLemLambdaSquare}
\begin{split}
\int_{\R_-} g\rho_0' \phi^2(\vk,x_3)  &\leq g \rho_+(\phi(\vk,0))^2 + \Lambda^2 \int_{\R_-}\rho_0\Big(\phi^2+\frac{(\phi')^2}{k^2} \Big)(\vk,x_3) \\
&\qquad+ \Lambda\mu \int_{\R_-}\Big( \Big(\frac{\phi''}k+k\phi\Big)^2+4(\phi')^2\Big)(\vk,x_3) 
\end{split}
\end{equation}
for all $\vk \in    (L^{-1}\bZ)^2\setminus\{0\}$. 

Meanwhile, for $\vk \neq 0$, notice that $k_1\psi+k_2\varphi+\phi'=0$. One thus has 
\begin{equation}\label{1_InePsiTheta}
(\phi')^2 \leq (k_1\psi+k_2\varphi)^2+ (k_1\varphi-k_2\psi)^2 = k^2(\psi^2+ \varphi^2),
\end{equation}
and 
\begin{equation}\label{3_InePsiTheta}
\begin{split}
2(\phi')^2 =2k_1^2\psi^2 +2k_2^2\varphi^2 + 4k_1k_2\psi\varphi \leq 2k_1^2\psi^2 +2k_2^2\varphi^2 + (k_1\varphi+k_2\psi)^2.
\end{split}
\end{equation}
Furthermore, we obtain that 
\[
(\phi'')^2 \leq  (k_1\psi'+k_2\varphi')^2+ (k_1\varphi'-k_2\psi')^2 = k^2((\psi')^2+(\varphi')^2).
\]
This yields
\[
\Big( \frac1k \phi''+k\phi \Big)^2= \frac1{k^2}(\phi'')^2 + 2\phi\phi''+ k^2\phi^2 \leq (\psi')^2+(\varphi')^2-2\phi(k_1\psi'+k_2\varphi') + k^2\phi^2,
\]
so that
\begin{equation}\label{2_InePsiTheta}
\Big( \frac1k \phi''+k\phi \Big)^2\leq  (k_1\phi-\psi')^2+(k_2\phi-\varphi')^2. 
\end{equation}
Then, in view of  Fubini's and Parseval's theorem again, we find that due to \eqref{1_InePsiTheta},
\begin{equation}\label{Ine2ndLemLambdaSquare}
\begin{split}
 \int_{\Omega}\rho_0|u|^2 &= \frac1{4\pi^2L^2} \sum_{\vk  \in (L^{-1}\bZ)^2} \int_{\R_-}\rho_0(\psi^2+ \varphi^2+ \phi^2)(\vk,x_3) \\
&\geq \frac1{4\pi^2L^2} \sum_{\vk  \in (L^{-1}\bZ)^2\setminus\{0\}} \int_{\R_-}\rho_0\Big( \phi^2+\frac{(\phi')^2}{k^2} \Big)(\vk,x_3), 
\end{split}
\end{equation}
and that due to \eqref{3_InePsiTheta} and \eqref{2_InePsiTheta},
\begin{equation}\label{Ine3rdLemLambdaSquare}
\begin{split}
\frac12  \int_{\Omega}\mu |\bS u|^2  &=   \frac{\mu}{4\pi^2L^2} \sum_{\vk  \in (L^{-1}\bZ)^2 } \int_{\R_-} 
\left( \begin{split}
& 2(\phi')^2+ 2k_1^2\psi^2+2k_2^2\varphi^2 + (k_1\varphi+k_2\psi)^2 \\
&+(k_1\phi-\psi')^2+ (k_2\phi-\varphi')^2
\end{split} \right)\\
&\geq   \frac{\mu}{4\pi^2L^2} \sum_{\vk  \in (L^{-1}\bZ)^2\setminus\{0\}} \int_{\R_-}\Big( \Big( \frac{\phi''}k+k\phi\Big)^2+4(\phi')^2\Big) .
\end{split}
\end{equation}
Combining \eqref{Eq_grhou_3}, \eqref{Ine1stLemLambdaSquare}, \eqref{Ine2ndLemLambdaSquare} and \eqref{Ine3rdLemLambdaSquare}, the inequality  \eqref{IneLambdaSquare} follows, we end the proof here. 
\end{proof}

We are in position to prove Proposition \ref{PropSharpGrowthRate}. 
\begin{proof}[Proof of Proposition \ref{PropSharpGrowthRate}]
Owing to \eqref{EqD_tU_linearized} and \eqref{IneLambdaSquare}, we have that
\begin{equation}\label{1stInePropSGR}
 \begin{split}
\int_{\Omega}\rho_0|\partial_t u(t)|^2 + \int_0^t \int_{\Omega}\mu|\bS\partial_t u(s)|^2ds &= y_1 + \int_{\Omega}g\rho_0'|u_3(t)|^2-\int_{\Gamma}g\rho_+|u_3(t)|^2 \\
&\leq y_1+ \Lambda^2 \int_{\Omega}\rho_0|u(t)|^2 + \frac12\Lambda \int_{\Omega}\mu|\bS u(t)|^2,
\end{split} 
\end{equation}
where 
\[
y_1 = \int_{\Omega}\rho_0|\partial_t u(0)|^2 - \int_{\Omega}g\rho_0'|u_3(0)|^2 +\int_{\Gamma}g\rho_+|u_3(0)|^2 .
\]
Using Cauchy-Schwarz's inequality, we have that
\begin{equation}\label{2ndInePropSGR}
\begin{split}
 \int_{\Omega}\mu|\bS u(t)|^2 &=\int_{\Omega}\mu|\bS u(0)|^2+ 2\int_0^t\int_{\Omega} \mu \bS u(s) : \bS \partial_t u(s) ds \\
&\leq  \int_{\Omega}\mu|\bS u(0)|^2 + \frac1{\Lambda} \int_0^t \int_{\Omega}\mu|\bS \partial_t u(s)|^2 ds + \Lambda \int_0^t\int_{\Omega} \mu |\bS u(s)|^2 ds
\end{split}
\end{equation}
and that
\begin{equation}\label{3rdInePropSGR}
\frac{d}{dt} \int_{\Omega}\rho_0|u|^2  \leq \frac1{\Lambda} \int_{\Omega}\rho_0|\partial_t u|^2 + \Lambda \int_{\Omega}\rho_0|u|^2.
\end{equation}
The three inequalities \eqref{1stInePropSGR}, \eqref{2ndInePropSGR} and \eqref{3rdInePropSGR} imply that 
\begin{equation}\label{4thInePropSGR}
\begin{split}
\frac{d}{dt} \int_{\Omega}\rho_0|u(t)|^2 + \frac12 \int_{\Omega}\mu|\bS u(t)|^2 &\leq y_2 +  2\Lambda\int_{\Omega}\rho_0|u(t)|^2 + \Lambda \int_0^t\int_{\Omega} \mu |\bS u(s)|^2ds.
\end{split}
\end{equation}
where 
\[
y_2= \frac{y_1}{\Lambda} + \int_{\Omega}\mu|\bS u(0)|^2.
\]
In view of Gronwall's inequality, we obtain from \eqref{4thInePropSGR} that
\begin{equation}\label{5thInePropSGR}
\begin{split}
 \int_{\Omega}\rho_0|u(t)|^2  + \frac12 \int_0^t \int_{\Omega}\mu|\bS u(s)|^2ds &\leq e^{2\Lambda t}\int_{\Omega}\rho_0|u(0)|^2 + \frac{y_2}{2\Lambda}(e^{2\Lambda t}-1).
\end{split}
\end{equation}
Hence, 
\begin{equation}\label{6thInePropSGR}
\begin{split}
\frac1{\Lambda} \int_{\Omega}\rho_0|\partial_t u(t)|^2  + \frac12\int_{\Omega}\mu|\bS u(t)|^2
&\leq y_2 +\Lambda\int_{\Omega}\rho_0|u(t)|^2+ \Lambda\int_0^t \int_{\Omega}\mu|\bS u(s)|^2ds \\
&\leq \Big(y_2+2\Lambda \int_{\Omega}\rho_0|u(0)|^2\Big) e^{2\Lambda t}
\end{split}
\end{equation}
Using the trace theorem, we have 
\begin{equation}\label{7thInePropSGR}
y_1 + y_2 \lesssim \|u(0)\|_{H^1(\Omega)}^2 + \|\partial_t u(0)\|_{L^2(\Omega)}^2. 
\end{equation}
Because of \eqref{5thInePropSGR}, \eqref{6thInePropSGR} and \eqref{7thInePropSGR}, we observe
\begin{equation}\label{9thInePropSGR}
\begin{split}
&\|(u, \bS u, \partial_t u)(t)\|_{L^2(\Omega)}^2+\int_0^t \|\bS u(s)\|_{L^2(\Omega)}^2 ds  \lesssim( \|\partial_t u(0)\|_{L^2(\Omega)}^2 +\|u(0)\|_{H^1(\Omega)}^2) e^{2\Lambda t}.
\end{split}
\end{equation}
In view of Korn's inequality (see \eqref{KornIne}), that implies
\begin{equation}\label{9thInePropSGR}
\begin{split}
&\|(u,\nabla u, \partial_t u)(t)\|_{L^2(\Omega)}^2 +\int_0^t \|\nabla u(s)\|_{L^2(\Omega)}^2 ds   \lesssim( \|\partial_t u(0)\|_{L^2(\Omega)}^2 +\|u(0)\|_{H^1(\Omega)}^2) e^{2\Lambda t}.
\end{split}
\end{equation}
Using $\eqref{EqLinearized}_1$ and \eqref{9thInePropSGR} also, we get 
\begin{equation}\label{10thInePropSGR}
\begin{split}
\|\zeta(t)\|_{H^1(\Omega)}^2 &\lesssim \|\zeta(0)\|_{H^1(\Omega)}^2+ \int_0^t \|u(s)\|_{H^1(\Omega)}^2 ds  \\
&\lesssim (\|\zeta(0)\|_{H^1(\Omega)}^2 +\|\partial_t u(0)\|_{L^2(\Omega)}^2 +\|u(0)\|_{H^1(\Omega)}^2)e^{2\Lambda t}.
\end{split}
\end{equation}
The inequality \eqref{IneSGR_ZetaU} follows from \eqref{9thInePropSGR} and \eqref{10thInePropSGR}. 

To prove \eqref{IneSGR_Eta}, we use the trace theorem to obtain that
\[
\begin{split}
\| \partial_t \eta(t)\|_{H^{1/2}( \Gamma)}^2 + \int_0^t\|\partial_t \eta(s)\|_{H^{1/2}( \Gamma)}^2 ds &=\|u_3(t)\|_{H^{1/2}( \Gamma)}^2 + \int_0^t\|u_3(s)\|_{H^{1/2}( \Gamma)}^2 ds \\
&\leq \|u_3(t)\|_{H^1(\Omega)}^2 + \int_0^t\|u_3(s)\|_{H^1(\Omega)}^2 ds.
\end{split}
\]
Together with \eqref{5thInePropSGR}, \eqref{7thInePropSGR} and \eqref{9thInePropSGR}, we deduce that
\begin{equation}\label{11thInePropSGR}
\| \partial_t \eta(t)\|_{H^{1/2}( \Gamma)}^2 + \int_0^t\|\partial_t \eta(s)\|_{H^{1/2}( \Gamma)}^2 ds \lesssim (\|\partial_t u(0)\|_{L^2(\Omega)}^2 +\|u(0)\|_{H^1(\Omega)}^2) e^{2\Lambda t}.
\end{equation}
The resulting inequality tells us that
\begin{equation}\label{12thInePropSGR}
\begin{split}
\|\eta(t)\|_{H^{1/2}(\Gamma)}^2 &\leq \|\eta(0)\|_{H^{1/2}(\Gamma)}^2 + \int_0^t \|\partial_t \eta(s)\|_{H^{1/2}(\Gamma)}^2 ds \\
&\lesssim (\|\eta(0)\|_{H^{1/2}(\Gamma)}^2 +\|\partial_t u(0)\|_{L^2(\Omega)}^2 +\|u(0)\|_{H^1(\Omega)}^2 )e^{2\Lambda t}.
\end{split}
\end{equation}
The inequality \eqref{IneSGR_Eta} follows from \eqref{11thInePropSGR} and \eqref{12thInePropSGR}.  Proposition \ref{PropSharpGrowthRate} is proven. 
\end{proof}

\section{A priori energy estimates}\label{SectAprioriEnergy}

Let us recall the perturbation terms $U=(\zeta, u, q, \eta)$ from \eqref{PerturTerms}. We define the full energy functional $\cE_f(U(t))>0$ such that 
\begin{equation}\label{EnergyFull}
\begin{split}
\cE_f^2(U(t)) &:= \|\eta(t)\|_{H^{9/2}(\Gamma)}^2+\sum_{l=0}^2 \|\partial_t^l \eta\|_{H^{4-2l}(\Gamma)}^2 +\sum_{l=0}^2 \|\partial_t^l (\zeta, u)(t)\|_{H^{4-2l}(\Omega)}^2\\
&\qquad +\|q(t)\|_{H^3(\Omega)}^2 +\|\partial_tq(t)\|_{H^1(\Omega)}^2
\end{split}
\end{equation}
and its corresponding dissipation $\cD_f(U(t))>0$, 
\begin{equation}\label{DissipationFull}
\begin{split}
\cD_f^2(U(t)) &:= \sum_{l=0}^2 \|\partial_t^l u(t)\|_{H^{5-2l}(\Omega)}^2 +\|\partial_t q\|_{H^2(\Omega)}^2 + \| q\|_{H^4(\Omega)}^2.
\end{split}
\end{equation}
For notational convenience, we only write $\cE_f(t)$ and $\cD_f(t)$ in this section.

The local existence of regular solution to \eqref{EqPertur} then follows from \cite[Theorem 6.3]{GT13}, that we restate below.  Let us recall the definition of $K$ from \eqref{NotationABKJ} and of $\cA$ from  \eqref{MatrixA} and define $\mathbb{M}= K\cA$, $R= \partial_t \mathbb{M} \mathbb{M}^{-1}$ and $D_t u= \partial_t u - Ru$.  We also define an orthogonal projection onto the tangent space of the surface $\{x_3 = \eta_0(x_1,x_2)\}$ according to
\begin{equation}
\Pi_0 v = v- \frac{v\cdot \cN_0}{|\cN_0|^2} \cN_0 \quad\text{for }\cN_0=(-\partial_1 \eta_0,\partial_2 \eta_0,1)^T.
\end{equation}
  Let us write
\[
\begin{split}
G^{2,0} &= g\rho_+\eta\cN \text{ on }\Gamma,\\
G^{2,1} &= D_t G^{2,0} + \mu \bS_\cA(Ru)\cN + (\mu \bS_\cA u-q\text{Id})\partial_t\cN +\mu \bS_{\partial_t\cA}u \cN \text{ on }\Gamma.
\end{split}
\]

\begin{proposition}\label{PropLocalSolution}
Suppose that there is a sufficiently small constant $\nu_1 \in (0,1)$ such that $(\zeta_0, u_0, q_0, \eta_0)$ satisfying
\[
\|\zeta_0\|_{H^4(\Omega)}^2 +\|u_0\|_{H^4(\Omega)}^2+\|q_0\|_{H^3(\Omega)}^2+\|\eta_0\|_{H^{9/2}(\Gamma)}^2 \leq \nu_1.
\]
Suppose also that   the following compatibility conditions hold for $j=0$ and 1,
\begin{equation}\label{CompaCond}
\begin{cases}
\text{div}_{\cA_0} D_t^j u_0 =0 &\quad\text{in }\Omega,\\
\Pi_0(G^{2,j}(0) +\mu \bS_{\cA_0} D_t^j u_0 \cN_0) =0 &\quad\text{on }\Gamma.\end{cases}
\end{equation}
Then,  there exist $\nu_2>0$  and $T_{\max} >0$ such that if $\cE_f(0) \leq \nu_2$,  Eq. \eqref{EqPertur} with the initial data $(\zeta_0, u_0, q_0, \eta_0)$ satisfying the compatibility conditions \eqref{CompaCond}   has  a unique solution $(\zeta, u,q, \eta)$  on the time interval $[0, T_{\max})$. Moreover, we have
\[
\cE_f(t) \lesssim(1+T_{\max})\cE_f(0),
\]
 and $\eta$ is such that the mapping $\Theta(\cdot, t)$ defined by \eqref{ThetaTransform} is a $C^2$-diffeomorphism for each $t\in [0,T_{\max})$. 
\end{proposition}

With that  regular solution $(\zeta,u,q,\eta)$ of \eqref{EqPertur} on a finite time interval $[0,T_{\max})$,  we aim at showing the \textit{a priori} energy estimates for the nonlinear  equations \eqref{EqPertur}.

 \begin{proposition}\label{PropAprioriEnergy}
Let $\delta_0$ be sufficiently small,  there are $\varepsilon>0$ sufficiently small and another $C_0>0$  such that for all $\delta \in (0,\delta_0)$ if $\sup_{0\leq s\leq t}\cE_f(s)\leq\delta$, we have
\begin{equation}\label{EstAprioriEnergy}
\begin{split}
\cE_f^2(t)+\int_0^t\cD_f^2(s)ds &\leq C_0 \Big( \varepsilon^{-5} \cE_f^2(0)+ \varepsilon \int_0^t \cE_f^2(s)ds+  \varepsilon^{-5}  \int_0^t \cE_f(\cE_f^2+\cD_f^2)(s)ds+\varepsilon^{-5} \cE_f^3(t) \Big)\\
&\qquad\quad  + C_0 \varepsilon^{-59} \int_0^t (\|(\zeta,u)(s)\|_{L^2(\Omega)}^2 +\|\eta(s)\|_{L^2(\Gamma)}^2 )ds.
\end{split}
\end{equation}
\end{proposition}
\noindent\textbf{Strategy of the proof}. 
Respectively, we derive the \textit{a priori} energy estimates for the space-time derivatives of $\eta$  in Propositions \ref{PropEstTransportEtaH^4}, \ref{PropEstEta^9/2}, for the temporal derivatives of $u$ in Proposition \ref{PropEstPartial_t^lU_L2}, for the horizontal space-time derivatives of $u$ in Proposition \ref{PropEstHorizonU} and for the space-time derivatives of $\zeta$ in Proposition \ref{PropEstZeta}.  Then, we derive some estimates thanks to the elliptic regularity theory (see Propositions \ref{PropCompareE}, \ref{PropCompareD}). In view of these above estimates,  we obtain \eqref{EstAprioriEnergy} and  complete the proof of Proposition \ref{PropAprioriEnergy}.

In what follows, the constants $C_i$ $(i\geq 1)$ are to indicate some constants, which are referred later. 

\subsection{Energy estimates of the perturbation transport}
We first derive the \textit{a priori} energy estimates for $\eta$.
\begin{proposition}\label{PropEstTransportEtaH^4}
The following inequalities hold
\begin{equation}\label{EstTransportEtaH^4}
\begin{split}
\|\eta(t)\|_{H^4(\Gamma)}^2 &\leq C_1\Big( \cE_f^2(0) +  \int_0^t( \varepsilon \|\eta(s)\|_{H^4(\Gamma)}^2 + \varepsilon^{-1} \|\nabla u(s)\|_{H^4(\Omega)}^2)ds\int_0^t \cE_f^3(s)ds\Big),
\end{split}
\end{equation}
\begin{equation}\label{EstTransportEta_tH^2}
\begin{split}
\|\partial_t\eta(t)\|_{H^2(\Gamma)}^2 &\leq C_2 \Big( \cE_f^2(0)+ \int_0^t (\varepsilon\|\partial_t\eta(s)\|_{H^2(\Gamma)}^2 +\varepsilon^{-1} \|\nabla \partial_t u(s)\|_{H^2(\Omega)}^2)ds+ \int_0^t \cE_f^3(s) ds \Big),
\end{split}
\end{equation}
and
 \begin{equation}\label{EstTransportEta_t^2L^2}
\begin{split}
\|\partial_t^2\eta (t)\|_{L^2(\Gamma)}^2 &\leq  C_3 \Big( \cE_f^2(0)+ \int_0^t (\varepsilon \|\partial_t^2\eta(s)\|_{L^2(\Gamma)}^2 +\varepsilon^{-1} \|\nabla \partial_t^2 u(s)\|_{L^2(\Omega)}^2 )ds +\int_0^t \cE_f^3(s) ds\Big).
\end{split}
\end{equation}
\end{proposition}
\begin{proof}
Let us prove \eqref{EstTransportEtaH^4}. For any $\alpha \in \N^2, |\alpha|\leq 4$, we have by $\eqref{EqPertur}_4$, 
\[
\partial_t\partial^\alpha \eta = \partial^\alpha u_3- (u_1\partial^\alpha \partial_1\eta +u_2 \partial^\alpha \partial_2\eta) - \underbrace{\sum_{0\neq \beta \leq \alpha}(\partial^\beta u_1\partial^{\alpha-\beta} \partial_1\eta +\partial^\beta u_2 \partial^{\alpha-\beta} \partial_2\eta)}_{=:R_1^\alpha}.
\]
Using the integration by parts, we obtain
\[
\begin{split}
\frac12 \frac{d}{dt}\|\partial^\alpha \eta\|_{L^2(\Gamma)}^2 &= \frac12 \int_\Gamma (\partial_1 u_1+\partial_2 u_2) |\partial^\alpha\eta|^2 +\int_\Gamma (\partial^\alpha u_3- R_1^\alpha)\partial^\alpha\eta .
\end{split}
\]
So that, we have
\begin{equation}\label{EqDtPartialEta}
\begin{split}
\frac12 \frac{d}{dt}\|\partial^\alpha \eta\|_{L^2(\Gamma)}^2 &\lesssim (\|\partial_1 u_1\|_{L^\infty(\Gamma)}+\|\partial_2 u_2\|_{L^\infty(\Gamma)})\|\partial^\alpha \eta\|_{L^2(\Gamma)}^2 \\
&\qquad+( \|\partial^\alpha u_3\|_{L^2(\Gamma)} +\|R_1^\alpha\|_{L^2(\Gamma)}) \|\partial^\alpha\eta\|_{L^2(\Gamma)}.
\end{split}
\end{equation}
We make use of the trace theorem to obtain that 
\begin{equation}\label{1stEstPropTransport}
\|\partial_j u_j\|_{L^\infty(\Gamma)} \lesssim \|u\|_{H^3(\Gamma)} \lesssim \|u\|_{H^4(\Omega)}, \quad
\|\partial^\alpha u_3\|_{L^2(\Gamma)} \lesssim \|\partial^\alpha u\|_{H^1(\Omega)}
\end{equation}
and that
\begin{equation}\label{3rdEstPropTransport}
\|R_1^\alpha\|_{L^2(\Gamma)} \lesssim \sum_{0\neq \beta\leq \alpha} \|\partial^\beta u\|_{L^2(\Gamma)} \|\partial^{\alpha-\beta}\eta\|_{H^1(\Gamma)} \lesssim \|\partial^\alpha u\|_{H^1(\Omega)} \|\eta\|_{H^{|\alpha|}(\Gamma)}.
\end{equation}
By summing over $\alpha\in \N^2, |\alpha|\leq 4$, it follows from \eqref{EqDtPartialEta}, \eqref{1stEstPropTransport}\ and \eqref{3rdEstPropTransport} that 
\[
\begin{split}
\frac{d}{dt}\|\eta\|_{H^4(\Gamma)}^2 &\lesssim \|\nabla u\|_{H^4(\Omega)}\|\eta\|_{H^4(\Gamma)} +\|u\|_{H^4(\Omega)} \|\eta\|_{H^4(\Gamma)}^2\lesssim \|\nabla u\|_{H^4(\Omega)}\|\eta\|_{H^4(\Gamma)} + \cE_f^3.
\end{split}
\]
Using Cauchy-Schwarz's inequality and then integrating the resulting inequality from 0 to $t$, we obtain  $\eqref{EstTransportEtaH^4}$.

We  show \eqref{EstTransportEta_tH^2}. Let $\alpha \in \N^2, |\alpha|\leq 2$, we get 
\[ \begin{split}
\partial_t^2 \partial^\alpha\eta &= \partial^\alpha\partial_t u_3-( u_1\partial^\alpha \partial_1\partial_t \eta+ u_2\partial^\alpha\partial_2\partial_t\eta )- (\partial_t u_1\partial^\alpha\partial_1\eta + \partial_t u_2 \partial^\alpha \partial_2\eta) \\
&\qquad- \underbrace{ \sum_{0\neq \beta \leq \alpha}(\partial^\beta u_1\partial^{\alpha-\beta} \partial_1\partial_t \eta +\partial^\beta u_2 \partial^{\alpha-\beta} \partial_2\partial_t\eta)}_{=:R_2^\alpha} \\
&\qquad- \underbrace{  \sum_{0\leq \beta \leq \alpha}(\partial^\beta \partial_t u_1\partial^{\alpha-\beta} \partial_1\eta +\partial^\beta\partial_t  u_2 \partial^{\alpha-\beta} \partial_2\eta)}_{=:R_3^\alpha}.
\end{split}
\]
Via the integration by parts, this yields
\[
\begin{split}
\frac12 \frac{d}{dt}\|\partial^\alpha \partial_t \eta\|_{L^2(\Gamma)}^2 &= \frac12 \int_\Gamma (\partial_1 u_1+\partial_2 u_2) |\partial^\alpha\partial_t \eta|^2+ \int_\Gamma 
\partial_t (\partial_1u_1+\partial_2u_2) \partial^\alpha \eta  \partial^\alpha\partial_t\eta \\
&\qquad + \int_\Gamma (\partial^\alpha\partial_t u_3 - R_2^\alpha-R_3^\alpha)\partial^\alpha\partial_t \eta.
\end{split}
\]
Using Sobolev embedding and the trace theorem, we have 
\begin{equation}\label{8thEstPropTransport}
\int_\Gamma (\partial_1 u_1+\partial_2 u_2) |\partial^\alpha\partial_t \eta|^2  \lesssim \|u\|_{H^3(\Gamma)}\|\partial^\alpha\partial_t \eta\|_{L^2(\Gamma)}^2 \lesssim \|u\|_{H^4(\Omega)} \|\partial_t\eta\|_{H^2(\Gamma)}^2 \lesssim \cE_f^3.
\end{equation}
\begin{equation}\label{7thEstPropTransport}
\begin{split}
\int_\Gamma  \partial_t(\partial_1u_1+\partial_2u_2)  \partial^\alpha \eta  \partial^\alpha\partial_t\eta &\lesssim  \|\partial_t u\|_{H^1(\Gamma)}\|\partial^\alpha \eta\|_{H^2(\Gamma)}\|\partial_t\partial^\alpha\eta\|_{L^2(\Gamma)}\\
& \lesssim \|\partial_t u\|_{H^2(\Omega)}\|\eta\|_{H^4(\Gamma)}\|\partial_t\eta\|_{H^2(\Gamma)},
\end{split}
\end{equation}
and 
\begin{equation}\label{6thEstPropTransport}
\begin{split}
\int_\Gamma \partial^\alpha\partial_t u_3 \partial^\alpha\partial_t \eta  \lesssim \|\partial_t u_3\|_{H^2(\Gamma)}\|\partial_t\eta\|_{H^2(\Gamma)} \lesssim \|\partial_t u\|_{H^3(\Omega)}\|\partial_t\eta\|_{H^2(\Gamma)}. 
\end{split}
\end{equation}
We obtain also 
\begin{equation}\label{4thEstPropTransport}
\begin{split}
\|R_2^\alpha\|_{L^2(\Gamma)} &\lesssim  \sum_{0\neq \beta\leq \alpha} \|\partial^\beta u\|_{L^2(\Gamma)} \|\partial^{\alpha-\beta}\partial_t \eta\|_{H^1(\Gamma)}  \lesssim \|u\|_{H^3(\Omega)} \|\partial_t \eta\|_{H^2(\Gamma)}
\end{split}
\end{equation}
and that
\begin{equation}\label{5thEstPropTransport}
\begin{split}
\|R_3^\alpha\|_{L^2(\Gamma)} &\lesssim \sum_{0\leq \beta\leq \alpha} \|\partial^\beta \partial_t u\|_{L^2(\Gamma)}  \|\partial^{\alpha-\beta}\eta\|_{H^1(\Gamma)}\lesssim \| \partial_tu \|_{H^3(\Omega)} \|\eta\|_{H^3(\Gamma)}.
\end{split}
\end{equation}
Combining \eqref{8thEstPropTransport}, \eqref{7thEstPropTransport}, \eqref{6thEstPropTransport}, \eqref{4thEstPropTransport} and \eqref{5thEstPropTransport}, we deduce that 
\[
 \frac{d}{dt} \|\partial_t\eta\|_{H^2(\Gamma)}^2 \lesssim   \|\partial_t u\|_{H^3(\Omega)} \|\partial_t\eta\|_{H^2(\Gamma)}+\cE_f^3.
\]
Combining \eqref{6thEstPropTransport}, \eqref{4thEstPropTransport} and \eqref{5thEstPropTransport}, we deduce that 
\[
 \frac{d}{dt} \|\partial_t\eta\|_{H^2(\Gamma)}^2 \lesssim   \|\nabla \partial_t u\|_{H^2(\Omega)} \|\partial_t\eta\|_{H^2(\Gamma)}+\cE_f^3.
\]
Using Cauchy-Schwarz's inequality and then integrating from 0 to $t$, we obtain $\eqref{EstTransportEta_tH^2}$.

We have $\eqref{EstTransportEta_t^2L^2}$ by following the same strategy as for  proving $\eqref{EstTransportEta_tH^2}$. The proof of Proposition \ref{PropEstTransportEtaH^4} is complete.
\end{proof}

\begin{proposition}\label{PropEstEta^9/2}
There holds 
\begin{equation}\label{EstEtaH^9/2}
\begin{split}
 \|\eta(t)\|_{H^{9/2}(\Gamma)}^2 &\leq  C_4\Big( \cE_f^2(0)+ \int_0^t (\varepsilon \|\eta(s)\|_{H^{9/2}(\Gamma)}^2 +\varepsilon^{-1} \|u(s)\|_{H^5(\Omega)}^2)ds+ \int_0^t \cE_f^3(s)  ds\Big).
\end{split}
\end{equation}
\end{proposition}
\begin{proof}
To prove \eqref{EstEtaH^9/2}, we borrow the idea from \cite[Lemma 3.9]{WT12}. Let 
\[
\cJ = \sqrt{1-\partial_1^2-\partial_2^2}.
\]
We apply $\cJ^{9/2}$ to $\eqref{EqPertur}_4$ and then multiply the resulting equation by $\cJ^{9/2}\eta$. Hence,  we find that
\[\begin{split}
\frac{d}{dt}\|\eta\|_{H^{9/2}(\Gamma)}^2 &= \frac12\int_\Gamma(\partial_1u_1+\partial_2 u_2)|\cJ^{9/2}\eta|^2 +\int_\Gamma (\cJ^{9/2}u_3 - [\cJ^{9/2}, u_1]\partial_1\eta - [\cJ^{9/2}, u_2]\partial_2\eta)\cJ^{9/2}\eta.
\end{split}\]
Thanks to \eqref{EstCommutator}, we have the following estimates,
\begin{equation}\label{1stEstPropEta^9/2}
\int_\Gamma \partial_ju_j|\cJ^{9/2}\eta|^2 \lesssim \|\partial_j u_j\|_{L^\infty(\Gamma)}\|\cJ^{9/2}\eta\|_{L^2(\Gamma)}^2 \lesssim \|u\|_{H^3(\Gamma)} \|\eta\|_{H^{9/2}(\Gamma)}^2,
\end{equation}
\begin{equation}\label{2ndEstPropEta^9/2}
\int_\Gamma \cJ^{9/2}u_3\cJ^{9/2}\eta \lesssim \|\cJ^{9/2}u_3\|_{L^2(\Gamma)} \|\cJ^{9/2}\eta\|_{L^2(\Gamma)} \lesssim \|\cJ^4 u\|_{H^1(\Omega)}\|\eta\|_{H^{9/2}(\Gamma)},
\end{equation}
and
\begin{equation}\label{3rdEstPropEta^9/2}
\begin{split}
\int_\Gamma [\cJ^{9/2}, u_j]\partial_j\eta \cJ^{9/2}\eta &\lesssim \|\partial_j u_j\|_{L^\infty(\Gamma)}\|\cJ^{7/2}\eta\|_{L^2(\Gamma)}\|\cJ^{9/2}\eta\|_{L^2(\Gamma)}\\
&\qquad\qquad +\|\cJ^{9/2}u\|_{L^2(\Gamma)} \|\partial_j\eta\|_{L^\infty(\Gamma)}\|\cJ^{9/2}\eta\|_{L^2(\Gamma)}\\
&\lesssim \|u\|_{H^3(\Gamma)} \|\eta\|_{H^{9/2}(\Gamma)}^2 +
\|\cJ^4 u\|_{H^1(\Omega) }\|\eta\|_{H^3(\Gamma)} \|\eta\|_{H^{9/2}(\Gamma)}.
\end{split}
\end{equation}
In view of \eqref{1stEstPropEta^9/2}, \eqref{2ndEstPropEta^9/2} and \eqref{3rdEstPropEta^9/2},  we get
\[
\begin{split}
\frac{d}{dt}\|\eta\|_{H^{9/2}(\Gamma)}^2   &\lesssim   \|u\|_{H^3(\Gamma)} \|\eta\|_{H^{9/2}(\Gamma)}^2 + \|\cJ^4 u\|_{H^1(\Omega) }(1+\|\eta\|_{H^3(\Gamma)}) \|\eta\|_{H^{9/2}(\Gamma)} \\
&\lesssim \|u\|_{H^5(\Omega)} \|\eta\|_{H^{9/2}(\Gamma)}+ \cE_f^3.
\end{split}
\]
Using Cauchy-Schwarz's inequality and then integrating from 0 to $t$, we obtain \eqref{EstEtaH^9/2}.
\end{proof}

\subsection{Temporal estimates for the perturbation velocity}

If we use the nonlinear equations in the perturbed form \eqref{EqPertur}, there will be no control of the highest temporal derivative of $q$ appearing in the nonlinear term $\cQ^2$.   Instead, we switch our original nonlinear equations \eqref{EqNS_Lagrangian} to a new formulation using a geometric transformation of the domain. The equations are
 \begin{equation}\label{1stEqPerturGeometric}
\begin{cases}
\partial_t\zeta + \text{div}_{\cA}(\rho_0 u) = F^1 &\quad\text{in }\Omega,\\
(\rho_0+\rho_0'\theta+\zeta)\partial_tu+ \nabla_{\cA}q - \mu\text{div}_\cA \bS_\cA u+ g\zeta e_3 = F^2 &\quad\text{in }\Omega,\\
\text{div}_\cA u=0 &\quad\text{in }\Omega,\\
\partial_t \eta= u\cdot \cN &\quad\text{on }\Gamma,\\
(qId- \mu\bS_\cA u)\cN =g\rho_+ \eta\cN, &\quad\text{on }\Gamma.
\end{cases}
\end{equation}
Here,
\begin{equation}\label{TermsF12}
\begin{split}
F^1 &= K\partial_t\theta(  \rho_0''  \theta +\partial_3\zeta) -\text{div}_\cA((\rho_0' \theta+\zeta)u),\\
F^2 &= -(\rho_0+\rho_0'\theta+\zeta)(-K\partial_t\theta \partial_3u + u\cdot\nabla_\cA u)- g\rho_0' (AK\theta, BK \theta, (1-K)\theta)^T.
\end{split}
\end{equation}

Applying the temporal differential operator $\partial_t^l$ $(l\geq 1)$ to \eqref{1stEqPerturGeometric}, the resulting equations are 
\begin{equation}\label{EqPerturGeometric}
\begin{cases}
\partial_t(\partial_t^l \zeta)+ \text{div}_\cA(\rho_0 \partial_t^l u)=F^{1,l} &\quad\text{in }\Omega,\\
(\rho_0+\rho_0'\theta+\zeta)\partial_t(\partial_t^l u)+ \nabla_{\cA}\partial_t^l q- \mu\text{div}_\cA \bS_\cA \partial_t^l u+ g \partial_t^l \zeta e_3 = F^{2,l} &\quad\text{in }\Omega,\\
\text{div}_\cA \partial_t^l u= F^{3,l} &\quad\text{in }\Omega,\\
\partial_t( \partial_t^l \eta)= \partial_t^l u\cdot \cN+ F^{4,l} &\quad\text{on }\Gamma,\\
(\partial_t^l qId- \mu  \bS_\cA\partial_t^l u) \cN=g\rho_+ \partial_t^l\eta \cN +F^{5,l} &\quad\text{on }\Gamma.
\end{cases}
\end{equation}
The terms $F^{j,l}$ with $l\geq 1$ and $1\leq j\leq 5$ are given in Appendix \ref{AppNonlinearTerms} (see \eqref{TermF1}, \eqref{TermF2}, \eqref{TermF34} and \eqref{TermF5}). We use the convention that $F^{1,0}=F^1, F^{2,0}= F^2$ and $F^{j,0}=0$ for $3\leq j\leq 5$. We now derive the following proposition.
\begin{proposition}\label{PropEstPartial_t^lU_L2}
For $l=0$ and 1, we have
\begin{equation}\label{EstPartial_t^lU_L2}
\begin{split}
&\|\partial_t^l u(t)\|_{L^2(\Omega)}^2 + \|\partial_t^l \eta(t)\|_{L^2(\Gamma)}^2+ \int_0^t \|\nabla \partial_t^l u(s)\|_{L^2(\Omega)}^2 ds \\
&\qquad \leq C_5\Big( \cE_f^2(0)+\int_0^t \|(u,\zeta)(s)\|_{L^2(\Omega)}^2 ds +\int_0^t \cE_f^3(s)ds\Big).
\end{split}
\end{equation}
We also have
\begin{equation}\label{EstPartial_t^2U_L2}
\begin{split}
&\|\partial_t^2 u(t)\|_{L^2(\Omega)}^2 + \|\partial_t^2 \eta(t)\|_{L^2(\Gamma)}^2+ \int_0^t \|\nabla \partial_t^2 u(s)\|_{L^2(\Omega)}^2 ds \\
&\qquad \leq C_6 \Big(\cE_f^2(0)+\int_0^t \|(u,\zeta)(s)\|_{L^2(\Omega)}^2 ds +\int_0^t \cE_f(\cE_f^2+\cD_f^2)(s)ds+\cE_f^3(t)\Big).
\end{split}
\end{equation}
\end{proposition}

The proof of Proposition \ref{PropEstPartial_t^lU_L2} relies on Lemmas \ref{LemIntByPartsA}, \ref{LemEqU_L^2} and \ref{LemF_lTermsL^2} below.
\begin{lemma}\label{LemIntByPartsA}
Let $J$ be defined as in \eqref{NotationABKJ}. For any scalar function $\vartheta \in \R$ and any vector function $ \varrho \in \R^3$, there holds
\begin{equation}\label{IntByPartsA}
\int_\Omega (\nabla_\cA \vartheta) \cdot J \varrho = \int_\Gamma \vartheta (\cN\cdot \varrho) -\int_\Omega J \vartheta\text{div}_\cA\varrho.
\end{equation}
\end{lemma}
\begin{proof}
We have from the integration by parts that
\begin{equation}\label{1_IntByPartsA}
\begin{split}
\int_\Omega  (\nabla_\cA \vartheta) \cdot J \varrho &= \int_\Omega  J \cA_{ij}\partial_j \vartheta \varrho_i = \int_\Gamma \vartheta (J\cA_{i3} \varrho_i) -\int_\Omega \vartheta \partial_j (J\cA_{ij}\varrho_i) 
\end{split}
\end{equation}
Note that $J\cA_{i3}\varrho_i = \cN \cdot \varrho$ and $\partial_j(J \cA_{ij}) =  0$, hence
\begin{equation}\label{3_IntByPartsA}
\int_\Omega  (\nabla_\cA \vartheta) \cdot J \varrho=  \int_\Gamma \vartheta(\cN \cdot\varrho)  -\int_\Omega \vartheta J\cA_{ij}\partial_j  \varrho_i = \int_\Gamma \vartheta(\cN \cdot\varrho)- \int_\Omega  J\vartheta \text{div}_\cA \varrho.
\end{equation}
We obtain \eqref{IntByPartsA}, i.e. Lemma \ref{LemIntByPartsA}.\end{proof}

\begin{lemma}\label{LemEqU_L^2}
There holds for all $l\geq 0$, 
\begin{equation}\label{EqU_L^2_L=0}
\begin{split}
&\frac12 \frac{d}{dt}\Big( \int_{\Omega}(\rho_0 +\rho_0'\theta+\zeta) J|\partial_t^l u|^2 + \int_\Gamma g\rho_+|\partial_t^l \eta |^2 \Big) +\frac12 \mu\int_{\Omega} J|\bS_\cA\partial_t^l u|^2 \\
&= \frac12 \int_{\Omega}\partial_t((\rho_0+\rho_0'\theta+\zeta)J)|\partial_t^l u|^2 +\int_{\Omega} J(F^{2,l}\cdot \partial_t^l u- g\partial_t^l\zeta \partial_t^l u_3+ F^{3,l} \partial_t^l q) \\
&\qquad-\int_\Gamma( g\rho_+\partial_t^l \eta F^{4,l} +F^{5,l}\cdot \partial_t^l u).
\end{split}
\end{equation}
If $\l\geq 1$, one has
\begin{equation}\label{EqU_L^2_L=1}
\begin{split}
&\frac12 \frac{d}{dt}\Big( \int_{\Omega}(\rho_0 +\rho_0'\theta+\zeta) J|\partial_t^l u|^2 + \int_\Gamma g\rho_+|\partial_t^l \eta |^2 -\int_{\Omega}g\rho_0'|\partial_t^{l-1}u_3|^2\Big) +\frac12 \mu\int_{\Omega} J|\bS_\cA\partial_t^l u|^2 \\
&= \frac12 \int_{\Omega}\partial_t((\rho_0+\rho_0'\theta+\zeta)J)|\partial_t^l u|^2 +\int_{\Omega} J (F^{2,l}\cdot \partial_t^l u+  F^{3,l} \partial_t^l q)-\int_\Gamma( g\rho_+\partial_t^l \eta F^{4,l} +F^{5,l}\cdot \partial_t^l u)\\
&\quad + \int_\Omega g\rho_0 J F^{3,l-1}\partial_t^l u_3 - \int_{\Omega}g\rho_0'(A\partial_t^{l-1}\partial_3 u_1+ B\partial_t^{l-1} \partial_3 u_2)\partial_t^l u_3 - \int_\Omega gJ F^{1,l-1}\partial_t^lu_3. 
\end{split}
\end{equation}
\end{lemma}
\begin{proof}
We multiply by $J\partial_t^l u$ on both sides of $\eqref{EqPerturGeometric}_2$ to have that 
\begin{equation}\label{1stEqLemEqU_L^2}
\begin{split}
&\frac12 \frac{d}{dt} \Big( \int_{\Omega}(\rho_0 +\rho_0'\theta+\zeta) J|\partial_t^l u|^2\Big)\\
 &= \frac12 \int_\Omega \partial_t((\rho_0+ \rho_0'\theta+\zeta)J)|\partial_t^l u|^2 -\int_\Omega\nabla_\cA\partial_t^l q\cdot J\partial_t^l u  \\
&\quad+ \int_\Omega \mu (\text{div}_\cA \bS_\cA \partial_t^l u) \cdot J\partial_t^l u  -\int_\Omega gJ \partial_t^l \zeta\partial_t^l u_3 +\int_\Omega J F^{2,l} \cdot \partial_t^l u.
\end{split}
\end{equation}
Thanks to Lemma \ref{LemIntByPartsA}, one deduces 
\begin{equation}\label{2ndEqLemEqU_L^2}
\begin{split}
&-\int_\Omega\nabla_\cA\partial_t^l q\cdot J\partial_t^l u  + \int_\Omega \mu (\text{div}_\cA \bS_\cA \partial_t^l u) \cdot J\partial_t^l u \\
&\quad=\int_\Gamma (\mu \bS_\cA\partial_t^l u-\partial_t^l q Id)\cN \cdot \partial_t^l u+\int_\Omega J (\text{div}_\cA \partial_t^l u) \partial_t^l q -\frac12 \int_\Omega \mu J| \bS_\cA \partial_t^l u|^2
\end{split}
\end{equation}
Substituting  \eqref{2ndEqLemEqU_L^2} into \eqref{1stEqLemEqU_L^2}, we have  
\begin{equation}\label{3rdEqLemEqU_L^2}
\begin{split}
&\frac12 \frac{d}{dt} \Big( \int_\Omega(\rho_0 +\rho_0'\theta+\zeta) J|\partial_t^l u|^2\Big) + \frac12 \int_\Omega \mu J| \bS_\cA \partial_t^l u|^2 \\
&= \frac12 \int_\Omega \partial_t((\rho_0+\rho_0'\theta+\zeta)J)|\partial_t^l u|^2 +\int_\Omega J (\text{div}_\cA \partial_t^l u) \partial_t^l q+\int_\Gamma (\mu \bS_\cA\partial_t^l u-\partial_t^l q Id)\cN \cdot \partial_t^l u  \\
& \quad-\int_\Omega gJ \partial_t^l\zeta \partial_t^l u_3+\int_\Omega J F^{2,l} \cdot \partial_t^l u.  
\end{split}
\end{equation}
Using $\eqref{EqPerturGeometric}_{3,4,5}$, we obtain  \eqref{EqU_L^2_L=0} from \eqref{3rdEqLemEqU_L^2}.

To prove \eqref{EqU_L^2_L=1}, we use $\eqref{EqPerturGeometric}_1$ at order $l-1$ to get that 
\begin{equation}\label{4thEqLemEqU_L^2}
\begin{split}
-\int_\Omega gJ\partial_t^l\zeta \partial_t^l u_3 &= \int_\Omega gJ\text{div}_\cA(\rho_0\partial_t^{l-1} u) \partial_t^l u_3 - \int_\Omega gJ F^{1,l-1}\partial_t^lu_3 \\
&=  \int_\Omega g\rho_0' \partial_t^{l-1}u_3 \partial_t^l u_3 + \int_\Omega g\rho_0 J F^{3,l-1}\partial_t^l u_3 \\
&\quad - \int_{\Omega}g\rho_0'(A\partial_t^{l-1}\partial_3 u_1+ B\partial_t^{l-1} \partial_3 u_2)\partial_t^l u_3  - \int_\Omega gJ F^{1,l-1}\partial_t^lu_3.
\end{split}
\end{equation}
Combining \eqref{4thEqLemEqU_L^2} and \eqref{EqU_L^2_L=0}, we obtain \eqref{EqU_L^2_L=1}.  Lemma \ref{LemEqU_L^2} is proven.
\end{proof}

\begin{lemma}\label{LemF_lTermsL^2}
The following inequalities hold
\begin{equation}\label{F_lTermsL^2}
\sum_{l=0}^1\Big( \|(F^{1,l}, F^{2,l}, F^{3,l})\|_{L^2(\Omega)}+ \|(F^{4,l},F^{5,l})\|_{L^2(\Gamma)}\Big) \lesssim \cE_f^2,
\end{equation}
\begin{equation}\label{F_2TermsL^2}
 \|(F^{1,2}, F^{2,2})\|_{L^2(\Omega)}+ \|(F^{4,2},F^{5,2})\|_{L^2(\Gamma)} \lesssim \cE_f(\cE_f+\|\nabla \partial_t u\|_{H^2(\Omega)}+ \|\nabla \partial_t^2 u\|_{L^2(\Omega)}),
\end{equation}
and 
\begin{equation}\label{BoundD_tJF}
\|(F^{3,2}, JF^{3,2})\|_{L^2(\Omega)} \lesssim \cE_f^2, \quad  \| \partial_t(JF^{3,2})\|_{L^2(\Omega)} \lesssim \cE_f(\cE_f+\|\nabla \partial_t u\|_{H^2(\Omega)}+ \|\nabla \partial_t^2 u\|_{L^2(\Omega)}).
\end{equation}
\end{lemma}
\begin{proof}

For $\Sigma=\Omega$ or $\Gamma$, all quadratic terms $\|X_1X_2\|_{L^2(\Sigma)}$ or cubic ones $\|X_1X_2X_3\|_{L^2(\Sigma)}$ appearing in $F^{j,l}$ with $1\leq j\leq 5$ will be bounded by  using  Sobolev embedding, Lemma \ref{LemEstNablaQ_Pf} and other inequalities in  Appendix \ref{AppUsefulEst}. Precisely,  we have
\[\begin{split}
\|X_1X_2\|_{L^2(\Sigma)} &\lesssim \|X_1\|_{L^\infty(\Sigma)}\|X_2\|_{L^2(\Sigma)} \lesssim \|X_1\|_{H^2(\Sigma)}\|X_2\|_{L^2(\Sigma)}
\end{split}\]
and
\[\begin{split}
\|X_1X_2X_3\|_{L^2(\Sigma)}  &\lesssim \|X_1\|_{L^\infty(\Sigma)} \|X_2\|_{L^\infty(\Sigma)}\|X_3\|_{L^2(\Sigma)}\lesssim \|X_1\|_{H^2(\Sigma)} \|X_2\|_{H^2(\Sigma)}\|X_3\|_{L^2(\Sigma)}. 
\end{split}\]
We only show the estimates of the term $F^{2,l} (0\leq l\leq 2)$ (see  \eqref{TermsF12} and \eqref{TermF2}),  the estimates of  others terms are proven in the same way.  

For $F^2$ (see \eqref{TermsF12}), we have
\[
(\rho_0+\rho_0'\theta+\zeta)K \partial_t\theta\partial_3 u = (\rho_0+\rho_0'\theta+\zeta)(K-1) \partial_t\theta\partial_3 u +(\rho_0+\rho_0'\theta+\zeta) \partial_t\theta\partial_3 u.
\]
Thanks to Lemma \ref{LemEstNablaQ_Pf} and \eqref{CoefEst_22}, we obtain
\begin{equation}\label{1_F_lTermsL^2}
\begin{split}
\|(\rho_0&+\rho_0'\theta+\zeta)K \partial_t\theta\partial_3 u\|_{L^2(\Omega)} \\
&\lesssim (1+ \|(\theta,\zeta)\|_{H^2(\Omega)}) (1+\|K-1\|_{H^2(\Omega)}) \|\partial_t \theta\|_{H^2(\Omega)} \|u\|_{H^1(\Omega)} \\
&\lesssim (1+\|\zeta\|_{H^2(\Omega)}+\|\eta\|_{H^{3/2}(\Gamma)})(1+\|\eta\|_{H^{5/2}(\Gamma)})\|\partial_t\eta\|_{H^{3/2}(\Gamma)}\|u\|_{H^1(\Omega)} \\
&\lesssim \cE_f^2.
\end{split}
\end{equation}
Note that $u\cdot\nabla_\cA u = u\cdot\nabla_{\cA-\text{Id}} u+ u\cdot \nabla u$, we  use Lemma \ref{LemEstNablaQ_Pf} and \eqref{CoefEst_24} to get that
\begin{equation}\label{2_F_lTermsL^2}
\begin{split}
\|(\rho_0&+\rho_0'\theta +\zeta)  u\cdot \nabla_\cA u\|_{L^2(\Omega)} \\
&\lesssim \|(\rho_0+\rho_0'\theta+\zeta) u\cdot \nabla_{\cA-\text{Id}} u\|_{L^2(\Omega)}+\|(\rho_0+\rho_0'\theta+\zeta) u\cdot \nabla  u\|_{L^2(\Omega)} \\
&\lesssim   (1+\|\zeta\|_{H^2(\Omega)}+\|\eta\|_{H^{3/2}(\Gamma)}) (1+\|\cA-\text{Id}\|_{H^2(\Omega)}) \|u\|_{H^2(\Omega)}\|u\|_{H^1(\Omega)}\\
&\lesssim (1+\|\zeta\|_{H^2(\Omega)}+\|\eta\|_{H^{3/2}(\Gamma)}) (1+\|\eta \|_{H^{5/2}(\Gamma)})\|u\|_{H^2(\Omega)}\|u\|_{H^1(\Omega)}\\
&\lesssim \cE_f^2.
\end{split}
\end{equation}
Due to  Lemma \ref{LemEstNablaQ_Pf} again and \eqref{CoefEst_22}, \eqref{CoefEst_23},  we have
\begin{equation}\label{3_F_lTermsL^2}
\begin{split}
\|(AK\theta,BK\theta,(1-K)\theta)\|_{L^2(\Omega)} &\lesssim \|(AK,BK,K-1)\|_{H^2(\Omega)} \|\theta\|_{L^2(\Omega)}\\
&\lesssim \|\eta\|_{H^{5/2}(\Gamma)} \|\eta\|_{L^2(\Gamma)}\lesssim \cE_f^2.
\end{split}
\end{equation}
It follows from \eqref{1_F_lTermsL^2}, \eqref{2_F_lTermsL^2} and \eqref{3_F_lTermsL^2}  that $\|F^2 \|_{L^2(\Omega)} \lesssim \cE_f^2$.

For $F^{2,1}$ (see \eqref{TermF2}),  we  obtain
\begin{equation}\label{4_F_lTermsL^2}
\begin{split}
\|F^{2,1}\|_{L^2(\Omega)} &\lesssim \|\partial_t F^2\|_{L^2(\Omega)} + (1+\|\cA-\text{Id} \|_{H^2(\Omega)}) \|\partial_t\cA\|_{H^1(\Omega)} \|\nabla u\|_{H^2(\Omega)}  \\
&\qquad+(1+ \|\cA-\text{Id}\|_{H^3(\Omega)})  \|\nabla^2 u\|_{H^2(\Omega)} \|\partial_t \cA\|_{H^1(\Omega)}  +\|\zeta \|_{H^3(\Omega)}  \|\partial_t \cA \|_{L^2(\Omega)} \\
&\qquad+ (\|\partial_t \zeta\|_{H^2(\Omega)} +\|\partial_t\theta\|_{H^2(\Omega)} )\|\partial_t u\|_{L^2(\Omega)}.
\end{split}
\end{equation}
According to  Lemma \ref{LemEstNablaQ_Pf} and \eqref{CoefEst_24}, it follows from \eqref{4_F_lTermsL^2} that
\begin{equation}\label{6_F_lTermsL^2}
\begin{split}
\|F^{2,1}\|_{L^2(\Omega)}&\lesssim  \|\partial_t F^2\|_{L^2(\Omega)} +(1+\|\eta\|_{H^{7/2}(\Gamma)}) \|\partial_t \eta\|_{H^{3/2}(\Gamma)} \| u\|_{H^4(\Omega)}+\|\zeta\|_{H^3(\Omega)}\|\partial_t\eta\|_{H^{1/2}(\Gamma)}\\
&\quad+ (\|\partial_t \zeta\|_{H^2(\Omega)} +\|\partial_t\eta\|_{H^{3/2}(\Gamma)} )\|\partial_t u\|_{L^2(\Omega)}\\
&\lesssim  \|\partial_t F^2\|_{L^2(\Omega)} + \cE_f^2.
\end{split}
\end{equation}
We calculate each term of $\partial_tF^2$, 
\[
\begin{split}
\partial_t ((\rho_0+\rho_0'\theta+\zeta)K\partial_t\theta\partial_3 u) &= (\rho_0+\rho_0'\theta+\zeta)(\partial_t K\partial_t\theta\partial_3 u+ K\partial_t^2 \theta\partial_3 u +K\partial_t\theta \partial_t\partial_3 u)\\
&\qquad +(\rho_0' \partial_t\theta+ \partial_t\zeta)K\partial_t\theta\partial_3 u,
\end{split}
\]
which will be bounded as follows
\[
\begin{split}
\|\partial_t& ((\rho_0+\rho_0'\theta+\zeta)K\partial_t\theta\partial_3 u)\|_{L^2(\Omega)}\\
&\lesssim   (1+\|(\theta,\zeta)\|_{H^2(\Omega)})  \|\partial_t K\|_{L^2(\Omega)}\|\partial_t\theta\|_{H^2(\Omega)} \|u\|_{H^1(\Omega)}\\
&+  (1+\|(\theta,\zeta)\|_{H^2(\Omega)})  (\|K-1\|_{H^2(\Omega)}+1) \\
&\qquad\qquad\times( \|\partial_3 u\|_{H^2(\Omega)} \|\partial_t^2\theta\|_{L^2(\Omega)}+\|\partial_t \theta\|_{H^2(\Omega)}\|(u,\partial_t u)\|_{H^1(\Omega)})\\
&+\|(\partial_t \theta,\partial_t \zeta)\|_{H^2(\Omega)}  (\|K-1\|_{H^2(\Omega)}+1) \|\partial_t \theta\|_{H^2(\Omega)}\|u\|_{H^1(\Omega)}.
\end{split}
\]
Using  Lemma \ref{LemEstNablaQ_Pf} and \eqref{CoefEst_22}, we deduce that
\begin{equation}\label{7_F_lTermsL^2}
\begin{split}
\|\partial_t &((\rho_0+\rho_0'\theta+\zeta)K\partial_t\theta\partial_3 u)\|_{L^2(\Omega)}\\&\lesssim (1+\|\zeta\|_{H^2(\Omega)}+\|\eta\|_{H^{3/2}(\Gamma)}) \|\eta\|_{H^{1/2}(\Gamma)} \|\partial_t \eta\|_{H^{3/2}(\Gamma)}\|u\|_{H^1(\Omega)}\\
&\quad+ (1+\|\zeta\|_{H^2(\Omega)}+\|\eta\|_{H^{3/2}(\Gamma)}) (1+\|\eta\|_{H^{5/2}(\Gamma)}) \\
&\qquad\qquad\times(\|u\|_{H^3(\Omega)} \|\partial_t^2\eta\|_{L^2(\Gamma)} + \|\partial_t\eta\|_{H^{3/2}(\Gamma)}\| (u,\partial_t u)\|_{H^1(\Omega)}) \\
&\quad + (\|\partial_t\zeta\|_{H^2(\Omega)} +\|\partial_t\eta\|_{H^{3/2}(\Gamma)})(1+\|\eta\|_{H^{5/2}(\Gamma)}) \|u\|_{H^1(\Omega)} \|\partial_t\eta\|_{H^{3/2}(\Gamma)}\\
&\lesssim \cE_f^2.
\end{split}
\end{equation}
Next, we compute 
\[
\begin{split}
\partial_t ((\rho_0+\rho_0'\theta+\zeta)u\cdot\nabla_\cA u) &=  (\rho_0+\rho_0'\theta+\zeta) (\partial_t u_i \cA_{ij}\partial_j u_k + u_i \partial_t \cA_{ij} \partial_j u_k + u_i \cA_{ij}\partial_t \partial_j u_k)\\
&\qquad+(\rho_0'\partial_t\theta+\partial_t\zeta) u_i \cA_{ij}\partial_j u_k.
\end{split}
\]
Hence, it follows from Lemma \ref{LemEstNablaQ_Pf} and \eqref{CoefEst_24} that
\begin{equation}\label{8_F_lTermsL^2}
\begin{split}
\|\partial_t &((\rho_0+\rho_0'\theta+\zeta)u\cdot\nabla_\cA u)\|_{L^2(\Omega)}\\
&\lesssim (1+\|(\theta,\zeta)\|_{H^2(\Omega)})  (\|\cA-\text{Id}\|_{H^2(\Omega)} +1) \|u\|_{H^3(\Omega)}\|\partial_t u\|_{H^1(\Omega)}\\
&\qquad +(1+\|(\theta,\zeta)\|_{H^2(\Omega)})  \|u\|_{H^3(\Omega)}^2 \|\partial_t\cA\|_{L^2(\Omega)} \\
&\qquad+ \|(\partial_t \theta,\partial_t \zeta)\|_{H^2(\Omega)}(\|\cA-\text{Id}\|_{H^2(\Omega)} +1) \|u\|_{H^2(\Omega)}\|u\|_{H^1(\Omega)}\\
&\lesssim (1+\|\zeta\|_{H^2(\Omega)}+\|\eta\|_{H^{3/2}(\Gamma)}) (\|\eta\|_{H^{5/2}(\Gamma)}+1) \|u\|_{H^3(\Omega)}\|\partial_t u\|_{H^1(\Omega)}\\
&\qquad + (1+\|\zeta\|_{H^2(\Omega)}+\|\eta\|_{H^{3/2}(\Gamma)})   \|u\|_{H^3(\Omega)}^2\|\partial_t\eta\|_{H^{1/2}(\Gamma)}\\
&\qquad+  (\|\partial_t\zeta\|_{H^2(\Omega)} +\|\partial_t\eta\|_{H^{3/2}(\Gamma)})(\|\eta\|_{H^{5/2}(\Gamma)}+1)\|u\|_{H^2(\Omega)}\|u\|_{H^1(\Omega)}\\
&\lesssim \cE_f^2.
\end{split}
\end{equation}
Using again Lemma \ref{LemEstNablaQ_Pf} and \eqref{CoefEst_22}, \eqref{CoefEst_23}, one has
\begin{equation}\label{9_F_lTermsL^2}
\begin{split}
\|\partial_t(AK\theta,BK\theta,(1-K)\theta)\|_{L^2(\Omega)} &\lesssim \|(AK, BK, K-1)\|_{H^2(\Omega)} \|\partial_t \theta\|_{L^2(\Omega)}\\
&\qquad +\|\partial_t (AK,BK,K-1)\|_{L^2(\Omega)}\|\theta\|_{H^2(\Omega)}\\
&\lesssim \|\eta\|_{H^{5/2}(\Gamma)}\|\partial_t\eta\|_{L^2(\Gamma)} + \|\partial_t\eta\|_{H^{1/2}(\Gamma)} \|\eta\|_{H^{3/2}(\Gamma)} \\
&\lesssim \cE_f^2.
\end{split}
\end{equation}
We deduce  $\|\partial_tF^2\|_{L^2(\Omega)}\lesssim \cE_f^2$ from \eqref{6_F_lTermsL^2}, \eqref{7_F_lTermsL^2}, \eqref{8_F_lTermsL^2} and \eqref{9_F_lTermsL^2}. So that, $\|F^{2,1}\|_{L^2(\Omega)}\lesssim \cE_f^2$.

For $F^{2,2}$ (see again \eqref{TermsF12}), we use  the product estimate \eqref{ProductEst} and Sobolev embedding  to obtain that  
\[
\begin{split}
 \|F^{2,2} \|_{L^2(\Omega)}  &\lesssim \|\partial_t^2 F^2\|_{L^2(\Omega)} + \|\cA\|_{H^2(\Omega)}( \|\partial_t\cA\|_{H^2(\Omega)}\|\nabla \partial_t u\|_{H^2(\Omega)}+ \|\partial_t^2 \cA\|_{H^1(\Omega)}\|\nabla u\|_{H^3(\Omega)}) \\
&+\|\partial_t\cA\|_{H^2(\Omega)} (\|\partial_t\cA\|_{H^2(\Omega)}\|\nabla u\|_{H^2(\Omega)}+ \|\cA\|_{H^3(\Omega)} \|\nabla\partial_t u\|_{H^1(\Omega)})\\
&+\|\partial_t^2 \cA\|_{L^2(\Omega)} \|\cA\|_{H^3(\Omega)}\|\nabla u\|_{H^3(\Omega)}+\|\partial_t\cA\|_{H^2(\Omega)}\|\partial_t\zeta\|_{H^1(\Omega)}+\|\partial_t^2\cA\|_{L^2(\Omega)}\|\nabla \zeta\|_{H^2(\Omega)} \\
& +\|(\partial_t\zeta,\partial_t\theta)\|_{H^2(\Omega)}\|\partial_t^2u\|_{L^2(\Omega)} +\|(\partial_t^2\zeta,\partial_t^2\theta)\|_{L^2(\Omega)}\|\partial_t u\|_{H^2(\Omega)}.
\end{split}
\]
We make use of Lemma  \ref{LemEstNablaQ_Pf} and \eqref{CoefEst_24} to further get that
\[
\begin{split}
\|F^{2,2}\|_{L^2(\Omega)} &\lesssim \|\partial_t^2 F^2\|_{L^2(\Omega)} + \|\eta\|_{H^{3/2}(\Gamma)} \|\partial_t\eta\|_{H^{3/2}(\Gamma)}\|\nabla\partial_t u\|_{H^2(\Omega)}\\
&\qquad+ \|\eta\|_{H^{3/2}(\Gamma)}\|\partial_t^2\eta\|_{H^{1/2}(\Gamma)}\|u\|_{H^4(\Omega)}+\|\partial_t\eta\|_{H^{3/2}(\Gamma)}^2 \|u\|_{H^3(\Omega)}  \\
&\qquad+\|\partial_t\eta\|_{H^{3/2}(\Gamma)} \|\eta\|_{H^{7/2}(\Gamma)}\|\partial_t u\|_{H^2(\Omega)}+ \|\partial_t^2\eta\|_{L^2(\Gamma)}\|\eta\|_{H^{7/2}(\Gamma)}\|u\|_{H^4(\Omega)} \\
&\qquad +\|\partial_t\eta\|_{H^{3/2}(\Gamma)}\|\partial_t\zeta\|_{H^1(\Omega)}+\|\partial_t^2\eta\|_{L^2(\Gamma)}\|\zeta\|_{H^3(\Omega)} \\
&\qquad +(\|\partial_t\zeta\|_{H^2(\Omega)}+\|\partial_t\eta\|_{H^{3/2}(\Gamma)})\|\partial_t^2u\|_{L^2(\Omega)}\\
&\qquad +(\|\partial_t^2\zeta\|_{L^2(\Omega)}+\|\partial_t^2\eta\|_{L^2(\Gamma)})\|\partial_t u\|_{H^2(\Omega)}\\
&\lesssim  \|\partial_t^2 F^2\|_{L^2(\Omega)} +\cE_f(\cE_f+ \|\partial_t^2\eta\|_{H^{1/2}(\Gamma)} +\|\nabla\partial_t u\|_{H^2(\Omega)}).
\end{split}
\]
Together with \eqref{EstEta_H^3/2}, we deduce from the resulting inequality that
\[
\|F^{2,2}\|_{L^2(\Omega)} \lesssim \|\partial_t^2 F^2\|_{L^2(\Omega)} +\cE_f(\cE_f+\|\nabla\partial_t u\|_{H^2(\Omega)}).
\]
Hence, in order to show that 
\[
\|F^{2,2}\|_{L^2(\Omega)} \lesssim \cE_f(\cE_f+\|\nabla\partial_t u\|_{H^2(\Omega)}+\|\nabla \partial_t^2u\|_{L^2(\Omega)}),
\]
we will prove that
\[
\|\partial_t^2 F^2\|_{L^2(\Omega)} \lesssim \cE_f(\cE_f+\|\nabla\partial_t^2 u\|_{L^2(\Omega)}).
\]
We now estimate each term of $\partial_t^2 F^2$.  Due to Sobolev embedding, one has that
\[
\begin{split}
\|\partial_t^2 &((\rho_0+\rho_0'\theta+\zeta)K\partial_t\theta\partial_3 u)\|_{L^2(\Omega)}\\ &\lesssim   (1+\|(\theta,\zeta)\|_{H^2(\Omega)}) (\|K-1\|_{H^2(\Omega)}+1)  \\
&\qquad\quad \times  (\|\partial_3 u\|_{H^2(\Omega)} \|\partial_t^3 \theta\|_{L^2(\Omega)} +\|\partial_t\partial_3u\|_{L^2(\Omega)} \|\partial_t^2 \theta\|_{H^2(\Omega)}
+ \|\partial_t^2 \partial_3 u\|_{L^2(\Omega)}\|\partial_t\theta\|_{H^2(\Omega)})\\
&+  (1+\|(\theta,\zeta)\|_{H^2(\Omega)}) \|\partial_t K\|_{H^2(\Omega)}(\|\partial_t\theta\|_{H^2(\Omega)}\|\partial_tu\|_{H^1(\Omega)}+ \|\partial_t^2\theta\|_{L^2(\Omega)} \|u\|_{H^3(\Omega)}) \\
&+  (1+\|(\theta,\zeta)\|_{H^2(\Omega)})\|\partial_t^2 K\|_{L^2(\Omega)}  \|\partial_t \theta\|_{H^2(\Omega)}\|u\|_{H^3(\Omega)}\\
& +\|(\partial_t \theta,\partial_t \zeta)\|_{H^2(\Omega)}(\|K-1\|_{H^2(\Omega)}+1) ( \|\partial_t^2 \theta\|_{L^2(\Omega)}\|u\|_{H^3(\Omega)}+\|\partial_t\theta\|_{H^2(\Omega)}\|\partial_t u\|_{H^1(\Omega)})\\
&+\|(\partial_t \theta,\partial_t \zeta)\|_{H^2(\Omega)}\|\partial_tK\|_{L^2(\Omega)}\|\partial_t \theta\|_{H^2(\Omega)}\|u\|_{H^3(\Omega)}\\
&+ \|(\partial_t^2\theta,\partial_t^2\zeta)\|_{L^2(\Omega)}  (\|K-1\|_{H^2(\Omega)}+1)\|\partial_t\theta\|_{H^2(\Omega)}\|u\|_{H^3(\Omega)}.
\end{split}
\]
Thanks to Lemma \ref{LemEstNablaQ_Pf} and \eqref{CoefEst_22}, this yields
 \begin{equation}\label{010_F_lTermsL^2}
\begin{split}
&\|\partial_t^2 ((\rho_0+\rho_0'\theta+\zeta)K\partial_t\theta\partial_3 u)\|_{L^2(\Omega)} \\
&\lesssim (1+\|\eta\|_{H^{3/2}(\Gamma)}+\|\zeta\|_{H^2(\Omega)}) (\|\eta\|_{H^{5/2}(\Gamma)}+1) \\
&\qquad\quad \times (\|u\|_{H^3(\Omega)} \|\partial_t^3\eta\|_{L^2(\Gamma)} + \|\partial_t u\|_{H^1(\Omega)}\|\partial_t^2\eta\|_{H^{3/2}(\Gamma)} +\|\nabla \partial_t^2 u\|_{L^2(\Omega)} \|\partial_t\eta\|_{H^{3/2}(\Gamma)}) \\
&\quad+(1+\|\eta\|_{H^{3/2}(\Gamma)}+\|\zeta\|_{H^2(\Omega)}) \|\partial_t \eta\|_{H^{5/2}(\Gamma)} (\|\partial_t\eta\|_{H^{3/2}(\Gamma)}\|\partial_t u\|_{H^1(\Omega)} +\|\partial_t^2\eta\|_{L^2(\Gamma)}\|u\|_{H^3(\Omega)} )\\
&\quad+ (1+\|\eta\|_{H^{3/2}(\Gamma)}+\|\zeta\|_{H^2(\Omega)})  (\|\partial_t^2 \eta\|_{H^{1/2}(\Gamma)} +\|\partial_t\eta\|_{H^{1/2}(\Gamma)}^2) \|\partial_t\eta\|_{H^{3/2}(\Gamma)}\|u\|_{H^3(\Omega)} \\
&\quad +(\|\partial_t\eta\|_{H^{3/2}(\Gamma)}+\|\partial_t \zeta\|_{H^2(\Omega)}) (\|\eta\|_{H^{5/2}(\Gamma)}+1)  (\|\partial_t^2 \eta\|_{L^2(\Gamma)}\|u\|_{H^3(\Omega)}+\|\partial_t\eta\|_{H^{3/2}(\Gamma)}\|\partial_t u\|_{H^1(\Omega)})\\
&\quad+(\|\partial_t\eta\|_{H^{3/2}(\Gamma)}+\|\partial_t \zeta\|_{H^2(\Omega)}) \|\partial_t\eta\|_{H^{1/2}(\Gamma)}\|\partial_t\eta\|_{H^{3/2}(\Gamma)}\|u\|_{H^3(\Omega)}\\
&\quad+ (\|\partial_t^2\eta\|_{L^2(\Gamma)}+\|\partial_t^2\zeta\|_{L^2(\Omega)})(\|\eta\|_{H^{5/2}(\Gamma)}+1)\|\partial_t\eta\|_{H^{3/2}(\Gamma)}\|u\|_{H^3(\Omega)}.
\end{split}
\end{equation}
 Using \eqref{EstEta_H^3/2} and \eqref{EstEtaD_t^3}, we thus have from \eqref{010_F_lTermsL^2} that
 \begin{equation}\label{10_F_lTermsL^2}
\begin{split}
&\|\partial_t^2 ((\rho_0+\rho_0'\theta+\zeta)K\partial_t\theta\partial_3 u)\|_{L^2(\Omega)}\\ &\lesssim \cE_f(\cE_f+  \|\nabla \partial_t^2 u\|_{L^2(\Omega)}+ \|\partial_t^3\eta\|_{L^2(\Gamma)}+\|\partial_t^2\eta\|_{H^{1/2}(\Gamma)}+ \|\partial_t\eta\|_{H^{5/2}(\Gamma)})\\
&\lesssim \cE_f(\cE_f+ \|\nabla \partial_t^2 u\|_{L^2(\Omega)}).
\end{split}
\end{equation}
In a same way, we have
\[
\begin{split}
&\|\partial_t^2 ((\rho_0+\rho_0'\theta+\zeta)u\cdot\nabla_\cA u)\|_{L^2(\Omega)}\\
&\lesssim   (1+\|(\theta,\zeta)\|_{H^2(\Omega)})(\|\cA-\text{Id}\|_{H^2(\Omega)}+1) (\|\partial_t^2 u\|_{H^1(\Omega)} \| u\|_{H^3(\Omega)} +\|\partial_t u\|_{H^2(\Omega)}^2 )\\
&\qquad+ (1+\|(\theta,\zeta)\|_{H^2(\Omega)})(\|\partial_t\cA\|_{L^2(\Omega)} \|\partial_tu\|_{H^3(\Omega)} +\|\partial_t^2\cA\|_{L^2(\Omega)} \|u\|_{H^2(\Omega)}) \| u\|_{H^3(\Omega)}\\
&\qquad+ \|(\partial_t \theta,\partial_t \zeta)\|_{H^2(\Omega)}(\|\cA-\text{Id}\|_{H^2(\Omega)}+1)\|\partial_t u\|_{H^1(\Omega)}\|u\|_{H^3(\Omega)} \\
&\qquad+ \|(\partial_t^2\theta, \partial_t^2\zeta)\|_{L^2(\Omega)} (\|\cA-\text{Id}\|_{H^2(\Omega)}+1)\|u\|_{H^3(\Omega)}^2.
\end{split}\]
Thanks to Lemma \ref{LemEstNablaQ_Pf} and \eqref{CoefEst_24}, we further get that
\[\begin{split}
&\|\partial_t^2 ((\rho_0+\rho_0'\theta+\zeta)u\cdot\nabla_\cA u)\|_{L^2(\Omega)}\\
&\lesssim  (1+\|\eta\|_{H^{3/2}(\Gamma)}+ \|\zeta\|_{H^2(\Omega)})(\|\eta\|_{H^{5/2}(\Gamma)}+1) (\|\partial_t^2 u\|_{H^1(\Omega)} \| u\|_{H^3(\Omega)} +\|\partial_t u\|_{H^2(\Omega)}^2 )\\
&\quad+ (1+\|\eta\|_{H^{3/2}(\Gamma)}+ \|\zeta\|_{H^2(\Omega)}) ( \|\partial_t\eta\|_{H^{1/2}(\Gamma)} \|\partial_t u\|_{H^3(\Omega)}\\
&\quad+( \|\partial_t^2\eta\|_{H^{1/2}(\Gamma)}+\|\partial_t\eta\|_{H^{5/2}(\Gamma)}^2)\|u\|_{H^2(\Omega)})\|u\|_{H^3(\Omega)}\\
&\quad+ (\|\partial_t\eta\|_{H^{3/2}(\Gamma)}+ \|\partial_t \zeta\|_{H^2(\Omega)} )(\|\eta\|_{H^{5/2}(\Gamma)}+1) \|\partial_t u\|_{H^3(\Omega)}\|u\|_{H^3(\Omega)}\\
&\quad+( \|\partial_t^2\eta\|_{L^2(\Gamma)}+ \| \partial_t^2\zeta\|_{L^2(\Omega)}) (\|\eta\|_{H^{5/2}(\Gamma)}+1)\|u\|_{H^3(\Omega)}^2 \\
&\lesssim \cE_f(\cE_f+\|\partial_t^2\eta\|_{H^{1/2}(\Gamma)}+ \|\partial_t\eta\|_{H^{5/2}(\Gamma)}^2+\|\nabla \partial_t^2u\|_{L^2(\Omega)}). 
\end{split}
\]
Due to  \eqref{EstEta_H^3/2}, we obtain
\begin{equation}\label{11_F_lTermsL^2}
\|\partial_t^2 ((\rho_0+\rho_0'\theta+\zeta)u\cdot\nabla_\cA u)\|_{L^2(\Omega)} \lesssim \cE_f(\cE_f+\|\nabla \partial_t^2u\|_{L^2(\Omega)}). 
\end{equation}
Furthermore, thanks to \eqref{EstEta_H^3/2} again and Lemma \ref{LemEstNablaQ_Pf}, \eqref{CoefEst_22}, \eqref{CoefEst_23}, one has
\begin{equation}\label{12_F_lTermsL^2}
\begin{split}
\|\partial_t^2 &(AK\theta,BK\theta,(1-K)\theta)\|_{L^2(\Omega)}\\
 &\lesssim \|\partial_t^2(AK,BK,K-1)\|_{L^2(\Omega)}\| \theta\|_{H^2(\Omega)}+ \|\partial_t(AK,BK,K-1)\|_{L^2(\Omega)}\|\partial_t\theta\|_{H^2(\Omega)}\\
&\qquad + \|(AK,BK,K-1)\|_{H^2(\Omega)}\|\partial_t^2 \theta\|_{L^2(\Omega)} \\
&\lesssim (\|\partial_t^2\eta\|_{H^{1/2}(\Gamma)}+\|\partial_t\eta\|_{H^{5/2}(\Gamma)}^2) \|\eta\|_{H^{3/2}(\Gamma)} + \|\partial_t\eta\|_{H^{3/2}(\Gamma)} \|\partial_t\eta\|_{H^{3/2}(\Gamma)} +\|\eta\|_{H^{5/2}(\Gamma)} \|\partial_t^2\eta\|_{L^2(\Gamma)}\\
&\lesssim \cE_f(\cE_f+ \|\partial_t^2\eta\|_{H^{1/2}(\Gamma)}+\|\partial_t\eta\|_{H^{5/2}(\Gamma)}^2)\\
&\lesssim \cE_f^2. 
\end{split}
\end{equation}
 Consequently, there holds
\[
\|\partial_t^2 F^2\|_{L^2(\Omega)} \lesssim \cE_f(\cE_f+ \|\nabla \partial_t^2 u\|_{L^2(\Omega)})
\]
thanks to \eqref{10_F_lTermsL^2}, \eqref{11_F_lTermsL^2} and \eqref{12_F_lTermsL^2}.

We are left to prove \eqref{BoundD_tJF}. From the formula of $F^{3,2}$ (see \eqref{TermF34}), we use Sobolev embedding and \eqref{CoefEst_24} to get
\[
\begin{split}
\|F^{3,2}\|_{L^2(\Omega)} &\lesssim \|\partial_t^2\cA\|_{L^2(\Omega)} \|\nabla u\|_{H^2(\Omega)} + \|\partial_t\cA\|_{H^2(\Omega)}\|\nabla \partial_t u\|_{L^2(\Omega)} \\
&\lesssim (\|\partial_t^2\eta\|_{H^{1/2}(\Gamma)}+\|\partial_t\eta\|_{H^{5/2}(\Gamma)}^2) \|u\|_{H^3(\Omega)}+ \|\partial_t\eta\|_{H^{5/2}(\Gamma)}\|\partial_t u\|_{H^1(\Omega)}.
\end{split}
\]
Owing to \eqref{EstEta_H^1/2} and \eqref{EstEta_H^3/2}, we deduce that 
\begin{equation}\label{F^32Est_1}
\begin{split}
\|F^{3,2}\|_{L^2(\Omega)} &\lesssim (\|\partial_t^2\eta\|_{H^{1/2}(\Gamma)}+\|\partial_t\eta\|_{H^{5/2}(\Gamma)}^2)\cE_f+ \|\partial_t\eta\|_{H^{5/2}(\Gamma)}\cE_f \lesssim \cE_f^2. 
\end{split}
\end{equation}
Together with \eqref{CoefEstimates}, this yields 
\begin{equation}\label{F^32Est_2}
\|JF^{3,2}\|_{L^2(\Omega)} \lesssim (1+\|J-1\|_{L^\infty(\Omega)})\|F^{3,2}\|_{L^2(\Omega)} \lesssim \cE_f^2. 
\end{equation}
We continue using \eqref{CoefEstimates}, Lemma \ref{LemEstNablaQ_Pf} and \eqref{F^32Est_1} to get
\begin{equation}\label{F^32Est_4}
\begin{split}
\|\partial_t(JF^{3,2})\|_{L^2(\Omega)} &\lesssim \|\partial_t J\|_{L^\infty(\Omega)} \|F^{3,2}\|_{L^2(\Omega)} + (1+\|J-1\|_{L^\infty(\Omega)})\|\partial_tF^{3,2}\|_{L^2(\Omega)} \\
&\lesssim \|\partial_t \partial_3\theta\|_{H^2(\Omega)}\cE_f^2 + \|\partial_tF^{3,2}\|_{L^2(\Omega)}\\
&\lesssim \|\partial_t \eta\|_{H^{5/2}(\Omega)}\cE_f^2 + \|\partial_tF^{3,2}\|_{L^2(\Omega)}.
\end{split}
\end{equation}
By Sobolev embedding and \eqref{CoefEst_24}, let us estimate that 
\begin{equation}\label{F^32Est_3}
\begin{split}
\|\partial_tF^{3,2}\|_{L^2(\Omega)} &\lesssim \|\partial_t^2\cA\|_{L^2(\Omega)}\|\nabla\partial_t u\|_{H^2(\Omega)} + \|\partial_t\cA\|_{H^2(\Omega)}\|\nabla\partial_t^2 u\|_{L^2(\Omega)}  + \|\partial_t^3\cA\|_{L^2(\Omega)}\|\nabla u\|_{H^2(\Omega)} \\
&\lesssim (\|\partial_t^2\eta\|_{H^{1/2}(\Gamma)}+\|\partial_t\eta\|_{H^{5/2}(\Gamma)}^2) \|\nabla \partial_t u\|_{H^2(\Omega)} + \|\partial_t\eta\|_{H^{5/2}(\Gamma)}\|\nabla\partial_t^2 u\|_{L^2(\Omega)} \\
&\quad + (\|\partial_t^3 \eta\|_{H^{1/2}(\Gamma)} +\|\partial_t\eta\|_{H^{5/2}(\Gamma)}\|\partial_t^2\eta\|_{H^{1/2}(\Gamma)} + \|\partial_t\eta\|_{H^{5/2}(\Gamma)}^3)\|u\|_{H^3(\Omega)}.
\end{split}
\end{equation}
Combining the resulting inequality \eqref{F^32Est_2} with all inequalities in Lemma \ref{LemEstEta_H^1/2}, we obtain 
\begin{equation}\label{F^32Est_5}
\|\partial_t F^{3,2}\|_{L^2(\Omega)} \lesssim \cE_f (\cE_f+\|\nabla \partial_t u\|_{H^2(\Omega)}+ \|\nabla \partial_t^2 u\|_{L^2(\Omega)}).
\end{equation}
Combining \eqref{EstEta_H^1/2}, \eqref{F^32Est_4} and \eqref{F^32Est_5} gives us that 
\begin{equation}\label{F^32Est_6}
\|\partial_t(JF^{3,2})\|_{L^2(\Omega)}  \lesssim \cE_f (\cE_f+\|\nabla \partial_t u\|_{H^2(\Omega)}+ \|\nabla \partial_t^2 u\|_{L^2(\Omega)}).
\end{equation}
The inequality \eqref{BoundD_tJF} follows from \eqref{F^32Est_1}, \eqref{F^32Est_2} and \eqref{F^32Est_6}.
\end{proof}

We are in position to prove Proposition \ref{PropEstPartial_t^lU_L2}.
\begin{proof}[Proof of Proposition \ref{PropEstPartial_t^lU_L2}]
In view of \eqref{EqU_L^2_L=0} at order $l=0$, we have 
\begin{equation}\label{1stEstPropEstD_tU}
\begin{split}
&\frac12 \Big( \int_{\Omega}(\rho_0 +\rho_0'\theta+\zeta) J| u(t)|^2 + \int_\Gamma g\rho_+ |\eta(t) |^2 \Big) +\frac12 \mu\int_0^t \int_{\Omega} J|\bS_\cA u(s)|^2 ds\\
&= \frac12 \Big( \int_{\Omega}(\rho_0 +\rho_0'\theta+\zeta) J| u(t)|^2 + \int_\Gamma g\rho_+ |\eta(t) |^2 \Big)\Big|_{t=0}+\frac12 \int_0^t \int_{\Omega}\partial_t((\rho_0+\rho_0'\theta+\zeta)J)|u|^2(s) ds  \\
&\quad+\int_0^t \int_{\Omega} J(F^2\cdot u- g\zeta u_3)ds
\end{split}
\end{equation}
We first estimate the l.h.s of \eqref{EstPartial_t^lU_L2}.
Notice that
\[
\begin{split}
 J\|\bS_\cA u\|_{L^2(\Omega)}^2 &=  \|\bS u\|_{L^2(\Omega)}^2 +\int_\Omega (J-1)|\bS u|^2 +\int_\Omega J(\bS_\cA u+\bS u) : (\bS_\cA u-\bS u),
 \end{split}
 \]
and that
\[\begin{split}
\bS_\cA u\pm \bS u = (\cA_{ik}\pm \delta_{ik})\partial_k u_j + (\cA_{jk}\pm \delta_{jk})\partial_k u_j,
\end{split}\]
we  use \eqref{CoefEst_21} to obtain
\[\begin{split}
\int_\Omega J(\bS_\cA u+\bS u) : (\bS_\cA u-\bS u) &= 4 \int_\Omega (A^2 (\partial_1u_2+\partial_2 u_1)^2 + B^2(\partial_1 u_3+\partial_3 u_1)^2)\\
&\lesssim \|(A,B)\|_{H^2(\Omega)}^2 \|\nabla u\|_{L^2(\Omega)}^2\\
&\lesssim \cE_f^4.
\end{split}\]
 Note also that $\|J-1\|_{L^\infty(\Omega)}\lesssim 1$ (see \eqref{CoefEstimates}), we use Korn's inequality \eqref{KornIne} to have
\begin{equation}\label{2ndEstPropEstD_tU}
\begin{split}
 J\|\bS_\cA u\|_{L^2(\Omega)}^2  &\gtrsim \|\nabla u\|_{L^2(\Omega)}^2  -\cE_f^3.
\end{split}
\end{equation}
Due to Sobolev embedding, we then have 
\begin{equation}\label{1_2ndEstPropEstD_tU}
\begin{split}
\inf_\Omega (\rho_0+\rho_0'\theta+\zeta) &\geq \rho_- -C \|(\theta,\zeta)\|_{H^2(\Omega)} \geq  \frac12\rho_-.
\end{split}
\end{equation}
The l.h.s of \eqref{1stEstPropEstD_tU} will be estimated as
\begin{equation}\label{2_2ndEstPropEstD_tU}
\begin{split}
& \int_{\Omega}(\rho_0 +\rho_0'\theta+\zeta) J| u(t)|^2 + \int_\Gamma g\rho_+ |\eta(t) |^2 + \mu\int_0^t \int_{\Omega} J|\bS_\cA u(s)|^2 ds \\
&\qquad\gtrsim \|u(t)\|_{L^2(\Omega)}^2+ \|\eta(t)\|_{L^2(\Gamma)}^2+ \int_0^t \|\nabla  u(s)\|_{L^2(\Omega)}^2 ds-\int_0^t\cE_f^3(s)ds. 
\end{split}
\end{equation}
We now estimate the r.h.s of \eqref{1stEstPropEstD_tU}. By Gagliardo-Nirenberg's inequality (see \eqref{GN_ine}) and Sobolev embedding, one has
\[\begin{split}
&\|\partial_t ((\rho_0+\rho_0'\theta+\zeta)J)\|_{L^\infty(\Omega)}\\
&\lesssim \|(\rho_0+\rho_0'\theta+\zeta)\partial_t J\|_{H^2(\Omega)} + \| (\rho_0'\partial_t\theta+\partial_t\zeta)J\|_{L^\infty(\Omega)}\\
&\lesssim (1+\|(\theta,\zeta)\|_{H^2(\Omega)})\|\partial_t\theta\|_{H^3(\Omega)} +\|(\partial_t\theta,\partial_t\zeta)\|_{H^2(\Omega)}(1+\|J-1\|_{L^\infty(\Omega)}).
\end{split}\]
Together with Lemma \ref{LemEstNablaQ_Pf}, \eqref{EstEta_H^1/2} and \eqref{CoefEstimates}, we observe 
\begin{equation}\label{EstDtJ_L}
\begin{split}
\|\partial_t ((\rho_0+\rho_0'\theta+\zeta)J)\|_{L^\infty(\Omega)}  &\lesssim (1+\|\eta\|_{H^{3/2}(\Gamma)}+\|\zeta\|_{H^2(\Omega)}) \|\partial_t\eta\|_{H^{5/2}(\Gamma)}\\
&\qquad +\|\partial_t\eta\|_{H^{3/2}(\Gamma)}+ \|\partial_t\zeta\|_{H^2(\Omega)}\\
&\lesssim \cE_f,
\end{split}
\end{equation}
which yields 
\begin{equation}\label{3rdEstPropEstD_tU}
\int_0^t \int_{\Omega}\partial_t((\rho_0+\rho_0'\theta+\zeta)J)|u|^2(s) ds\lesssim \int_0^t  \cE_f^3(s)ds.
\end{equation}
Furthermore, thanks to \eqref{CoefEstimates} and \eqref{F_lTermsL^2}, we get 
\begin{equation}\label{2_3rdEstPropEstD_tU}
\begin{split}
\int_{\Omega} J(F^2\cdot u- g\zeta u_3) &\lesssim   (\|J-1\|_{L^\infty(\Omega)}+1)(  \|F^2\|_{L^2(\Omega)}\|u\|_{L^2(\Omega)} +\|\zeta\|_{L^2(\Omega)}\|u_3\|_{L^2(\Omega)} ) \\
&\lesssim \cE_f^3 + \|\zeta\|_{L^2(\Omega)}\|u_3\|_{L^2(\Omega)}.
\end{split}
\end{equation}
Substituting  \eqref{2ndEstPropEstD_tU}, \eqref{2_2ndEstPropEstD_tU}, \eqref{3rdEstPropEstD_tU} and \eqref{2_3rdEstPropEstD_tU} into \eqref{1stEstPropEstD_tU}, we deduce $\eqref{EstPartial_t^lU_L2}_{l=0}$.

For $l=1$,  we make use of \eqref{EqU_L^2_L=1} at order $l=1$ to have that
\begin{equation}\label{4thEstPropEstD_tU}
\begin{split}
&\frac12 \Big( \int_\Omega (\rho_0+\rho_0'\theta+\zeta)J|\partial_t u|^2(t) + \int_\Gamma g\rho_+ |\partial_t\eta(t)|^2 - \int_\Omega g\rho_0'|u_3(t)|^2 \Big) +\frac{\mu}2 \int_0^t \|\bS_\cA \partial_t u(s)\|_{L^2(\Omega)}^2 ds\\
&= \frac12 \Big( \int_{\Omega}(\rho_0 +\rho_0'\theta+\zeta) J|\partial_t u|^2 + \int_\Gamma g\rho_+|\partial_t \eta |^2 -\int_{\Omega}g\rho_0'|u_3|^2\Big)\Big|_{t=0}- \int_0^t \int_\Omega gJ F^1\partial_t u_3(s) ds. \\
&\quad+\frac12 \int_0^t \int_{\Omega}\partial_t((\rho_0+\rho_0'\theta+\zeta)J)|\partial_t u|(s)^2ds  +\int_0^t \int_{\Omega} J (F^{2,1}\cdot \partial_t u+  F^{3,1} \partial_t q)(s)ds\\
&\quad -\int_0^t \int_\Gamma( g\rho_+\partial_t \eta F^{4,1} +F^{5,1}\cdot \partial_t u)(s)ds -\int_0^t \int_{\Omega}g\rho_0'(A\partial_3 u_1+ B \partial_3 u_2)\partial_t u_3(s)ds. 
\end{split}
\end{equation}
By a similar argument as  the proof of \eqref{2_2ndEstPropEstD_tU}, we estimate the l.h.s of \eqref{4thEstPropEstD_tU} as 
\begin{equation}\label{5thEstPropEstD_tU}
\begin{split}
& \int_\Omega (\rho_0+\rho_0'\theta+\zeta)J|\partial_t u|^2(t) + \int_\Gamma g\rho_+ |\partial_t\eta(t)|^2 - \int_\Omega g\rho_0'|u_3(t)|^2  +\mu \int_0^t \|\bS_\cA \partial_t u(s)\|_{L^2(\Omega)}^2 ds\\
&\gtrsim \|\partial_t u(t)\|_{L^2(\Omega)}^2 +\|\partial_t \eta(t)\|_{L^2(\Gamma)}^2 + \int_0^t \|\nabla \partial_t u(s)\|_{L^2(\Omega)}^2-\|u_3(t)\|_{L^2(\Omega)}^2 -\int_0^t\cE_f^3(s)ds.
\end{split}
\end{equation}
For the r.h.s of \eqref{4thEstPropEstD_tU}, we use \eqref{EstDtJ_L} to obtain that 
\begin{equation}\label{6thEstPropEstD_tU}
\int_0^t \int_{\Omega}\partial_t((\rho_0+\rho_0'\theta+\zeta)J)|\partial_t u|^2(s) ds\lesssim \int_0^t  \cE_f^3(s)ds.
\end{equation}
Next, thanks to \eqref{CoefEst_21}, we see that 
\begin{equation}
\int_0^t \int_{\Omega}g\rho_0'(A\partial_3 u_1+ B \partial_3 u_2)\partial_t u_3(s)ds \lesssim \int_0^t \cE_f^2\|(A,B)\|_{H^2(\Omega)}(s) ds \lesssim \int_0^t \cE_f^3(s)ds.
\end{equation}
Let us use \eqref{CoefEstimates} and \eqref{F_lTermsL^2} to estimate that
\begin{equation}\label{7thEstPropEstD_tU}
\begin{split}
&\int_0^t \int_{\Omega} J (F^{2,1}\cdot \partial_t u+  F^{3,1} \partial_t q-gF^1 \partial_3 u)(s)ds -\int_0^t \int_\Gamma( g\rho_+\partial_t \eta F^{4,1} +F^{5,1}\cdot \partial_t u)(s)ds\\
&\lesssim  \int_0^t ((\|J-1\|_{L^\infty(\Omega)}+1)\|(F^1,F^{2,1}, F^{3,1})(s)\|_{L^2(\Omega)}+\|(F^{4,1}, F^{5,1})(s)\|_{L^2(\Gamma)})\cE_f(s)ds \\
&\lesssim \int_0^t \cE_f^3(s)ds.
\end{split}
\end{equation}
Combining \eqref{5thEstPropEstD_tU}, \eqref{6thEstPropEstD_tU} and \eqref{7thEstPropEstD_tU}, we obtain
\begin{equation}\label{8thEstPropEstD_tU}
\begin{split}
\|\partial_t u(t)\|_{L^2(\Omega)}^2+\|\partial_t\eta(t)\|_{L^2(\Gamma)}^2 +\int_0^t \|\nabla\partial_t u(s)\|_{L^2(\Omega)}^2 ds &\lesssim \cE_f^2(0)+ \|u_3(t)\|_{L^2(\Omega)}^2  + \int_0^t \cE_f^3(s) ds.
\end{split}
\end{equation}
As a consequence of \eqref{8thEstPropEstD_tU} and $\eqref{EstPartial_t^lU_L2}_{l=0}$,  the inequality $\eqref{EstPartial_t^lU_L2}_{l=1}$ follows.

For $l=2$, we use \eqref{EqU_L^2_L=1} at order $l=2$ to have that 
\begin{equation}\label{1_EqU_L^2_L=2}
\begin{split}
&\frac12 \Big( \int_{\Omega}(\rho_0 +\rho_0'\theta+\zeta) J|\partial_t^2 u|^2(t) + \int_\Gamma g\rho_+|\partial_t^2 \eta(t) |^2 -\int_{\Omega}g\rho_0'|\partial_t u_3(t)|^2\Big) \\
&\qquad+\frac12 \mu\int_0^t \int_{\Omega} J|\bS_\cA\partial_t^2 u(s)|^2ds \\
&=\frac12 \Big( \int_{\Omega}(\rho_0 +\rho_0'\theta+\zeta) J|\partial_t^2 u|^2 + \int_\Gamma g\rho_+|\partial_t^2 \eta |^2 -\int_{\Omega}g\rho_0'|\partial_t u_3|^2\Big)\Big|_{t=0} \\
&\quad+ \frac12 \int_0^t \int_{\Omega}\partial_t((\rho_0+\rho_0'\theta+\zeta)J)|\partial_t^2 u(s)|^2ds  +\int_0^t \int_{\Omega} J (F^{2,2}\cdot \partial_t^2 u+  F^{3,2} \partial_t^2 q)(s)ds\\
&\quad -\int_0^t \int_\Gamma( g\rho_+\partial_t^2 \eta F^{4,2} +F^{5,2}\cdot \partial_t^2 u)(s) ds+ \int_0^t \int_\Omega g\rho_0 J F^{3,1}\partial_t^2 u_3(s)ds \\
&\quad- \int_0^t \int_{\Omega}g\rho_0'(A\partial_t \partial_3 u_1+ B\partial_t \partial_3 u_2)\partial_t^2 u_3(s) ds - \int_0^t \int_\Omega gJ F^{1,1}\partial_t^2 u_3(s)ds. 
\end{split}
\end{equation}
We follow the previous arguments to observe that
\begin{equation}\label{1_EstDtU_L2}
\begin{split}
&\|\partial_t^2 u(t)\|_{L^2(\Omega)}^2 +\|\partial_t^2\eta(t)\|_{L^2(\Gamma)}^2 + \int_0^t \|\nabla \partial_t^2 u(s)\|_{L^2(\Omega)}^2 ds \\
&\lesssim \cE_f^2(0) + \|\partial_t u_3(t)\|_{L^2(\Omega)}^2 +  \int_0^t \cE_f^2(s) \|(A,B)(s)\|_{H^2(\Omega)} ds  +\int_0^t \|(F^{4,2}, F^{5,2})(s)\|_{L^2(\Gamma)} \cE_f(s) ds\\
&\qquad+\int_0^t (\|J-1\|_{L^\infty(\Omega)}+1)\|(F^{1,1}, F^{2,2}, F^{3,1})(s)\|_{L^2(\Omega)}ds +\int_0^t \int_\Omega (JF^{3,2}\partial_t^2q)(s)ds.
\end{split}
\end{equation}
Since $\partial_t^2 q$ does not appear in $\cE_f$ or $\cD_f$, we use the integration in time to have 
\[\begin{split}
\int_0^t \int_\Omega (JF^{3,2}\partial_t^2q)(s)ds &=  \int_\Omega(\partial_t q JF^{3,2})(t) - \int_\Omega (\partial_t q JF^{3,2})(0)  -\int_0^t \int_\Omega \partial_tq(s)  \partial_t(JF^{3,2})(s)ds.
\end{split}\]
Thanks to \eqref{BoundD_tJF}, we observe 
\begin{equation}\label{2_EstDtU_L2}
\int_0^t \int_\Omega (JF^{3,2}\partial_t^2q)(s)ds  \lesssim \cE_f^2(0) + \cE_f^3(t) +\int_0^t \cE_f^2(\cE_f+\|\nabla \partial_t u\|_{H^2(\Omega)}+ \|\nabla \partial_t^2 u\|_{L^2(\Omega)})(s) ds.
\end{equation}
It follows from  \eqref{CoefEstimates}, \eqref{F_lTermsL^2} and \eqref{F_2TermsL^2} that 
\begin{equation}\label{3_EstDtU_L2}
\begin{split}
&\int_0^t ((\|J-1\|_{L^\infty(\Omega)}+1)\|(F^{2,2}, F^{3,1}, F^{1,1})(s)\|_{L^2(\Omega)}+\|(F^{4,2}, F^{5,2})(s)\|_{L^2(\Gamma)}) \cE_f(s) ds \\
&\lesssim\int_0^t \cE_f^2(\cE_f+\|\nabla \partial_t u\|_{H^2(\Omega)}+ \|\nabla \partial_t^2 u\|_{L^2(\Omega)})(s) ds.
\end{split}
\end{equation}
Using \eqref{1_EstDtU_L2}, \eqref{2_EstDtU_L2}, \eqref{3_EstDtU_L2} and \eqref{CoefEst_21},  we deduce that 
\[
\begin{split}
&\|\partial_t^2 u(t)\|_{L^2(\Omega)}^2 +\|\partial_t^2\eta(t)\|_{L^2(\Gamma)}^2 + \int_0^t \|\nabla \partial_t^2 u(s)\|_{L^2(\Omega)}^2 ds \\
&\lesssim  \cE_f^2(0) + \|\partial_t u_3(t)\|_{L^2(\Omega)}^2 +\cE_f^3(t)+\int_0^t \cE_f^2(\cE_f+\|\nabla \partial_t u\|_{H^2(\Omega)}+ \|\nabla \partial_t^2 u\|_{L^2(\Omega)})(s) ds.\end{split}
\]
We obtain $\eqref{EstPartial_t^2U_L2}$ thanks to the resulting inequality and $\eqref{EstPartial_t^lU_L2}_{l=1}$.
\end{proof}

\subsection{Horizontal estimates of the perturbation velocity} 
We continue deriving  the mixed horizontal space-time derivatives of $u$. Note that $\rho_0$ only depends on $x_3$. Let $\beta=(\beta_1,\beta_2)\in \N^2$ and let us apply the horizontal derivative $\partial_h^\beta=\partial_1^{\beta_1}\partial_2^{\beta_2}$ to \eqref{EqPertur}, we obtain the following equations.
\begin{equation}\label{EqHorizontal}
\begin{cases}
\partial_t \partial_h^\beta \zeta +\rho_0' \partial_h^\beta  u_3= \partial_h^\beta \cQ^1\quad&\text{in } \Omega,\\
\rho_0\partial_t \partial_h^\beta  u+ \nabla\partial_h^\beta  q -\mu \Delta \partial_h^\beta  u +g\partial_h^\beta \zeta  e_3= \partial_h^\beta  \cQ^2 \quad&\text{in }\Omega,\\
\text{div} \partial_h^\beta u=\partial_h^\beta  \cQ^3 \quad&\text{in } \Omega,\\
\partial_t \partial_h^\beta \eta -\partial_h^\beta  u_3 =\partial_h^\beta  \cQ^4 \quad&\text{on } \Gamma,\\
(\partial_h^\beta q \text{Id}-\mu \bS\partial_h^\beta  u)  e_3= g\rho_+\partial_h^\beta\eta e_3+ \partial_h^\beta  \cQ^5\quad&\text{on } \Gamma.
\end{cases}
\end{equation}

\begin{proposition}\label{PropEstHorizonU}
The following inequalities hold 
\begin{equation}\label{EstU_tH^4}
\begin{split}
&\sum_{\beta \in \N^2, 1\leq |\beta|\leq 4} \Big(\|\partial_h^\beta u(t)\|_{L^2(\Omega)}^2  + \int_0^t \|\nabla \partial_h^\beta u(s)\|_{L^2(\Omega)}^2 ds \Big)\\
&\quad\leq C_7\Big( \cE_f^2(0)+ \varepsilon^3  \int_0^t (\cE_f^2(s)+ \|\nabla u_3(s)\|_{H^4(\Omega)}^2) ds+ \int_0^t \cE_f(\cE_f^2+\cD_f^2)(s) ds\Big)\\
&\qquad\quad+ C_7\varepsilon^{-27} \int_0^t (\|u_3(s)\|_{L^2(\Omega)}^2+\|\eta(s)\|_{L^2(\Omega)}^2)ds,
\end{split}
\end{equation}
and 
\begin{equation}\label{EstPartial_tU_H^2}
\begin{split}
&\sum_{\beta \in \N^2, 1\leq |\beta|\leq 2} \Big( \|\partial_h^\beta \partial_t u(t)\|_{L^2(\Omega)}^2+\int_0^t \| \nabla \partial_h^\beta\partial_t u(s)\|_{L^2(\Omega)}^2 ds  \Big)\\
&\quad\leq C_8 \Big(\cE_f^2(0)+  \varepsilon^3 \int_0^t (\cE_f^2(s)+ \|\nabla u_3(s)\|_{H^4(\Omega)}^2 )ds+\int_0^t\cE_f(\cE_f^2+\cD_f^2)(s)ds\Big)\\
&\qquad\quad+C_8 \varepsilon^{-27} \int_0^t (\|u_3(s)\|_{L^2(\Omega)}^2+\|\eta(s)\|_{L^2(\Omega)}^2)ds.
\end{split}
\end{equation}
\end{proposition}
To prove Proposition \ref{PropEstHorizonU}, we need the following lemma. 
\begin{lemma}\label{LemEstTermsQ}
The following inequalities hold
\begin{equation}\label{1stIneLemHorizon}
\begin{split}
&\|\cQ^1\|_{H^2(\Omega)}+\|\partial_t\cQ^1\|_{L^2(\Omega)} +\|\partial_t\cQ^2\|_{L^2(\Omega)}+\|\cQ^2\|_{H^2(\Omega)}+\|\partial_t\cQ^3\|_{H^1(\Omega)}\\
&\qquad\qquad+\|\cQ^3\|_{H^3(\Omega)} +\|\cQ^4\|_{H^{7/2}(\Gamma)}+\|\cQ^5\|_{H^{5/2}(\Gamma)}+ \|\partial_t \cQ^5\|_{H^{1/2}(\Gamma)} \lesssim \cE_f^2,
\end{split}
\end{equation}
and 
\begin{equation}\label{EstTildeQ1}
\begin{split}
\|\cQ^2\|_{H^3(\Omega)}+ \|\partial_t\cQ^2\|_{H^1(\Omega)}&+ \|\cQ^3\|_{H^4(\Omega)} + \| \partial_t\cQ^3\|_{H^2(\Omega)} \\
&+\|\cQ^5\|_{H^{7/2}(\Gamma)}+\|\partial_t\cQ^5\|_{H^{3/2}(\Gamma)}\lesssim  \cE_f(\cE_f+\cD_f).
\end{split}
\end{equation}
\end{lemma}
\begin{proof}
For \eqref{1stIneLemHorizon}, we only present estimates for some terms of the l.h.s, precisely, 
\[
\|\partial_t\cQ^3\|_{H^1(\Omega)}+\|\cQ^4\|_{H^{7/2}(\Gamma)}+\|\cQ_1^5\|_{H^{5/2}(\Gamma)} \lesssim \cE_f^2,
\]
the estimates of the other terms in the l.h.s of \eqref{1stIneLemHorizon} follow the same way.  To get $\|\partial_t\cQ^3\|_{H^1(\Omega)}\lesssim \cE_f^2$, we use  \eqref{ProductEst} and \eqref{CoefEst_22}, \eqref{CoefEst_23}  to bound each term of $\partial_t\cQ^3$ \eqref{TermQ34}. Indeed, we have 
\begin{equation}\label{1_BoundD_tQ^3}
\begin{split}
\|\partial_t ((1 -K)\partial_3u_3)\|_{H^1(\Omega)} &\lesssim  \| \partial_t K\|_{H^1(\Omega)}\|\partial_3 u_3\|_{H^3(\Omega)} + \|K-1\|_{H^3(\Omega)}\|\partial_t \partial_3 u_3\|_{H^1(\Omega)}\\
&\lesssim \|\partial_t\eta\|_{H^{3/2}(\Gamma)} \|u_3\|_{H^4(\Omega)}+ \|\eta\|_{H^{7/2}(\Gamma)} \|\partial_t u_3\|_{H^2(\Omega)}\\
&\lesssim \cE_f^2,
\end{split}
\end{equation}
and
\begin{equation}\label{2_BoundD_tQ^3}
\begin{split}
\|\partial_t&(AK \partial_3u_1+BK \partial_3u_2 )\|_{H^1(\Omega)} \\
&\lesssim \|\partial_t(AK,BK)\|_{H^1(\Omega)} \|\partial_3 u\|_{H^3(\Omega)} + \|(AK,BK)\|_{H^3(\Omega)} \|\partial_t\partial_3u\|_{H^1(\Omega)}\\
&\lesssim \|\partial_t\eta\|_{H^{3/2}(\Gamma)}\|u\|_{H^4(\Omega)}+\| \eta\|_{H^{7/2}(\Gamma)} \|\partial_t u\|_{H^2(\Omega)}\\
&\lesssim \cE_f^2.
\end{split}
\end{equation}
Hence, $\|\partial_t\cQ^3\|_{H^1(\Omega)}\lesssim \cE_f^2$ follows from \eqref{1_BoundD_tQ^3} and \eqref{2_BoundD_tQ^3}. We apply the product estimate \eqref{ProductEst} and the trace theorem to have that 
\[\begin{split}
\|\cQ^4\|_{H^{7/2}(\Gamma)} &\lesssim \|u_1\|_{H^{7/2}(\Gamma)}\|\partial_1\eta\|_{H^{7/2}(\Gamma)}+  \|u_2\|_{H^{7/2}(\Gamma)}\|\partial_2\eta\|_{H^{7/2}(\Gamma)}\\
&\lesssim \|u\|_{H^4(\Omega)} \|\eta\|_{H^{9/2}(\Gamma)}\\
&\lesssim \cE_f^2.
\end{split}\]
Moreover, using \eqref{ProductEst}, \eqref{CoefEst_22}, \eqref{CoefEst_23} again and the trace theorem, we show 
\[
\| \cQ_1^5\|_{H^{5/2}(\Gamma)} \lesssim \cE_f^2.
\]
 From the expression of $\cQ_1^5$ \eqref{TermQ5}, we have that
\begin{equation}\label{1_2ndIneLemHorizon}
\begin{split}
&\|\partial_1\eta (q- g\rho_+\eta-2\mu(\partial_1 u_1-AK \partial_3 u_1))\|_{H^{5/2}(\Gamma)}\\
&\qquad\lesssim \|\partial_1\eta\|_{H^{5/2}(\Gamma)}( \|(q,\eta, \partial_1 u_1)\|_{H^{5/2}(\Gamma)}+\|AK\|_{H^{5/2}(\Gamma)} \|\partial_3u_1\|_{H^{5/2}(\Gamma)})\\
&\qquad \lesssim \|\eta\|_{H^{7/2}(\Gamma)} (\|q\|_{H^3(\Omega)}+\|u_1\|_{H^4(\Omega)}+ \|\eta\|_{H^{5/2}(\Gamma)}+\|AK\|_{H^3(\Omega)}\|u_1\|_{H^4(\Omega)})\\
&\qquad\lesssim \|\eta\|_{H^{7/2}(\Gamma)}(\|q\|_{H^3(\Omega)}+\|u_1\|_{H^4(\Omega)}+ \|\eta\|_{H^{5/2}(\Gamma)}+ \|\eta\|_{H^{7/2}(\Gamma)}\|u_1\|_{H^4(\Omega)}),
\end{split}
\end{equation}
that
\begin{equation}\label{3_2ndIneLemHorizon}
\begin{split}
\|\partial_2 &\eta  (\partial_1 u_2+\partial_2u_1-AK\partial_3u_2-BK\partial_3u_1)\|_{H^{5/2}(\Gamma)} \\
&\lesssim \|\partial_2\eta\|_{H^{5/2}(\Gamma)}\|u\|_{H^{7/2}(\Gamma)}(1+ \|(AK,BK)\|_{H^{5/2}(\Gamma)}) \\
&\lesssim  \|\partial_2\eta\|_{H^{5/2}(\Gamma)}\|u\|_{H^{7/2}(\Gamma)}(1+ \|(AK,BK)\|_{H^3(\Omega)}) \\
& \lesssim \|\eta\|_{H^{7/2}(\Gamma)} \|u\|_{H^4(\Omega)}(1+\|\eta\|_{H^{7/2}(\Gamma)} ), 
\end{split}
\end{equation}
and that
\begin{equation}\label{5_2ndIneLemHorizon}
\begin{split}
\|(1-K)\partial_3 u_1+AK \partial_3 u_3\|_{H^{5/2}(\Gamma)} &\lesssim \|(K-1,AK)\|_{H^{5/2}(\Gamma)} \|\partial_3 u\|_{H^{5/2}(\Gamma)}\\
&\lesssim \|(K-1,AK)\|_{H^3(\Omega)} \|u\|_{H^4(\Omega)}\\
&\lesssim \|\eta\|_{H^{7/2}(\Gamma)}\| u\|_{H^4(\Omega)}
\end{split}
\end{equation}
Hence, the inequality $\| \cQ_1^5\|_{H^{5/2}(\Gamma)} \lesssim \cE_f^2$ follows from the three above estimates \eqref{1_2ndIneLemHorizon}, \eqref{3_2ndIneLemHorizon} and \eqref{5_2ndIneLemHorizon}.

Similarly, for \eqref{EstTildeQ1}, we show only
\[
 \|\partial_t\cQ_1^2\|_{H^1(\Omega)} + \| \cQ^3\|_{H^4(\Omega)}\lesssim \cE_f(\cE_f+\cD_f).
\]
The inequality $\|\cQ^3\|_{H^4(\Omega)}\lesssim \cE_f\cD_f$ (see $\cQ^3$ in \eqref{TermQ34}) is proven by  using \eqref{ProductEst} and \eqref{CoefEst_22}, \eqref{CoefEst_23}, 
\[
\begin{split}
\|\cQ^3\|_{H^4(\Omega)} \lesssim \|(AK,BK,K-1)\|_{H^4(\Omega)}
\|\partial_3 u\|_{H^4(\Omega)} \lesssim \|\eta\|_{H^{9/2}(\Gamma)}\|\nabla u\|_{H^4(\Omega)}.
\end{split}
\]
Let us prove $\|\partial_t \cQ_1^2\|_{H^1(\Omega)}\lesssim \cE_f(\cE_f+\cD_f)$ (see $\cQ_1^2$ in \eqref{TermQ2_1}). In view of \eqref{ProductEst} and Lemma \ref{LemEstNablaQ_Pf}, we obtain that
\begin{equation}\label{1_BoundD_tQ^1}
\begin{split}
\|\partial_t((\zeta+\rho_0'\theta)\partial_tu_1)\|_{H^1(\Omega)} &\lesssim \|(\partial_t\zeta,\partial_t\theta)\|_{H^1(\Omega)}\|\partial_t u_1\|_{H^3(\Omega)}+ \|(\zeta,\theta)\|_{H^3(\Omega)} \|\partial_t^2 u_1\|_{H^1(\Omega)}\\
&\lesssim (\|\partial_t\zeta\|_{H^1(\Omega)}+ \|\partial_t\eta\|_{H^{1/2}(\Gamma)}) \|\partial_t u_1\|_{H^3(\Omega)}  \\
&\qquad+ (\|\zeta\|_{H^3(\Omega)}+ \|\eta\|_{H^{5/2}(\Gamma)})\|\partial_t^2 u_1\|_{H^1(\Omega)}\\
&\lesssim \cE_f(\cE_f+\|\nabla\partial_t u_1\|_{H^2(\Omega)}+\|\nabla \partial_t^2 u_1\|_{L^2(\Omega)}).
\end{split}
\end{equation}
We further use \eqref{ProductEst}  to have
\[
\begin{split}
&\| \partial_t ((\rho_0+ \rho_0'\theta+\zeta)Ku_3\partial_3 u_1) \|_{H^1(\Omega)} \\
&\lesssim (1+\|(\theta,\zeta)\|_{H^3(\Omega)}) \| \partial_t (Ku_3\partial_3 u_1)\|_{H^1(\Omega)} +\|(\partial_t\zeta,\partial_t\theta)\|_{H^1(\Omega)} \| K u_3\partial_3u_1\|_{H^3(\Omega)}\\
&\lesssim (1+\|(\theta,\zeta)\|_{H^3(\Omega)}) \|\partial_t K\|_{H^1(\Omega)}\|u_3\|_{H^3(\Omega)}\|\partial_3 u_1\|_{H^3(\Omega)}  \\
&\quad+ (1+\|(\theta,\zeta)\|_{H^3(\Omega)})(\|K-1\|_{H^3(\Omega)}+1) ( \|\partial_t u_3\|_{H^1(\Omega)}\|\partial_3 u_1\|_{H^3(\Omega)}+ \|\partial_t\partial_3 u_1\|_{H^1(\Omega)}\|u_3\|_{H^3(\Omega)}) \\
&\quad+ \|(\partial_t\zeta,\partial_t\theta)\|_{H^1(\Omega)}(\|K-1\|_{H^3(\Omega)}+1)\|u_3\|_{H^3(\Omega)}\|\partial_3u_1\|_{H^3(\Omega)}.
\end{split}
\]
Thanks to Lemma \ref{LemEstNablaQ_Pf} and \eqref{CoefEst_22}, we deduce
\begin{equation}\label{2_BoundD_tQ^1}
\begin{split}
&\| \partial_t ((\rho_0+ \rho_0'\theta+\zeta)Ku_3\partial_3 u_1) \|_{H^1(\Omega)} \\
&\lesssim (1+\|\zeta\|_{H^3(\Omega)}+\|\eta\|_{H^{5/2}(\Gamma)}) \|\partial_t\eta\|_{H^{3/2}(\Gamma)}\|u\|_{H^4(\Omega)}^2\\
&\quad +(1+\|\zeta\|_{H^3(\Omega)}+\|\eta\|_{H^{5/2}(\Gamma)}) (\|\eta\|_{H^{7/2}(\Gamma)}+1) \|\partial_t u\|_{H^2(\Omega)} \|u\|_{H^4(\Omega)}\\
&\quad+ (\|\partial_t\zeta\|_{H^1(\Omega)}+\|\partial_t\eta\|_{H^{1/2}(\Gamma)}) (\|\eta\|_{H^{7/2}(\Gamma)}+1)\|u\|_{H^4(\Omega)}^2\\
&\lesssim \cE_f^2.
\end{split}
\end{equation}
Since $K^2-1=-J^{-2}(2\partial_3\theta +(\partial_3\theta)^2)$, let us use \eqref{ProductEst} to obtain 
\[
\begin{split}
&\|\partial_t ((K^2+A^2+B^2-1)\partial_{33}^2 u_1  - 2AK \partial_{13}^2 u_1-2BK \partial_{23}^2 u_1)\|_{H^1(\Omega)} \\
&\lesssim (\|(A^2,B^2,AK,BK)\|_{H^3(\Omega)}+\|K^2-1\|_{H^3(\Omega)}) \|\partial_t u_1\|_{H^3(\Omega)} \\
&\quad+ \|\partial_t(A^2,B^2,K^2-1, AK,BK)\|_{H^1(\Omega)}\|\nabla^2  u_1\|_{H^3(\Omega)}\\
&\lesssim (\|(A,B)\|_{H^3(\Omega)}^2 + \|(AK,BK)\|_{H^3(\Omega)} + \|\partial_3\theta\|_{H^3(\Omega)}(1+ \|\partial_3\theta\|_{H^3(\Omega)}))\|\partial_t u_1\|_{H^3(\Omega)}\\
&\quad+ (\|(A,B)\|_{H^3(\Omega)}\|(\partial_t A,\partial_tB)\|_{H^1(\Omega)} + \|\partial_t\partial_3\theta\|_{H^1(\Omega)}(1+\|\partial_3\theta\|_{H^3(\Omega)}) ) \|\nabla^2u_1\|_{H^3(\Omega)}\\
&\quad+ \|\partial_t(AK,BK)\|_{H^1(\Omega)} \|\nabla^2u_1\|_{H^3(\Omega)}.
\end{split}
\]
Owing to  \eqref{CoefEst_21} and \eqref{CoefEst_23}, we deduce 
\begin{equation}\label{4_BoundD_tQ^1}
\begin{split}
&\|\partial_t ((K^2+A^2+B^2-1)\partial_{33}^2 u_1  - 2AK \partial_{13}^2 u_1-2BK \partial_{23}^2 u_1)\|_{H^1(\Omega)} \\
&\lesssim \|\eta\|_{H^{7/2}(\Gamma)}(1+\|\eta\|_{H^{7/2}(\Gamma)})  \|\partial_t u_1\|_{H^3(\Omega)} + \|\partial_t\eta\|_{H^{3/2}(\Gamma)}(1+\|\eta\|_{H^{7/2}(\Gamma)})  \|\nabla^2 u_1\|_{H^3(\Omega)} \\
&\lesssim \cE_f (\cE_f+\|\nabla \partial_t u_1\|_{H^2(\Omega)}+ \|\nabla u_1\|_{H^4(\Omega)}).
\end{split}
\end{equation}
We continue using \eqref{ProductEst}, Lemma \ref{LemEstNablaQ_Pf} and \eqref{CoefEst_23} to get
\begin{equation}\label{3_BoundD_tQ^1}
\begin{split}
\|\partial_t &(AK(\partial_3 q- g\rho_0'\theta))\|_{H^1(\Omega)} \\
&\lesssim \|\partial_t(AK)\|_{H^1(\Omega)} \|(q,\theta)\|_{H^3(\Omega)} +  \|AK\|_{H^3(\Omega)} (\|\partial_t \partial_3 q\|_{H^1(\Omega)} +\|\partial_t\theta\|_{H^1(\Omega)})\\
&\lesssim \|\partial_t\eta\|_{H^{3/2}(\Gamma)} (\|q\|_{H^3(\Omega)}+\|\eta\|_{H^{5/2}(\Gamma)})+ \|\eta\|_{H^{7/2}(\Gamma)}( \|\partial_tq\|_{H^2(\Omega)} + \|\partial_t\eta\|_{H^{1/2}(\Gamma)})\\
&\lesssim \cE_f(\cE_f+ \|\partial_tq\|_{H^2(\Omega)}).
\end{split}
\end{equation}
From the product estimate \eqref{ProductEst}, we obtain also
\begin{equation}\label{0_5_BoundD_tQ^1}
\begin{split}
&\|\partial_t ((K\partial_3K (A^2+B^2+1) -\partial_1(AK)-\partial_2(BK)-A\partial_1K-B\partial_2K)\partial_3 u_1)\|_{H^1(\Omega)} \\
&\lesssim (\|K\partial_3K(A^2+B^2+1)\|_{H^3(\Omega)} +\|\nabla(AK,BK)\|_{H^3(\Omega)} )\|\partial_t\partial_3 u_1\|_{H^1(\Omega)}\\
&\qquad +(\|A\partial_1K\|_{H^3(\Omega)}+\|B\partial_2K\|_{H^3(\Omega)})\|\partial_t\partial_3 u_1\|_{H^1(\Omega)} \\
&\qquad +(\|\partial_t(K\partial_3K)(A^2+B^2+1)\|_{H^1(\Omega)}+\|\nabla\partial_t(AK,BK)\|_{H^1(\Omega)})\|\partial_3u_1\|_{H^3(\Omega)}\\
&\qquad +(\|(\partial_tA \partial_1K,\partial_tB\partial_2K)\|_{H^1(\Omega)}+ \|(A\partial_t\partial_1K,B\partial_t\partial_2K)\|_{H^1(\Omega)})\|\partial_3u_1\|_{H^3(\Omega)}.
\end{split}
\end{equation}
We will bound each term in the r.h.s of \eqref{0_5_BoundD_tQ^1}. Thanks to the product estimate \eqref{ProductEst} again and \eqref{CoefEst_21}, \eqref{CoefEst_22}, we have 
\begin{equation}\label{1_5_BoundD_tQ^1}
\begin{split}
\|K\partial_3K(A^2+B^2+1)\|_{H^3(\Omega)} &\lesssim \|K\partial_3K\|_{H^3(\Omega)}(1+\|(A,B)\|_{H^3(\Omega)}^2) \\
&\lesssim (\|K-1\|_{H^3(\Omega)}+1)\|\partial_3K\|_{H^3(\Omega)}(1+\|(A,B)\|_{H^3(\Omega)}^2) \\
&\lesssim (1+\|\eta\|_{H^{7/2}(\Gamma)}) \|\eta\|_{H^{9/2}(\Gamma)}(1+\|\eta\|_{H^{7/2}(\Gamma)}^2) \\
&\lesssim \cE_f.
\end{split}
\end{equation}
In a same way, we have
\begin{equation}\label{2_5_BoundD_tQ^1}
\begin{split}
\|A\partial_1K\|_{H^3(\Omega)}+\|B\partial_2K\|_{H^3(\Omega)}  &\lesssim \|(A,B)\|_{H^3(\Omega)}\|\nabla K\|_{H^3(\Omega)} \lesssim \|\eta\|_{H^{7/2}(\Gamma)}\|\eta\|_{H^{9/2}(\Gamma)}\lesssim \cE_f^2,
\end{split}
\end{equation}
and
\begin{equation}\label{3_5_BoundD_tQ^1}
\begin{split}
\|(\partial_tA \partial_1K,\partial_tB\partial_2K)\|_{H^1(\Omega)} &\lesssim \|(\partial_tA,\partial_tB)\|_{H^1(\Omega)} \|\nabla K\|_{H^3(\Omega)} \lesssim \|\partial_t\eta\|_{H^{3/2}(\Gamma)}\|\eta\|_{H^{9/2}(\Gamma)}\lesssim \cE_f^2. 
\end{split}
\end{equation}
Using also \eqref{EstEta_H^1/2}, we deduce 
\begin{equation}\label{0_5_BoundD_tQ^1}
\begin{split}
 \|(A\partial_t\partial_1K,B\partial_t\partial_2K)\|_{H^1(\Omega)} &\lesssim \|(A,B)\|_{H^3(\Omega)} \|\partial_t \nabla K\|_{H^1(\Omega)} \lesssim \|\eta\|_{H^{7/2}(\Gamma)}\|\partial_t \eta\|_{H^{5/2}(\Gamma)} \lesssim \cE_f^2,
\end{split}
\end{equation}
and
\begin{equation}\label{4_5_BoundD_tQ^1}
\begin{split}
\|\nabla(AK,BK)\|_{H^3(\Omega)}+ \|\nabla\partial_t(AK,BK)\|_{H^1(\Omega)} &\lesssim \|\eta\|_{H^{9/2}(\Gamma)} + \|\partial_t\eta\|_{H^{5/2}(\Gamma)} \lesssim \cE_f.
\end{split}
\end{equation}
Thanks to \eqref{ProductEst}, \eqref{CoefEst_22} and \eqref{EstEta_H^1/2} again, let us estimate the term $\|\partial_t(K\partial_3K)\|_{H^1(\Omega)}$ as follows
\[\begin{split}
\|\partial_t(K\partial_3K)\|_{H^1(\Omega)} &\lesssim \|\partial_tK\|_{H^1(\Omega)}\|\partial_3K\|_{H^3(\Omega)} + (1+\|K-1\|_{H^3(\Omega)})\|\partial_t\partial_3K\|_{H^1(\Omega)} \\
&\lesssim \|\partial_t\eta\|_{H^{3/2}(\Gamma)} \|\eta\|_{H^{9/2}(\Gamma)}+ (1+ \|\eta\|_{H^{7/2}(\Gamma)})\|\partial_t\eta\|_{H^{5/2}(\Gamma)}\\
&\lesssim \cE_f.
\end{split}\]
Owing to \eqref{CoefEst_21}, this yields
\begin{equation}\label{5_5_BoundD_tQ^1}
\begin{split}
\|\partial_t(K\partial_3K)(A^2+B^2+1)\|_{H^1(\Omega)} &\lesssim \|\partial_t(K\partial_3K)\|_{H^1(\Omega)} (1+\|(A,B)\|_{H^3(\Omega)}^2) \\
&\lesssim  \cE_f(1+\|\eta\|_{H^{7/2}(\Gamma)}^2)\lesssim \cE_f.
\end{split}
\end{equation}
The five inequalities above \eqref{1_5_BoundD_tQ^1}, \eqref{2_5_BoundD_tQ^1}, \eqref{3_5_BoundD_tQ^1}, \eqref{4_5_BoundD_tQ^1} and \eqref{5_5_BoundD_tQ^1} help us to obtain from \eqref{0_5_BoundD_tQ^1} that
\begin{equation}\label{5_BoundD_tQ^1}
\begin{split}
&\|\partial_t ((K\partial_3K (A^2+B^2+1) -\partial_1(AK)-\partial_2(BK)-A\partial_1K-B\partial_2K)\partial_3 u_1)\|_{H^1(\Omega)}\\
&\lesssim \cE_f(\|\partial_t u_1\|_{H^2(\Omega)}+ \|u_1\|_{H^4(\Omega)}) \\
&\lesssim \cE_f^2.
\end{split}
\end{equation}
Combining \eqref{1_BoundD_tQ^1}, \eqref{2_BoundD_tQ^1}, \eqref{4_BoundD_tQ^1}, \eqref{3_BoundD_tQ^1} and \eqref{5_BoundD_tQ^1}, we conclude 
\[
\begin{split}
\|\partial_t\cQ_1^2\|_{H^1(\Omega)} &\lesssim \cE_f(\cE_f+ \|\nabla \partial_t^2 u_1\|_{L^2(\Omega)}+\|\partial_tq\|_{H^2(\Omega)}+\|\nabla\partial_t u_1\|_{H^2(\Omega)})\lesssim \cE_f(\cE_f+\cD_f).
\end{split}
\]
\end{proof}

We are in position to show Proposition \ref{PropEstHorizonU}.
\begin{proof}[Proof of Proposition \ref{PropEstHorizonU}]
For any $\beta \in \N^2$ such that $1\leq |\beta|\leq 4$, multiplying by $\partial_h^\beta u$ on both sides of $\eqref{EqHorizontal}_2$ and integrating over $\Omega$, one has the identity
\[\begin{split}
\frac12 \frac{d}{dt} \int_\Omega \rho_0|\partial_h^\beta u|^2 + \int_\Omega(\nabla\partial_h^\beta q - \mu \Delta\partial_h^\beta u)\cdot \partial_h^\beta u +\int_\Omega g\rho_0'\partial_h^\beta \zeta \partial_h^\beta u_3 =\int_\Omega \partial_h^\beta u\cdot \partial_h^\beta \cQ^2.
\end{split}\]
Using the integration by parts and $\eqref{EqHorizontal}_{3,5}$, one has 
\begin{equation}\label{1stEstU_tH^4}
\begin{split}
\frac12 \frac{d}{dt} \int_\Omega \rho_0|\partial_h^\beta u|^2 + \frac12 \int_\Omega \mu|\bS \partial_h^\beta u|^2 
&=-\int_\Omega g\rho_0' \partial_h^\beta \zeta \partial_h^\beta u_3+ \int_\Omega \partial_h^\beta u\cdot \partial_h^\beta \cQ^2 - \int_\Gamma g\rho_+\partial_h^\beta\eta \partial_h^\beta u_3 \\
&\qquad+ \int_\Omega \partial_h^\beta q\partial_h^\beta \cQ^3 - \int_\Gamma\partial_h^\beta u\cdot\partial_h^\beta \cQ^5.
\end{split}
\end{equation}
We estimate each integral in the r.h.s of \eqref{1stEstU_tH^4}.  For the first integral, we use Young's inequality and \eqref{EstJensen} to get that
\begin{equation}\label{0_1stEstU_tH^4}
\begin{split}
\int_\Omega g\rho_0' \partial_h^\beta \zeta \partial_h^\beta u_3  \lesssim \|\zeta\|_{H^4(\Omega)}\|u_3\|_{H^4(\Omega)} 
&\lesssim  \varepsilon^3  \|\zeta\|_{H^4(\Omega)}^2+ \varepsilon^{-3} \|u_3\|_{H^4(\Omega)}^2 \\
&\lesssim \varepsilon^3 ( \|\zeta\|_{H^4(\Omega)}^2 +\|u_3\|_{H^5(\Omega)}^2 ) +  \varepsilon^{-27}\|u_3\|_{L^2(\Omega)}^2.
\end{split}
\end{equation}
For the third integral, it follows from the trace theorem and Young's inequality  that
\[
\begin{split}
\int_\Gamma g\rho_+\partial_h^\beta\eta \partial_h^\beta u_3 \lesssim \|\partial_h^\beta \eta\|_{H^{-1/2}(\Gamma)} \|\partial_h^\beta u_3\|_{H^{1/2}(\Gamma)}  &\lesssim \|\eta\|_{H^{|\beta|-1/2}(\Gamma)} \|u_3\|_{H^{|\beta|+1/2}(\Gamma)}\\
&\lesssim \|\eta\|_{H^{7/2}(\Gamma)} \|u_3\|_{H^5(\Omega)} \\
&\lesssim \varepsilon^3 \|u_3\|_{H^5(\Omega)}^2 + \varepsilon^{-3}\|\eta\|_{H^{7/2}(\Gamma)}^2.
\end{split}
\]
Thanks to \eqref{EstJensen} again, we have
\[
\|\eta\|_{H^{7/2}(\Gamma)}^2 \lesssim \varepsilon^6 \|\eta\|_{H^{9/2}(\Gamma)}^2 +\varepsilon^{-21} \|\eta\|_{L^2(\Gamma)}^2. 
\]
Hence,
\begin{equation}\label{1_1stEstU_tH^4}
\begin{split}
\int_\Gamma g\rho_+\partial_h^\beta\eta \partial_h^\beta u_3 \lesssim  \varepsilon^3 (\|\eta\|_{H^{9/2}(\Gamma)}^2 + \| u_3\|_{H^5(\Omega)}^2 )+ \varepsilon^{-24} \|\eta\|_{L^2(\Gamma)}^2.
\end{split}
\end{equation}
For the fourth integral, we use Cauchy-Schwarz's inequality and \eqref{EstTildeQ1} to have 
\begin{equation}\label{2_1stEstU_tH^4}
 \int_\Omega \partial_h^\beta q\partial_h^\beta \cQ^3 \lesssim \|q\|_{H^4(\Omega)} \|\cQ^3\|_{H^4(\Omega)} \lesssim \cE_f\cD_f(\cE_f+\cD_f).
\end{equation}
For the second and fifth integral, we split into two cases.  For $\beta \in \N^2$ such that $1\leq |\beta|\leq 3 $, we use trace theorem and \eqref{EstTildeQ1} to bound
\begin{equation}\label{3rdEstU_tH^4}
\begin{split}
\int_\Omega \partial_h^\beta u\cdot \partial_h^\beta \cQ^2 -\int_\Gamma \partial_h^\beta u\cdot\partial_h^\beta \cQ^5 &\lesssim \|u\|_{H^3(\Omega)} \|\cQ^2\|_{H^3(\Omega)}+\|u\|_{H^3(\Gamma)} \|\cQ^5\|_{H^3(\Gamma)} \\
&\lesssim  \|u\|_{H^3(\Omega)} \|\cQ^2\|_{H^3(\Omega)}+\|u\|_{H^4(\Omega)} \|\cQ^5\|_{H^3(\Gamma)}\\
&\lesssim \cE_f^2(\cE_f+\cD_f).
\end{split}
\end{equation}
For $\beta=(\beta_1,\beta_2) \in \N^2$ such that $|\beta|=4$, we assume $\beta_1\geq 1$ and  write 
\begin{equation}\label{BetaPM}
\beta_- =(\beta_1-1,\beta_2) \quad\text{and}\quad \beta_+=(\beta_1+1, \beta_2).
\end{equation}
Thanks to \eqref{EstTildeQ1} again, we estimate that 
\begin{equation}\label{4thEstU_tH^4}
\begin{split}
\Big|\int_\Omega  \partial_h^\beta u\cdot\partial_h^\beta \cQ^2\Big|= \Big| \int_\Omega \partial_h^{\beta_+} u\cdot \partial_h^{\beta_-} \cQ^2\Big| &\lesssim \|\partial_h^{\beta_+} u\|_{L^2(\Omega)} \|\partial_h^{\beta_-} \cQ^2\|_{L^2(\Omega)} \\
& \lesssim \|u\|_{H^5(\Omega)} \| \cQ^2\|_{H^3(\Omega)}\lesssim \cE_f(\cE_f^2+\cD_f^2).
\end{split}
\end{equation}
Using the trace theorem also, we have
\begin{equation}\label{5thEstU_tH^4}
\begin{split}
\Big| \int_\Gamma \partial_h^\beta u\cdot  \partial_h^\beta \cQ^5 \Big|= \Big|\int_\Gamma  \partial_h^{\beta_+} u\cdot  \partial_h^{\beta_-}\cQ^5 \Big| &\lesssim \|\partial_h^{\beta_+} u\|_{H^{-1/2}(\Gamma)} \|\partial_h^{\beta_-} \cQ^5\|_{H^{1/2}(\Gamma)}\\
&\lesssim \|u\|_{H^{|\beta_+|-1/2}(\Gamma)} \|\cQ^5\|_{H^{|\beta_-|+1/2}(\Gamma)}\\
&\lesssim \|u\|_{H^5(\Omega)} \|\cQ^5\|_{H^{7/2}(\Gamma)}\\
&\lesssim \cE_f(\cE_f^2+\cD_f^2).
\end{split}
\end{equation}
Substituting \eqref{0_1stEstU_tH^4}, \eqref{1_1stEstU_tH^4}, \eqref{2_1stEstU_tH^4},   \eqref{3rdEstU_tH^4}, \eqref{4thEstU_tH^4}, \eqref{5thEstU_tH^4} into \eqref{1stEstU_tH^4},  we obtain 
\begin{equation}\label{6thEstU_tH^4}
\begin{split}
\frac{d}{dt} \int_\Omega \rho_0|\partial_h^\beta u|^2 +  \int_\Omega \mu|\bS \partial_h^\beta u|^2 &\lesssim \varepsilon^3 (\cE_f^2 + \|\nabla u_3\|_{H^4(\Omega)}^2) + \cE_f(\cE_f^2+\cD_f^2)  \\
&\qquad+\varepsilon^{-27} (\|u_3\|_{L^2(\Omega)}^2+ \|\eta\|_{L^2(\Gamma)}^2).
\end{split}
\end{equation}
By Korn's inequality \eqref{KornIne}, one has 
\[
\int_\Omega \mu|\bS \partial_h^\beta u|^2 \gtrsim \|\nabla \partial_h^\beta u\|_{L^2(\Omega)}^2.
\]
Hence, we deduce from \eqref{6thEstU_tH^4} that
\[
\begin{split}
\frac{d}{dt} \int_\Omega \rho_0|\partial_h^\beta u|^2 +  \|\nabla \partial_h^\beta u\|_{L^2(\Omega)}^2 &\lesssim \varepsilon^3 (\cE_f^2 + \|\nabla u_3\|_{H^4(\Omega)}^2)+\varepsilon^{-27} (\|u_3\|_{L^2(\Omega)}^2+ \|\eta\|_{L^2(\Gamma)}^2) \\
&\qquad+ \cE_f(\cE_f^2+\cD_f^2).
\end{split}
\]
Integrating the resulting inequality in time, we obtain \eqref{EstU_tH^4}.

To prove \eqref{EstPartial_tU_H^2}, we  compute  from \eqref{EqHorizontal} that
\begin{equation}\label{EqDtHorizontal}
\begin{cases}
\rho_0 \partial_t^2 \partial_h^\beta u+ \nabla\partial_t \partial_h^\beta q -\mu\Delta\partial_t \partial_h^\beta u -g\rho_0'\partial_h^\beta u_3 e_3 =-g\partial_h^\beta \cQ^1 e_3+ \partial_t\partial_h^\beta \cQ^2 &\quad\text{in }\Omega, \\
\text{div}\partial_t\partial_h^\beta u =\partial_t\partial_h^\beta \cQ^3 &\quad\text{in }\Omega, \\
(\partial_t \partial_h^\beta q\text{Id} -\mu\bS \partial_t\partial_h^\beta u)e_3 = g\rho_+\partial_h^\beta u_3 e_3+ g\rho_+\partial_h^\beta \cQ^4 e_3+\partial_h^\beta \cQ^5 &\quad\text{on }\Gamma.
\end{cases}
\end{equation}
For any $\beta \in \N^2$ with $ |\beta|=1$ or 2, multiplying by $\partial_t\partial_h^\beta u$ on both sides of $\eqref{EqDtHorizontal}_1$ and integrating over $\Omega$, one has the identity 
\[\begin{split}
&\frac12 \frac{d}{dt} \int_\Omega( \rho_0|\partial_t \partial_h^\beta u|^2 -g\rho_0'|\partial_h^\beta u_3|^2)+ \int_\Omega(\nabla\partial_t\partial_h^\beta q - \mu \Delta\partial_t\partial_h^\beta u)\cdot \partial_t\partial_h^\beta u \\
&=\int_\Omega (-g \partial_h^\beta\cQ^1 \partial_t\partial_h^\beta u_3 +  \partial_t\partial_h^\beta \cQ^2 \cdot \partial_t \partial_h^\beta u).
\end{split}\]
Using the integration by parts, one has 
\[\begin{split}
&\frac12 \frac{d}{dt}\Big( \int_\Omega \rho_0|\partial_t \partial_h^\beta u|^2 - \int_\Omega g\rho_0'|\partial_h^\beta u_3|^2 \Big)  + \int_\Gamma(\partial_t\partial_h^\beta q\text{Id}-\mu\bS\partial_t\partial_h^\beta u)e_3 \cdot \partial_t\partial_h^\beta u \\
&= \int_\Omega \partial_t\partial_h^\beta q\text{div}\partial_t\partial_h^\beta u- \frac{\mu}2\int_\Omega |\bS\partial_t\partial_h^\beta u|^2 + \int_\Omega (-g \partial_h^\beta\cQ^1 \partial_t\partial_h^\beta u_3 +  \partial_t\partial_h^\beta \cQ^2 \cdot \partial_t \partial_h^\beta u)
\end{split}\]
By $\eqref{EqDtHorizontal}_{2,3}$, we observe
\begin{equation}\label{1stEstD_tU_H^2}
\begin{split}
&\frac12 \frac{d}{dt}\Big(\int_\Omega \rho_0| \partial_t\partial_h^\beta  u|^2 +\int_\Gamma g\rho_+ |\partial_h^\beta u_3|^2 - \int_\Omega g\rho_0'|\partial_h^\beta u_3|^2 \Big) + \frac12 \int_\Omega\mu|\bS\partial_t\partial_h^\beta u|^2 \\
&= \int_\Omega (-g \partial_h^\beta\cQ^1 \partial_t\partial_h^\beta u_3 +  \partial_t\partial_h^\beta \cQ^2 \cdot \partial_t \partial_h^\beta u) +\int_\Omega \partial_t\partial_h^\beta q \partial_t\partial_h^\beta \cQ^3 \\
&\qquad-\int_\Gamma (\partial_t\partial_h^\beta \cQ^5+g\rho_+\partial_h^\beta  \cQ^4e_3)\cdot \partial_t\partial_h^\beta u.
\end{split}
\end{equation}
We now estimate each integral in the r.h.s of \eqref{1stEstD_tU_H^2}. For the first and third integral, we use Cauchy-Schwarz's inequality and \eqref{1stIneLemHorizon}, \eqref{EstTildeQ1} to have
\begin{equation}\label{0_EstD_tU_H^2}
\begin{split}
 \int_\Omega  \partial_h^\beta\cQ^1 \partial_t\partial_h^\beta u_3 +\int_\Omega \partial_t\partial_h^\beta q \partial_t\partial_h^\beta \cQ^3 &\lesssim \|\partial_t u_3\|_{H^2(\Omega)} \|\cQ^1\|_{H^2(\Omega)} + \|\partial_tq\|_{H^2(\Omega)}\|\partial_t \cQ^3\|_{H^2(\Omega)} \\
&\lesssim \cE_f(\cE_f^2+\cD_f^2).
\end{split}
\end{equation}
With the same notations $\beta_\pm$ \eqref{BetaPM}, we use Cauchy-Schwarz's inequality again and \eqref{EstTildeQ1} to bound the second integral as
\begin{equation}\label{1_EstD_tU_H^2}
\begin{split}
\Big|\int_\Omega  \partial_t\partial_h^\beta u\cdot  \partial_t\partial_h^\beta  \cQ^2 \Big| =\Big|\int_\Omega  \partial_t\partial_h^{\beta_+} u\cdot \partial_t\partial_h^{\beta_-} \cQ^2  \Big| &\lesssim \|\partial_t u\|_{H^{|\beta_+|}(\Gamma)}\|\partial_t\cQ^2\|_{H^{|\beta_-|}(\Omega)} \\
&\lesssim \|\partial_t u\|_{H^3(\Omega)} \|\partial_t\cQ^2\|_{H^1(\Omega)}\\
&\lesssim \cE_f(\cE_f^2+\cD_f^2),
\end{split}
\end{equation}
and to bound the fourth integral as
\begin{equation}\label{3rdEstD_tU_H^2}
\begin{split}
\Big|\int_\Gamma  \partial_t\partial_h^\beta u\cdot  \partial_t\partial_h^\beta \cQ^5 \Big|=\Big|\int_\Gamma \partial_t\partial_h^{\beta_+} u\cdot \partial_t\partial_h^{\beta_-} \cQ^5\Big| &\lesssim \|\partial_t\partial_h^{\beta_+} u\|_{H^{-1/2}(\Gamma)} \| \partial_t\partial_h^{\beta_-} \cQ^5\|_{H^{1/2}(\Gamma)} \\
&\lesssim  \|\partial_t u\|_{H^{|\beta_+|-1/2}(\Gamma)}\|\partial_t\cQ^5\|_{H^{|\beta_-|+1/2}(\Gamma)}\\
&\lesssim \|\partial_t u\|_{H^3(\Omega)} \|\partial_t\cQ^5\|_{H^{3/2}(\Gamma)}\\
&\lesssim \cE_f(\cE_f^2+\cD_f^2).
\end{split}
\end{equation}
Thanks to the trace theorem and \eqref{1stIneLemHorizon}, we bound the fifth integral as
\begin{equation}\label{2ndEstD_tU_H^2}
\begin{split}
\int_\Gamma\partial_t\partial_h^\beta u_3 \partial_h^\beta \cQ^4 \lesssim \|\partial_t\partial^\beta u_3\|_{H^{-1/2}(\Gamma)} \|\partial^\beta \cQ^4\|_{H^{1/2}(\Gamma)} 
&\lesssim \|\partial_t u_3\|_{H^{|\beta|-1/2}(\Omega)} \| \cQ^4\|_{H^{|\beta|+1/2}(\Gamma)}\\
&\lesssim \|\partial_t u_3\|_{H^2(\Omega)} \|\cQ^4\|_{H^{5/2}(\Gamma)}\\
&\lesssim \cE_f^3.
\end{split}
\end{equation}
In view of \eqref{0_EstD_tU_H^2}, \eqref{1_EstD_tU_H^2}, \eqref{3rdEstD_tU_H^2} and \eqref{2ndEstD_tU_H^2}, we get
\[
\begin{split}
\frac{d}{dt}\Big(\int_\Omega \rho_0| \partial_t\partial_h^\beta  u|^2 +\int_\Gamma g\rho_+ |\partial_h^\beta u_3|^2 - \int_\Omega g\rho_0'|\partial_h^\beta u_3|^2 \Big) + \int_\Omega |\bS\partial_t\partial_h^\beta u|^2 \lesssim \cE_f(\cE_f^2+\cD_f^2).
\end{split}
\]
Integrating in time and using Korn's inequality \eqref{KornIne}, we obtain
\[\begin{split}
&\| \partial_t \partial_h^\beta u(t)\|_{L^2(\Omega)}^2 +\int_0^t \|\nabla \partial_t \partial_h^\beta  u(s)\|_{L^2(\Omega)}^2 ds  \lesssim \cE_f^2(0) +\|\partial_h^\beta u_3(t)\|_{L^2(\Omega)}^2 + \int_0^t \cE_f(\cE_f^2+\cD_f^2)(s)ds.
\end{split}\]
Combining the resulting inequality and \eqref{EstU_tH^4}, the inequality \eqref{EstPartial_tU_H^2} follows. Proof of Proposition \ref{PropEstHorizonU} is complete.
\end{proof}

\subsection{Estimates of the perturbation density} 

 We continue deriving the energy evolution of the  space-time derivatives of $\zeta$. We rewrite $\eqref{EqNS_Lagrangian}_1$ as
\begin{equation}\label{EqZetaTilde}
\partial_t\zeta = K\partial_t\theta\partial_3\zeta - u_j\cA_{jk}\partial_k\zeta -\rho_0' u_3 -\rho_0 \cQ^3 +  \tilde\cQ^1,
\end{equation}
where 
\begin{equation}\label{TildeQ1}
\tilde \cQ^1=\rho_0'' K\theta\partial_t\theta -q\cA_{lk}\partial_k u_l-\cA_{lk}\partial_k(\rho_0'\theta u_l)-(\cA_{lk}-\delta_{lk})\partial_k(\rho_0 u_l).
\end{equation}
We first present the estimate of  $\tilde \cQ^1$.
\begin{lemma}\label{LemH^4TildeQ1}
There holds 
\begin{equation}\label{H^4TildeQ1}
\|\tilde \cQ^1\|_{H^4(\Omega)} \lesssim \cE_f(\cE_f+\cD_f).
\end{equation}
\end{lemma}
\begin{proof}
We use  \eqref{ProductEst}, \eqref{CoefEst_22} and Lemma \ref{LemEstNablaQ_Pf} to have that 
\[
\begin{split}
\|\rho_0''K\theta\partial_t \theta\|_{H^4(\Omega)} &\lesssim (1+\|K-1\|_{H^4(\Omega)}) \|\theta\|_{H^4(\Omega)}\|\partial_t\theta\|_{H^4(\Omega)} \\
&\lesssim (1+\|\eta\|_{H^{9/2}(\Gamma)})\|\eta\|_{H^{7/2}(\Gamma)} \|\partial_t\eta\|_{H^{7/2}(\Gamma)}.
\end{split}
\]
Combining  \eqref{EstEta_H^1/2} and the resulting inequality, we have
\begin{equation}\label{1_BoundQ^1}
\|\rho_0''K\theta\partial_t \theta\|_{H^4(\Omega)}  \lesssim \cE_f^2. 
\end{equation}
Using  Lemma \ref{LemEstNablaQ_Pf} again and \eqref{ProductEst}, \eqref{CoefEst_24}, one has
\begin{equation}\label{3_BoundQ^1}
\begin{split}
\|\cA_{lk}\partial_k(\rho_0'\theta u_l)\|_{H^4(\Omega)} &\lesssim (1+\|\cA-\text{Id}\|_{H^4(\Omega)} )\|\theta\|_{H^5(\Omega)} \|u\|_{H^5(\Omega)}\\ &\lesssim (1+\|\eta\|_{H^{9/2}(\Gamma)})\|\eta\|_{H^{9/2}(\Gamma)}\|u\|_{H^5(\Omega)}\\
&\lesssim \cE_f(\cE_f+\|\nabla u\|_{H^4(\Omega)})
\end{split}
\end{equation}
and 
\begin{equation}\label{4_BoundQ^1}
\|(\cA_{lk}-\delta_{lk})\partial_k(\rho_0 u_l)\|_{H^4(\Omega)} \lesssim \|\cA-\text{Id}\|_{H^4(\Omega)} \|u\|_{H^5(\Omega)} \lesssim  \cE_f(\cE_f+\|\nabla u\|_{H^4(\Omega)}).
\end{equation}
Thanks to  Gagliardo-Nireberg's inequality also and  \eqref{CoefEst_24}, we obtain
\begin{equation}\label{2_BoundQ^1}
\begin{split}
\|q\cA_{lk}\partial_k u_l\|_{H^4(\Omega)}&\lesssim (1+\|\cA-\text{Id}\|_{H^4(\Omega)} )(\|q\|_{H^2(\Omega)} \|\nabla u\|_{H^4(\Omega)}+\|q\|_{H^4(\Omega)}\|\nabla u\|_{H^2(\Omega)})\\
&\lesssim \cE_f(\cE_f+\|q\|_{H^4(\Omega)}+\|\nabla u\|_{H^4(\Omega)}),
\end{split}
\end{equation}
Those above estimates, \eqref{1_BoundQ^1}, \eqref{3_BoundQ^1},  \eqref{4_BoundQ^1} and \eqref{2_BoundQ^1} imply 
\[
\|\tilde\cQ^1\|_{H^4(\Omega)}\lesssim \cE_f(\cE_f+\|q\|_{H^4(\Omega)}+\|\nabla u\|_{H^4(\Omega)})\lesssim \cE_f(\cE_f+\cD_f).
\]
Lemma \ref{LemH^4TildeQ1} is proven.
\end{proof}
We derive the following proposition. 
\begin{proposition}\label{PropEstZeta}
The following inequality holds
\begin{equation}\label{EstZeta_H^4}
\begin{split}
\| \zeta(t)\|_{H^4(\Omega)}^2 &\leq C_9\Big( \cE_f^2(0) +  \varepsilon^3 \int_0^t (\|\zeta(s)\|_{H^4(\Omega)}^2 +  \|u_3(s)\|_{H^5(\Omega)}^2) ds \Big)\\
&\qquad + C_9\Big( \varepsilon^{-27}\int_0^t \|(u_3,\zeta)(s)\|_{L^2(\Omega)}^2ds+ \int_0^t \cE_f(\cE_f^2+\cD_f^2)(s) ds\Big).
\end{split}
\end{equation}
\end{proposition}
\begin{proof}
 It can be seen from $\eqref{EqPertur}_1$   that 
\[
\frac12\frac{d}{dt}\|\zeta\|_{L^2(\Omega)}^2 = -\int_\Omega \rho_0' u_3\zeta +\int_\Omega \cQ^1\zeta \lesssim (\|u_3\|_{L^2(\Omega)}+\|\cQ^1\|_{L^2(\Omega)})\|\zeta\|_{L^2(\Omega)}.
\]
Due to \eqref{1stIneLemHorizon}, we thus have
\[
\frac{d}{dt}\|\zeta\|_{L^2(\Omega)}^2 \lesssim \|u_3\|_{L^2(\Omega)}\|\zeta\|_{L^2(\Omega)}+ \cE_f^3.
\]
This yields
\begin{equation}\label{-1_InePropEstZeta}
\begin{split}
\|\zeta(t)\|_{L^2(\Omega)}^2 &\lesssim \cE_f^2(0)+ \int_0^t \|(u_3,\zeta)(s)\|_{L^2(\Omega)}^2 ds + \int_0^t\cE_f^3(s)ds \\
&\lesssim  \cE_f^2(0)+ \int_0^t (\varepsilon^3 \|u_3(s)\|_{L^2(\Omega)}^2 +\varepsilon^{-3}\|\zeta(s)\|_{L^2(\Omega)}^2) ds + \int_0^t\cE_f^3(s)ds
\end{split}
\end{equation}

For $\alpha\in \N^3, 1\leq |\alpha|\leq 4$,  we have from \eqref{EqZetaTilde} that
\begin{equation}\label{EqPerturZeta}
\begin{split}
\partial_t\partial^\alpha\zeta &= K\partial^\alpha(\partial_t\theta\partial_3\zeta) + \sum_{0\neq \beta \leq \alpha}\partial^\beta K \partial^{\alpha-\beta}(\partial_t\theta\partial_3\zeta) - \partial^\alpha(u_j\cA_{jk}\partial_k\zeta) +\partial^\alpha(-\rho_0' u_3-\rho_0\cQ^3 + \tilde \cQ^1)\\
&= (K\partial_t\theta\partial_3\partial^\alpha\zeta - u_j\cA_{jk}\partial_k\partial^\alpha \zeta ) +\sum_{0\neq \beta\leq \alpha} K \partial^\beta \partial_t\theta \partial^{\alpha-\beta}\partial_3\zeta +  \sum_{0\neq \beta\leq \alpha}\partial^\beta K\partial^{\alpha-\beta}(\partial_t\theta\partial_3\zeta)   \\
&\qquad -\sum_{0\neq \beta\leq \alpha}\partial^\beta ( u_j\cA_{jk})\partial^{\alpha-\beta}\partial_k\zeta+\partial^\alpha (-\rho_0' u_3-\rho_0\cQ^3 + \tilde \cQ^1).
\end{split}
\end{equation}
We deduce from \eqref{EqPerturZeta} that
\begin{equation}\label{0thInePropEstZeta}
\begin{split}
\frac12 \frac{d}{dt}\|\partial^{\alpha}\zeta\|_{L^2(\Omega)}^2 &= \int_{\Omega}(K\partial_t\theta\partial_3\partial^\alpha\zeta - u_j\cA_{jk}\partial_k\partial^\alpha \zeta) \partial^\alpha \zeta + \sum_{0\neq \beta\leq \alpha} \int_\Omega K \partial^\beta \partial_t\theta \partial^{\alpha-\beta}\partial_3\zeta \partial^\alpha\zeta \\
&\quad +  \sum_{0\neq \beta\leq \alpha}\int_\Omega \partial^\beta K\partial^{\alpha-\beta}(\partial_t\theta\partial_3\zeta) \partial^\alpha\zeta  - \sum_{0\neq \beta\leq \alpha}\int_\Omega\partial^\beta ( u_j\cA_{jk})\partial^{\alpha-\beta}\partial_k\zeta\partial^\alpha\zeta \\ 
&\quad+ \int_\Omega \partial^\alpha (-\rho_0' u_3-\rho_0\cQ^3 + \tilde \cQ^1)\partial^\alpha\zeta.
\end{split}
\end{equation}
We bound each integral in the r.h.s of \eqref{0thInePropEstZeta}. For the first integral, using the integration by parts, one has 
\[
\begin{split}
&2\int_{\Omega}(K\partial_t\theta\partial_3\partial^\alpha\zeta - u_j\cA_{jk}\partial_k\partial^\alpha \zeta) \partial^\alpha \zeta \\
& = \int_\Omega ( K\partial_t\theta \partial_3|\partial^\alpha \zeta|^2 - u_j\cA_{jk}\partial_k|\partial^\alpha\zeta|^2) \\
&= \int_\Gamma (K\partial_t\theta -u_j\cA_{j3})|\partial^\alpha\zeta|^2 - \int_\Omega (\partial_3(K\theta) -\partial_k(u_j\cA_{jk}))|\partial^\alpha\zeta|^2.
\end{split}
\]
On $\Gamma$, we have  $K\partial_t\theta -u_j\cA_{j3}  = 0$ by the definition of $\cA$ \eqref{MatrixA} and by $\eqref{EqNS_Lagrangian}_4$. This yields
\[
2\int_{\Omega}(K\partial_t\theta\partial_3\partial^\alpha\zeta - u_j\cA_{jk}\partial_k\partial^\alpha \zeta) \partial^\alpha \zeta  = - \int_\Omega (\partial_3(K\theta) -\partial_k(u_j\cA_{jk}))|\partial^\alpha\zeta|^2.
\]
Due to Sobolev embedding and the product estimate \eqref{ProductEst}, it can be seen that
\[\begin{split}
\int_\Omega &(\partial_3(K\theta) -\partial_k(u_j\cA_{jk}))|\partial^\alpha\zeta|^2\\
 &\lesssim \|\partial_3(K\theta) -\partial_k(u_j\cA_{jk})\|_{H^2(\Omega)} \|\zeta\|_{H^4(\Omega)}^2 \\
&\lesssim ((\|K-1\|_{H^3(\Omega)}+1)\|\theta\|_{H^3(\Omega)}+ \|u\|_{H^3(\Omega)}(\|\cA-\text{Id}\|_{H^3(\Omega)}+1))\|\zeta\|_{H^4(\Omega)}^2.
\end{split}\]
Owing to Lemma \ref{LemEstNablaQ_Pf} and \eqref{CoefEst_22}, \eqref{CoefEst_24}, we have
\begin{equation}\label{1stInePropEstZeta}
\begin{split}
\int_\Omega(\partial_3(K\theta) -\partial_k(u_j\cA_{jk}))|\partial^\alpha\zeta|^2 &\lesssim (1+ \|\eta\|_{H^{7/2}(\Gamma)})(\|\eta\|_{H^{5/2}(\Gamma)} + \|u\|_{H^3(\Omega)}) \|\zeta\|_{H^4(\Omega)}^2 \lesssim \cE_f^3.
\end{split}
\end{equation}
For the second integral in the r.h.s of \eqref{0thInePropEstZeta}, we use Cauchy-Schwarz's inequality, \eqref{CoefEstimates} and Lemma \ref{LemEstNablaQ_Pf} to obtain
\[
\begin{split}
 \sum_{0\neq \beta\leq \alpha} \int_\Omega K \partial^\beta \partial_t\theta \partial^{\alpha-\beta}\partial_3\zeta \partial^\alpha\zeta &\lesssim  (1+\|K-1\|_{L^\infty(\Omega)}) \|\partial_t\theta\|_{H^3(\Omega)}  \|\partial_3\zeta\|_{H^3(\Omega)} \|\zeta\|_{H^4(\Omega)}\\
 &\lesssim \|\partial_t\eta\|_{H^{5/2}(\Gamma)}\|\zeta\|_{H^4(\Omega)}^2.
\end{split}
\]
In view of \eqref{EstEta_H^1/2}, one has 
\begin{equation}\label{2ndInePropEstZeta}
 \sum_{0\neq \beta\leq \alpha} \int_\Omega K \partial^\beta \partial_t\theta \partial^{\alpha-\beta}\partial_3\zeta \partial^\alpha\zeta  \lesssim \cE_f^3.
\end{equation}
Next, for the third integral, we use Cauchy-Schwarz's inequality and the product estimate \eqref{ProductEst} to have
\[
\begin{split}
 \sum_{0\neq \beta\leq \alpha}\int_\Omega \partial^\beta K\partial^{\alpha-\beta}(\partial_t\theta\partial_3\zeta) \partial^\alpha\zeta  &\lesssim \|\nabla K\|_{H^2(\Omega)} \|\partial_t\theta\partial_3\zeta\|_{H^3(\Omega)} \|\zeta\|_{H^4(\Omega)}\\
 &\lesssim  \|\nabla K\|_{H^2(\Omega)} \|\partial_t\theta\|_{H^3(\Omega)}\|\zeta\|_{H^4(\Omega)}^2.
\end{split}
\]
Owing to \eqref{CoefEst_22}, Lemma \ref{LemEstNablaQ_Pf}, and \eqref{EstEta_H^1/2} again, we deduce
\begin{equation}\label{3rdInePropEstZeta}
\begin{split}
 \sum_{0\neq \beta\leq \alpha}\int_\Omega \partial^\beta K\partial^{\alpha-\beta}(\partial_t\theta\partial_3\zeta) \partial^\alpha\zeta &\lesssim \|\eta\|_{H^{7/2}(\Gamma)}\|\partial_t\eta\|_{H^{5/2}(\Gamma)} \|\zeta\|_{H^4(\Omega)}^2\lesssim \cE_f^4.
\end{split}
\end{equation}
For the fourth integral, we continue using Cauchy-Schwarz's inequality and \eqref{CoefEst_24} also to observe
\begin{equation}\label{4thInePropEstZeta}
\begin{split}
 \sum_{0\neq \beta\leq \alpha}\int_\Omega\partial^\beta ( u_j\cA_{jk})\partial^{\alpha-\beta}\partial_k\zeta \partial^\alpha\zeta &\lesssim (\|\cA-\text{Id}\|_{H^4(\Omega)}+1) \|u\|_{H^4(\Omega)}\|\zeta\|_{H^4(\Omega)}^2\\
&\lesssim (1+\|\eta\|_{H^{9/2}(\Gamma)})\|u\|_{H^4(\Omega)}\|\zeta\|_{H^4(\Omega)}^2\lesssim \cE_f^3.
\end{split}
\end{equation}
Let us bound the fifth integral. Thanks  to Young's inequality, we have
\[\begin{split}
\int_\Omega \partial^\alpha(\rho_0' u_3)\partial^\alpha\zeta  &\lesssim \|\partial^\alpha\zeta\|_{L^2(\Omega)}\|u_3\|_{H^{|\alpha|}(\Omega)} \lesssim \varepsilon^3  \|\partial^\alpha\zeta\|_{L^2(\Omega)}^2+ \varepsilon^{-3} \|u_3\|_{H^4(\Omega)}^2\\
\end{split}\]
By Young's inequality again and \eqref{EstJensen}, this yields
\begin{equation}\label{1_5thInePropEstZeta}
 \begin{split}
 \int_\Omega \partial^\alpha(\rho_0' u_3)\partial^\alpha\zeta   & \lesssim \varepsilon^3 \|\partial^\alpha\zeta\|_{L^2(\Omega)}^2+ \varepsilon^{-3} (\varepsilon^6 \|u_3\|_{H^5(\Omega)}^2 + \varepsilon^{-24} \|u_3\|_{L^2(\Omega)}^2) \\
&\lesssim \varepsilon^3 ( \|\partial^\alpha\zeta\|_{L^2(\Omega)}^2+\| u_3\|_{H^5(\Omega)}^2 )+ \varepsilon^{-27} \|u_3\|_{L^2(\Omega)}^2.
\end{split}
\end{equation}
Thanks to \eqref{EstTildeQ1} and \eqref{H^4TildeQ1}, we have
\begin{equation}\label{2_5thInePropEstZeta}
\begin{split}
\int_\Omega \partial^\alpha (-\rho_0\cQ^3 + \tilde \cQ^1)\partial^\alpha\zeta &\lesssim( \|\cQ^3\|_{H^4(\Omega)}+\|\tilde \cQ^1\|_{H^4(\Omega)})\|\zeta\|_{H^4(\Omega)} 
\lesssim \cE_f^2(\cE_f+\cD_f).
\end{split}
\end{equation}

We substitute \eqref{1stInePropEstZeta}, \eqref{2ndInePropEstZeta} \eqref{3rdInePropEstZeta}, \eqref{4thInePropEstZeta}, \eqref{1_5thInePropEstZeta} and \eqref{2_5thInePropEstZeta} into \eqref{0thInePropEstZeta} to have
\[
\frac{d}{dt}\|\partial^{\alpha}\zeta\|_{L^2(\Omega)}^2 \lesssim  \varepsilon^3 (\|\zeta\|_{H^4(\Omega)}^2 +  \|u_3\|_{H^5(\Omega)}^2) +\varepsilon^{-27} \|u_3\|_{L^2(\Omega)}^2+ \cE_f(\cE_f^2+\cD_f^2).
\]
Integrating the resulting inequality from 0 to $t$, together with \eqref{-1_InePropEstZeta}, one has \eqref{EstZeta_H^4}.  Proof of Proposition \ref{PropEstZeta} is complete.
\end{proof}

In addition, we have the following estimate.
\begin{proposition}\label{PropEstD_tZetaH^2}
There holds 
\begin{equation}\label{EstD_tZetaH^2}
\|\partial_t\zeta \|_{H^2(\Omega)}^2 +\|\partial_t^2\zeta\|_{L^2(\Omega)}^2\leq C_{10}(\|u_3\|_{H^2(\Omega)}^2 + \|\partial_t u_3\|_{L^2(\Omega)}^2 +\cE_f^4).
\end{equation}
\end{proposition} 
\begin{proof}
It follows directly from $\eqref{EqPertur}_1$ and \eqref{1stIneLemHorizon} that
\begin{equation}\label{EstPartial_tZeta_H^2}
\|\partial_t\zeta\|_{H^2(\Omega)}^2 \lesssim \| u_3\|_{H^2(\Omega)}^2+\|\cQ^1\|_{H^2(\Omega)}^2 \lesssim \| u_3\|_{H^2(\Omega)}^2+ \cE_f^4
\end{equation}
and
\begin{equation}\label{EstPartial_t^2Zeta_L^2}
\|\partial_t^2 \zeta\|_{L^2(\Omega)}^2 \lesssim \|\partial_t u_3\|_{L^2(\Omega)}^2 + \|\partial_t\cQ^1\|_{L^2(\Omega)}^2 \lesssim \|\partial_t u_3\|_{L^2(\Omega)}^2 +\cE_f^4.
\end{equation}
Hence, we obtain  \eqref{EstD_tZetaH^2} by combining \eqref{EstPartial_tZeta_H^2}  and \eqref{EstPartial_t^2Zeta_L^2}. 
\end{proof}

\subsection{Elliptic  estimates}
We use the elliptic  estimate \eqref{EllipticEst} to derive some inequalities.
\begin{proposition}\label{PropCompareE}
 There holds
\begin{equation}\label{EstCompareE_fh}
\begin{split}
&\|u\|_{H^4(\Omega)}^2 +\| q\|_{H^3(\Omega)}^2 + \|\partial_t u\|_{H^2(\Omega)}^2 +\| \partial_t q\|_{H^1(\Omega)}^2  \\
&\quad \leq C_{11} \Big(\|\partial_t^2 u\|_{L^2(\Omega)}^2+\|u_3\|_{L^2(\Omega)}^2 +\|\zeta\|_{H^2(\Omega)}^2 + \|\eta\|_{H^{5/2}(\Gamma)}^2+ \|\partial_t\eta\|_{H^{1/2}(\Gamma)}^2 +\cE_f^4\Big).
\end{split}
\end{equation}
\end{proposition}
\begin{proof}
We derive from \eqref{EqPertur} that
\begin{equation}\label{EllipticEqD_tU}
\begin{cases}
-\mu\Delta \partial_t u +\nabla \partial_t q = -\rho_0\partial_t^2 u-g\partial_t\zeta  e_3+ \partial_t\cQ^2 &\quad\text{in }\Omega,\\
\text{div} \partial_t u = \partial_t \cQ^3  &\quad\text{in }\Omega,\\
(\partial_t q \text{Id}-\mu\bS\partial_t u) e_3=g\rho_+\partial_t\eta e_3+ \partial_t \cQ^5  &\quad\text{on }\Gamma.
\end{cases}
\end{equation}
Applying the elliptic estimate \eqref{EllipticEst} to \eqref{EllipticEqD_tU}, it tells us that 
\[
\begin{split}
\|\partial_t u\|_{H^2(\Omega)}^2 +\| \partial_t q\|_{H^1(\Omega)}^2 &\lesssim \|\partial_t^2 u\|_{L^2(\Omega)}^2 +\|\partial_t\zeta\|_{L^2(\Omega)}^2+\|\partial_t\eta\|_{H^{1/2}(\Gamma)}^2 \\
&\qquad+\|\partial_t\cQ^2\|_{L^2(\Omega)}^2 + \|\partial_t\cQ^3\|_{H^1(\Omega)}^2+\|\partial_t\cQ^5\|_{H^{1/2}(\Gamma)}^2.
\end{split}\]
Note that from $\eqref{EqPertur}_1$  that , 
\[
\|\partial_t\zeta\|_{L^2(\Omega)}^2 \lesssim \|u_3\|_{L^2(\Omega)}^2+\|\cQ^1\|_{L^2(\Omega)}^2.
\]
Hence, we have
\[
\begin{split}
\|\partial_t u\|_{H^2(\Omega)}^2 +\| \partial_t q\|_{H^1(\Omega)}^2 &\lesssim \|\partial_t^2 u\|_{L^2(\Omega)}^2 + \|\partial_t\eta\|_{H^{1/2}(\Gamma)}^2+\|u_3\|_{L^2(\Omega)}^2+\|\cQ^1\|_{L^2(\Omega)}^2\\
&\qquad+  \|\partial_t\cQ^2\|_{L^2(\Omega)}^2+ \|\partial_t\cQ^3\|_{H^1(\Omega)}^2+\|\partial_t\cQ^5\|_{H^{1/2}(\Gamma)}^2.
\end{split}
\]
Due to \eqref{1stIneLemHorizon}, this yields
\begin{equation}\label{EstEllipticU_tH^2}
\begin{split}
\|\partial_t u\|_{H^2(\Omega)}^2 +\| \partial_t q\|_{H^1(\Omega)}^2 &\lesssim \|\partial_t^2 u\|_{L^2(\Omega)}^2+ \|u_3\|_{L^2(\Omega)}^2 +\|\partial_t\eta\|_{H^{1/2}(\Gamma)}^2 +\cE_f^4.
\end{split}
\end{equation}

Meanwhile,  we obtain from  \eqref{EqPertur} that
\begin{equation}\label{EllipticEqU}
\begin{cases}
-\Delta u+\nabla q=-\rho_0\partial_t u -g\zeta  e_3+ \cQ^2 &\quad\text{in }\Omega,\\
\text{div} u=\cQ^3 &\quad\text{in }\Omega,\\
(q \text{Id}-\mu \bS u) e_3=g\rho_+\eta e_3+ \cQ^5 &\quad\text{on }\Gamma.
\end{cases}
\end{equation}
Owing to \eqref{1stIneLemHorizon}  and by applying the elliptic estimate \eqref{EllipticEst} again to \eqref{EllipticEqU}, we observe that
\begin{equation}\label{EstEllipticU_H^4}
\begin{split}
\|u\|_{H^4(\Omega)}^2 +\| q\|_{H^3(\Omega)}^2 &\lesssim \|\partial_t u\|_{H^2(\Omega)}^2 +\|\zeta\|_{H^2(\Omega)}^2+\|\cQ^2\|_{H^2(\Omega)}^2+ \|\cQ^3\|_{H^3(\Omega)}^2\\
&\qquad\quad+ \|\eta\|_{H^{5/2}(\Gamma)}^2+\|\cQ^5\|_{H^{5/2}(\Gamma)}^2\\
&\lesssim \|\partial_t u\|_{H^2(\Omega)}^2+\|\zeta\|_{H^2(\Omega)}^2 +\|\eta\|_{H^{5/2}(\Gamma)}^2+\cE_f^4.
\end{split}
\end{equation}
Combining \eqref{EstEllipticU_tH^2} and  \eqref{EstEllipticU_H^4}, one has \eqref{EstCompareE_fh}. Proof of Proposition \ref{PropCompareE} is complete.
\end{proof}

Let us define the "horizontal" dissipation $\cD_h>0$ as follows,
\begin{equation}\label{DissipationHor}
\cD_h^2 :=  \sum_{\beta\in \N^2, |\beta|\leq   4} \|\nabla \partial_h^\beta u\|_{L^2(\Omega)}^2+ \sum_{\beta\in \N^2, |\beta|\leq 2} \| \nabla \partial_h^\beta \partial_t u\|_{L^2(\Omega)}^2 +\|\nabla \partial_t^2 u\|_{L^2(\Omega)}^2.
\end{equation}
The next proposition is to compare $\cD_f$ \eqref{DissipationFull} and $\cD_h$ \eqref{DissipationHor}. 
\begin{proposition}\label{PropCompareD}
Assuming $\delta_0$ sufficiently small, there holds
\begin{equation}\label{EstCompareD_fh}
\begin{split}
\cD_f^2 \leq C_{12}\Big(\cD_h^2  + \varepsilon^3 \cE_f^2 + \varepsilon^{-9}(\|(\zeta,u)\|_{L^2(\Omega)}^2 + \|\eta\|_{L^2(\Gamma)}^2) +\cE_f^3\Big).
\end{split}
\end{equation}
\end{proposition}
To prove Proposition \ref{PropCompareD}, we use the two lemmas below. 
\begin{lemma}
For any $s\geq 0$, there holds
\begin{equation}\label{FourierIne}
\|f\|_{H^{s+1/2}(\Gamma)}^2 \lesssim \|f\|_{H^{1/2}(\Gamma)}^2 +\sum_{\beta \in \N^2, |\beta|=s} \|\partial_h^\beta f\|_{H^{1/2}(\Gamma)}^2. 
\end{equation}
\end{lemma}
\begin{proof}
Since $\Gamma= \fT^2 \times\{0\}$, we exploit the definition of the Sobolev norm on $\fT^2$ to have that 
\[
\|f\|_{H^{s+1/2}(\Gamma)}^2 \approx \sum_{n\in (L^{-1}\bZ)^2}(1+|n|^2)^{s+1/2}|\hat f(n)|^2,
\]
where $\hat f$ is the Fourier series of $f$. By Cauchy-Schwarz's inequality, one has
\[
\|f\|_{H^{s+1/2}(\Gamma)}^2 \lesssim \sum_{n\in (L^{-1}\bZ)^2}(1+|n|^2)^{1/2}|\hat f(n)|^2 +\sum_{\beta\in \N^2, |\beta|=s} \sum_{n\in (L^{-1}\bZ)^2}(1+|n|^2)^{1/2}|n^\beta\hat f(n)|^2,
\]
which immediately yields \eqref{FourierIne}.
\end{proof}
\begin{lemma}\label{LemBoundDtU_L2}
Let us write 
\[
\sW := \sqrt{ \|\zeta\|_{L^2(\Omega)}^2 + \|u\|_{H^1(\Omega)}^2+\|\eta\|_{L^2(\Gamma)}^2+\cE_f^3}.
\]
The following inequalities holds
\begin{equation}\label{BoundDtU_L2}
\|\partial_t u\|_{L^2(\Omega)} \lesssim \|\nabla \partial_t u\|_{L^2(\Omega)}+\sW,
\end{equation}
and
\begin{equation}\label{BoundDt^2U_L2}
\|\partial_t^2 u\|_{L^2(\Omega)} \lesssim  \|\nabla \partial_t u\|_{L^2(\Omega)}+\|\nabla\partial_t^2 u\|_{L^2(\Omega)}+\sW.
\end{equation}
\end{lemma}
\begin{proof}
Let us show \eqref{BoundDtU_L2} first. Multiplying by $\partial_t u$ on both sides of $\eqref{EqPertur}_2$, we obtain  
\[
\begin{split}
\int_\Omega\rho_0|\partial_t u|^2 &= - \int_\Omega \nabla q \cdot \partial_t u + \mu \int_\Omega \Delta u \cdot \partial_t u -\int_\Omega g\zeta \partial_t u_3  +  \int_\Omega \cQ^2 \cdot \partial_t u\\
&=-\int_\Gamma (q\text{Id}-\mu\bS u)e_3 \cdot \partial_t u + \int_\Omega q \text{div}\partial_t u  - \frac{\mu}2 \int_\Omega \bS u:\bS \partial_t u \\
&\qquad \quad-\int_\Omega g\zeta \partial_t u_3 +  \int_\Omega \cQ^2 \cdot \partial_t u,
\end{split}
\]
after using the integration by parts. Using $\eqref{EqPertur}_{3,5}$, this yields
\begin{equation}\label{1_BoundDtU_L2}
\begin{split}
\int_\Omega\rho_0|\partial_t u|^2&=- \int_\Gamma g\rho_+ \eta \partial_t u_3 - \int_\Gamma \cQ^5\cdot\partial_t u  + \int_\Omega q \partial_t \cQ^3 + \mu \int_\Omega\bS u:\bS \partial_t u \\
&\qquad\quad-\int_\Omega g\zeta \partial_t u_3 + \int_\Omega \cQ^2\cdot \partial_t u.
\end{split}
\end{equation} 
By Cauchy-Schwarz's inequality, we have 
\begin{equation}\label{2_BoundDtU_L2}
\begin{split}
\frac{\mu}2  \int_\Omega \bS u:\bS \partial_t u-\int_\Omega g\zeta \partial_t u_3 &\lesssim \|u\|_{H^1(\Omega)}\|\partial_t u\|_{H^1(\Omega)}+ \|\zeta\|_{L^2(\Omega)}\|\partial_t u_3\|_{L^2(\Omega)}. 
 \end{split}
\end{equation}
Using also the trace theorem, we have
\begin{equation}\label{3_BoundDtU_L2}
\begin{split}
\int_\Gamma g\rho_+ \eta \partial_t u_3  \lesssim \|\eta\|_{L^2(\Gamma)} \|\partial_t u\|_{H^1(\Omega)}.
\end{split}
\end{equation}
Because of  \eqref{1stIneLemHorizon} and the trace theorem again, one has 
\begin{equation}\label{4_BoundDtU_L2}
\begin{split}
\int_\Gamma& \cQ^5 \cdot \partial_t u +  \int_\Omega q \partial_t \cQ^3+ \int_\Omega \cQ^2\cdot \partial_t u \\
&\lesssim (\|\cQ^5\|_{L^2(\Gamma)}+ \|\cQ^2\|_{L^2(\Omega)}) \|\partial_t u\|_{H^1(\Omega)} + \|\partial_t\cQ^3\|_{L^2(\Omega)} \|q\|_{L^2(\Omega)}\lesssim \cE_f^3.
 \end{split}
\end{equation}
Combining \eqref{2_BoundDtU_L2}, \eqref{3_BoundDtU_L2} and \eqref{4_BoundDtU_L2}, we obtain  from \eqref{1_BoundDtU_L2} that
\[
\begin{split}
\|\partial_t u\|_{L^2(\Omega)}^2 &\lesssim  \|\partial_t u\|_{L^2(\Omega)} (\|\eta\|_{L^2(\Gamma)}+ \|u\|_{H^1(\Omega)}+\|\zeta\|_{L^2(\Omega)})\\
&\qquad+\|\nabla \partial_t u\|_{L^2(\Omega)} (\|\eta\|_{L^2(\Gamma)}+\|u\|_{H^1(\Omega)})  +\cE_f^3.
\end{split}
\]
Using Young's inequality, we get further that for any $\nu>0$,
\begin{equation}\label{5_BoundDtU_L2}
\begin{split}
\|\partial_t u\|_{L^2(\Omega)}^2 &\lesssim  \nu \|\partial_t u\|_{L^2(\Omega)}^2+\|\nabla \partial_t u\|_{L^2(\Omega)}^2 +(1+\nu^{-1})\sW^2.
\end{split}
\end{equation}
Let $\nu>0$ be sufficiently small, the inequality \eqref{BoundDtU_L2} follows from \eqref{5_BoundDtU_L2}.

To prove \eqref{BoundDt^2U_L2}, we differentiate   $\eqref{EqPertur}_{2,5}$ with respect to $t$ and then eliminate the  terms $\partial_t\zeta, \partial_t \eta$ by using $\eqref{EqPertur}_{1,4}$ to  deduce  that 
\begin{equation}\label{EqDt^2U}
\begin{cases}
\rho_0 \partial_t^2 u+\nabla \partial_tq -\mu\Delta\partial_t u-g\rho_0'u_3 e_3 = \partial_t\cQ^2 -g\cQ^1 e_3 &\quad\text{in }\Omega,\\
\text{div}\partial_t u= \partial_t\cQ^3 &\quad\text{in }\Omega,\\
(\partial_t q \text{Id}-\mu \bS \partial_t u)e_3 =g\rho_+ u e_3  + g\rho_+\cQ^4e_3+ \partial_t\cQ^5&\quad\text{on }\Gamma.
\end{cases}
\end{equation}
Multiplying both sides of $\eqref{EqDt^2U}_1$ by $\partial_t^2 u$,  we obtain that 
\begin{equation}\label{1_EqDt^2U}
\begin{split}
& \int_\Omega\rho_0|\partial_t^2 u|^2 + \int_\Omega (\nabla\partial_t q-\mu \Delta \partial_t u)\cdot\partial_t^2 u -\int_\Omega g\rho_0' u_3\partial_t^2 u_3 =  \int_\Omega (\partial_t\cQ^2 -g\cQ^1e_3) \cdot\partial_t^2 u.
\end{split}
\end{equation}
Using the integration by parts, we have that
\[\begin{split}
 \int_{\Omega}\rho_0|\partial_t^2 u|^2 &=- \int_\Gamma (\partial_t q \text{Id}-\mu \bS \partial_t u)e_3 \cdot \partial_t u + \int_\Omega \partial_t q \text{div}\partial_t u  -\frac{\mu}2 \int_\Omega \bS\partial_t u : \bS\partial_t^2 u \\
&\qquad+\int_\Omega g\rho_0' u_3\partial_t^2 u_3+   \int_\Omega (\partial_t\cQ^2 -g\cQ^1 e_3) \cdot\partial_t^2 u.
\end{split}\]
Substituting $\eqref{EqDt^2U}_{2,3}$ into the resulting equality yields
\begin{equation}\label{2_EqDt^2U}
\begin{split}
 \int_{\Omega}\rho_0|\partial_t^2 u|^2 &= - \int_\Gamma g\rho_+ u_3 \partial_t^2 u_3 - \int_\Gamma( g\rho_+\cQ^4e_3+\partial_t\cQ^5) \cdot\partial_t^2 u + \int_\Omega\partial_tq \partial_t\cQ^3\\
&\qquad  -\frac{\mu}2 \int_\Omega \bS\partial_t u : \bS\partial_t^2 u  +\int_\Omega g\rho_0' u_3\partial_t^2 u_3+  \int_\Omega (\partial_t\cQ^2 -g\cQ^1 e_3) \cdot\partial_t^2 u.
\end{split}
\end{equation}
We estimate each integral in the r.h.s of \eqref{2_EqDt^2U}. For the first integral, we use Cauchy-Schwarz's inequality and the trace theorem to have
\begin{equation}\label{3_EqDt^2U}
- \int_\Gamma g\rho_+ u_3 \partial_t^2 u_3 \lesssim \|u_3\|_{L^2(\Gamma)}\|\partial_t^2u_3\|_{L^2(\Gamma)}\lesssim  \|u_3\|_{H^1(\Omega)}\|\partial_t^2 u_3\|_{H^1(\Omega)}.
\end{equation}
For the fourth and fifth integral, we bound as
\begin{equation}\label{4_EqDt^2U}
\begin{split}
-\frac{\mu}2 \int_\Omega \bS\partial_t u : \bS\partial_t^2 u  +\int_\Omega g\rho_0' u_3\partial_t^2 u_3 &\lesssim \|\partial_t u\|_{H^1(\Omega)} \|\partial_t^2 u\|_{H^1(\Omega)} + \|u_3\|_{L^2(\Omega)}\|\partial_t^2 u_3\|_{L^2(\Omega)}.
\end{split}
\end{equation}
For the other integrals, we use \eqref{1stIneLemHorizon} to obtain
\begin{equation}\label{5_EqDt^2U}
\begin{split}
\int_\Omega & (\partial_t\cQ^2 -g\cQ^1e_3) \cdot\partial_t^2 u +\int_\Omega \partial_t q\partial_t \cQ^3 \\
&\lesssim (\|\partial_t\cQ^2\|_{L^2(\Omega)}+\|\cQ^1\|_{L^2(\Omega)}) \|\partial_t^2 u\|_{L^2(\Omega)} + \|\partial_t q\|_{L^2(\Omega)}\|\partial_t\cQ^3\|_{L^2(\Omega)} \lesssim \cE_f^3,
\end{split}
\end{equation} 
and use the trace theorem also to obtain
\begin{equation}\label{6_EqDt^2U}
\begin{split}
\int_\Gamma( g\rho_+\cQ^4e_3+\partial_t\cQ^5) \cdot\partial_t^2 u &\lesssim (\|\cQ^4\|_{L^2(\Gamma)}+\|\partial_t\cQ^5\|_{L^2(\Gamma)}) \|\partial_t^2 u\|_{H^1(\Omega)}\lesssim \cE_f^2 \|\partial_t^2 u\|_{H^1(\Omega)}.
\end{split}
\end{equation}
Substituting \eqref{3_EqDt^2U}, \eqref{4_EqDt^2U}, \eqref{5_EqDt^2U} and \eqref{6_EqDt^2U} into \eqref{2_EqDt^2U} yields
\begin{equation}\label{7_EqDt^2U}
\begin{split}
\|\partial_t^2 u\|_{L^2(\Omega)}^2 &\lesssim (\|u\|_{H^1(\Omega)}+\|\partial_t u\|_{H^1(\Omega)}+\cE_f^2) \|\partial_t^2 u\|_{H^1(\Omega)}+\cE_f^3\\
&\lesssim  (\|u\|_{H^1(\Omega)}+\|\partial_t u\|_{H^1(\Omega)}+\cE_f^2)( \|\nabla \partial_t^2 u\|_{L^2(\Omega)}+\cE_f).
\end{split}
\end{equation}
Combining \eqref{7_EqDt^2U} and \eqref{BoundDtU_L2} gives us that 
\begin{equation}\label{8_EqDt^2U}
\|\partial_t^2 u\|_{L^2(\Omega)}^2 \lesssim (\|\nabla \partial_t u\|_{L^2(\Omega)}+\sW) (\|\nabla\partial_t^2 u\|_{L^2(\Omega)} +\sW).
\end{equation}
Thanks to Cauchy-Schwarz's inequality, the resulting inequality \eqref{8_EqDt^2U} implies \eqref{BoundDt^2U_L2}. Lemma \ref{LemBoundDtU_L2} is shown.
\end{proof}

We are able to show Proposition \ref{PropCompareD}.
\begin{proof}[Proof of Proposition \ref{PropCompareD}]
We apply the elliptic estimate \eqref{EstElliptic} to 
\[
\begin{cases}
-\mu\Delta \partial_t u +\nabla \partial_t q = -\rho_0\partial_t^2 u+ g(\rho_0'u_3-\cQ^1)+ \partial_t\cQ^2 &\quad\text{in }\Omega,\\
\text{div} \partial_t u = \partial_t \cQ^3  &\quad\text{in }\Omega,\\
\partial_t u =\partial_t u &\quad\text{on }\Gamma,
\end{cases}
\]
to have that
\[
\begin{split}
\| \partial_t u\|_{H^3(\Omega)}^2+\| \partial_t q\|_{H^2(\Omega)}^2 &\lesssim  \|\partial_t^2 u\|_{H^1(\Omega)}^2 +\| u_3 \|_{H^1(\Omega)}^2+\|(\cQ^1,\partial_t\cQ^2)\|_{H^1(\Omega)}^2\\
&\qquad+ \| \partial_t\cQ^3\|_{H^2(\Omega)}^2+\| \partial_tu \|_{H^{5/2}(\Gamma)}^2.
\end{split}
\]
This yields 
\begin{equation}\label{0_CompareD_fh}
\begin{split}
\| \partial_t u\|_{H^3(\Omega)}^2+\|\partial_t q\|_{H^2(\Omega)}^2 &\lesssim  \|\partial_t^2 u\|_{H^1(\Omega)}^2 +\|u_3 \|_{H^1(\Omega)}^2 + \| \partial_tu  \|_{H^{5/2}(\Gamma)}^2+\cE_f^2(\cE_f^2+\cD_f^2),
\end{split}
\end{equation}
due to \eqref{EstTildeQ1} also. It follows from  \eqref{FourierIne} and the trace theorem that
\begin{equation}\label{4_CompareD_fh}
\begin{split}
\| \partial_t u \|_{H^{5/2}(\Gamma)}^2 &\lesssim \| \partial_t u \|_{H^{1/2}(\Gamma)}^2 +\sum_{\beta\in \N^2, |\beta|=2} \|  \partial_h^\beta \partial_t u \|_{H^{1/2}(\Gamma)}^2 \lesssim \|\partial_t u\|_{H^1(\Omega)}^2+ \sum_{\beta\in \N^2, |\beta|=2} \| \partial_h^\beta \partial_t u\|_{H^1(\Omega)}^2.
\end{split}
\end{equation}
Combining \eqref{0_CompareD_fh} and \eqref{4_CompareD_fh} gives us that
\begin{equation}\label{5_CompareD_fh}
\begin{split}
\| \partial_t u\|_{H^3(\Omega)}^2+\|\partial_t q\|_{H^2(\Omega)}^2 &\lesssim \|(u,\partial_t u,\partial_t^2 u)\|_{H^1(\Omega)}^2 + \sum_{\beta\in \N^2, |\beta|=2} \| \partial_h^\beta \partial_t u\|_{H^1(\Omega)}^2  +\cE_f^2(\cE_f^2+\cD_f^2).
\end{split}
\end{equation}
Thanks to the interpolation inequality \eqref{EstJensen}, we  get that, for $\nu>0$, 
\[
\|\partial_h^\beta \partial_t u\|_{L^2(\Omega)}^2 \lesssim \|\partial_t u\|_{H^2(\Omega)}^2 \lesssim \nu \|\partial_tu\|_{H^3(\Omega)}^2 + \nu^{-2}\|\partial_t u\|_{L^2(\Omega)}^2. 
\]
Hence, it follows from \eqref{5_CompareD_fh} that
\[
\begin{split}
\| \partial_t u\|_{H^3(\Omega)}^2+\|\partial_t q\|_{H^2(\Omega)}^2 &\lesssim \|(u,\partial_t u,\partial_t^2 u)\|_{H^1(\Omega)}^2+\nu \|\partial_tu\|_{H^3(\Omega)}^2 + \sum_{\beta\in \N^2, |\beta|=2} \| \nabla \partial_h^\beta \partial_t u\|_{L^2(\Omega)}^2 \\
&\qquad+\cE_f^2(\cE_f^2+\cD_f^2).
\end{split}
\]
Let $\nu>0$ be sufficiently small, one has
\begin{equation}\label{6_CompareD_fh}
\begin{split}
\| \partial_t u\|_{H^3(\Omega)}^2+\|\partial_t q\|_{H^2(\Omega)}^2 &\lesssim \|(u,\partial_t u,\partial_t^2 u)\|_{H^1(\Omega)}^2  + \sum_{\beta\in \N^2, |\beta|=2} \| \nabla \partial_h^\beta \partial_t u\|_{L^2(\Omega)}^2 \\
&\qquad\qquad+\cE_f^2(\cE_f^2+\cD_f^2).
\end{split}
\end{equation}

Meanwhile, applying the elliptic estimate  \eqref{EstElliptic} to
\[
\begin{cases}
-\mu\Delta u +\nabla q = -\rho_0\partial_t u -g\zeta e_3+ \cQ^2 &\quad\text{in }\Omega,\\
\text{div}  u = \cQ^3  &\quad\text{in }\Omega,\\
 u = u &\quad\text{on }\Gamma,
\end{cases}
\]
 we have
\[
\begin{split}
\|u\|_{H^5(\Omega)}^2+\|q\|_{H^4(\Omega)}^2 &\lesssim \|\partial_t u\|_{H^3(\Omega)}^2+\| \zeta\|_{H^3(\Omega)}^2+\| \cQ^2\|_{H^3(\Omega)}^2 +\|\cQ^3\|_{H^4(\Omega)}^2 +\|u\|_{H^{9/2}(\Gamma)}^2.
\end{split}
\]
Using \eqref{6_CompareD_fh} and \eqref{EstTildeQ1}, we further obtain 
\begin{equation}\label{7_CompareD_fh}
\begin{split}
\|u\|_{H^5(\Omega)}^2+\| q\|_{H^4(\Omega)}^2 &\lesssim  \| (u,\partial_t u, \partial_t^2 u)\|_{H^1(\Omega)}^2 + \|\zeta\|_{H^3(\Omega)}^2 +\|u\|_{H^{9/2}(\Gamma)}^2\\
&\qquad+ \sum_{\beta\in \N^2, |\beta|=2} \| \nabla \partial_h^\beta \partial_t u\|_{L^2(\Omega)}^2+ \cE_f^2(\cE_f^2+\cD_f^2).
\end{split}
\end{equation}
Using  \eqref{FourierIne} again and the trace theorem, we obtain that
\begin{equation}\label{8_CompareD_fh}
\begin{split}
\|u \|_{H^{9/2}(\Gamma)}^2 &\lesssim \|u \|_{H^{1/2}(\Gamma)}^2 +\sum_{\beta\in \N^2, |\beta|=4} \|  \partial_h^\beta  u \|_{H^{1/2}(\Gamma)}^2 \lesssim \| u\|_{H^1(\Omega)}^2+ \sum_{\beta\in \N^2, |\beta|=4} \| \partial_h^\beta u\|_{H^1(\Omega)}^2.
\end{split}
\end{equation}
Notice from \eqref{EstJensen} again that,  for $\nu>0$, 
\begin{equation}\label{9_CompareD_fh}
\|\partial_h^\beta u\|_{L^2(\Omega)}^2 \lesssim \|u\|_{H^4(\Omega)}^2 \lesssim \nu \|u\|_{H^5(\Omega)}^2+\nu^{-4}\|u\|_{L^2(\Omega)}^2.
\end{equation}
In view of \eqref{8_CompareD_fh} and \eqref{9_CompareD_fh}, we deduce from \eqref{7_CompareD_fh} that
\begin{equation}\label{10_CompareD_fh}
\begin{split}
\|u\|_{H^5(\Omega)}^2+\| q\|_{H^4(\Omega)}^2 &\lesssim  \| (u,\partial_t u, \partial_t^2 u)\|_{H^1(\Omega)}^2 + \|\zeta\|_{H^3(\Omega)}^2 +\nu \|u\|_{H^5(\Omega)}^2+ \sum_{\beta\in \N^2, |\beta|=4} \| \nabla \partial_h^\beta  u\|_{L^2(\Omega)}^2\\
&\qquad  + \sum_{\beta\in \N^2, |\beta|=2} \| \nabla \partial_h^\beta \partial_t u\|_{L^2(\Omega)}^2 + \cE_f^2(\cE_f^2+\cD_f^2).
\end{split}
\end{equation}
Let $\nu>0$ be sufficiently small, the inequality \eqref{10_CompareD_fh} implies that 
\begin{equation}\label{11_CompareD_fh}
\begin{split}
\|u\|_{H^5(\Omega)}^2+\| q\|_{H^4(\Omega)}^2 &\lesssim  \| (u,\partial_t u, \partial_t^2 u)\|_{H^1(\Omega)}^2+ \|\zeta\|_{H^3(\Omega)}^2 + \sum_{\beta\in \N^2, |\beta|=2} \| \nabla \partial_h^\beta \partial_t u\|_{L^2(\Omega)}^2\\
&\qquad + \sum_{\beta\in \N^2, |\beta|=4} \| \nabla \partial_h^\beta u\|_{L^2(\Omega)}^2 + \cE_f^2(\cE_f^2+\cD_f^2).
\end{split}
\end{equation}
Keep in mind the definition of $\cD_h$ \eqref{DissipationHor}, we obtain from \eqref{6_CompareD_fh} and \eqref{11_CompareD_fh} that 
\[
\begin{split}
\| u\|_{H^5(\Omega)}^2&+\| \partial_t u\|_{H^3(\Omega)}^2+ \|q\|_{H^4(\Omega)}^2 +\|\partial_t q\|_{H^2(\Omega)}^2 \\
 &\lesssim \cD_h^2 +   \| (u,\partial_t u, \partial_t^2 u)\|_{H^1(\Omega)}^2+ \|\zeta\|_{H^3(\Omega)}^2 +\cE_f^2(\cE_f^2+\cD_f^2).
\end{split}
\]
Thanks to the definition of $\cD_f$ \eqref{DissipationFull}, that implies 
\begin{equation}\label{12_CompareD_fh}
\cD_f^2 \lesssim \cD_h^2 + \| (u,\partial_t u, \partial_t^2 u)\|_{L^2(\Omega)}^2+ \|\zeta\|_{H^3(\Omega)}^2 + \cE_f^2(\cE_f^2+\cD_f^2).
\end{equation}
Thanks to Lemma \ref{LemBoundDtU_L2}, we deduce from \eqref{12_CompareD_fh} that
\[
\cD_f^2 \lesssim  \cD_h^2+ \|u\|_{H^1(\Omega)}^2 + \|\zeta\|_{H^3(\Omega)}^2 +\|\eta\|_{L^2(\Gamma)}^2  + \cE_f^3 +\cE_f^2\cD_f^2.
\]
We continue using  \eqref{EstJensen} to further have
\begin{equation}\label{13_CompareD_fh}
\cD_f^2 \leq C_{13}\Big(\cD_h^2 + \varepsilon^3 \cE_f^2 + \varepsilon^{-9}(\|(\zeta,u)\|_{L^2(\Omega)}^2+\|\eta\|_{L^2(\Gamma)}^2) +\cE_f^3 + \cE_f^2\cD_f^2\Big).
\end{equation}
Restricting further $C_{13}\delta_0^2 \leq \frac12$,
we obtain \eqref{EstCompareD_fh} from \eqref{13_CompareD_fh}. Proposition \ref{PropCompareD} is proven.
\end{proof}

\subsection{Proof of Proposition \ref{PropAprioriEnergy}}
Let us denote
\[\begin{cases}
C_{1,\varepsilon} = (C_1+C_2+C_3+C_4) \varepsilon^3+C_5+C_6+C_7+C_8+C_9,\\
C_{14}=\sum_{i=1}^9 C_i,\\
C_{2,\varepsilon} = (C_1+C_2+C_3+C_4)\varepsilon+(C_7+C_8+C_9)\varepsilon^3,\\
C_{3,\varepsilon}= C_5+C_6+(C_7+C_8+C_9)\varepsilon^{-27}.
\end{cases}\]
Keep in mind the definition of $\cD_h$ \eqref{DissipationHor}, we obtain from Propositions \ref{PropEstTransportEtaH^4}, \ref{PropEstEta^9/2}, \ref{PropEstPartial_t^lU_L2}, \ref{PropEstHorizonU},  \ref{PropEstZeta}  that 
\begin{equation}\label{0_HorizonEnergy}
\begin{split}
&\varepsilon^2 \Big(\sum_{l=0}^2 \|\partial_t^l \eta(t)\|_{H^{4-2l}(\Gamma)}^2 +\|\eta(t)\|_{H^{9/2}(\Gamma)}^2\Big)+\|\zeta(t)\|_{H^4(\Omega)}^2 \\
&\qquad\qquad+ \|(u,\partial_t u, \partial_t^2u)(t)\|_{L^2(\Omega)}^2+ \int_0^t \cD_h^2(s) ds\\
&\leq C_{1,\varepsilon}  \cE_f^2(0)+C_6\cE_f^3(t)+ C_{14} \varepsilon^3  \int_0^t \cE_f^2(s)ds +C_{2,\varepsilon}  \int_0^t \cD_f^2(s)ds \\
&\qquad+ C_{3,\varepsilon} \int_0^t (\|(u,\zeta)(s)\|_{L^2(\Omega)}^2+\|\eta(s)\|_{L^2(\Gamma)}^2) ds +  C_{1,\varepsilon} \int_0^t \cE_f(\cE_f^2+\cD_f^2)(s)ds. 
\end{split}
\end{equation}
Chaining \eqref{0_HorizonEnergy} with   \eqref{EstCompareD_fh} in Proposition \ref{PropCompareD}, we get that
\begin{equation}\label{EstHorizonEnergy}
\begin{split}
&\varepsilon^2 \Big(\sum_{l=0}^2 \|\partial_t^l \eta(t)\|_{H^{4-2l}(\Gamma)}^2 +\|\eta(t)\|_{H^{9/2}(\Gamma)}^2 + \|\zeta(t)\|_{H^4(\Omega)}^2\Big) \\
&\qquad\quad+\|(u,\partial_t u,\partial_t^2 u)(t)\|_{L^2(\Omega)}^2 + \frac1{C_{12}} \int_0^t \cD_f^2(s)ds   \\
&\leq C_{1,\varepsilon}  \cE_f^2(0)+C_6\cE_f^3(t)+(C_{14}+1) \varepsilon^3 \int_0^t \cE_f^2(s)ds + C_{2,\varepsilon}  \int_0^t \cD_f^2(s)ds \\
&\quad+  C_{4,\varepsilon} \int_0^t  (\|(u,\zeta)(s)\|_{L^2(\Omega)}^2+\|\eta(s)\|_{L^2(\Gamma)}^2) ds +  (C_{1,\varepsilon}+1) \int_0^t \cE_f(\cE_f^2+\cD_f^2)(s)ds,
\end{split}
\end{equation}
where  $C_{4,\varepsilon}= C_{3,\varepsilon}+\varepsilon^{-9}$. Let $0<\varepsilon\leq 1$ be sufficiently small such that  $C_{2,\varepsilon} \leq \frac1{2C_{12}}$. So that, the inequality \eqref{EstHorizonEnergy} implies 
\begin{equation}\label{1_HorizonEnergy}
\begin{split}
&\sum_{l=0}^2 \|\partial_t^l \eta(t)\|_{H^{4-2l}(\Gamma)}^2 + \|\eta(t)\|_{H^{9/2}(\Gamma)}^2  +\|\zeta(t)\|_{H^4(\Omega)}^2 +\|(u,\partial_t u,\partial_t^2 u)(t)\|_{L^2(\Omega)}^2+\int_0^t\cD_f^2(s)ds \\
&\leq C_{15}\Big( \varepsilon^{-2}  \cE_f^2(0)+\varepsilon \int_0^t \cE_f^2(s)ds+  \varepsilon^{-2} \int_0^t \cE_f(\cE_f^2+\cD_f^2)(s)ds\Big)\\
&\qquad\quad  +C_{15}\varepsilon^{-29} \int_0^t (\|(\zeta,u)(s)\|_{L^2(\Omega)}^2 +\|\eta(s)\|_{L^2(\Gamma)}^2 )ds+C_{15}\varepsilon^{-2}\cE_f^3(t).
\end{split}
\end{equation}
Combining   \eqref{EstCompareE_fh} and \eqref{1_HorizonEnergy}, one has 
\begin{equation}\label{EstHorizonEnergy_1}
\begin{split}
& \varepsilon^{1/4}\Big(\|u(t)\|_{H^4(\Omega)}^2 +\| q(t)\|_{H^3(\Omega)}^2 + \|\partial_t u(t)\|_{H^2(\Omega)}^2 +\| \partial_t q(t)\|_{H^1(\Omega)}^2 \Big) + \sum_{l=0}^2 \|\partial_t^l \eta(t)\|_{H^{4-2l}(\Gamma)}^2   \\
&\qquad+ \|\eta(t)\|_{H^{9/2}(\Gamma)}^2  + \|\zeta(t)\|_{H^4(\Omega)}^2+\|u(t)\|_{L^2(\Omega)}^2+ \|\partial_t^2 u(t)\|_{L^2(\Omega)}^2+\int_0^t\cD_f^2(s)ds \\
&\leq C_{11}\varepsilon^{1/4} \Big(\|(\partial_t^2 u, u_3)(t)\|_{L^2(\Omega)}^2+\|\zeta(t)\|_{H^2(\Omega)}^2 + \|\eta(t)\|_{H^{5/2}(\Gamma)}^2+ \|\partial_t\eta(t)\|_{H^{1/2}(\Gamma)}^2 \Big) \\
&\qquad+ C_{15}\Big( \varepsilon^{-2}  \cE_f^2(0)+\varepsilon \int_0^t \cE_f^2(s)ds+  \varepsilon^{-2} \int_0^t \cE_f(\cE_f^2+\cD_f^2)(s) ds\Big)\\
&\qquad +C_{15}\varepsilon^{-29} \int_0^t (\|(\zeta,u)(s)\|_{L^2(\Omega)}^2 +\|\eta(s)\|_{L^2(\Gamma)}^2 )ds+C_{15}\varepsilon^{-2}\cE_f^3(t)+C_{11}\varepsilon^{1/4} \cE_f^4(t).
\end{split}
\end{equation}
Let us refine $\varepsilon$ so that $C_{11}\varepsilon^{1/4} \leq \frac12$,
it follows from \eqref{EstHorizonEnergy_1} that
\begin{equation}\label{EstHorizonEnergy_3}
\begin{split}
&\sum_{j=0}^2( \|\partial_t^j u(t)\|_{H^{4-2j}(\Omega)}^2 + \|\partial_t^j \eta(t)\|_{H^{4-2j}(\Gamma)}^2 ) +\|\zeta(t)\|_{H^4(\Omega)}^2 +\|\eta\|_{H^{9/2}(\Gamma)}^2\\
&\qquad +\| q(t)\|_{H^3(\Omega)}^2 +\| \partial_t q(t)\|_{H^1(\Omega)}^2 + \int_0^t\cD_f^2(s)ds  \\
&\leq C_{16}\Big( \varepsilon^{-9/4} \cE_f^2(0)+  \varepsilon^{3/4} \int_0^t \cE_f^2(s)ds+  \varepsilon^{-9/4}  \int_0^t \cE_f(\cE_f^2+\cD_f^2)(s)ds \Big)\\
&\qquad\quad  + C_{16} \varepsilon^{-117/4} \int_0^t (\|(\zeta,u)(s)\|_{L^2(\Omega)}^2 +\|\eta(s)\|_{L^2(\Gamma)}^2 )ds+C_{16}\varepsilon^{-9/4}\cE_f^3(t).
\end{split}
\end{equation}
Combining \eqref{EstHorizonEnergy_3} and  \eqref{EstD_tZetaH^2} in Proposition \ref{PropEstD_tZetaH^2}, we obtain
\begin{equation}\label{EstHorizonEnergy_4}
\begin{split}
&\sum_{j=0}^2( \|\partial_t^j u(t)\|_{H^{4-2j}(\Omega)}^2 + \|\partial_t^j \eta(t)\|_{H^{4-2j}(\Gamma)}^2 ) +\|\zeta(t)\|_{H^4(\Omega)}^2 +\|\eta\|_{H^{9/2}(\Gamma)}^2\\
&\qquad +\| q(t)\|_{H^3(\Omega)}^2 +\| \partial_t q(t)\|_{H^1(\Omega)}^2 + \int_0^t\cD_f^2(s)ds  +\varepsilon^{1/4}(\|\partial_t\zeta\|_{H^2(\Omega)}^2+\|\partial_t^2\zeta\|_{L^2(\Omega)}^2) \\
&\leq C_{16}\Big( \varepsilon^{-9/4} \cE_f^2(0)+  \varepsilon^{3/4} \int_0^t \cE_f^2(s)ds+  \varepsilon^{-9/4}  \int_0^t \cE_f(\cE_f^2+\cD_f^2)(s)ds \Big)\\
&\qquad + C_{16} \varepsilon^{-117/4} \int_0^t (\|(\zeta,u)(s)\|_{L^2(\Omega)}^2 +\|\eta(s)\|_{L^2(\Gamma)}^2 )ds +C_{16}\varepsilon^{-9/4}\cE_f^3(t)\\
&\qquad+ C_{10}\varepsilon^{1/4}(\|u_3\|_{H^2(\Omega)}^2+\|\partial_t u_3\|_{L^2(\Omega)}^2+\cE_f^4(t)).
\end{split}
\end{equation}
We continue refining $\varepsilon$ so that  $C_{10}\varepsilon^{1/4} \leq \frac12$.
It follows from \eqref{EstHorizonEnergy_4} that
\begin{equation}\label{EstHorizonEnergy_5}
\begin{split}
&\sum_{j=0}^2( \|\partial_t^j u(t)\|_{H^{4-2j}(\Omega)}^2 + \|\partial_t^j \eta(t)\|_{H^{4-2j}(\Gamma)}^2 ) +\|\zeta(t)\|_{H^4(\Omega)}^2 +\|\eta\|_{H^{9/2}(\Gamma)}^2\\
&\quad +\| q(t)\|_{H^3(\Omega)}^2 +\| \partial_t q(t)\|_{H^1(\Omega)}^2 + \int_0^t\cD_f^2(s)ds  +\varepsilon^{1/4}(\|\partial_t\zeta\|_{H^2(\Omega)}^2+\|\partial_t^2\zeta\|_{L^2(\Omega)}^2) \\
&\leq C_{17}\Big( \varepsilon^{-9/4} \cE_f^2(0)+  \varepsilon^{3/4} \int_0^t \cE_f^2(s)ds+  \varepsilon^{-9/4}  \int_0^t \cE_f(\cE_f^2+\cD_f^2)(s)ds \Big)\\
&\quad + C_{17} \varepsilon^{-117/4} \int_0^t (\|(\zeta,u)(s)\|_{L^2(\Omega)}^2 +\|\eta(s)\|_{L^2(\Gamma)}^2 )ds + C_{17}(\varepsilon^{-9/4} \cE_f^3(t)+\varepsilon^{1/4}\cE_f^4(t)).
\end{split}
\end{equation}
Let us recall the definition of $\cE_f$ \eqref{EnergyFull} and divide both sides of  \eqref{EstHorizonEnergy_5} by $\varepsilon^{1/4}$  to deduce 
\begin{equation}\label{EstHorizonEnergy_6}
\begin{split}
\cE_f^2(t)+ \int_0^t\cD_f^2(s)ds  &\leq C_{18}\Big( \varepsilon^{-5/2} \cE_f^2(0)+  \varepsilon^{1/2} \int_0^t \cE_f^2(s)ds+  \varepsilon^{-5/2}  \int_0^t \cE_f(\cE_f^2+\cD_f^2)(s)ds \Big)\\
&\qquad\quad  + C_{18} \varepsilon^{-59/2} \int_0^t (\|(\zeta,u)(s)\|_{L^2(\Omega)}^2 +\|\eta(s)\|_{L^2(\Gamma)}^2 )ds + C_{18}\varepsilon^{-5/2} \cE_f^3(t).
\end{split}
\end{equation}
Switching $\varepsilon^{1/2}$ by $\varepsilon$ in \eqref{EstHorizonEnergy_3}, one has \eqref{EstAprioriEnergy}. Proof of Proposition \ref{PropAprioriEnergy} is finished.

\section{Nonlinear instability}\label{SectNonlinear}

Note again that, we compactly write $U=(\zeta, u,  q,\eta)$ throughout this paper. 

Owing to Proposition \ref{PropSolEqLinear}, there exist infinitely many normal modes $e^{\lambda_n(\vk)}V_n(\vk, x)$ $(n\geq 1)$ to the linearized equations \eqref{EqLinearized}, being arranged as \eqref{AssumeLambdaN}.  Let us fix a $\vk=\vk_0 \in S_\Lambda$ such that \eqref{AssumeLambdaN} holds. Dropping $\vk_0$ in the notations since it is fixed, we recall \eqref{GeneralInitialCond_Ido}
\[
U^\sN(t, x) :=  \sum_{j=1}^\sN \Csf_j e^{\lambda_j t}V_j( x)
\]
to approximate the nonlinear equations \eqref{EqPertur} and require that the coefficients $\Csf_j$ satisfying \eqref{NormalizedCond_1}-\eqref{NormalizedCond_2}. Due to the compatibility conditions \eqref{CompaCond}, we cannot set $U^\sN(0,x)$ as the initial data for the nonlinear equations \eqref{EqPertur}. With the help of an abstract argument in \cite[Section 5C]{JT13}, we obtain the modified initial data $U_0^{\delta,\sN}(x)$.
\begin{proposition}\label{PropModifyData}
There exist a number $\delta_0>0$ and a family of initial data 
\[
U_0^{\delta,\sN}(x) = \delta U^\sN(0,x)+ \delta^2 U_\star^{\delta,\sN}(x)
\]
for $\delta \in (0,\delta_0)$ such that 
\begin{enumerate}
\item   $U_0^{\delta,\sN}$ satisfies the compatibility conditions \eqref{CompaCond} and $\cE_f(U_\star^{\delta,\sN}) \leq C_\sN^\star <\infty$ with $C_\sN^\star$ being independent of $\delta$,
\item  the nonlinear equations \eqref{EqPertur} with the above initial data  $U_0^{\delta, \sN}$ has a unique solution   $U^{\delta,\sN}$ satisfying that $\sup_{0\leq t<T^{\max}}\cE_f(U^{\delta,\sN}(t))<\infty$.
\end{enumerate}
\end{proposition}

\subsection{The difference functions} 

Set  $U^d(t,x)= U^{\delta,\sN}(t,x) - \delta U^\sN(t,x)$.
Since $U^{\delta,\sN}$ solves the nonlinear equations \eqref{EqPertur} and $U^\sN$ solves the linearized equations \eqref{EqLinearized}, we obtain that  $U^d$ satisfies 
\begin{equation}\label{EqDiff}
\begin{cases}
\partial_t \zeta^d+ \rho_0' u_3^d =\cQ^1(U^{\delta,\sN}) &\quad\text{in }\Omega,\\
\rho_0 \partial_t  u^d-\mu \Delta  u^d+ \nabla q^d+g\zeta^d e_3=\cQ^2(U^{\delta,\sN}) &\quad\text{in }\Omega,\\
\text{div}u^d= \cQ^3(U^{\delta,\sN}) &\quad\text{in }\Omega,\\
\partial_t\eta^d=u_3^d+\cQ^4(U^{\delta,\sN}) &\quad\text{on }\Gamma,\\
((q^d-g\rho_+\eta^d)\text{Id}-\mu\bS u^d)e_3= \cQ^5(U^{\delta,\sN}) &\quad\text{on }\Gamma.
\end{cases}
\end{equation}
The initial condition for \eqref{EqDiff} is 
\begin{equation}\label{InitialEqDiff}
U^d(0)= (\zeta^d, u^d,\eta^d,q^d)(0)=\delta^2 U_\star^{\delta,\sN}.
\end{equation}

Let $\|U\|_{\cE_f} := \cE_f(U)$, which is defined as in \eqref{EnergyFull}. 
Let $F_\sN(t) = \sum_{j=j_m}^\sN |\Csf_j| e^{\lambda_j t}$ and $0<\nu_0 \ll 1$ be fixed later \eqref{ChoiceEps}. There exists  a unique $T^{\delta}$ such that $\delta F_\sN(T^{\delta})=\nu_0$. 
Let
\begin{equation}\label{DefC_19}
\begin{split}
C_{19} &= \|U^\sN(0)\|_{\cE_f}, \quad
 C_{20} = \sqrt{\|(\zeta^\sN,u^\sN)(0)\|_{L^2(\Omega)}^2 +\|\eta^\sN(0)\|_{L^2(\Gamma)}^2}.
\end{split}
\end{equation}
We define
\begin{equation}\label{DefTstar}
\begin{split}
T^{\star} &:= \sup\Big\{t\in (0,T^{\max})| \|U^{\delta,\sN}(t)\|_{\cE_f}\leq 2C_{19} \delta_0\},\\
T^{\star\star} &:= \sup \{t\in (0,T^{\max})| \|(\zeta^{\delta,\sN},u^{\delta,\sN})(t)\|_{L^2(\Omega)}+\|\eta^{\delta,\sN}(t)\|_{L^2(\Gamma)} \leq 2 C_{20} \delta F_\sN(t)\Big\}.
\end{split}
\end{equation}
Note that 
\[
\|U^{\delta,\sN}(0)\|_{\cE_f} \leq  \delta \|U^\sN(0)\|_{\cE_f}+\|U^d(0)\|_{\cE_f}  \leq C_{19}\delta+ C_\sN^\star \delta^2 < 2C_{19} \delta_0,
\]
 we  then have $T^{\star}>0$. Similarly, we have $T^{\star\star}>0$.

The aim of this part is to derive the bound in time of $\|(\zeta^d,u^d)(t)\|_{L^2(\Omega)}+\|\eta^d(t)\|_{L^2(\Gamma)}$ in the following proposition. 
\begin{proposition}\label{PropL2_NormU^d}
For all $t \leq \min(T^{\delta},T^{\star}, T^{\star\star})$, there holds
\begin{equation}\label{L2_NormU^d}
\begin{split}
&\|(\zeta^d, u^d)(t)\|_{L^2(\Omega)}^2+\|\eta^d(t)\|_{L^2(\Gamma)}^2 \leq C_{21} \delta^3 \Big(\sum_{j=j_m}^\sP |\Csf_j| e^{\lambda_j t}+ \max(0,\sN-\sP)\max_{\sP+1\leq j\leq \sN}|\Csf_j| e^{\frac23 \Lambda t}\Big)^3.
\end{split}
\end{equation}
\end{proposition}

In order to prove Proposition \ref{PropL2_NormU^d}, we need the following bound in time of $\|U^{\delta,\sN}(t)\|_{\cE_f}$. 
 \begin{proposition}\label{PropNormU^delta_M}
For all $t\leq \min\{T^\delta,T^\star,T^{\star\star}\}$, there holds 
\begin{equation}\label{BoundNormU^delta_M}
\|U^{\delta,\sN}(t)\|_{\cE_f} \leq C_{22}\delta F_\sN(t)\quad\text{for all }t\leq \min\{T^{\delta}, T^{\star}, T^{\star\star}\}.
\end{equation}
\end{proposition}
\begin{proof}
 We fix a sufficiently small constant $\varepsilon$ such that $C_0\varepsilon \leq \frac{\lambda_\sN}4$ and  Proposition \ref{PropAprioriEnergy} holds. Hence, it follows from \eqref{EstAprioriEnergy} that 
\begin{equation}\label{4_HorizonEnergy}
\begin{split}
&\cE_f^2(U^{\delta,\sN}(t))+\int_0^t\cD_f^2(U^{\delta,\sN}(s))ds \\
&\leq \frac{\lambda_\sN}4  \int_0^t \cE_f^2(U^{\delta,\sN}(s))ds + C_{\lambda_\sN} \Big( \cE_f^2(U^{\delta,\sN}(0))+    \int_0^t \cE_f(\cE_f^2+\cD_f^2)(U^{\delta,\sN}(s))ds \Big)\\
&\qquad  + C_{\lambda_\sN} \Big( \cE_f^3(U^{\delta,\sN}(t))+ \int_0^t (\|(\zeta^{\delta,\sN},u^{\delta,\sN})(s)\|_{L^2(\Omega)}^2 +\|\eta^{\delta,\sN}(s)\|_{L^2(\Gamma)}^2 )ds\Big).
\end{split}
\end{equation}
Refining also $\delta_0$, we get  $C_{\lambda_\sN}\delta_0 \leq \frac12$ and $ C_{\lambda_\sN}\delta_0 \leq \frac{\lambda_\sN}4$,
one thus has
\begin{equation}\label{5_HorizonEnergy}
\begin{split}
\frac12 \cE_f^2(U^{\delta,\sN}(t))+\frac12 \int_0^t\cD_f^2(U^{\delta,\sN}(s))ds&\leq C_{\lambda_\sN} \cE_f^2(U^{\delta,\sN}(0)) + \Big(\frac{\lambda_\sN}4+  \delta C_{\lambda_\sN}\Big) \int_0^t \cE_f^2(U^{\delta,\sN}(s))ds \\
&\quad+ C_{\lambda_\sN} \int_0^t (\|(\zeta^{\delta,\sN},u^{\delta,\sN})(s)\|_{L^2(\Omega)}^2 +\|\eta^{\delta,\sN}(s)\|_{L^2(\Gamma)}^2 )ds \\
&\leq C_{\lambda_\sN} \cE_f^2(U^{\delta,\sN}(0)) + \frac{\lambda_\sN}2 \int_0^t \cE_f^2(U^{\delta,\sN}(s))ds \\
&\quad+ C_{\lambda_\sN} \int_0^t (\|(\zeta^{\delta,\sN},u^{\delta,\sN})(s)\|_{L^2(\Omega)}^2 +\|\eta^{\delta,\sN}(s)\|_{L^2(\Gamma)}^2 )ds.
\end{split}
\end{equation}
Consequently, for all $t\leq \min\{T^{\delta}, T^{\star}, T^{\star\star}\}$,
\[\begin{split}
\|U^{\delta,\sN}(t)\|_{\cE_f}^2  &\leq 2C_{\lambda_\sN} \|U^{\delta,\sN}(0)\|_{\cE_f}^2 +\lambda_\sN \int_0^t \|U^{\delta, \sN}(s)\|_{\cE_f}^2ds  \\
&\qquad+ 2C_{\lambda_\sN}  \int_0^t (\|(\zeta^{\delta,\sN},u^{\delta, \sN})(s)\|_{L^2(\Omega)}^2+\|\eta^{\delta,\sN}(s)\|_{L^2(\Gamma)}^2)ds \\
&\leq \lambda_\sN \int_0^t \|U^{\delta, \sN}(s)\|_{\cE_f}^2ds+ C_{23}\delta^2 F_\sN^2(t).
\end{split}\]
Applying Gronwall’s inequality, the resulting inequality tells us that
\begin{equation}\label{2ndEnergyNormDelta}
\|U^{\delta,\sN}(t)\|_{\cE_f}^2 \leq C_{23}\Big( \delta^2 F_\sN^2(t) + \delta^2\int_0^t e^{\lambda_\sN(t-s)} F_\sN^2(s) ds\Big).
\end{equation}
Note that $\lambda_\sN <\lambda_j$ for all $1\leq j\leq \sN-1$, we have
\begin{equation}\label{3rdEnergyNormDelta}
\begin{split}
\int_0^t e^{\lambda_\sN(t-s)} F_\sN^2(s) ds &\leq \sN^2 \sum_{j=j_m}^\sN \int_0^t e^{\lambda_\sN(t-s)} |\Csf_j|^2 e^{2\lambda_j s} ds \leq \sN^2  e^{\lambda_\sN t}\sum_{j=j_m}^\sN |\Csf_j|^2 \frac{e^{(2\lambda_j-\lambda_\sN)t}}{2\lambda_j-\lambda_\sN}.
\end{split}
\end{equation}
Substituting \eqref{3rdEnergyNormDelta} into \eqref{2ndEnergyNormDelta}, this yields \eqref{BoundNormU^delta_M}. We deduce Proposition \ref{PropNormU^delta_M}.
\end{proof}

We now prove Proposition \ref{PropL2_NormU^d}. 
\begin{proof}[Proof of Proposition \ref{PropL2_NormU^d}]
Differentiating   $\eqref{EqDiff}_{2,5}$ with respect to $t$ and then eliminating the  terms $\partial_t\zeta^d, \partial_t \eta^d$ by using $\eqref{EqDiff}_{1,4}$, we deduce from \eqref{EqDiff} that 
\begin{equation}\label{EqDtV_DtQ}
\begin{cases}
\rho_0 \partial_t^2 u^d+\nabla \partial_tq^d -\mu\Delta\partial_t u^d-g\rho_0'u_3^d e_3 = \partial_t\cQ^2(U^{\delta,\sN}) -g\cQ^1(U^{\delta,\sN}) e_3 &\quad\text{in }\Omega,\\
\text{div}\partial_t u^d= \partial_t\cQ^3(U^{\delta,\sN}) &\quad\text{in }\Omega,\\
(\partial_t q^d \text{Id}-\mu \bS \partial_t u^d)e_3 =g\rho_+ u^d e_3  + g\rho_+\cQ^4(U^{\delta,\sN})e_3+ \partial_t\cQ^5(U^{\delta,\sN})&\quad\text{on }\Gamma.
\end{cases}
\end{equation}
Multiplying both sides of $\eqref{EqDtV_DtQ}_1$ by $\partial_t u^d$,  we obtain that 
\begin{equation}\label{1EqD_tV_linearized}
\begin{split}
&\frac12 \frac{d}{dt} \int_\Omega\rho_0|\partial_t u^d|^2 + \int_\Omega (\nabla\partial_t q^d-\mu \Delta\partial_t u^d)\cdot\partial_t u^d -\int_\Omega g\rho_0' u_3^d\partial_t u_3^d \\
&=  \int_\Omega (\partial_t\cQ^2(U^{\delta,\sN}) -g\cQ^1(U^{\delta,\sN}) e_3) \cdot\partial_t u^d.
\end{split}
\end{equation}
Using the integration by parts, we deduce from the resulting equality that
\[\begin{split}
&\frac12 \frac{d}{dt}\Big( \int_{\Omega}\rho_0|\partial_t u^d|^2 -\int_{\Omega}g\rho_0'|u_3^d|^2)  + \int_\Gamma (\partial_t q \text{Id}-\mu \bS \partial_t u^d)e_3 \cdot \partial_t u^d \\
&= \int_\Omega \partial_t q^d \text{div}\partial_t u^d  -\frac{\mu}2 \int_\Omega|\bS\partial_t u^d|^2+   \int_\Omega (\partial_t\cQ^2(U^{\delta,\sN}) -g\cQ^1(U^{\delta,\sN}) e_3) \cdot\partial_t u^d.
\end{split}\]
Substituting $\eqref{EqDtV_DtQ}_{2,3}$ into \eqref{1EqD_tV_linearized}, we have 
\[\begin{split}
&\frac12 \frac{d}{dt}\Big( \int_{\Omega}\rho_0|\partial_t u^d|^2 -\int_{\Omega}g\rho_0'|u_3^d|^2 \Big)  + \int_\Gamma g\rho_+  \partial_t u_3^d \partial_t u_3^d +\int_\Gamma (\partial_t\cQ^5(U^{\delta,\sN}) + g\rho_+\cQ^4(U^{\delta,\sN}) e_3)\cdot \partial_t u^d \\
&= \int_\Omega \partial_tq^d \partial_t \cQ^3(U^{\delta,\sN}) -\frac{\mu}2 \int_{\Omega} |\bS\partial_t u^d|^2 +  \int_\Omega (\partial_t\cQ^2(U^{\delta,\sN}) -g\cQ^1(U^{\delta,\sN}) e_3) \cdot\partial_t u^d.
\end{split}\]
This yields
\begin{equation}\label{EqD_tV_linearized}
\begin{split}
&\frac12 \frac{d}{dt}\Big( \int_{\Omega}\rho_0|\partial_t u^d|^2 -\int_{\Omega}g\rho_0'|u_3^d|^2  +\int_{\Gamma}g\rho_+|u_3^d|^2 \Big) +\frac{\mu}2 \int_{\Omega} |\bS\partial_t u^d|^2 \\
&= \int_\Omega (\partial_t\cQ^2(U^{\delta,\sN}) -g\cQ^1(U^{\delta,\sN}) e_3) \cdot\partial_t u^d + \int_\Omega \partial_tq^d \partial_t \cQ^3(U^{\delta,\sN}) \\
&\qquad - \int_\Gamma (\partial_t\cQ^5(U^{\delta,\sN}) + g\rho_+\cQ^4(U^{\delta,\sN}) e_3)\cdot \partial_t u^d,
\end{split}
\end{equation}

Note that 
\[
\begin{split}
 &\int_\Omega (\partial_t\cQ^2(U^{\delta,\sN}) -g\cQ^1(U^{\delta,\sN}) e_3) \cdot\partial_t u^d \\
 &\qquad\lesssim   (\|\partial_t\cQ^2(U^{\delta,\sN})\|_{L^2(\Omega)}+\|\cQ^1(U^{\delta,\sN})\|_{L^2(\Omega)}) (\|\partial_t u^{\delta,\sN}\|_{L^2(\Omega)}+\delta \|\partial_t u^\sN\|_{L^2(\Omega)}).
 \end{split}
\]
 In view of \eqref{1stIneLemHorizon}, \eqref{PropNormU^delta_M} and the definition of $U^M$, we have
\begin{equation}\label{1stEqDtV_DtQ}
\begin{split}
\int_\Omega (\partial_t\cQ^2(U^{\delta,\sN}) -g\cQ^1(U^{\delta,\sN}) e_3) \cdot\partial_t u^d &\lesssim \|U^{\delta,\sN}\|_{\cE_f}^2 (\|U^{\delta,\sN}\|_{\cE_f} +\delta\|\partial_t u^\sN\|_{L^2(\Omega)}) \lesssim \delta^3F_\sN^3(t).
\end{split}
\end{equation} 
Similarly, we observe  
\begin{equation}\label{01stEqDtV_DtQ}
\begin{split}
\int_\Omega \partial_tq^d \partial_t \cQ^3(U^{\delta,\sN}) &\lesssim  \|\partial_t \cQ^3(U^{\delta,\sN})\|_{L^2(\Omega)} (\|\partial_t q^{\delta,\sN}\|_{L^2(\Omega)} +\delta \|\partial_t q^\sN\|_{L^2(\Omega)}) \lesssim \delta^3 F_\sN^3(t).
\end{split}
\end{equation}
We continue applying \eqref{1stIneLemHorizon}, \eqref{PropNormU^delta_M} and use the trace theorem to get
\begin{equation}\label{2ndEqDtV_DtQ}
\begin{split}
\int_\Gamma &(\partial_t\cQ^5(U^{\delta,\sN}) + g\rho_+\cQ^4(U^{\delta,\sN})e_3)\cdot \partial_t u^d\\
&\lesssim (\|\partial_t\cQ^5(U^{\delta,\sN})\|_{L^2(\Gamma)} + \|\cQ^4(U^{\delta,\sN})\|_{L^2(\Gamma)}) \|\partial_tu^d\|_{L^2(\Gamma)}\\
 &\lesssim (\|\partial_t\cQ^5(U^{\delta,\sN})\|_{H^{1/2}(\Gamma)} + \|\cQ^4(U^{\delta,\sN})\|_{H^{1/2}(\Gamma)}) \|\partial_tu^d\|_{H^1(\Omega)}\\
 &\lesssim \delta^3F_\sN^3(t).
\end{split}
\end{equation} 
Substituting \eqref{1stEqDtV_DtQ}, \eqref{01stEqDtV_DtQ} and  \eqref{2ndEqDtV_DtQ} into \eqref{EqD_tV_linearized}, we obtain that 
\begin{equation}\label{1stIneSharpV}
 \begin{split}
&\int_{\Omega}\rho_0|\partial_t u^d(t)|^2 + \int_0^t \int_{\Omega}\mu|\bS\partial_t u^d(s)|^2 ds\leq z_1+  \int_{\Omega}g\rho_0'|u_3^d(t)|^2 -\int_{\Gamma}g\rho_+|u_3^d(t)|^2 +C_{24}\delta^3F_\sN^3(t),
\end{split} 
\end{equation}
where 
\[
z_1 = \int_{\Omega}\rho_0|\partial_t u^d(0)|^2 - \int_{\Omega}g\rho_0'|u_3^d(0)|^2+\int_{\Gamma}g\rho_+|u_3^d(0)|^2.
\]
Note that $u^d$ is not divergence-free, we cannot apply Lemma \ref{LemIneLambdaSquare} directly. In view of \cite[Lemma A.9]{Wan19}, we thus  decompose $u^d$  as 
\[
u^d= \underline w+ \widehat w
\]
such that $ \underline  w$ is divergence-free and $\widehat w$ satisfies that 
\[
\text{div}\widehat w= \cQ^3(U^\delta,M) \quad\text{and}\quad \|\widehat w\|_{H^s(\Omega)} \lesssim \|\cQ^3(U^{\delta,M})\|_{H^{s-1}(\Omega)}\quad\text{for any }s\geq 1.
\]
We thus have
\begin{equation}\label{1-1IneSharpV}
\begin{split}
\int_{\Omega}g\rho_0'|u_3^d|^2 -\int_{\Gamma}g\rho_+|u_3^d|^2 &= \Big( \int_\Omega g\rho_0'|\underline w|^2 -\int_\Gamma g\rho_+|\underline w|^2\Big)  +2\int_\Omega g\rho_0' \underline w\cdot\widehat w +\int_\Omega g\rho_0'|\widehat w|^2  \\
&\qquad- 2\int_\Gamma g\rho_+ \underline w\cdot \widehat w- \int_\Gamma g\rho_+|\widehat w|^2,
\end{split}
\end{equation}
and
\begin{equation}\label{2-1IneSharpV}
\begin{split}
\Lambda^2 \int_\Omega\rho_0|u_3^d|^2 + \frac{\Lambda}2\int_\Omega \mu|\bS  u_3^d|^2 &= \Lambda^2 \int_\Omega\rho_0|\widehat w|^2 + \frac12\Lambda\int_\Omega \mu|\bS \widehat w|^2+ \Lambda^2 \Big( 2\int_\Omega\rho_0 \underline w \cdot \widehat w+ \int_\Omega \rho_0 |\widehat w|^2\Big) \\
&\qquad+ \frac{\Lambda}2 \Big(2\int_\Omega \mu \bS \underline w: \bS \widehat w+\int \mu|\bS\widehat w|^2\Big).
\end{split}
\end{equation}
We apply Lemma \ref{LemIneLambdaSquare} to $\underline w$ to obtain 
\begin{equation}\label{3-1IneSharpV}
 \int_{\Omega}g\rho_0'|\underline w|^2 -\int_{\Gamma}g\rho_+|\underline w|^2 \leq \Lambda^2 \int_\Omega\rho_0|\underline w|^2 + \frac12\Lambda\int_\Omega \mu|\bS \underline w|^2. 
\end{equation}
Combining \eqref{1-1IneSharpV}, \eqref{2-1IneSharpV} and \eqref{3-1IneSharpV} gives us that
\begin{equation}\label{4-1IneSharpV}
\begin{split}
\int_{\Omega}g\rho_0'|u_3^d|^2 -\int_{\Gamma}g\rho_+|u_3^d|^2 &\leq \Lambda^2 \int_\Omega\rho_0|u_3^d|^2 + \frac12\Lambda\int_\Omega \mu|\bS u_3^d|^2 - \Lambda^2 \Big( 2\int_\Omega\rho_0 \underline w\cdot \widehat w+ \int_\Omega \rho_0 |\widehat w|^2\Big) \\
&\quad - \frac12\Lambda \Big(2\int_\Omega \mu \bS \underline w: \bS \widehat w+\int \mu|\bS\widehat w|^2\Big)+2\int_\Omega g\rho_0' \underline w\cdot\widehat w +\int_\Omega g\rho_0'|\widehat w|^2  \\
&\quad - 2\int_\Gamma g\rho_+ \underline w\cdot \widehat w- \int_\Gamma g\rho_+|\widehat w|^2.
\end{split}
\end{equation}
By the trace theorem, we have
\[
\begin{split}
-\Lambda^2  \int_\Omega\rho_0 \underline w\cdot \widehat w &- \frac12\Lambda \int_\Omega \mu \bS \underline w: \bS \widehat w + \int_\Omega g\rho_0' \underline w\cdot\widehat w -\int_\Gamma g\rho_+ \underline w\cdot \widehat w\\
&\lesssim \|\underline w\|_{H^1(\Omega)} \|\widehat w\|_{H^1(\Omega)}+\|\widehat w\|_{H^1(\Omega)}^2.
\end{split}
\]
Together with \eqref{1stIneLemHorizon}-\eqref{BoundNormU^delta_M}, that implies 
\begin{equation}\label{6-1IneSharpV}
\begin{split}
-\Lambda^2  \int_\Omega&\rho_0 \underline w\cdot \widehat w- \frac12\Lambda \int_\Omega \mu \bS \underline w: \bS \widehat w + \int_\Omega g\rho_0' \underline w\cdot\widehat w -\int_\Gamma g\rho_+ \underline w\cdot \widehat w\\
&\lesssim (\|u_3^d\|_{H^1(\Omega)}+ \|\widehat w\|_{H^1(\Omega)} ) \|\widehat w\|_{H^1(\Omega)}+ \|\widehat w\|_{H^1(\Omega)}^2\\
&\lesssim (\|u^{\delta,\sN}\|_{H^1(\Omega)}+ \delta \|u^\sN\|_{H^1(\Omega)})\|\cQ^3(U^{\delta,\sN})\|_{L^2(\Omega)} +\|\cQ^3(U^{\delta,\sN})\|_{L^2(\Omega)}^2 \\
&\lesssim \delta^3 F_\sN^3,
\end{split}
\end{equation}
and
\begin{equation}\label{5-1IneSharpV}
\begin{split}
-\Lambda^2 \int_\Omega\rho_0|\widehat w|^2 - \frac12\Lambda\int_\Omega \mu|\bS \widehat w|^2 +  \int_\Omega g\rho_0'|\widehat w|^2 +\int_\Gamma g\rho_+|\widehat w|^2&\lesssim \|\widehat w\|_{H^1(\Omega)}^2 \\
&\lesssim \|\cQ^3(U^{\delta,M})\|_{L^2(\Omega)}^2  \lesssim \delta^4 F_\sN^4,
\end{split}
\end{equation}
Plugging \eqref{6-1IneSharpV} and \eqref{5-1IneSharpV} into \eqref{4-1IneSharpV}, we arrive at
\begin{equation}\label{7-1IneSharpV}
\begin{split}
\int_{\Omega}g\rho_0'|u_3^d|^2 -\int_{\Gamma}g\rho_+|u_3^d|^2 \leq \Lambda^2 \int_\Omega\rho_0|u_3^d|^2 + \frac12\Lambda\int_\Omega \mu|\bS u_3^d|^2+ C_{25}\delta^3F_\sN^3.
\end{split}
\end{equation}
It follows from \eqref{7-1IneSharpV} and \eqref{1stIneSharpV} that 
\begin{equation}\label{2ndIneSharpV}
 \begin{split}
\int_{\Omega}\rho_0|\partial_t u^d(t)|^2  + \int_0^t \int_{\Omega}\mu|\bS\partial_t u^d(s)|^2 ds &\leq z_1+\Lambda^2 \int_{\Omega}\rho_0|u^d(t)|^2  + \frac12\Lambda \int_{\Omega}\mu|\bS u^d(t)|^2 \\
&\qquad+(C_{24}+C_{25})\delta^3F_\sN^3(t).
 \end{split} 
\end{equation}
Using Cauchy-Schwarz's inequality, we have that
\begin{equation}\label{2ndInePropNormU^d}
\begin{split}
 \int_{\Omega}\mu|\bS u^d(t)|^2 &=\int_{\Omega}\mu|\bS u^d(0)|^2 + 2\int_0^t\int_{\Omega} \mu \bS u^d(s) : \bS \partial_t u^d(s)  ds \\
&\leq  \int_{\Omega}\mu|\bS u^d(0)|^2 + \frac1{\Lambda} \int_0^t \int_{\Omega}\mu|\bS \partial_t u^d(s)|^2  ds + \Lambda \int_0^t\int_{\Omega} \mu |\bS u^d(s)|^2 ds
\end{split}
\end{equation}
and that
\begin{equation}\label{3rdInePropNormU^d}
\frac{d}{dt} \int_{\Omega}\rho_0|u^d|^2  \leq \frac1{\Lambda} \int_{\Omega}\rho_0|\partial_t u^d|^2 + \Lambda \int_{\Omega}\rho_0|u^d|^2.
\end{equation}
Three above inequalities \eqref{2ndIneSharpV}, \eqref{2ndInePropNormU^d} and \eqref{3rdInePropNormU^d} imply that 
\begin{equation}\label{4thInePropNormU^d}
\begin{split}
\frac{d}{dt} \int_{\Omega}\rho_0|u^d(t)|^2 + \frac12 \int_{\Omega}\mu|\bS u^d(t)|^2 &\leq \frac{z_1}{\Lambda} + \int_{\Omega}\mu|\bS u^d(0)|^2 + 2\Lambda\int_{\Omega}\rho_0|u^d(t)|^2\\
&\qquad+ \Lambda \int_0^t\int_{\Omega} \mu |\bS u^d(s)|^2 ds+(C_{24}+C_{25})\delta^3F_\sN^3(t),
\end{split}
\end{equation}
It follows from $U^d(0)=\delta^2 U_\star^{\delta,\sN}$  that  $z_1\lesssim \delta^4$, this yields
\[
\frac{z_1}{\Lambda} + \int_{\Omega}\mu|\bS u^d(0)|^2 \lesssim \delta^4.
\]
 Hence, the inequality \eqref{4thInePropNormU^d} implies 
\begin{equation}\label{5thInePropNormU^d}
\begin{split}
\frac{d}{dt} \int_{\Omega}\rho_0|u^d(t)|^2 + \frac12 \int_{\Omega}\mu|\bS u^d(t)|^2 &\leq 2\Lambda\int_{\Omega} \rho_0|u^d(t)|^2 +\Lambda \int_0^t\int_{\Omega} \mu |\bS u^d(s)|^2 ds\\
&\qquad+C_{26}\delta^3F_\sN^3(t).
\end{split}
\end{equation}
In view of Gronwall's inequality, we obtain from \eqref{5thInePropNormU^d} that
\begin{equation}\label{6thInePropNormU^d}
\begin{split}
 \int_{\Omega}\rho_0|u^d(t)|^2  + \frac12 \int_0^t \int_{\Omega}\mu|\bS u^d(s)|^2 ds&\leq  C_{26} \delta^3 \int_0^te^{2\Lambda (t-s)}F_\sN^3(s)ds \\
 &\leq C_{27} \delta^3 \int_0^t e^{2\Lambda (t-s)} F_\sN(3s)ds.
\end{split}
\end{equation}
Due to \eqref{AssumeLambdaN}, we obtain for $1\leq j\leq \sP$,
\begin{equation}\label{7thInePropNormU^d}
 \int_0^t  e^{(3\lambda_j-2\Lambda)s}ds = \frac1{3\lambda_j-2\Lambda}(e^{(3\lambda_j-2\Lambda)t}-1) \leq \frac1{3\lambda_j-2\Lambda}e^{(3\lambda_j-2\Lambda)t}
\end{equation}
and for $j\geq \sP+1$, 
\begin{equation}\label{8thInePropNormU^d}
 \int_0^t  e^{(3\lambda_j-2\Lambda)s}ds = \frac1{3\lambda_j-2\Lambda} (e^{(3\lambda_j-2\Lambda)t}-1)\leq \frac1{2\Lambda-3\lambda_j}.
\end{equation}
Substituting \eqref{7thInePropNormU^d} and \eqref{8thInePropNormU^d} into \eqref{6thInePropNormU^d}, we observe that  
\begin{equation}\label{13InePropNormU^d}
\begin{split}
&\| u^d(t)\|_{L^2(\Omega)}^2 +\int_0^t \|\nabla u^d(s)\|_{L^2(\Omega)}^2 ds \leq C_{27}  \delta^3 \Big( \sum_{j=j_m}^\sP \frac{|\Csf_j|}{3\lambda_j-2\Lambda } e^{3\lambda_j t} + \sum_{j=\sP+1}^\sN \frac{|\Csf_j|}{2\Lambda-3\lambda_j} e^{2\Lambda t} \Big).
\end{split}
\end{equation}

We then estimate $\|\zeta^d(t)\|_{L^2(\Omega)}$. Due to $\eqref{EqDiff}_1$,  we obtain
\begin{equation}\label{10InePropNormU^d}
\|\zeta^d(t)\|_{L^2(\Omega)}^2 \lesssim  \|\zeta^d(0)\|_{L^2(\Omega)}^2 + \int_0^t (\|u_3^d(s)\|_{L^2(\Omega)}^2 +\|\cQ^1(U^{\delta,\sN})(s)\|_{L^2(\Omega)}^2)ds.
\end{equation}
Using  $U^d(0)=\delta^2 U_\star^{\delta,\sN}$ again and  using \eqref{1stIneLemHorizon} also,  the inequality \eqref{10InePropNormU^d} implies
\begin{equation}\label{11InePropNormU^d}
\begin{split}
\|\zeta^d(t)\|_{L^2(\Omega)}^2 &\lesssim  \delta^4+ \int_0^t (\|u_3^d(s)\|_{L^2(\Omega)}^2 + \|U^{\delta,\sN}(s)\|_{\cE_f}^4 )ds \\
&\lesssim \delta^4 +\int_0^t (\|u_3^d(s)\|_{L^2(\Omega)}^2 + \delta^4 F_\sN^4(s))ds.
\end{split}
\end{equation}
Note that $\delta F_\sN(t) \leq \varepsilon_0 \leq 1$ for any $t\leq T^\delta$. Hence, we have
\[
\delta^4 \int_0^t F_\sN^4(s)ds \lesssim \delta^3 \int_0^t F_\sN^3(s)ds \lesssim \delta^3 \int_0^t F_\sN(3s)ds, 
\]
this yields
\begin{equation}\label{EstDelta^4FM^4}
\delta^4 \int_0^t F_\sN^4(s)ds \lesssim \delta^3 \sum_{j=j_m}^\sN |\Csf_j| e^{3\lambda_j t}. 
\end{equation}
Combining \eqref{13InePropNormU^d} and  \eqref{EstDelta^4FM^4}, we deduce from \eqref{11InePropNormU^d} that
\begin{equation}\label{14InePropNormU^d}
\|\zeta^d(t)\|_{L^2(\Omega)}^2 \leq C_{28} \delta^3 \Big(\sum_{j=j_m}^\sN |\Csf_j|e^{3\lambda_j t}  +\sum_{j=\sP+1}^\sN |\Csf_j| e^{2\Lambda t}\Big).
\end{equation}

To estimate $\|\eta^d(t)\|_{L^2(\Gamma)}$, we  use $\eqref{EqDiff}_4$ and the trace theorem to obtain  
\[\begin{split}
 \frac{d}{dt} \|\eta^d\|_{L^2(\Gamma)}^2 &\leq \|\eta^d\|_{L^2(\Gamma)}(\|u_3^d\|_{L^2(\Gamma)} +\|\cQ^4(U^{\delta,\sN})\|_{L^2(\Gamma)})\\
 &\lesssim  \|\eta^d\|_{L^2(\Gamma)}( \|u_3^d\|_{H^1(\Omega)} +\|\cQ^4(U^{\delta,\sN})\|_{L^2(\Gamma)}).
\end{split}\]
This yields
\[
 \frac{d}{dt} \|\eta^d\|_{L^2(\Gamma)} \lesssim \|u_3^d\|_{H^1(\Omega)} +\|\cQ^4(U^{\delta,\sN})\|_{L^2(\Gamma)}.
\]
Thanks to \eqref{1stIneLemHorizon} and \eqref{PropNormU^delta_M}, we further get
\[
\begin{split}
\|\eta^d(t)\|_{L^2(\Gamma)}^2 &\lesssim \|\eta^d(0)\|_{L^2(\Gamma)}^2 + \int_0^t (\|u_3^d(s)\|_{H^1(\Omega)}^2 + \|U^{\delta,\sN}\|_{\cE_f}^2 )ds \\
&\lesssim  \delta^4 + \int_0^t (\|u_3^d(s)\|_{H^1(\Omega)}^2 + \delta^4 F_\sN^4(s))ds.
\end{split}
\]
Using  \eqref{13InePropNormU^d} and \eqref{EstDelta^4FM^4} again,  we have that  $\|\eta^d(t)\|_{L^2(\Gamma)}^2$ is bounded above like \eqref{14InePropNormU^d}.  Proposition \ref{PropL2_NormU^d} is proven.
\end{proof} 
\subsection{Proof of Theorem \ref{ThmNonlinear}} 
Since $j_m =\min\{j : 1\leq j\leq \sP, \Csf_j\neq 0\}$, we have 
\[
\begin{split}
\| u^\sN(t)\|_{L^2(\Omega)}^2 &=  \sum_{i=j_m}^\sN\Csf_i^2  e^{2\lambda_i t} \| u_i\|_{L^2(\Omega)}^2 +  2 \sum_{j_m\leq i<j\leq \sN} \Csf_i \Csf_j e^{(\lambda_i+\lambda_j)t} \int_{\Omega}  u_i \cdot  u_j\\
&\geq  \sum_{j=j_m}^\sN \Csf_j^2 e^{2\lambda_j t}\| u_j\|_{L^2(\Omega)}^2  + 2\sum_{j_m+1\leq i<j\leq \sN} \Csf_i\Csf_j e^{(\lambda_i+\lambda_j)t}   \int_{\Omega}  u_i \cdot  u_j \\
&\qquad  - |\Csf_{j_m}|\| u_{j_m}\|_{L^2(\Omega)} \Big(\sum_{j=j_m+1}^\sN|\Csf_j|  \| u_j\|_{L^2(\Omega)}\Big) e^{(\lambda_{j_m}+\lambda_{j_m+1})t}.
\end{split}
\]
By Cauchy-Schwarz's inequality, we obtain 
\[
\begin{split}
2\sum_{j_m+1 \leq i<j\leq \sN} \Csf_i\Csf_j e^{(\lambda_i+\lambda_j)t}   \int_{\Omega}  u_i\cdot  u_j &\geq -2 \sum_{j_m\leq i<j\leq \sN} |\Csf_i||\Csf_j| e^{(\lambda_{j_m+1}+\lambda_{j_m+2})t} \| u_i\|_{L^2(\Omega)}\| u_j\|_{L^2(\Omega)}\\
&\qquad \geq - e^{(\lambda_{j_m+1}+\lambda_{j_m+2})t} \Big(\sum_{j=j_m+1}^\sN |\Csf_j|\| u_j\|_{L^2(\Omega)}\Big)^2.
\end{split}
\]
This implies 
\[
\begin{split}
\| u^\sN(t)\|_{L^2(\Omega)}^2 &\geq \sum_{j=j_m}^\sN \Csf_j^2 e^{2\lambda_j t}\| u_j\|_{L^2(\Omega)}^2  -e^{(\lambda_{j_m+1}+\lambda_{j_m+2})t} \Big(\sum_{j=j_m+1}^\sN |\Csf_j| \| u_j\|_{L^2(\Omega)}\Big)^2 \\
&\qquad - |\Csf_{j_m}| e^{(\lambda_{j_m}+\lambda_{j_m+1})t}\| u_{j_m}\|_{L^2(\Omega)} \Big(\sum_{j=j_m+1}^\sN|\Csf_j|  \| u_j\|_{L^2(\Omega)}\Big).
\end{split}
\]
Due to the assumption \eqref{NormalizedCond_2}, we deduce that
\[
\begin{split}
\| u^\sN(t)\|_{L^2(\Omega)}^2 &\geq \sum_{j=j_m}^\sN \Csf_j^2 e^{2\lambda_j t}\| u_j\|_{L^2(\Omega)}^2 -\Csf_{j_m}^2 e^{(\lambda_{j_m}+\lambda_{j_m+1})t} \| u_{j_m}\|_{L^2(\Omega)}^2 \\
&\qquad -  \Csf_{j_m}^2 e^{(\lambda_{j_m+1}+\lambda_{j_m+2})t} \| u_{j_m}\|_{L^2(\Omega)}^2.
\end{split}
\]
This yields
\[
\begin{split}
\| u^\sN(t)\|_{L^2(\Omega)}^2 &\geq \Csf_{j_m}^2\Big(e^{2\lambda_{j_m} t} -\frac12 e^{(\lambda_{j_m}+\lambda_{j_m+1})t}-\frac14 e^{(\lambda_{j_m+1}+\lambda_{j_m+2})t} \Big) \| u_{j_m}\|_{L^2(\Omega)}^2  \\
&\qquad+  \sum_{j=j_m+1}^\sN \Csf_j^2 e^{2\lambda_j t}\| u_j\|_{L^2(\Omega)}^2  .
\end{split}
\]
Notice that for all $t\geq 0$,
\[
e^{2\lambda_{j_m} t}  -\frac12 e^{(\lambda_{j_m}+\lambda_{j_m+1})t}-  \frac14 e^{(\lambda_{j_m+1}+\lambda_{j_m+2})t}\geq \frac14 e^{2\lambda_{j_m} t}.
\]
Hence, for all $t\leq \min\{T^\delta,T^\star, T^{\star\star}\}$, we have 
\begin{equation}\label{L^2NormU_2^M}
\| u^\sN(t)\|_{L^2(\Omega)} \geq C_{29} F_\sN(t).
\end{equation}

Let 
\[
\tilde \Csf(\sN) = \max_{\sP+1\leq j\leq \sN}\frac{|\Csf_j|}{|\Csf_{j_m}|}\geq 0.
\]
Now, we show that 
\begin{equation}\label{T_deltaMin}
T^\delta \leq \min\{T^\star, T^{\star\star}\}
\end{equation}
by choosing 
\begin{equation}\label{ChoiceEps}
\nu_0 <\min\Big( \frac{2C_{19}\delta_0}{C_{22}},  \frac{ C_{20}^2}{C_{21}(1+\sN\tilde\Csf(\sN))^3}, \frac{C_{29}^2}{4C_{21}(1+\sN\tilde\Csf(\sN))^2}  \Big).
\end{equation}
Indeed, if $T^\star< T^\delta$, we have from \eqref{BoundNormU^delta_M} that
\[
\|U^{\delta,\sN}(T^\star)\|_{\cE_f} \leq C_{22} \delta F_\sN(T^\star)  < C_{22} \delta F_\sN(T^\delta )=C_{22}\nu_0 <2C_{19} \delta_0,
\]
which contradicts the definition of $T^\star$ in \eqref{DefTstar}. If $T^{\star\star} <T^\delta$, we obtain from the definition of $C_{20}$ \eqref{DefC_19} and the inequality \eqref{L2_NormU^d} that
\begin{equation}\label{EstBoundSigmaDelta}
\begin{split}
&\|(\zeta^{\delta,\sN},u^{\delta,\sN})(T^\delta)\|_{L^2(\Omega)} +\|\eta^{\delta,\sN}(T^\delta)\|_{L^2(\Gamma)}\\
 &\leq \| (\zeta^d,u^d)(T^\delta)\|_{L^2(\Omega)}+\|\eta^d(T^\delta)\|_{L^2(\Omega)}+\delta (\|(\zeta^\sN, u^\sN)(T^\delta)\|_{L^2(\Omega)}+\|\eta^\sN(T^\delta)\|_{L^2(\Gamma)})\\
&\leq  \sqrt{C_{21}}  \delta^{\frac32} \Big( \sum_{j=j_m}^\sP |\Csf_j| e^{\lambda_j T^{\delta}} + (\sN-\sP)\Big(\max_{\sP+1\leq j\leq \sN}|\Csf_j|\Big) e^{2\Lambda T^{\delta}/3}\Big)^{3/2} +  C_{20} \delta F_\sN(T^\delta).
\end{split}
\end{equation}
Notice from \eqref{AssumeLambdaN} that for $\sP+1\leq j\leq \sN$,
\[
|\Csf_j| \delta e^{\frac23 \Lambda T^{\delta}} < \frac{|\Csf_j|}{|\Csf_{j_m}|}( \delta |\Csf_{j_m}|  e^{\lambda_{j_m} T^{\delta}})<  \frac{|\Csf_j|}{|\Csf_{j_m}|} \delta F_\sN(T^{\delta}) =  \frac{|\Csf_j|}{|\Csf_{j_m}|} \nu_0.
\] 
Then, it follows from \eqref{EstBoundSigmaDelta} that
\[
\begin{split}
\| &(\zeta^{\delta,\sN},  u^{\delta,\sN})(T^{\delta})\|_{L^2(\Omega)} +\|\eta^{\delta,\sN}(T^\delta)\|_{L^2(\Gamma)}\\
 &\leq   C_{20} \delta F_\sN(T^{\delta}) + \sqrt{C_{21}}\delta^{3/2}(1+ \sN\tilde \Csf(\sN))^{3/2} F_\sN^{3/2}(T^\delta) \\
&\leq   C_{20}\nu_0+  \sqrt{C_{21}}(1+ \sN\tilde \Csf(\sN))^{3/2}\nu_0^{3/2}.
\end{split}
\]
Using \eqref{ChoiceEps} again, we deduce 
\[
\|(\zeta^{\delta,\sN},  u^{\delta,\sN})(T^{\delta})\|_{L^2(\Omega)} +\|\eta^{\delta,\sN}(T^\delta)\|_{L^2(\Gamma)}< 2  C_{20} \nu_0 = 2 C_{20} \delta F_\sN(T^{\delta}),
\]
which also contradicts  the definition of $T^{\star\star}$ in \eqref{DefTstar}.  So, \eqref{T_deltaMin} holds.

Once we have \eqref{T_deltaMin}, it follows from \eqref{L2_NormU^d} and \eqref{L^2NormU_2^M} that
\[\begin{split}
\|
& u^{\delta,\sN}(T^{\delta})\|_{L^2(\Omega)} \\
&\geq \delta\|u^\sN(T^\delta)\|_{L^2(\Omega)}- \| u^d(T^\delta)\|_{L^2(\Omega)} \\
&\geq C_{29} \delta F_\sN(T^{\delta})- \sqrt{C_{21}}  \delta^{3/2}\Big( \sum_{j=j_m}^\sP |\Csf_j| e^{\lambda_j T^{\delta}} + (\sN-\sP)\Big(\max_{\sP+1\leq j\leq \sN}|\Csf_j|\Big) e^{2 \Lambda T^{\delta}/3}\Big)^{3/2}.
\end{split}\]
Thanks to \eqref{ChoiceEps} again, we conclude that 
\[
\| u^{\delta,\sN}(T^{\delta})\|_{L^2(\Omega)} \geq C_{29}\nu_0 - \sqrt{C_{21}}(1+ \sN\tilde \Csf(\sN))^{3/2}\nu_0^{3/2} \geq \frac12C_{29}\nu_0 >0.
\]
Theorem \ref{ThmNonlinear} follows by taking $\delta_0$ satisfying  Propositions  \ref{PropAprioriEnergy}, \ref{PropModifyData}, $\nu_0$ satisfying \eqref{ChoiceEps} and $m_0=\frac12 C_{29}$. We complete the proof.

\appendix
\section{Poisson extension}\label{AppPoisson}

We  define the appropriate Poisson sum that allows us to extend $\eta$ defined on $\Gamma$ to a function $\theta$ defined on $\Omega$,
\[
(\cP f)(x_h,x_3) =\sum_{\vk \in (L^{-1}\bZ)^2} \frac{e^{i\vk\cdot x_h}}{2\pi L}e^{|\vk|x_3}\hat f(\vk), \quad\text{where } \hat f(\vk) =\int_{\fT^2} f(x_h)\frac{e^{-i\vk\cdot x_h}}{2\pi L} dx_h.
\]
We then have $\cP : H^s(\Gamma)\to H^{s+1/2}(\Omega)$ is a bounded linear operator for $s>0$. We extend $\eta$ defined on $\Gamma$ to be a function defined on $\Omega$, 
\[
\theta(t,x) := (\cP\eta)(t,x_h,x_3)
\]
for all $x_h \in \fT^2, x_3\leq 0$.  Lemma \ref{LemEstNablaQ_Pf} below implies in particular that if $\eta \in H^{q-1/2}(\Gamma)$, then $\theta \in H^q(\Omega)$ for $q\geq 0$. 
\begin{lemma}\label{LemEstNablaQ_Pf}
For $q\in \N$, let $H_h^q$ be the usual homogeneous Sobolev space of order $q$ and $\cP f$ be the Poisson sum of a function $f$ in $H_h^{q-1/2}(\Gamma)$. There holds 
\begin{equation}\label{EstNablaQ_Pf}
\|\nabla^q \cP f\|_{L^2(\Omega)}^2 \lesssim \|f\|_{H_h^{q-1/2}(\Gamma)}^2.
\end{equation}
\end{lemma}
\begin{proof}
Thanks to Fubini's theorem and Parseval's formula, we obtain 
\[
\begin{split}
\|\nabla^q \cP f\|_{L^2(\Omega)}^2 &\lesssim \sum_{\vk \in (L^{-1}\bZ)^2} \int_{-\infty}^0 |\vk|^{2q}|\hat f(\vk)|^2 e^{2|\vk|x_3} \lesssim \sum_{\vk \in (L^{-1}\bZ)^2} |\vk|^{2q-1}|\hat f(\vk)|^2.
\end{split}
\]
The inequality \eqref{EstNablaQ_Pf} then follows.
\end{proof}

\section{Nonlinear terms}\label{AppNonlinearTerms}

The nonlinear terms $\cQ^i (1\leq i\leq 5)$ in \eqref{EqPertur} are given by
\begin{equation}\label{TermQ}
\begin{split}
\cQ^1&= \rho_0'u_3 -\rho_0' \partial_t\theta+ K \partial_t\theta( \partial_3\zeta+ \rho_0'+ \rho_0''\theta+\rho_0'\partial_3\theta) - \text{div}_{\cA}((\rho_0+\rho_0'\theta+ \zeta) u) \\
\cQ^2 &= -(\zeta +\rho_0'\theta) \partial_t u- (\zeta+\rho_0+ \rho_0'\theta )K \partial_t\theta \partial_3 u -(\nabla_{\cA}p-\nabla q-g\zeta e_3) \\
&\qquad - (\zeta+\rho_0+ \rho_0'\theta )  u\cdot \nabla_{\cA} u - \mu (\Delta  u-\text{div}_{\cA}(\bS_{\cA} u)),\\
\cQ^3 &= \text{div}  u-\text{div}_{\cA}u, \\
\cQ^4 &= -u_1 \partial_1\eta -u_2 \partial_2\eta, \\
\cQ^5 &=  (q-g\rho_+\eta)\text{Id}\cdot ( e_3-\cN) -\mu \bS u  e_3 + \mu (\bS_{\cA} u) \cN.
\end{split}
\end{equation}
Alternatively, we have
\begin{equation}\label{TermQ1}
\begin{split}
\cQ^1  &= - K\rho_0'' \theta  u_3 + K\partial_t\theta(\partial_3\zeta+ \rho_0''\theta) -K\rho_0'\theta(Au_1+Bu_2) \\
&\qquad - u_1\partial_1\zeta -u_2 \partial_2\zeta -Ku_3\partial_3\zeta + K \partial_3\zeta(A u_1+ B u_2)
\end{split}
\end{equation}
that
\begin{equation}\label{TermQ2_1}
\begin{split}
\cQ_1^2 &=  -(\zeta +\rho_0'\theta) \partial_t u_1 -  (\rho_0+ \rho_0'\theta +\zeta)K \partial_t\theta \partial_3u_1 + AK(\partial_3q-g\rho_0'\theta) \\
&\quad- (\zeta+\rho_0+ \rho_0'\theta ) \Big( u_1(\partial_1u_1 -AK \partial_3u_1) + u_2(\partial_2u_1 -BK \partial_3u_1)+Ku_3 \partial_3u_1   \Big)\\
&\quad+ \mu \left(\begin{split}
 &(K^2+A^2+B^2-1)\partial_{33}^2 u_1 - 2AK \partial_{13}^2 u_1-2BK \partial_{23}^2 u_1 \\
 &(K\partial_3K (A^2+B^2+1) -\partial_1(AK)-\partial_2(BK)-A\partial_1K-B\partial_2K)\partial_3 u_1
 \end{split}\right),
\end{split}
\end{equation}
\begin{equation}\label{TermQ2_2}
\begin{split}
\cQ_2^2 &=  -(\zeta +\rho_0'\theta) \partial_t u_2 -  (\rho_0+ \rho_0'\theta+\zeta )K \partial_t\theta \partial_3u_1 + BK(\partial_3q-g\rho_0'\theta) \\
&\quad - (\zeta+\rho_0+ \rho_0'\theta ) \Big( u_1(\partial_1u_2 -AK \partial_3u_2) + u_2(\partial_2u_2 -BK \partial_3u_2)+Ku_3 \partial_3u_2   \Big)\\
&\quad+ \mu \left(\begin{split}
&(K^2+A^2+B^2-1)\partial_{33}^2 u_2    - 2AK \partial_{13}^2 u_2-2BK \partial_{23}^2 u_2 \\ 
&(K\partial_3K (A^2+B^2+1) -\partial_1(AK)-\partial_2(BK)-A\partial_1K-B\partial_2K)\partial_3 u_2
 \end{split}\right),
\end{split}
\end{equation}
\begin{equation}\label{TermQ2_3}
\begin{split}
\cQ_3^2&=-(\zeta+\rho_0'\theta)  \partial_t u_3 -  (\rho_0+ \rho_0'\theta+\zeta )K \partial_t\theta \partial_3u_3 + (1-K)(\partial_3q-g\rho_0'\theta)\\
&\quad - (\zeta+\rho_0+ \rho_0'\theta ) \Big( u_1(\partial_1u_3 -AK \partial_3u_3) + u_2(\partial_2u_3 -BK \partial_3u_3)+K u_3 \partial_3u_3   \Big)\\
&\quad + (K-1)(\partial_{13}^2 u_1+\partial_{23}^2u_2) +\partial_1K \partial_3u_1 + \partial_2K \partial_3u_2  -\partial_1(AK \partial_3u_3) \\
&\quad- AK \partial_3(K\partial_3u_1+\partial_1u_3- AK \partial_3u_3) - \partial_2(BK \partial_3u_3) \\
&\quad- BK \partial_3(K\partial_3u_2+\partial_2u_3- BK \partial_3u_3) +2(K^2-1) \partial_{33}^2 u_3 +2K \partial_3K \partial_3u_3,
\end{split}
\end{equation}
that
\begin{equation}\label{TermQ34}
\begin{split}
\cQ^3  &= (1 -K)\partial_3u_3 +AK \partial_3u_1+BK \partial_3u_2, \\
 \cQ^4 &= -u_1 \partial_1\eta-u_2 \partial_2\eta,
\end{split}
\end{equation}
and that
\begin{equation}\label{TermQ5}
\begin{split}
\cQ_1^5 &=\partial_1\eta \big(q -g\rho_+\eta -2\mu(\partial_1 u_1-AK \partial_3 u_1)\big)- \mu \partial_2\eta (\partial_1u_2 +\partial_2u_1 -AK \partial_3u_2 -BK \partial_3u_1) \\
&\qquad -\mu\big((1-K)\partial_3u_1+ AK \partial_3u_3\big),\\
\cQ_2^5 &=-\mu \partial_1\eta (\partial_1u_2 +\partial_2u_1 -AK \partial_3u_2 -BK \partial_3u_1) + \partial_2\eta \big(q -g\rho_+\eta -2\mu(\partial_2u_2 -BK \partial_3u_2)\big)\\
&\qquad -\mu\big((1-K)\partial_3 u_2 +BK \partial_3 u_3\big),\\
\cQ_3^5 &= -\mu\partial_1\eta (\partial_1 u_3-AK \partial_3u_3+K\partial_3 u_1) -\mu \partial_2\eta ( \partial_2u_3-AK \partial_3u_3+K\partial_3 u_2)- 2\mu(1-K)\partial_3u_3.
\end{split}
 \end{equation}
The terms $F^{j,l} (1\leq j\leq 5)$ in \eqref{EqPerturGeometric} are given by 
\begin{equation}\label{TermF1}
F^{1,l} = \partial_t^l F^1 - \sum_{0<j\leq l}C_l^j \partial_t^j \cA_{jk}\partial_k(\rho_0 \partial_t^{l-j}u_j),
\end{equation}
\begin{equation}\label{TermF2}
\begin{split}
F_i^{2,l} &= \partial_t^l F^2+ \sum_{0<j\leq l}C_l^j \mu(\cA_{jk}\partial_k(\partial_t^j\cA_{jm}\partial_t^{l-j}\partial_m u_i)+ \partial_t^j\cA_{jk}\partial_t^{l-j}\partial_k(\cA_{jm}\partial_m u_i)) \\
&\qquad -  \sum_{0<j\leq l}C_l^j (\rho_0 \partial_t^j\cA_{ik}\partial_k\partial_t^{l-j}\zeta +\partial_t^j(\zeta+\rho_0'\theta)\partial_t(\partial_t^{l-j}u_i)), 
\end{split}
\end{equation}
\begin{equation}\label{TermF34}
\begin{split}
F^{3,l} &= -\sum_{0<j\leq l} C_l^j \partial_t^j \cA_{ik}\partial_k(\partial_t^{l-j} u_i), \quad F^{4,l} = \sum_{0<j\leq l} C_l^j \partial_t^j \cN \cdot \partial_t^{l-j}u,\\
\end{split}
\end{equation}
\begin{equation}\label{TermF5}
\begin{split}
F_i^{5,l} &=  \mu \sum_{0<j\leq l} C_l^j( \partial_t^j(\cA_{ik}\cN_m) \partial_k\partial_t^{l-j}u_m + \partial_t^j (\cA_{mk}\cN_m)\partial_k \partial_t^{l-j}u_i)\\
&\qquad+\sum_{0<j\leq l} C_l^j \partial_t^j\cN_i \partial_t^{l-j} ( g\rho_+\eta -q ).
\end{split}
\end{equation}

\section{Some useful estimates}\label{AppUsefulEst}

\noindent{\textbf{Product estimate}}. 
Suppose that $\Sigma=\Omega$ or $\Gamma$, let $f\in H^{s_1}(\Sigma) , g\in H^{s_2}(\Sigma)$, 
\begin{enumerate}
\item if $0\leq r\leq s_1\leq s_2$ and $s_2>r+3/2$, then  $fg\in H^r(\Sigma)$,
\item if $0\leq r\leq s_1\leq s_2$ and $s_1>3/2$, then  $fg\in H^r(\Sigma)$.
\end{enumerate}
In both cases, we have 
\begin{equation}\label{ProductEst}
\|fg\|_{H^r(\Sigma)} \lesssim \|f\|_{H^{s_1}(\Sigma)}\|g\|_{H^{s_2}(\Sigma)},
\end{equation}
We refer to \cite[Lemma 10.1]{GT13} for the proof of \eqref{ProductEst}.

\noindent{\textbf{Gagliardo-Nirenberg's inequality}}. Let $s\geq 0$,  $\Sigma=\Omega$ or $\Gamma$ and $f,g \in H^s(\Sigma) \cap L^\infty(\Sigma)$, we have 
\begin{equation}\label{GN_ine}
\|fg\|_{H^s(\Sigma)} \lesssim \|f\|_{H^s(\Sigma)}\|g\|_{L^\infty(\Sigma)} +\|f\|_{L^\infty(\Sigma)}\|g\|_{H^s(\Sigma)}.
\end{equation}

\noindent{\textbf{Elliptic estimates}}. Let $r\geq 2$ and $\phi \in H^{r-2}(\Omega), \psi \in H^{r-1}(\Omega)$ and $\alpha \in H^{r-3/2}(\Gamma)$. There exist unique $u\in H^r(\Omega)$ and $q\in H^{r-1}(\Omega)$ solving
\[\begin{cases}
-\Delta u+\nabla q=\phi &\quad\text{in } \Omega,\\
\text{div}u=\psi &\quad\text{in } \Omega,\\
(q\text{Id}-\mu \bS u)e_3 =\alpha &\quad\text{on } \Gamma.
\end{cases}\]
Moreover, we have
\begin{equation}\label{EllipticEst}
\|u\|_{H^r(\Omega)}^2+\|q\|_{H^{r-1}(\Omega)}^2 \lesssim \|\phi\|_{H^{r-2}(\Omega)}^2 +\|\psi\|_{H^{r-1}(\Omega)}^2 +\|\alpha\|_{H^{r-3/2}(\Gamma)}^2.
\end{equation}
thanks to \cite[Lemma A.15]{GT13} for example.

We also recall the classical regularity theory for the Stokes problem with Dirichlet boundary conditions (see \cite[Theorem 2.4]{Tem84} after using the domain expansion technique). Let $r\geq 2$ and $f\in H^{r-2}(\Omega), g\in H^{r-1}(\Omega)$ and $h\in H^{r-1/2}(\Gamma)$ such that 
\[
\int_\Omega g = \int_\Gamma h \cdot \nu, \text{ where }\nu \text{ is the outward unit normal vector to the boundary}.
\]
There exist uniquely $u \in H^r(\Omega)$ and $q\in H^{r-1}(\Omega)$ solving
\[\begin{cases}
-\Delta u+\nabla q= f &\quad\text{in } \Omega,\\
\text{div}u=g &\quad\text{in } \Omega,\\
u=h &\quad\text{on } \Gamma.
\end{cases}\]
There also holds
\begin{equation}\label{EstElliptic}
\|u\|_{H^r(\Omega)}^2+ \|q\|_{H^{r-1}(\Omega)}^2 \lesssim \|f\|_{H^{r-2}(\Omega)}^2+\|g\|_{H^{r-1}(\Omega)}^2+ \|h\|_{H^{r-1/2}(\Gamma)}^2.
\end{equation}

\noindent{\textbf{Korn's inequality}}.  The following Korn's inequality is proven in \cite[Theorem 5.12]{KO88}, 
\begin{equation}\label{KornIne}
\|\nabla u\|_{L^2(\Omega)}^2 \lesssim \|\bS u\|_{L^2(\Omega)}^2.
\end{equation}

\noindent{\textbf{Commutator estimates}}. Let $\cJ =\sqrt{1-\partial_1^2-\partial_2^2}$ and let us define the commutator 
\[
[\cJ^s,f]g = \cJ^s(fg)-f\cJ^sg.
\]
We have 
\begin{equation}\label{EstCommutator}
\|[\cJ^s,f]g\|_{L^2(\Gamma)} \lesssim \|\nabla f\|_{L^\infty(\Gamma)}\|\cJ^{s-1}g\|_{L^2(\Gamma)}+\|\cJ^sf\|_{L^2(\Gamma)}\|g\|_{L^\infty(\Gamma)}. 
\end{equation}
The proof of \eqref{EstCommutator} is similar to that one of \cite[Lemma X1]{KP88}.

\noindent{\textbf{Interpolation inequality}}. It can be found in \cite[Chapter 5]{AF05} that 
\[
\|u\|_{H^j(\Omega)} \lesssim \|u\|_{L^2(\Omega)}^{1/(j+1)} \|u\|_{H^{j+1}(\Omega)}^{j/(j+1)}.
\]
That implies for  $\varepsilon>0$, there is a universal constant $C(j)$ such that
\begin{equation}\label{EstJensen}
\|u\|_{H^j(\Omega)}\leq \varepsilon \| u\|_{H^{j+1}(\Omega)} + C(j) \varepsilon^{-j} \|u\|_{L^2(\Omega)}.
\end{equation}

\noindent{\textbf{Coefficient estimates}}. If $\|\eta\|_{H^{5/2}(\Gamma)}\lesssim 1$, we have
\begin{equation}\label{CoefEstimates}
\begin{split}
&\|J-1\|_{L^\infty(\Omega)} +\|\cN-1\|_{L^\infty(\Gamma)}+ \|K-1\|_{L^\infty(\Gamma)}\lesssim \|\eta\|_{H^{5/2}(\Gamma)}.
\end{split}
\end{equation}
Also, the map $\Theta$ defined by \eqref{ThetaTransform} is a diffeomorphism. We refer to \cite[Lemma 2.4]{GT13_1} for the proof of \eqref{CoefEstimates}.  In the following lemmas, we provide some additional estimates.
\begin{lemma}\label{LemEstEta_H^1/2}
We have
\begin{equation}\label{EstEta_H^1/2}
\|\partial_t \eta\|_{H^{7/2}(\Gamma)}\lesssim \cE_f + \cE_f^2,
\end{equation}
\begin{equation}\label{EstEta_H^3/2}
\|\partial_t^2\eta\|_{H^{3/2}(\Gamma)}\lesssim \cE_f + \cE_f^2,
\end{equation}
and 
\begin{equation}\label{EstEtaD_t^3}
\|\partial_t^3\eta\|_{H^{1/2}(\Gamma)}\lesssim   \|\partial_t^2 u\|_{H^1(\Omega)} (1+ \cE_f)+\cE_f^2.
\end{equation}
\end{lemma}
\begin{proof}
By $\eqref{EqPertur}_4$, we have that
\begin{equation}\label{1_EstEta_H^1/2}
\|\partial_t\eta\|_{H^{7/2}(\Gamma)} \lesssim \|u_3\|_{H^{7/2}(\Gamma)} +\|\cQ^4\|_{H^{7/2}(\Gamma)} \lesssim \|u_3\|_{H^4(\Omega)} +\|\cQ^4\|_{H^{7/2}(\Gamma)}.
\end{equation}
We use \eqref{ProductEst} and  the trace theorem to estimate $\|\cQ^4\|_{H^{7/2}(\Gamma)}$ (see $\cQ^4$ in \eqref{TermQ}) as
\[
\|\cQ^4\|_{H^{7/2}(\Gamma)} \lesssim \|u\|_{H^{7/2}(\Gamma)}\|(\partial_1 \eta,\partial_2\eta)\|_{H^{7/2}(\Gamma)} \lesssim \|u\|_{H^4(\Omega)}\|\eta\|_{H^{9/2}(\Gamma)},
\]
Substituting the resulting inequality into \eqref{1_EstEta_H^1/2}, we have \eqref{EstEta_H^1/2}.

Using  \eqref{ProductEst} again,  we have 
\[
\begin{split}
\|\partial_t \cQ^4\|_{H^{3/2}(\Gamma)} &\lesssim \|\partial_t \partial_h \eta\|_{H^{3/2}(\Gamma)}\|u\|_{H^{7/2}(\Gamma)} + \|\partial_t u\|_{H^{3/2}(\Gamma)}\|\partial_h \eta\|_{H^{7/2}(\Gamma)}\\
& \lesssim \|\partial_t \eta\|_{H^{5/2}(\Gamma)} \|u\|_{H^4(\Omega)} + \|\eta\|_{H^{9/2}(\Gamma)}\|\partial_t u\|_{H^2(\Omega)}.
\end{split}
\]
Together with \eqref{EstEta_H^1/2}, that implies 
\[
\|\partial_t^2\eta\|_{H^{3/2}(\Gamma)}\lesssim \|\partial_t u_3\|_{H^2(\Omega)} +\|\partial_t \cQ^4\|_{H^{3/2}(\Gamma)} \lesssim \cE_f+\cE_f^2.
\]
One thus has \eqref{EstEta_H^3/2}.

We continue using  $\eqref{EqPertur}_4$ to have that 
\begin{equation}\label{2_EstEta_H^1/2}
\|\partial_t^3\eta\|_{H^{1/2}(\Gamma)} \lesssim \|\partial_t^2 u_3\|_{H^{1/2}(\Gamma)}+ \|\partial_t^2 \cQ^4\|_{H^{1/2}(\Gamma)} \lesssim  \|\partial_t^2 u_3\|_{H^1(\Omega)} +\|\partial_t^2 \cQ^4\|_{H^{1/2}(\Gamma)}.
\end{equation}
As a consequence of the product estimate \eqref{ProductEst} and Sobolev embedding, we obtain 
\[\begin{split}
\|\partial_t^2\cQ^4\|_{H^{1/2}(\Gamma)} &\lesssim \|\partial_t^2 u\|_{H^{1/2}(\Gamma)}\|(\partial_1\eta,\partial_2\eta)\|_{H^{5/2}(\Gamma)} + \|\partial_t u\|_{H^{1/2}(\Gamma)} \|\partial_t(\partial_1\eta,\partial_2\eta)\|_{H^{5/2}(\Gamma)} \\
&\qquad\quad+\|u\|_{H^{5/2}(\Gamma)}\|\partial_t^2(\partial_1\eta,\partial_2\eta)\|_{H^{1/2}(\Gamma)} \\
&\lesssim \|\partial_t^2 u\|_{H^1(\Omega)}\|\eta\|_{H^{7/2}(\Gamma)} + \|\partial_t u\|_{H^1(\Omega)} \|\partial_t\eta\|_{H^{7/2}(\Gamma)}+ \| u\|_{H^3(\Omega)} \|\partial_t^2 \eta\|_{H^{3/2}(\Gamma)}.
\end{split}\]
We continue using \eqref{EstEta_H^1/2} and \eqref{EstEta_H^3/2} to observe
\begin{equation}\label{3_EstEta_H^1/2}
\|\partial_t^2\cQ^4\|_{H^{1/2}(\Gamma)} \lesssim \|\partial_t^2 u\|_{H^1(\Omega)}\cE_f+\cE_f^2. 
\end{equation}
The inequality \eqref{EstEtaD_t^3} follows from \eqref{2_EstEta_H^1/2} and \eqref{3_EstEta_H^1/2}. Lemma \ref{LemEstEta_H^1/2} is proven.
\end{proof}

\begin{lemma}\label{LemCoefEst}
Under the assumption $\|\eta\|_{H^{9/2}(\Gamma)} \lesssim 1$, the following inequalities hold 
\begin{equation}\label{CoefEst_21}
\|\partial_t^l (A,B)\|_{H^s(\Omega)} \lesssim\|\partial_t^l \eta\|_{H^{s+1/2}(\Gamma)} \quad\text{for any }0\leq l\leq 2 \text{ and } 0\leq s\leq 4,
\end{equation}
and 
\begin{equation}\label{CoefEst_22}
\begin{cases}
\|K-1\|_{H^s(\Omega)} \lesssim \| \eta\|_{H^{s+1/2}(\Gamma)} \quad\text{for }0\leq s\leq 4, \\
\|\partial_t K \|_{H^s(\Omega)} \lesssim \|\partial_t\eta\|_{H^{s+1/2}(\Gamma)} \quad\text{for } 0\leq s\leq 2,\\
\|\partial_t^2 K\|_{L^2(\Omega)} \lesssim   \|\partial_t^2\eta\|_{H^{1/2}(\Gamma)}+ \|\partial_t\eta\|_{H^{5/2}(\Gamma)}^2, \\
\|\partial_t^3 K\|_{L^2(\Omega)} \lesssim \|\partial_t^3 \eta\|_{H^{1/2}(\Gamma)} +\|\partial_t\eta\|_{H^{5/2}(\Gamma)}\|\partial_t^2\eta\|_{H^{1/2}(\Gamma)} + \|\partial_t\eta\|_{H^{5/2}(\Gamma)}^3,
\end{cases}
\end{equation}
and
\begin{equation}\label{CoefEst_23}
\begin{cases}
\|(AK, BK)\|_{H^s(\Omega)} \lesssim  \|\eta\|_{H^{s+1/2}(\Gamma)}\quad\text{for }0\leq s\leq 4,\\
\|\partial_t(AK,BK)\|_{H^s(\Omega)} \lesssim \|\partial_t\eta\|_{H^{s+1/2}(\Gamma)}\quad\text{for } 0\leq s\leq 2,\\
\|\partial_t^2(AK,BK)\|_{L^2(\Omega)} \lesssim \|\partial_t^2\eta\|_{H^{1/2}(\Gamma)}+ \|\partial_t\eta\|_{H^{5/2}(\Gamma)}^2,\\
\|\partial_t^3 (AK,BK)\|_{L^2(\Omega)} \lesssim \|\partial_t^3 \eta\|_{H^{1/2}(\Gamma)} +\|\partial_t\eta\|_{H^{5/2}(\Gamma)}\|\partial_t^2\eta\|_{H^{1/2}(\Gamma)} + \|\partial_t\eta\|_{H^{5/2}(\Gamma)}^3,
\end{cases}
\end{equation}
and 
\begin{equation}\label{CoefEst_24}
\begin{cases}
 \|\cA-\text{Id}\|_{H^s(\Omega)}\lesssim  \|\eta\|_{H^{s+1/2}(\Gamma)}\quad\text{for }0\leq s\leq 4,\\
\|\partial_t \cA\|_{H^s(\Omega)} \lesssim \|\partial_t\eta\|_{H^{s+1/2}(\Gamma)}\quad\text{for } 0\leq s\leq 2,\\
\|\partial_t^2\cA\|_{L^2(\Omega)} \lesssim \|\partial_t^2\eta\|_{H^{1/2}(\Gamma)}+ \|\partial_t\eta\|_{H^{5/2}(\Gamma)}^2,\\
\|\partial_t^3\cA\|_{L^2(\Omega)} \lesssim \|\partial_t^3 \eta\|_{H^{1/2}(\Gamma)} +\|\partial_t\eta\|_{H^{5/2}(\Gamma)}\|\partial_t^2\eta\|_{H^{1/2}(\Gamma)} + \|\partial_t\eta\|_{H^{5/2}(\Gamma)}^3.
 \end{cases}
\end{equation}
\end{lemma}
\begin{proof}
 To prove \eqref{CoefEst_21}, we use Lemma \ref{LemEstNablaQ_Pf} to obtain
\[
\|\partial_t^l (A,B)\|_{H^s(\Omega)} = \|\partial_t^l ( \partial_1 \theta, \partial_2 \theta)\|_{H^s(\Omega)} \lesssim \|\partial_t^l \theta\|_{H^{s+1}(\Omega)} \lesssim \|\partial_t^l \eta\|_{H^{s+1/2}(\Gamma)}.
\]

We then claim \eqref{CoefEst_22}. Since $K-1= J^{-1}(1-J)=-J^{-1}\partial_3\theta$,  we have
\[
\|K-1\|_{H^s(\Omega)} \lesssim \|J^{-1}\partial_3\theta\|_{H^s(\Omega)} \lesssim \|\theta\|_{H^{s+1}(\Omega)}\lesssim \|\eta\|_{H^{s+1/2}(\Gamma)}. 
\]
Note that $K=J^{-1}$, we have $\partial_t K=-J^{-2}\partial_t J$. Owing to the product estimate \eqref{ProductEst},  Lemma \ref{LemEstNablaQ_Pf} and the fact that $\|J-1\|_{L^\infty(\Omega)}\lesssim 1$ \eqref{CoefEstimates}, we get
\[
\begin{split}
\|\partial_t K\|_{H^s(\Omega)} &\lesssim  \|J^{-2}  \partial_t  \partial_3\theta \|_{H^s(\Omega)}  \lesssim \|\partial_t\partial_3 \theta\|_{H^s(\Omega)} \lesssim \|\partial_t\eta\|_{H^{s+1/2}(\Gamma)}.
\end{split}
\]
Since $\partial_t^2 K= -J^{-2}\partial_t^2 J +2J^{-3}(\partial_t J)^2$, we continue applying Sobolev embedding,  Lemma \ref{LemEstNablaQ_Pf} and  \eqref{CoefEstimates} to obtain
\[
\begin{split}
\|\partial_t^2K\|_{L^2(\Omega)} &\lesssim \|J^{-2}\partial_t^2\partial_3\theta\|_{L^2(\Omega)} + \|J^{-3}(\partial_t\partial_3\theta )^2\|_{L^2(\Omega)} \\
&\lesssim \|\partial_t^2\partial_3\theta\|_{L^2(\Omega)} (1+\|\partial_3\theta\|_{H^2(\Omega)}) + \|\partial_t\partial_3\theta\|_{L^2(\Omega)} \|\partial_t\partial_3\theta\|_{H^2(\Omega)}\\
&\lesssim \|\partial_t^2\eta\|_{H^{1/2}(\Gamma)} (1+\|\eta\|_{H^{5/2}(\Gamma)})+ \|\partial_t\eta\|_{H^{1/2}(\Gamma)}\|\partial_t\eta\|_{H^{5/2}(\Gamma)}\\
&\lesssim \|\partial_t^2\eta\|_{H^{1/2}(\Gamma)}+ \|\partial_t\eta\|_{H^{5/2}(\Gamma)}^2.
\end{split}
\]
Similarly, we have 
\[
\partial_t^3 K= -J^{-2}\partial_t^3 J +6J^{-3}\partial_t J\partial_t^2 J-6J^{-4}(\partial_t J)^3.
\]
 This yields
\[
\begin{split}
\|\partial_t^3K\|_{L^2(\Omega)} &\lesssim  \|J^{-2}\partial_t^3\partial_3\theta\|_{L^2(\Omega)} + \|J^{-3} \partial_t\partial_3\theta   \partial_t^2\partial_3\theta\|_{L^2(\Omega)} + \|J^{-4} (\partial_t \partial_3\theta)^3\|_{L^2(\Omega)}\\
&\lesssim \|\partial_t^3\partial_3 \theta\|_{L^2(\Omega)} + \|\partial_t\partial_3\theta\|_{H^2(\Omega)}\|\partial_t^2 \partial_3\theta\|_{L^2(\Omega)}+  \|\partial_t\partial_3\theta\|_{H^2(\Omega)}^2  \|\partial_t\partial_3\theta\|_{L^2(\Omega)} \\
&\lesssim \|\partial_t^3 \eta\|_{H^{1/2}(\Gamma)} +\|\partial_t\eta\|_{H^{5/2}(\Gamma)}\|\partial_t^2\eta\|_{H^{1/2}(\Gamma)} + \|\partial_t\eta\|_{H^{5/2}(\Gamma)}^3.
\end{split}
\]
Hence,  \eqref{CoefEst_22} is proven.

We combine \eqref{CoefEst_21} and \eqref{CoefEst_22} to prove \eqref{CoefEst_23}. Note that $XK=X(K-1)+X$ for $X=A$ or $B$, we use Sobolev embedding and \eqref{CoefEst_22} to obtain that
\[
\begin{split}
\| XK\|_{L^2(\Omega)} &\lesssim \|X\|_{L^2(\Omega)}(1+\|K-1\|_{H^2(\Omega)})\lesssim \|\eta\|_{H^{1/2}(\Gamma)}.
\end{split}
\]
We make use \eqref{ProductEst} and \eqref{CoefEst_21}, \eqref{CoefEst_22} to obtain 
\[
\begin{split}
\|XK\|_{H^1(\Omega)} \lesssim \|X\|_{H^1(\Omega)}(1+\|K-1\|_{H^3(\Omega)}) \lesssim \|\eta\|_{H^{3/2}(\Gamma)}
\end{split}
\]
and  if $s=2,3$ or 4, we use also Gagliardo-Nirenberg's inequality to have
\[
\begin{split}
\|XK\|_{H^s(\Omega)} &\lesssim \|X\|_{H^s(\Omega)}(1+\|K-1\|_{H^2(\Omega)}) + \|X\|_{H^2(\Omega)}(1+\|K-1\|_{H^s(\Omega)}) \\
&\lesssim \|\eta\|_{H^{s+1/2}(\Gamma)}.
\end{split}
\]
We further obtain
\[
\begin{split}
\|\partial_t (XK)\|_{H^s(\Omega)} &\lesssim \|\partial_t X\|_{H^s(\Omega)}+ \|\partial_t X(K-1)\|_{H^s(\Omega)}+ \|X\partial_t(K-1)\|_{H^s(\Omega)}.
\end{split}
\]
If $s=0$, we use Sobolev embedding and \eqref{CoefEst_21}, \eqref{CoefEst_22} again to have
\[
\begin{split}
\|\partial_t(XK)\|_{L^2(\Omega)} &\lesssim \|\partial_tX\|_{L^2(\Omega)}(1+\|K-1\|_{H^2(\Omega)})+ \|X\|_{H^2(\Omega)} \|\partial_tK\|_{L^2(\Omega)} \\
&\lesssim \|\partial_t\eta\|_{H^{1/2}(\Gamma)}(1+\|\eta\|_{H^{5/2}(\Gamma)}).
\end{split}
\]
If $s=1$ or 2, we use \eqref{ProductEst} and also \eqref{CoefEst_21}, \eqref{CoefEst_22}  to obtain
\[
\begin{split}
\|\partial_t(XK)\|_{H^1(\Omega)} &\lesssim \|\partial_tX\|_{H^1(\Omega)}(1+\|K-1\|_{H^3(\Omega)})+ \|X\|_{H^3(\Omega)} \|\partial_tK\|_{H^1(\Omega)} \\
&\lesssim \|\partial_t\eta\|_{H^{3/2}(\Gamma)}(1+\|\eta\|_{H^{7/2}(\Gamma)}).
\end{split}
\]
or 
\[
\begin{split}
\|\partial_t(XK)\|_{H^2(\Omega)} &\lesssim \|\partial_tX\|_{H^2(\Omega)}(1+\|K-1\|_{H^2(\Omega)})+ \|X\|_{H^2(\Omega)} \|\partial_tK\|_{H^2(\Omega)} \\
&\lesssim \|\partial_t\eta\|_{H^{5/2}(\Gamma)}(1+\|\eta\|_{H^{5/2}(\Gamma)}).
\end{split}
\]
Similarly, it can be seen that
\[
\begin{split}
\|\partial_t^2(XK) \|_{L^2(\Omega)} &\lesssim \|\partial_t^2X\|_{L^2(\Omega)} + \|\partial_t^2X (K-1)\|_{L^2(\Omega)}+ \|\partial_t X\partial_t (K-1)\|_{L^2(\Omega)} \\
&\qquad\qquad+ \|X\partial_t^2(K-1)\|_{L^2(\Omega)}\\
&\lesssim \|\partial_t^2X\|_{L^2(\Omega)}( 1+\|K-1\|_{H^2(\Omega)}) + \|\partial_tX\|_{L^2(\Omega)}\|\partial_t K\|_{H^2(\Omega)} \\
&\qquad\qquad+ \|X\|_{H^2(\Omega)}\|\partial_t^2 K\|_{L^2(\Omega)}\\
&\lesssim \|\partial_t^2\eta\|_{H^{1/2}(\Gamma)}(1+\|\eta\|_{H^{5/2}(\Gamma)}) + \|\partial_t\eta\|_{H^{1/2}(\Gamma)} \|\partial_t\eta\|_{H^{5/2}(\Gamma)} \\
&\qquad\qquad+ \|\eta\|_{H^{5/2}(\Gamma)} (\|\partial_t^2\eta\|_{H^{1/2}(\Gamma)}+ \|\partial_t\eta\|_{H^{5/2}(\Gamma)}^2)\\
&\lesssim \|\partial_t^2\eta\|_{H^{1/2}(\Gamma)}+ \|\partial_t\eta\|_{H^{5/2}(\Gamma)}^2.
\end{split}
\]
In a same way, we have
\[
\begin{split}
\|\partial_t^3(XK) \|_{L^2(\Omega)}  &\lesssim \|\partial_t^3X\|_{L^2(\Omega)}( 1+\|K-1\|_{H^2(\Omega)}) + \|\partial_t^2X\|_{L^2(\Omega)}\|\partial_t K\|_{H^2(\Omega)} \\
&\qquad+ \|\partial_t X\|_{H^2(\Omega)}\|\partial_t^2 K\|_{L^2(\Omega)} +\|X\|_{H^2(\Omega)}\|\partial_t^3K\|_{L^2(\Omega)}\\
&\lesssim \|\partial_t^3\eta\|_{H^{1/2}(\Gamma)} + \|\partial_t\eta\|_{H^{5/2}(\Gamma)}\|\partial_t^2\eta\|_{H^{1/2}(\Gamma)} + \|\partial_t\eta\|_{H^{5/2}(\Gamma)}^3.
\end{split}
\]
Thus, the proof of \eqref{CoefEst_23} is complete.

Note that
\[
\|\partial_t^l (\cA-\text{Id})\|_{H^s(\Omega)} \leq \|\partial_t^l (K-1)\|_{H^s(\Omega)} + \|\partial_t^l(AK)\|_{H^s(\Omega)} + \|\partial_t^l(BK)\|_{H^s(\Omega)}.
\]
Hence,   \eqref{CoefEst_24} follows from \eqref{CoefEst_21}, \eqref{CoefEst_22} and  \eqref{CoefEst_23}.
\end{proof}

\section{Proof of Lemma \ref{LemFormulaD_k}}\label{AppFormulaD_k}
Note that the quotient 
\begin{equation}\label{Quotient}
\frac{2k^2(  \phi'(0)\phi(0)- \phi'(-a)\phi(-a))}{\int_{-a}^0 ((\phi'')^2 +2k^2(\phi')^2+k^4\phi^2)}
\end{equation}
is bounded because of the embedding $H^2((-a,0)) \hookrightarrow C^1((-a,0))$.
To prove Lemma \ref{LemFormulaD_k}, let us consider the Lagrangian functional 
\[
\cL_k(\phi,\beta) =  \beta\Big( \int_{-a}^0((\phi'')^2 +2k^2(\phi')^2+k^4\phi^2)-1\Big)-2k^2(\phi'(0)\phi(0)- \phi'(-a)\phi(-a)),
\]
for any $\phi \in H^2((-a,0))$. Using Lagrange multiplier theorem, we find that the extrema of the quotient \eqref{Quotient} are necessarily the stationary points  $(\phi_k,\beta_k)$ of $\cL_k$, which 
 satisfy  
\begin{equation}\label{Psi_kConstraintAt1}
 \int_{-a}^0((\phi_k'')^2 +2k^2(\phi_k')^2+k^4\phi_k^2)=1
\end{equation}
and
\begin{equation}\label{EqConstraintPsi_k}
\begin{split}
&\beta_k \int_{-a}^0  (\phi_k''\theta''+2k^2 \phi_k'\theta'+k^4 \phi_k \theta)\\  &=k^2(\phi_k'(0)\theta(0)+\phi_k(0)\theta'(0) -\phi_k'(-a)\theta(-a) -\phi_k(-a)\theta'(-a) ).
\end{split}
\end{equation}
for all $\theta\in H^2((-a,0))$.   Taking the integration by parts, we obtain  that 
\begin{equation}\label{1Eq_LemFormulaD_k}
\begin{split}
&\beta_k \int_{-a}^0 (\phi_k^{(4)}-2k^2 \phi_k''+k^4\phi_k )\theta  +\beta_k( \phi_k''\theta' - \phi_k''' \theta + 2k^2 \phi_k'\theta)\Big|_{-a}^0 \\
&\qquad=k^2(\phi_k'(0)\theta(0)+\phi_k(0)\theta'(0) - \phi_k'(-a)\theta(-a) -\phi_k(-a)\theta'(-a) ).
\end{split}
\end{equation}
Restricting $\theta \in C_0^{\infty}((-a,0))$, the resulting equality yields
\begin{equation}\label{EqPhi_kLemma32}
 \phi_k^{(4)}-2k^2 \phi_k''+k^4\phi_k =0 \quad\text{on } (-a,0).
\end{equation}
Hence, \eqref{1Eq_LemFormulaD_k} tells us  that 
\begin{equation}\label{2Eq_LemFormulaD_k}
\begin{cases}
\beta_k \phi_k''(0) = k^2\phi_k(0), \\
\beta_k(-\phi_k'''(0)+2k^2\phi_k'(0)) =k^2\phi_k'(0),\\
\beta_k\phi_k''(-a)= k^2\phi_k(-a),\\
\beta_k(-\phi_k'''(-a)+2k^2 \phi_k'(-a)) =k^2\phi_k'(-a).
\end{cases}
\end{equation}
Any solution $\phi_k$ of \eqref{EqPhi_kLemma32} is of the form 
\begin{equation}\label{FormPhi_k}
\phi_k(x_3) =(Ax_3+B) \sinh(kx_3) +(Cx_3+D)\cosh(kx_3),
\end{equation}
with $A,B,C,D$ are four constants such that $A^2+B^2+C^2+D^2>0$. Let us  compute from \eqref{FormPhi_k} that
\[
\phi_k'(x_3)= (A+kD+kCx_3) \sinh(kx_3) + (C+kB+kAx_3) \cosh(kx_3),
\]
\[
\phi_k''(x_3) = (2kC+k^2B+ k^2Ax_3) \sinh(kx_3) + (2kA+ k^2D +k^2 Cx_3) \cosh(kx_3).
\]
and
\[
\phi_k'''(x_3) = (3k^2A +k^3D+ k^3Cx_3)\sinh(kx_3) + (3k^2C+k^3B+k^3Ax_3) \cosh(kx_3).
\]
Substituting these formulas into \eqref{2Eq_LemFormulaD_k}, we obtain  
\begin{equation}\label{SystCB}
\begin{cases}
\beta_k(2kA +k^2D )=k^2D, \\
\beta_k(-k^2 C+k^3B)= k^2(C+kB),\\
\end{cases}
\end{equation}
and  
\begin{equation}\label{SystAD}
\begin{cases}
\beta_k\Big( - (2kC+k^2(B-Aa))\sinh(ka)+ (2kA+k^2(D-Ca))\cosh(ka) \Big) \\
\quad\quad =k^2(-(B-Aa)\sinh(ka)+(D-Ca)\cosh(ka) ),\\
\beta_k \Big( -(3k^2A+k^3(D-Ca))\sinh(ka)+ (3k^2C+k^3(B-Aa))\cosh(ka) \Big)\\
\quad\quad = k^2(2\beta_k-1) \Big( -(A+k(D-Ca))\sinh(ka)+ (C+k(B-Aa))\cosh(ka) \Big).\\
\end{cases}
\end{equation}

System \eqref{SystCB} is equivalent to 
\begin{equation}\label{SystCB_1}
\begin{cases}
k (\beta_k-1)B=(\beta_k+1)C,\\
 k(\beta_k-1)D =-2\beta_kA.
\end{cases}
\end{equation}
We also obtain that  \eqref{SystAD} is equivalent to 
\begin{equation}\label{SystAD_1}
\begin{cases}
\Big((-\beta_k(ka\sinh(ka)+2\cosh(ka))+ka\sinh(ka) \Big) A + (\beta_k-1)k\sinh(ka) B \\
\quad +\Big( (2\sinh(ka)+ka\cosh(ka))\beta_k-ka\cosh(ka)\Big)C + (-\beta_k+1)k\cosh(ka)D =0, \\
\Big(-(\beta_k+1)\sinh(ka) + (\beta_k-1)ka \cosh(ka) \Big) A + (-\beta_k+1)k\cosh(ka)B \\
\quad +\Big( (-\beta_k+1)ka\sinh(ka) + (\beta_k+1)\cosh(ka) \Big) C +(\beta_k-1)k\sinh(ka) D=0.
\end{cases}
\end{equation}
Substituting \eqref{SystCB_1} into \eqref{SystAD_1}, we deduce
\[
\begin{cases}
ka \tanh(ka)(-\beta_k+1)A + ((3\beta_k+1)\tanh(ka)+ka (\beta_k-1))C=0,\\
(-(3\beta_k+1)\tanh(ka) + ka(\beta_k-1))A + (-\beta_k+1)ka \tanh(ka)C=0.
\end{cases}
\]
Hence, system \eqref{2Eq_LemFormulaD_k} is equivalent to 
\begin{equation}\label{SystABCD_1}
\begin{cases}
(\beta_k+1)C - k (\beta_k-1)B=0,\\
2\beta_kA +k(\beta_k-1)D =0,\\
ka \tanh(ka)(-\beta_k+1)A + (\tanh(ka)(3\beta_k+1)+ka (\beta_k-1))C=0,\\
(-\tanh(ka)(3\beta_k+1) + ka(\beta_k-1))A +ka \tanh(ka) (-\beta_k+1) C=0.
\end{cases}
\end{equation}
System \eqref{SystABCD_1} admits a nontrivial solution $(A,C,B,D)$ if and only if the determinant of the corresponding matrix is equal to zero. This yields 
\[
k^2(\beta_k-1)^2 \Big( (ka)^2\tanh^2(ka) (\beta_k-1)^2 -\Big( (ka)^2(\beta_k-1)^2 -\tanh^2(ka)(3\beta_k+1)^2\Big)\Big)=0.
\]
Equivalently,
\begin{equation}\label{DetBeta_k}
k^2(\beta_k-1)^2 ( (ka)^2(\beta_k-1)^2 - \sinh^2(ka)(3\beta_k+1)^2 )=0.
\end{equation}
We have three possible values of $\beta_k$, which are solutions of \eqref{DetBeta_k} and  ordered as
\[
1 \text{ (multiplicity 2)}>  - \frac{\sinh(ka)-ka}{3\sinh(ka)+ka}>- \frac{\sinh(ka)+ ka}{3\sinh(ka)-ka}.
\]

Let us take the maximal value $\beta_k=1$. Clearly, we obtain  $A=C=0$ from \eqref{SystCB_1} and 
\[
\phi_k(x_3) = B\sinh(kx_3)+D\cosh(kx_3).
\]
Substituting the above $\phi_k$ into \eqref{Psi_kConstraintAt1}, we have
\[
\begin{split}
\int_{-a}^0 (B\sinh(kx_3)+D\cosh(kx_3))^2 + \int_{-a}^0 (D\sinh(kx_3)+B\cosh(kx_3))^2 =\frac1{2k^4}.
\end{split}
\]
Equivalently, 
\[
\frac12\sinh(2ka)  (B^2+D^2)- 2\sinh^2(ka) BD =\frac1{2k^3}.
\]
This yields  
\begin{equation}\label{EqBD_1}
\begin{cases}
D \text{ is arbitrary and }, \\
B= \frac{ 2\sinh^2(ka) \pm \sqrt{\sinh^2(ka)(2\cosh^2(ka)+\cosh(2ka)) D^2 + \frac1{k^3}\sinh(2ka)}}{2\sinh(2ka)}.
\end{cases}
\end{equation}

Let us consider the minimal value $\beta_k= - \frac{\sinh(ka)+ka}{3\sinh(ka)-ka}$. It can be seen from \eqref{SystABCD_1} that 
\begin{equation}\label{FindD}
\begin{split}
D &=- \frac{\sinh(ka)+ka}{2k\sinh(ka)}A, \\
C &= \frac{\cosh(ka)-1}{\sinh(ka)} A, \\
B &= -\frac{(\sinh(ka)-ka)(\cosh(ka)-1)}{2k\sinh^2(ka)}A.
\end{split}
\end{equation}
Hence, $\phi_k(x_3) = A z_k(x_3)$, where 
\[\begin{split}
z_k(x_3) &= \Big( x_3-\frac{(\sinh(ka)-ka)(\cosh(ka)-1)}{2k\sinh^2(ka)} \Big)\sinh(kx_3)\\
&\qquad + \Big( \frac{\cosh(ka)-1}{\sinh(ka)}x_3  - \frac{\sinh(ka)+ka}{2k\sinh(ka)}\Big)\cosh(kx_3).
\end{split}\]
To find $A$, we  trace back to \eqref{Psi_kConstraintAt1}. That means 
\begin{equation}\label{FindA}
A^2 \int_{-a}^0 ( (z_k'')^2 +2k^2(z_k')^2+k^4z_k^2)=1.
\end{equation}

From the  above cases, we conclude that 
\begin{itemize}
\item 
\[
\max_{\phi \in H^2((-a,0))} \frac{2k^2(\phi'(0)\phi(0)-\phi'(-a)\phi(-a))}{ \int_{-a}^0((\phi'')^2 +2k^2(\phi')^2+k^4\phi^2)} = 1.
\]
That  variational problem is attained by functions
\[
\phi(x_3) = B\sinh(kx_3)+D \cosh(kx_3),
\]
where $B,D$ satisfy \eqref{EqBD_1}.

\item 
\[
\min_{\phi \in H^2((-a,0))} \frac{2k^2(\phi'(0)\phi(0)-\phi'(-a)\phi(-a))}{ \int_{-a}^0( (\phi'')^2 +2k^2(\phi')^2+k^4\phi^2)}  =- \frac{\sinh(ka)+ka}{3\sinh(ka)-ka}.
\]
That  variational problem is attained by functions
\[
\phi(x_3) = (Ax_3+B)\sinh(kx_3)+(Cx_3+D)\cosh(kx_3),
\]
where $A,B,C,D$ satisfy \eqref{FindA} and \eqref{FindD}.
\end{itemize}

\section{Proof of Lemma \ref{LemVariationR}}\label{AppVariationR}
Let us consider the Lagrangian functional 
\[
\cL(\psi,\alpha) =  \alpha\Big( \int_{\R_-}((\psi''+k^2\psi)^2 +4k^2(\psi')^2)-1\Big)-k^2 (\psi(0))^2,
\]
for any $\psi \in H^2(\R_-)$. Using Lagrange multiplier theorem again, we find that the extrema of the quotient 
\[
\frac{k^2 (\psi(0))^2}{\int_{\R_-} ((\psi''+k^2\psi)^2 +4k^2(\psi')^2)}
\]
 are necessarily the stationary points  $(\psi_k,\alpha_k)$ of $\cL$, which 
 satisfy  
\[
 \int_{\R_-}((\psi_k''+k^2\psi_k)^2 +4k^2(\psi_k')^2)=1
\]
and for all $\theta\in H^2(\R_-)$,
\[
\begin{split}
\alpha_k \int_{\R_-}  (\psi_k''+k^2\psi_k) (\theta''+k^2\theta)+4k^2 \psi_k'\theta')   =k^2\psi_k(0)\theta(0).
\end{split}
\]
Taking the integration by parts, we obtain  that 
\begin{equation}\label{1Eq_LemAlpha_k}
\begin{split}
&\alpha_k \int_{\R_-} (\psi_k^{(4)}-2k^2 \psi_k''+k^4\psi_k )\theta  +\alpha_k(( \psi_k''+k^2\psi_k)\theta' - (\psi_k''' -3k^2 \psi_k')\theta)\Big|_{\R_-}\\ &=k^2\psi_k(0)\theta(0).
\end{split}
\end{equation}
Restricting $\theta \in C_0^{\infty}(\R_-)$, \eqref{1Eq_LemAlpha_k} yields that $\psi_k$ is a bounded solution of  
\[
 \psi_k^{(4)}-2k^2 \psi_k''+k^4\psi_k =0
 \] on $\R_-$ with the boundary conditions 
\begin{equation}\label{3EqPsi_k}
\begin{cases}
\alpha_k( \psi_k''(0)+k^2\psi_k(0) )=0, \\
\alpha_k(-\psi_k'''(0)+3k^2\psi_k'(0)) =k^2\psi_k(0).
\end{cases}
\end{equation}
We can see that $\psi_k$ is of the form 
\begin{equation}\label{FormPsi_k}
\psi_k(x_3) =(Ax_3+B) e^{kx_3} \quad(A^2+B^2>0).
\end{equation}
Substituting \eqref{FormPsi_k} into \eqref{3EqPsi_k}, we have
\begin{equation}\label{4EqPsi_k}
\begin{cases}
kB+A=0, \\
\alpha_k\big( -(k^3B+3k^2A)+3k^2(kB+A)\big) =k^2B.
\end{cases}
\end{equation}
Solving \eqref{4EqPsi_k}, we have $kB+A=0$ and $\alpha_k= \frac1{2k}$. It yields $\psi_k(x_3)= Cz_k(x_3)$, where 
\[
z_k(x_3)=(kx_3-1)e^{kx_3}
\]
 and $C$ satisfies
\[
C^2 \int_{\R_-}((z_k''+k^2z_k)^2+4k^2(z_k')^2)=1.
\]
We thus deduce \eqref{VariationTheta}.
 \end{document}